\documentclass[ejsv2,noshowframe]{imsart}

\usepackage{cancel}
\usepackage{threeparttable}
\usepackage{tabularx}
\usepackage{booktabs}
\usepackage{multirow}
\usepackage{enumitem}
\RequirePackage[numbers]{natbib}
\RequirePackage[colorlinks,citecolor=blue,urlcolor=blue]{hyperref}
\RequirePackage{graphicx}
\usepackage{float}
\usepackage{placeins}
\usepackage[ruled,vlined,linesnumbered]{algorithm2e}
\makeatletter
\renewcommand{\fps@algocf}{H}
\makeatother
\SetKwInOut{KwIn}{Input}
\SetKwInOut{KwOut}{Output}
\SetKw{KwRet}{Return}
\SetAlFnt{\footnotesize}
\SetAlCapFnt{\footnotesize}
\SetAlCapNameFnt{\footnotesize}

\startlocaldefs

\startlocaldefs
\theoremstyle{plain}
\newtheorem{theorem}{Theorem}[section]
\newtheorem{proposition}[theorem]{Proposition}
\newtheorem{lemma}[theorem]{Lemma}
\newtheorem{corollary}[theorem]{Corollary}

\theoremstyle{definition}

\theoremstyle{remark}

\DeclareMathOperator{\Var}{Var}
\DeclareMathOperator{\Cov}{Cov}

\newcommand{\E}{\mathbb{E}}

\renewcommand{\Pr}{\mathbb{P}}
\newcommand{\R}{\mathbb{R}}

\newcommand{\op}{\mathrm{op}}

\DeclareMathOperator{\tr}{tr}

\endlocaldefs

\endlocaldefs

\begin{document}
\begin{frontmatter}
\title{Statistical Optimality of Prediction-Powered Inference}\thankstext{t1}{This manuscript is a preprint and is currently under review.}
\runtitle{Statistical Optimality of Prediction-Powered Inference}

\begin{aug}
\author[A]{\fnms{Se Yoon}~\snm{Lee}\ead[label=e1]{seyoonlee.stat.math@gmail.com}}
\and
\author[B]{\fnms{Jae Kwang}~\snm{Kim}\ead[label=e2]{jkim@iastate.edu}}


\address[A]{Department of Statistics,
Texas A\&M University\printead[presep={,\ }]{e1}}

\address[B]{Department of Statistics,
Iowa State University\printead[presep={,\ }]{e2}}

\runauthor{Lee and Kim}
\end{aug}

\begin{abstract}
The prediction-powered inference (PPI) proposed
by Angelopoulos et al. (2023) is a popular method that leverages a small number of labeled samples and machine learning predictions for semi-supervised inference. While several variants of PPI have appeared in the literature, its rigorous statistical theory has not been fully developed. In this paper, we study the statistical optimality of PPI. Our contributions span both foundational theory and new methodology. First, we frame PPI as an $M$-estimation problem, revealing a link between the bias-corrected PPI estimating equation and the ideal full-data estimating equation. This connection leads to the consistency and asymptotic normality of the PPI estimator under simple random sampling without replacement. Next, we identify the efficient influence function and prove that PPI can attain the semiparametric efficiency lower bound when the predictor is score-calibrated, that is, when the predictor's output aligns with the true conditional expectation of the estimating function. Finally, for learned prediction rules, we develop asymptotic theory for cross-fitting and for a single-fit variant with variance correction in the special case of semiparametric mean estimation. Simulation experiments and a real-data application support these findings.
\end{abstract}

\begin{keyword}[class=MSC]
\kwd[Primary ]{62G20}
\kwd{62E20}
\kwd[; secondary ]{62G05}
\end{keyword}

\begin{keyword}
\kwd{Prediction-powered inference}
\kwd{Semi-supervised inference}
\kwd{Semiparametric efficiency theory}
\end{keyword}

\end{frontmatter}

\setcounter{tocdepth}{2}
\tableofcontents

\section{Introduction}


Prediction-powered inference (PPI) \citep{angelopous2023} is a semi-supervised inference technique \citep{van2020survey,zhu2009introduction,hady2013semi}: it uses a predictive model to impute missing responses for units observed only with covariates, and then corrects the resulting estimating equation using a comparatively small labeled subset. The estimator is defined as the root of a bias-corrected (or ``rectified'') estimating equation. PPI is well suited to settings with abundant covariates but scarce labels, where flexible machine learning (ML) predictors—e.g., deep learning, gradient-boosted trees, random forests, or BART \citep{LeCun2015,Breiman2001,Friedman2001,Chipman2010}—can be trained at modest cost.


While Angelopoulos et al. \cite{angelopous2023} established several important properties of PPI, their analysis primarily focuses on confidence interval validity under specific constructions (see Theorem~S1 in the supplement) and does not address its foundational statistical properties. Subsequent extensions—including PPI++ \citep{angelopoulos2024}, stratified PPI \citep{fisch2024stratified}, cross-fitted PPI \citep{zrnic2024cross}, and related variants \citep{fisch2024stratified,gu2024local,einbinder2025semi,luo2024federated,lee2026mec}—largely emphasize methodological refinements and empirical performance.

As a result, several fundamental questions remain unresolved. In particular, it is unclear whether PPI is statistically optimal, whether it admits a semiparametric efficiency characterization, or whether valid inference can be achieved under learned prediction rules without sample splitting. In this paper, we address this research gap under a superpopulation framework \citep{isaki1982} and rigorously investigate the statistical optimality of PPI. 

Our contributions span both foundational theory and new methodology:

\begin{itemize}
   \item[] \textbf{I. General moment-equation framework.}
We develop a statistical optimality framework for PPI under a general
moment-equation formulation and characterize its semiparametric efficiency.
  \item[] \textbf{II. Semiparametric mean estimation.}
We establish asymptotic theory for cross-fitted PPI with sample splitting, as well as for a single-fit variant with variance correction. The latter constitutes a new variant of PPI.
\end{itemize}

First (\textbf{I}), under a general moment-equation framework, we cast PPI as an $M$-estimation problem and establish an oracle bridge linking the computable bias-corrected score to the full-data moment. Under mild regularity conditions, we prove consistency and an asymptotically linear expansion for PPI, enabling valid confidence-interval construction for fixed prediction rules. We also identify the efficient influence function and give conditions under which PPI achieves the semiparametric efficiency bound.

Second (\textbf{II}), in the semiparametric mean setting, we study the statistical effect of label reuse, using the same labeled observations for both predictor training and evaluation, in the original PPI. We address this with two remedies: sample splitting and variance correction. We prove consistency and asymptotic normality for both under a learned prediction rule. For the former, we split the labeled sample into folds, train on the complement, and evaluate out-of-fold. For the latter, a linear-smoother degrees-of-freedom adjustment calibrates the plug-in variance with a single fit.

The article is organized as follows. In Section \ref{sec:setup} we introduce the problem setup and PPI methodology. Section \ref{sec:ASYMPTOTIC THEORY: A FIXED PREDICTION RULE} presents the asymptotic theory for a fixed predictor, Section \ref{sec:semiparametric-efficiency} discusses semiparametric efficiency, and Section \ref{sec:ASYMPTOTIC THEORY: A LEARNED PREDICTION RULE} addresses the learned predictor case, including cross-fitting and variance correction. Section~\ref{sec:Simulation experiments} presents the simulation experiments, Section~\ref{sec:Real data applications} presents the real-data application, and Section~\ref{sec:Conclusion} concludes the paper. The proofs of the main results are presented in the Appendix.

\section{Background}
\label{sec:setup}
\subsection{Setup and notation}\label{subsec:Setup and notaiton}
Let $\mathcal{F}_N=\{(X_i,Y_i)\}_{i=1}^N$ denote a finite population, viewed as a realization of $N$ i.i.d.\ draws from a super-population distribution $P_0$; that is, $(X_i,Y_i)\stackrel{\mathrm{i.i.d.}}{\sim} P_0$. 
We observe covariates $X_i\in\mathcal{X}\subset\mathbb{R}^d$ for all $i$, while responses $Y_i\in\mathcal{Y}\subset\mathbb{R}$ are observed only for a labeled subset $S\subset\{1,\ldots,N\}$ of size $n$. 
Let $\delta_i:=\mathbf{1}\{i\in S\}$ denote the inclusion indicator. 
Assume $S$ is drawn by simple random sampling without replacement (SRSWOR), selecting $n$ of the $N$ units. 
Let $f_N:=n/N$ denote the labeling fraction; throughout, asymptotics are taken with $N\to\infty$, $n=n(N)\to\infty$, and $n/N\to f\in(0,1)$. 
For notational convenience we henceforth write $f$ for $f_N$. 
For design-based statements we condition on $\mathcal{F}_N$; for large-sample statements we appeal to this super-population model. 



The inferential target is defined through a super-population moment condition. 
Let $\Theta\subset\mathbb{R}^p$ be open, and let $U:\Theta\times\mathcal{X}\times\mathcal{Y}\to\mathbb{R}^p$ be a measurable estimating function. 
The super-population parameter $\theta_0\in\Theta$ is defined as the unique solution to
\begin{equation}
\label{eq:pop-moment}
\mathbb{E}_{P_0}\!\big[U(\theta_0;X,Y)\big]
=\int_{\mathcal X\times\mathcal Y} U(\theta_0;x,y)\, dP_0(x,y)=0,
\end{equation}
and the finite-population analogue $\theta_N\in\Theta$ is defined as the unique solution to the finite-population counterpart of \eqref{eq:pop-moment},
\begin{equation}
\label{eq:UN}
U_N(\theta):=\frac{1}{N}\sum_{i=1}^N U(\theta;X_i,Y_i) = 0.
\end{equation}

Because the responses $Y_i$ are missing for many units, $U_N(\theta)$ is not computable in practice. The classical approach therefore relies only on the labeled sample and solves the labeled estimating equation
\begin{equation}
\nonumber
U_S(\theta):=\frac{1}{n}\sum_{j\in S} U(\theta;X_j,Y_j)=0,
\end{equation}
thereby discarding the information contained in the unlabeled covariates $\{X_j\}_{j\notin S}$.


The central goal of semi-supervised inference, including PPI, is to improve efficiency by exploiting unlabeled covariates while maintaining valid statistical inference (e.g., nominal \(95\%\) confidence-interval coverage, consistency of solutions to the relevant estimating equations, and asymptotic normality).

We note that the solution of the moment equation \eqref{eq:pop-moment} can also be characterized as the minimizer of a convex risk. Precisely, suppose there exists a measurable loss $L:\Theta\times\mathcal{X}\times\mathcal{Y}\to\mathbb{R}$ such that, for each $(x,y)$, the map $\theta\mapsto L(\theta;x,y)$ is convex and integrable, and\[U(\theta;x,y)\in \partial_\theta L(\theta;x,y)\qquad\text{for all }\theta\in\Theta,\]where $\partial_\theta L$ denotes the subdifferential with respect to $\theta$. Under standard interchange conditions for expectation and subgradient, it holds $0\in \partial_\theta \,\mathbb{E}\!\left[L(\theta;X,Y)\right]= \mathbb{E}\!\left[\partial_\theta L(\theta;X,Y)\right]= \mathbb{E}\!\left[U(\theta;X,Y)\right],$ so any $\theta_0$ solving \eqref{eq:pop-moment} also satisfies\[\theta_0 \in \arg\min_{\theta\in\Theta}\, \mathbb{E}\!\left[L(\theta;X,Y)\right].\] 

Conversely, if $\theta\mapsto \mathbb{E}[L(\theta;X,Y)]$ is strictly convex, its unique minimizer $\theta_0$ obeys $\mathbb{E}[U(\theta_0;X,Y)] $ $=0$. 



\subsection{Prediction-powered inference}\label{subsec:prediction-powered-inference}
PPI \citep{angelopous2023} employs the model-based estimating function
\begin{equation}
\label{eq:utilde}
\widetilde U(\theta;X) := U\!\big(\theta; X, m(X)\big),
\end{equation}
where $m(\cdot):\mathcal X\to\mathbb R$ is a predictive model. 


Note that $\widetilde U(\theta;X)$ in \eqref{eq:utilde} is computable for both labeled and unlabeled units once a predictor for $m$ is available. A distinctive feature of PPI is the systematic use of \eqref{eq:utilde} in semi-supervised settings—an aspect not typically addressed by standard supervised-learning methods \citep{zhu2009introduction,sohn2020fixmatch,song2024general}.


In what follows, we derive the PPI score; see \cite{angelopous2023} for a heuristic derivation.

Fix $\theta$ and write
\[
U_i(\theta):=U(\theta;X_i,Y_i),\quad 
\widetilde U_i(\theta):=U\big(\theta;X_i,m(X_i)\big).
\]

To develop a bias-calibration method, one can consider  the following intercept-only calibration  model
\begin{equation}
\nonumber
U_i(\theta)\;=\;\beta_0(\theta)\;+\;\widetilde U_i(\theta)\;+\;e_i,
\end{equation}
where $e_i$ is a zero-mean error term uncorrelated with $X_i$. The parameter $\beta_0(\theta)$ is a nuisance quantity capturing the average discrepancy between the model-based term and the observed estimating function; 
\[
\mathbb{E}\!\left\{ U(\theta;X,Y)-U\big(\theta;X,m(X)\big)\right\}\;=\;\beta_0(\theta).
\]

Estimate $\beta_0(\theta)$ on the labeled set $S$ by ordinary least squares, $Q(\beta_0):=\sum_{j\in S}\big\|U_j(\theta)-\beta_0-\widetilde U_j(\theta)\big \|_2^2$, which yields the closed-form solution
\begin{equation}
\label{eq:rectifier}
\widehat\beta_0(\theta)\;=\;\frac{1}{n}\sum_{j\in S}\Big\{U_j(\theta)-\widetilde U_j(\theta)\Big\}
\;=:\;\Delta_\theta,
\end{equation}
where  $\Delta_\theta$ is the average residual on labeled data, which serves as a bias correction term (the ‘rectifier’).  Intuitively, $\Delta_\theta$ estimates the bias in the model-based estimating function $\widetilde{U}(\theta;X)$, thereby correcting $\widetilde{U}(\theta;X)$ to better match the true full-data estimating function on average.


For any unit $i$--labeled or unlabeled--the debiased predictor of $U_i(\theta)$ is
\begin{equation}
\label{eq:Ui-debiased}
\widehat U_i(\theta)\;=\;\widetilde U_i(\theta)\;+\;\widehat\beta_0(\theta)
\;=\;\widetilde U_i(\theta)\;+\;\Delta_\theta.
\end{equation}

Averaging \eqref{eq:Ui-debiased} over the $N$ units gives the prediction-powered score
\begin{align}    
\widehat U_{\mathrm{PPI}}(\theta)
\;:= &\;\frac{1}{N}\sum_{i=1}^N \widehat U_i(\theta) 
\;=\;\underbrace{\frac{1}{N}\sum_{i=1}^N \widetilde U_i(\theta)}_{\text{measure of fit $m_\theta$} }
\;+\;\underbrace{\frac{1}{n}\sum_{j\in S}\Big\{U_j(\theta)-\widetilde U_j(\theta)\Big\}}_{\text{rectifier $\Delta_\theta$}},
\label{eq:PPI-score-derivation}
\end{align}
which coincides with the bias-corrected estimating equation used by PPI \citep{angelopous2023}. The PPI point estimator is then defined as any root of \eqref{eq:PPI-score-derivation}, i.e., $\hat\theta_{\mathrm{PPI}}\in\{\theta:\widehat U_{\mathrm{PPI}}(\theta)=0\}$. 

This derivation makes it explicit that PPI is a model-assisted, bias-corrected plug-in procedure with

\begin{itemize}
    \item[] \textbf{$\bullet$ Measure of model fit}: $m_\theta$ uses the prediction rule $m$ to impute outcomes for the unlabeled covariates. It captures how well the model explains the mean structure of the estimating function;
    \item[] \textbf{$\bullet$ Rectifier}: \(\Delta_\theta\) uses the labeled residuals to correct potential bias in the imputation by aligning with the observed labeled responses. 
\end{itemize}


In the PPI score \(\widehat U_{\mathrm{PPI}}(\theta)=m_\theta+\Delta_\theta\), the term \(m_\theta:=N^{-1}\sum_{i=1}^N U\!\big(\theta;x_i,m(x_i)\big)\) is a model-based fit computed on the imputed (prediction-filled) population. It exploits all \(N\) covariate points and therefore concentrates quickly, but is generally biased for the infeasible full-data moment because \(m(x)\neq Y\). The rectifier \(\Delta_\theta:=n^{-1}\sum_{j\in S}\{U(\theta;x_j,Y_j)-U(\theta;x_j,m(x_j))\}\) estimates that bias using only the labeled residuals; under simple random labeling it is conditionally unbiased for the average discrepancy in the estimating equation. Consequently, \(m_\theta+\Delta_\theta\) is, given \(\mathcal F_N\), an unbiased proxy for \(U_N(\theta)=N^{-1}\sum_{i=1}^N U(\theta;x_i,Y_i)\), so the PPI root solves a bias-corrected score that bridges the computable, prediction-assisted procedure to the oracle full-data target. 

\paragraph{PPI++ score derivation.}Likewise, one can obtain a PPI++ score \cite{angelopoulos2024} by considering the intercept-slope calibration model \(U_i(\theta)=\beta_0(\theta)+\beta_1(\theta)\,\widetilde U_i(\theta)+e_i\) with mean-zero errors. Fitting this linear regression on the labeled set \(S\) by ordinary least squares yields the slope \(\hat\beta_1(\theta)\) and the intercept \(\hat\beta_0(\theta)=n^{-1}\sum_{j\in S}\{U_j(\theta)-\hat\beta_1(\theta)\,\widetilde U_j(\theta)\}\). The debiased predictor for any unit is then \(\widehat U_i(\theta)=\hat\beta_1(\theta)\,\widetilde U_i(\theta)+\hat\beta_0(\theta)\), and averaging over the \(N\) units gives the PPI++ score 
\begin{align*}
    \widehat U_{\mathrm{PPI++}}(\theta)=\frac{1}{N}\sum_{i=1}^N \hat\beta_1(\theta)\,\widetilde U_i(\theta)+\frac{1}{n}\sum_{j\in S}\{U_j(\theta)-\hat\beta_1(\theta)\,\widetilde U_j(\theta)\}.
\end{align*}
This is the estimating function used by the efficient PPI estimator. 

In this paper, we focus on the statistical properties of the PPI \citep{angelopous2023}; the derivation and most arguments may extend to PPI++ \cite{angelopoulos2024}, with minor modifications, to slope-adjusted variants. A full treatment is left for future work.

\section{Asymptotic theory: a fixed prediction rule}\label{sec:ASYMPTOTIC THEORY: A FIXED PREDICTION RULE}
\subsection{Consistency of the PPI estimator}
\label{subsec:ppi-consistency}
The starting point for the asymptotic analysis is to understand design-unbiasedness of the
PPI score. Under SRSWOR, for any fixed
prediction rule $m$ and all $\theta$, we can show that
\begin{equation}
\label{eq:ppi-unbiased}
\mathbb{E}\!\left[\widehat U_{\mathrm{PPI}}(\theta)\,\middle|\,\mathcal{F}_N\right]=U_N(\theta).
\end{equation}
The expectation in (\ref{eq:ppi-unbiased}) is with respect to the randomization distribution generated by the sampling mechanism for selecting sample $S$ from the finite population $\mathcal{F}_N$. A detailed proof (as well as proofs of other results) is provided in the Appendix.

Identity \eqref{eq:ppi-unbiased} shows that the bias-corrected PPI score is design-unbiased for the full-data estimating equation, effectively acting as an  \emph{oracle bridge} between the computable PPI
score $\widehat U_{\mathrm{PPI}}(\theta)$ \eqref{eq:PPI-score-derivation} and the infeasible full-data moment $U_N(\theta)$ \eqref{eq:UN}.
In particular, letting $\theta_N\in\Theta$ be the unique zero of $U_N(\theta) = 0$, we have
$\mathbb{E}[\widehat U_{\mathrm{PPI}}(\theta_N)\mid\mathcal{F}_N]=0$.

If, in addition, $\sup_{\theta\in\Theta}\big\|\widehat U_{\mathrm{PPI}}(\theta)-U_N(\theta)\big\|_2\xrightarrow{p}0$ holds and $I_N(\theta_N):=-\partial_\theta U_N(\theta_N)$ is nonsingular, then the $Z$-estimation continuity theorem \citep{NeweyMcFadden1994,vanDerVaart1998}
implies $\hat\theta_{\mathrm{PPI}} \xrightarrow{p} \theta_N.$

To pass from the finite-population root to the super-population target, assume that for every compact $K\subset\Theta$, $\sup_{\theta\in K}\big\|U_N(\theta)-\mathbb{E}\{U(\theta;X,Y)\}\big\|_2\xrightarrow{p}0.$ Then, since $\theta_0$ is the unique zero of $\mathbb{E}\{U(\theta;X,Y)\}$, it holds $\theta_N\xrightarrow{p}\theta_0$, and hence $\hat\theta_{\mathrm{PPI}}\xrightarrow{p}\theta_0.$

Therefore,  the PPI estimator \(\hat\theta_{\mathrm{PPI}}\) remains consistent without requiring the predictor \(m\) to be correctly specified, provided the regularity conditions (see Assumptions 1–2 in Subsection \ref{subsec:Preliminaries: Assumptions and Setup.} in the Appendix) hold, thanks to the design-unbiasedness in \eqref{eq:ppi-unbiased}. Indeed, misspecification of \(m\) primarily affects efficiency--the asymptotic variance \(\Var(\hat\theta_{\mathrm{PPI}})\).

One can also understand this advantage of PPI through a survey-sampling lens. Define \(\Delta_i(\theta):=U(\theta;X_i,Y_i)-U(\theta;X_i,m(X_i))\).
Then the rectifier \(\Delta_\theta\) \eqref{eq:rectifier} equals
$(1/N)\sum_{i=1}^N (\delta_i/f)$ $\Delta_i(\theta)$ with
\(f=n/N\) the inclusion probability under SRSWOR (see the Appendix for the proof); that is, it is the Horvitz--Thompson estimator of the finite-population mean of \(\Delta_i(\theta)\) \citep{horvitz1952}. This identity clarifies that misspecification of \(m\) does not induce design bias in the score; it only affects the variance through \(\Delta_i(\theta)\). In survey sampling, the PPI estimator is often called the difference estimator \citep{breidt17}.

\subsection{General $M$-estimation theory}
\label{eq:M-estimation theory}
We state our theoretical result on $M$-estimation with an arbitrary fixed predictor $m$ (possibly misspecified). The following theorem provides an asymptotic linear expansion of the PPI estimator $\hat\theta_{\mathrm{PPI}}$, thereby characterizing its asymptotic variance.

\begin{theorem}\label{thm:ppi-known-m-compact}
Assume the regularity conditions (in Assumptions~1–3 in Subsection \ref{subsec:Preliminaries: Assumptions and Setup.} in the Appendix) hold under SRSWOR. Then the PPI estimator $\hat\theta_{\mathrm{PPI}}$ has the first-order asymptotic
expansion
\begin{align}
&\hat\theta_{\mathrm{PPI}}-\theta_0=\frac{1}{N}\sum_{i=1}^N\phi_i+o_p(N^{-1/2}),  \label{eq:LAN_PPI_estimator_m_estimation}  
\end{align}
with the influence function for the $i$-th observation
\begin{align}
\nonumber
\phi_i &= \phi(X_i,Y_i,\delta_i;\,\theta_0, m) \\&:= I(\theta_0)^{-1}\!\Big\{ U(\theta_0; X_i, m(X_i))  + \frac{\delta_i}{f}\,\big[ U(\theta_0; X_i, Y_i) - U(\theta_0; X_i, m(X_i)) \big] \Big\},
\nonumber
\end{align}
where $\delta_i=\mathbb{I} \{i\in S\}$ and $f=n/N$. Thus, we have
\begin{align}
\label{eq:asymptotic_normal_distribution_of_PPI_m_estimation}
\sqrt{N}\,(\hat\theta_{\mathrm{PPI}}-\theta_0)\ \xrightarrow{d}\ \mathcal N(0,\Sigma_f),
\end{align}
where the asymptotic variance $\Sigma_f$ can be decomposed as
\begin{align}
&\Sigma_f  = V_1+\big(f^{-1}-1\big)V_2,
    \label{eq:Sigma_known_m} \\
&V_1  := \Var \!\big(I(\theta_0)^{-1}U(\theta_0;X,Y)\big) = \; I(\theta_0)^{-1}\,\mathbb{E}\!\big[\,U(\theta_0;X,Y)^{\otimes 2}\big]\,
         I(\theta_0)^{-1} \label{eq:V1_known_m}\\
&V_2 := I(\theta_0)^{-1}\,
      \mathbb{E}\!\big[\Delta(\theta_0;X,Y)^{\otimes 2}\big]\,
      I(\theta_0)^{-1}, \label{eq:V2_known_m}\\
      &\Delta(\theta_0;X,Y):= U(\theta_0;X,Y)-U(\theta_0;X,m(X)),\nonumber
\end{align}
where $I(\theta)=-  E\left\{ \partial_\theta 
U \left( \theta; X,Y \right) 
\right\}$ and 
\(a^{\otimes 2} := a a^\top\) denotes the outer product.
\end{theorem}

The covariance decomposition in \eqref{eq:Sigma_known_m} isolates the two sources of uncertainty. 
The term \(V_1\) in \eqref{eq:V1_known_m} is the \emph{oracle/population variability} that would remain even if all responses were observed; it does not depend on the labeling fraction \(f\).
The term \(V_2\) in \eqref{eq:V2_known_m} is the variance of the \emph{labeled residual}
\(\Delta(\theta_0;X,Y)\), and it is scaled by \(f^{-1}-1=N/n-1=(N-n)/n\) $(\approx N/n $ if $N \gg n$), which quantifies the scarcity of labels.
Consequently:
(i) when \(f\to1\) (many labels), the residual component vanishes and
\(\Sigma_f\to V_1\);
(ii) when \(f\ll1\) (few labels), the residual term dominates unless the predictor \(m\) is highly accurate (so \(V_2\) is small);
(iii)  If the predictor is exactly correct for the estimating function (we term this ‘score-perfect’), then the bias term $V_2$ is zero and the PPI estimator is as efficient as if we had no missing labels. 

\paragraph{Confidence interval for $\theta_0$.} The asymptotic expansion in \eqref{eq:LAN_PPI_estimator_m_estimation} together with the variance
decomposition \eqref{eq:Sigma_known_m} also guides practical inference for any given predictor
$m$. Using the labeled set $S$, one may compute a plug-in estimate of (\ref{eq:Sigma_known_m})
\begin{align*}
&\widehat I \;=\; -\frac1n\sum_{j\in S}\partial_\theta U(\hat\theta_{\mathrm{PPI}};X_j,Y_j),\\    
&\widehat V_1 \;=\; \widehat I^{-1}\!\Big(\frac1n\sum_{j\in S} U(\hat\theta_{\mathrm{PPI}};X_j,Y_j)^{\otimes2}\Big)\widehat I^{-1}\\
&\widehat V_2 \;=\; \widehat I^{-1}\!\Big(\frac1n\sum_{j\in S}\Delta_j(\hat\theta_{\mathrm{PPI}})^{\otimes2}\Big)\widehat I^{-1},\\
&\Delta_j(\theta):=U(\theta;X_j,Y_j)-U(\theta;X_j,m(X_j)).
\end{align*}
Plugging \(\widehat V_1\) and \(\widehat V_2\) into \(\widehat\Sigma_f=\widehat V_1+(f^{-1}-1)\widehat V_2\) yields $100(1-\alpha)\%$ Wald-type confidence intervals:
\[
C^{\mathrm{PPI}}_{\alpha,j}
=\Big(\;\hat\theta_{\mathrm{PPI},j}\ \pm\ z_{1-\alpha/2}\,\sqrt{(\widehat\Sigma_f)_{jj}/N}\Big),\;\; j\in[p],
\]
where $z_{1-\alpha/2}$ is the $(1-\alpha/2)$-quantile of the standard normal distribution, $(\widehat\Sigma_f)_{jj}$ denotes the $j$-th diagonal element of $\widehat\Sigma_f$, and $[p]=\{1,\ldots,p\}$ indexes the components of $\theta$. For $p=1$, the set reduces to the usual Wald interval.

Note that the validity of the PPI, that is, 
\[
\lim_{N, n(N) \rightarrow \infty}
\mathbb P\!\left(\,\theta_{0,j}\in C^{\mathrm{PPI}}_{\alpha,j}\right)\ \geq 1-\alpha
\quad\text{for each } j\in[p].
\] holds due to the asymptotic normality of the $\hat{\theta}_{\mathrm{PPI}}$ (\ref{eq:asymptotic_normal_distribution_of_PPI_m_estimation}).

\paragraph{Semi-supervised mean estimation.}
For the mean parameter $\theta_0=\mathbb{E}[Y] \in \mathbb{R}$, we have $U(\theta;X,Y)=Y-\theta$, $I(\theta_0)=1$, and $\Delta(\theta_0;X,Y)=Y-m(X)$, so $V_1=\Var\{Y-\theta_0\}$ and $V_2=\Var\{Y-m(X)\}$. By Theorem~\ref{thm:ppi-known-m-compact},
\[
\Sigma_f \;=\; \Var\{Y-\theta_0\} \;+\; (f^{-1}-1)\,\Var\{Y-m(X)\}.
\]
Using the law of total variance, $\Var\{Y-\theta_0\}=\Var\{m(X)\}+\Var\{Y-m(X)\}$, this is equivalently
\begin{align}
\label{eq:variance_of_mean_fixed_m}
\sigma_f^2 := \Sigma_f \;=\; \Var\{m(X)\} \;+\; f^{-1}\,\Var\{Y-m(X)\}.    
\end{align}
This matches the semiparametric efficiency lower bound under missing data \citep{robins1994estimation}. 

We write $\sigma_f^2$ because $\Sigma_f$ is a positive real number.
Thus, a $(1-\alpha)\%$ confidence interval for $\theta_0$ is
\begin{align*}
C^{\mathrm{PPI}}_{\alpha}
\;=\;
\Big(
\hat\theta_{\mathrm{PPI}}
\;\pm\;
z_{1-\alpha/2}\,
\sqrt{\,\widehat\sigma_f^2/N\,}
\Big),
\end{align*}
where \(\widehat \Var(\hat\theta_{\mathrm{PPI}}) \approx \widehat\Sigma_f/N=\hat\sigma_m^2/N+\hat\sigma_\Delta^2/n\) with
\begin{align*}
\hat\sigma_m^2 \;&=\; \frac{1}{N}\sum_{i=1}^N\bigg(m(X_i)- \frac{1}{N}\sum_{i=1}^N m(X_i)\bigg)^2,
\end{align*}
computed from the unlabeled covariates \(\{X_i\}_{i=1}^N\), and
\begin{align*}
\hat\sigma_\Delta^2 \;=\; \frac{1}{n}\sum_{j\in S}\bigg(Y_j-m(X_j)-\frac{1}{n}\sum_{j\in S}\big\{Y_j-m(X_j)\big\}\bigg)^2    
\end{align*}
computed from the labeled pairs \(\{(X_j,Y_j)\}_{j\in S}\). Note that the confidence interval 
$C^{\mathrm{PPI}}_{\alpha}$ coincides with the heuristic formula derived in \citep{angelopous2023}.

\section{Semiparametric efficiency theory}\label{sec:semiparametric-efficiency}
The statistical theory in Section~\ref{sec:ASYMPTOTIC THEORY: A FIXED PREDICTION RULE} is valid for any choice of prediction function $m(x)$. We now ask: is the PPI estimator statistically optimal among all regular estimators in this missing-data setting? To answer this, we invoke semiparametric theory \citep{bickel1998,davidian2022methods,tsiatis2006}: the parameter of interest is finite-dimensional, $\theta\in\Theta\subset\mathbb{R}^p$, while the data law is otherwise unrestricted. Under known labeling fraction, the observed-data law factorizes as $P(dx, d\delta, dy)
= P_X(dx)\, P_\delta(d\delta)\, P_{Y\mid X}(dy \mid x),$
where $P_\delta = \mathrm{Bernoulli}(f)$ is independent of $(X,Y)$, so the nuisance parameter is the infinite-dimensional pair $\eta(P)=(P_X,P_{Y\mid X})$. The target $\theta(P)$ is defined by the full-data moment restriction $\mathbb{E}_P[U(\theta;X,Y)]=0$. 


The following theorem holds:
\begin{theorem}\label{thm:semiparametric-main}
Consider observed data $O=(X,\delta,\delta Y)$ from SRSWOR with known labeling fraction $\Pr(\delta=1)=f\in(0,1)$. Let $U(\theta;x,y)\in\mathbb{R}^p$ be a full-data estimating function with target $\theta_0\in\Theta$ defined by $\mathbb{E}\{U(\theta_0;X,Y)\}=0$, and let $I(\theta_0):=-\,\mathbb{E}\{\partial_\theta U(\theta_0;X,Y)\}$ be nonsingular. Assume the standard identification, smoothness, moment, and regularity conditions stated in the Appendix.

\noindent
Then:

\noindent\textup{\textbf{(i) Efficient influence function (EIF).}}
The parameter $\theta(P)$ defined implicitly by the observed data moment equation is pathwise differentiable at $P_0$ with efficient influence function
\begin{align*}
\phi^{\mathrm{eff}}(O) &= \phi^{\mathrm{eff}}(X,Y,\delta;\theta_0)\\
&= I(\theta_0)^{-1}\!\Big\{
\overline U(\theta_0;X;P_0)+\frac{\delta}{f}\Big(U(\theta_0;X,Y)-\overline U(\theta_0;X;P_0)\Big)
\Big\},
\end{align*}
where $\overline U(\theta;X;P_0):=\mathbb{E}\!\big[U(\theta;X,Y)\mid X\big]$.

\noindent\textup{\textbf{(ii) Linear expansion and asymptotic normality.}}
The PPI estimator $\widehat\theta_{\mathrm{PPI}}$ has the first-order asymptotic linear expansion with efficient influence function $\phi^{\mathrm{eff}}$,
\begin{align*}
&\widehat\theta_{\mathrm{PPI}}-\theta_0
= \frac{1}{N}\sum_{i=1}^N \phi^{\mathrm{eff}}(O_i)\;+\;o_p(N^{-1/2}),
\\
&\sqrt{N}\,\big(\widehat\theta_{\mathrm{PPI}}-\theta_0\big)
\;\xrightarrow{d}\;
\mathcal N\!\big(0,\ \Var_{P_0}(\phi^{\mathrm{eff}}(O))\big),
\end{align*}
with $\Var_{P_0}(\phi^{\mathrm{eff}}(O))
= I(\theta_0)^{-1}\,\Sigma\,I(\theta_0)^{-1}$ where
\begin{footnotesize}
\begin{align*}
&\Sigma
= \Var_{P_0}\!\big(\overline U(\theta_0;X;P_0)\big)
+\frac{1}{f}\,\mathbb{E}_{P_0}\!\Big[\Var\!\big(U(\theta_0;X,Y)\mid X\big)\Big].    
\end{align*}
\end{footnotesize}
\end{theorem}

\paragraph{Semiparametric efficiency lower bound.}
Theorem~\ref{thm:semiparametric-main} implies that, for any regular
asymptotically linear (RAL) estimator $\widetilde\theta$ of $\theta_0$ and influence function $\psi(O)$, it holds
\begin{align*}
\Var_{P_0}\{\psi(O)\}
\succeq
\Var_{P_0}\{\phi^{\mathrm{eff}}(O)\}
= I(\theta_0)^{-1}\,\Sigma\,I(\theta_0)^{-1},
\end{align*}
where $\phi^{\mathrm{eff}}$ and $\Sigma$ are as in Theorem~\ref{thm:semiparametric-main},
and $\succeq$ denotes the Loewner order on symmetric matrices (for $p=1$, it reduces to $\ge$). Equivalently, $\liminf_{N\to\infty}\Var_{P_0}\!\big(\sqrt{N}\,(\widetilde\theta-\theta_0)\big)
\;\succeq\; I(\theta_0)^{-1}\Sigma I(\theta_0)^{-1}.$ Thus $\Var_{P_0}\{\phi^{\mathrm{eff}}(O)\}$ is the semiparametric efficiency lower bound. Equality holds if and only if $\psi(O)=\phi^{\mathrm{eff}}(O)$ $P_0$-a.s., in which case the RAL estimator $\widetilde\theta$ is semiparametrically efficient. 

\paragraph{Semiparametric efficiency condition for PPI.} Assume the conditions of Theorems~\ref{thm:semiparametric-main} and~\ref{thm:ppi-known-m-compact} hold, and suppose that the predictor $m$ is \emph{score-calibrated} at the truth, i.e., plug-in score equals the conditional score:
\begin{small}
\begin{equation}
U(\theta_0;X,m(X))=\overline U(\theta_0;X;P_0)
=\E_{P_0}\!\big[U(\theta_0;X,Y)\mid X\big].
\label{score-calib}
\end{equation}
\end{small}
Then the semiparametric efficient variance matches the $M$-estimation variance:
\begin{small}
\begin{align*}
\Var_{P_0}\!\big(\phi^{\mathrm{eff}}(O)\big)
=I(\theta_0)^{-1}\,\Sigma\,I(\theta_0)^{-1}
=V_1+\big(f^{-1}-1\big)V_2,
\end{align*}
\end{small}
where $\Sigma$ is as in Theorem~\ref{thm:semiparametric-main} and
$V_1, V_2$ are as in Theorem~\ref{thm:ppi-known-m-compact}. Therefore, the PPI estimator attains the semiparametric efficiency bound asymptotically  when the predictor is score-calibrated in the sense of satisfying (\ref{score-calib}). In other words, the fitted model must output the correct conditional estimating function on average. This is a weaker condition than having a perfect predictor for $Y$, but it ensures no first-order bias in the score. 

In the special case of $\theta_0=\mathbb{E}_{P_0}[Y]$, the score-calibration condition reduces to $m(x)= \mathbb{E}_{P_0}[Y \mid X = x]$. Thus, as long as the ML predictor $m(x)$ is consistent for $\mathbb{E}_{P_0}[Y \mid X = x]$, we can expect that the resulting PPI estimator is semiparametrically efficient.  Of course, in practice $m$ is learned from data rather than fixed. We next examine how learning the predictor affects PPI’s properties and what can be done to maintain the statistical validity.

\section{Asymptotic theory: a learned prediction rule}\label{sec:ASYMPTOTIC THEORY: A LEARNED PREDICTION RULE}
 
\subsection{A limitation of vanilla prediction-powered inference}\label{sec:Limitation of prediction-powered inference}
Theorem~\ref{thm:ppi-known-m-compact} is stated for a \emph{fixed} prediction rule \(m\).
In practice, \(m\) is learned from the labeled data \(\{(X_j,Y_j): j\in S\}\), yielding a fitted rule \(\widehat m\). Generally speaking, the conclusions of Theorem~\ref{thm:ppi-known-m-compact} continue to hold if \(\widehat m(X)\) is sufficiently close to the oracle regression \(m_0(x):=\mathbb{E}_{P_0}[Y\mid X=x]\) and the associated function class has controlled complexity (e.g., is \(P\)-Donsker); see Section~10.4.2 of \cite{kennedy2024semiparametric} and Section~8.4.2 of \cite{kennedy2016semiparametric}. 


One central concern when using PPI with a learned prediction rule \(\widehat m\) is overfitting.
More specifically, the prediction rule \(m\) is an infinite-dimensional nuisance parameter; consequently, estimating \(m\) with a flexible black-box model (e.g., random forests, gradient-boosted trees, or neural networks) can be prone to overfitting.
If the same labeled outcomes are used both to train \(\widehat  m\) and to form the labeled residuals in the PPI score \(\widehat U_{\mathrm{PPI}}(\theta)\)~(\ref{eq:PPI-score-derivation}), this “double dipping” may allow training noise to leak into the rectifier \(\Delta_\theta\).
Such leakage may spoil the \(o_p(N^{-1/2})\) remainder required for the linear expansion in~(\ref{eq:LAN_PPI_estimator_m_estimation}).
See \cite{chernozhukov2017} for an illustration and relevant discussion.

To remove this leakage and preserve first-order validity, a standard remedy is \emph{cross-fitting} (or sample splitting). More broadly, cross-fitting is widely used in modern causal-inference workflows, including targeted learning and debiased/double machine learning, to mitigate overfitting of nuisance parameters and to ensure valid asymptotic inference \citep{newey2018cross,chernozhukov2017,ZhengVanDerLaan2010CVTMLE}. As an alternative, one may apply a \emph{variance adjustment} to the single-fit PPI estimator; to our knowledge, this constitutes a novel methodological development in the context of PPI.

\subsection{Cross-fit prediction-powered inference with sample-splitting}\label{sec:cf}
The cross-fit prediction-powered inference (CF-PPI) can be implemented as follows. First, we partition the labeled set \(S\) into \(K\) folds \(S_1,\ldots,S_K\). For the labeled covariates, for each fold \(k\in\{1,\ldots,K\}\), fit a predictor \(\widehat m^{(k)}\) using only the labels in \(S\setminus S_k\).
Use the resulting out-of-fold prediction
\(\widehat m^{(k)}(X_i)\) for every labeled index \(i\in S_k\). (See Figure \ref{fig:Cross_Fitting_Idea} in the Appendix.) For the unlabeled covariates, we use an aggregate predictor—either the single
fit on all labels $\widehat m^\star(x)=\widehat m_{\mathrm{all}}(x)$ or the average of the $K$
fold-specific fits $\widehat m^\star(x)=(1/K)\sum_{k=1}^K \widehat m^{(k)}(x)$. 

The out-of-fold predictor is then
\begin{align}
\nonumber
\widehat m^{(-)}(X_i):=
\begin{cases}
\widehat m^{(\kappa(i))}(X_i), & i\in S \,\, \text{(labeled)},\\
\widehat m^{\star}(X_i), & i\notin S\,\, \text{(unlabeled)},
\end{cases}    
\end{align}
where $\kappa(i)\in\{1,\dots,K\}$ is the fold map.

Define the cross-fitted model-fit term and rectifier by
\begin{align*}
m_\theta^{\mathrm{cf}}
&=\frac{1}{N}\sum_{i=1}^N U\!\big(\theta;X_i,\widehat m^\star(X_i)\big),\\
\Delta_\theta^{\mathrm{cf}}
&=\frac{1}{n}\sum_{k=1}^K\ \sum_{i\in S_k}\!\Big\{U(\theta;X_i,Y_i)-U\!\big(\theta;X_i,\widehat m^{(k)}(X_i)\big)\Big\}.    
\end{align*}

The CF-PPI score is then obtained by plugging-in $\widehat m^{(-)}(\cdot)$ to the PPI score \eqref{eq:PPI-score-derivation}
\begin{align*}
&\widehat U_{\mathrm{PPI}}^{\mathrm{cf}}(\theta):=
\widehat U_{\mathrm{PPI}}(\theta; \widehat m^{(-)}(\cdot)) =m_\theta^{\mathrm{cf}}+\Delta_\theta^{\mathrm{cf}},
\end{align*}
and the CF-PPI estimator is $\widehat\theta_{\mathrm{PPI}}^{\mathrm{cf}}
:= \arg\{\widehat U_{\mathrm{PPI}}^{\mathrm{cf}}(\theta)=0\}$.

One of the important benefits of using sample splitting in CF-PPI is the following. It ensures that, by conditioning on the trained models \(\{\widehat  m^{(k)}\}_{k=1}^K\), the labeled residuals in
\(\Delta_{\theta}^{\mathrm{cf}}\) are always evaluated on units that were \emph{not} used to
fit the corresponding model (out-of-fold). This preserves design-unbiasedness of the score and confines the learning error from \(\widehat m\) to a second-order remainder.

For the semi-supervised mean estimation, we can prove the following theorem.
\begin{theorem}\label{thm:cfppi_mean_unified}
Let $(X,Y)$ have joint law $P$ with $X\sim P_X$ and $Y=m_0(X)+\varepsilon$, where
$\mathbb E[\varepsilon\mid X]=0$ and $\mathbb E[Y^2]<\infty$. The target is the population mean
$\theta_0:=\mathbb E[Y]$. Let $\{X_i\}_{i=1}^N \stackrel{\mathrm{i.i.d.}}{\sim} P_X$ be an unlabeled
sample and $\{(X_j,Y_j)\}_{j\in S}$, $|S|=n$, an independent labeled sample from $P$, with
$N,n\to\infty$ and $n/N\to f\in(0,1)$. 

Let $\widehat m^{(-)}$ be the cross-fit predictor
(each labeled index is scored by a model trained without its own fold), and consider
\begin{equation*}
\widehat\theta^{\mathrm{cf}}_{\mathrm{PPI}}
=\frac{1}{N}\sum_{i=1}^N \widehat m^{(-)}(X_i)
+\frac{1}{n}\sum_{j\in S}\{Y_j-\widehat m^{(-)}(X_j)\}.
\end{equation*}
Then:

\noindent\textup{\textbf{(i) Consistency.}}
If the out-of-fold error is stochastically bounded in $L_2(P_X)$,
\[
\|\widehat m^{(-)}-m_0\|_{L_2(P_X)} = O_p(1),
\]
then $\widehat\theta^{\mathrm{cf}}_{\mathrm{PPI}}\xrightarrow{p}\theta_0$.

\noindent\textup{\textbf{(ii) Asymptotic normality.}}
If, in addition, the cross-fit predictor is $L_2(P_X)$-consistent,
\[
\|\widehat m^{(-)}-m_0\|_{L_2(P_X)} = o_p(1),
\]
then
\begin{align}
\nonumber
&\sqrt{N}\,\big(\widehat\theta^{\mathrm{cf}}_{\mathrm{PPI}}-\theta_0\big)
\;\xrightarrow{d}\; \mathcal{N}\!\big(0,\sigma_f^2\big),\\    
\sigma_f^2 \;&=\; \Var\!\big(m_0(X)\big) \;+\; \frac{1}{f}\,\Var\!\big(Y-m_0(X)\big).
\label{eq:variance_of_mean_CF_PPI}
\end{align}
\end{theorem}

Note that the variance formula \eqref{eq:variance_of_mean_CF_PPI} coincides with
\eqref{eq:variance_of_mean_fixed_m}, except that the fixed predictor $m$ is
replaced by the true regression function $m_0$. In practice, a confidence interval can be constructed by plugging the cross-fitted predictor
$\widehat m^{(-)}$ into this expression.

\paragraph{Relation to previous work.}
\cite{zrnic2024cross} proved central limit theorems (CLT) for CF-PPI for the mean
and for general $M$-estimation under stability conditions on the fold-specific learners
(see their Assumptions 1 and 2). Those conditions may capture algorithmic stability but are not standard in the empirical-process theory. Our analysis for the semi-supervised mean works under milder and more classical assumptions: (i) consistency follows from the minimal
$L_2$ stochastic boundedness $\|\widehat m^{(-)} -m_0\|_{L_2(P_X)}=O_p(1)$; and
(ii) asymptotic normality holds under the mere $L_2$ consistency
$\|\widehat m^{(-)}-m_0\|_{L_2(P_X)}=o_p(1)$. We do not impose stability, Donsker, or entropy
conditions. Our proof relies only on the basic decomposition, cross-fitting for conditional
independence and centering, and a conditional Chebyshev bound to control the remainder,
which yields the closed-form variance (\ref{eq:variance_of_mean_CF_PPI}) separating the unlabeled and labeled fluctuations. 

\paragraph{Relation to standard assumptions on nuisance parameters in causal inference.}
Under the setup of Theorem~\ref{thm:cfppi_mean_unified}, the nuisance remainder satisfies
\begin{align*}
(P_N-P_n)\big(\widehat m^{(-)}-m_0\big)= O_p\!\Big(\|\widehat m^{(-)}-m_0\|_{L_2(P_X)}
       \big(1/\sqrt N+1/\sqrt n\big)\Big),
\end{align*}
by the cross-fitted empirical-process bound
\((P_m-P)h = O_p(\|h\|_{L_2(P)}/\sqrt m)\) applied separately to the unlabeled and labeled
averages (see the Appendix, or Lemma~1 of \cite{kennedy2024semiparametric}).
Multiplying by $\sqrt N$ and using $n/N\to f\in(0,1)$ yields
\[
\sqrt N\,(P_N-P_n)\big(\widehat m^{(-)}-m_0\big)
 \;=\; O_p\!\big(\|\widehat m^{(-)}-m_0\|_{L_2(P_X)}\big).
\]
Hence the CLT in Theorem~\ref{thm:cfppi_mean_unified} holds as soon as the
out-of-fold predictor is $L_2(P_X)$-consistent,
$\|\widehat m^{(-)}-m_0\|_{L_2(P_X)}=o_p(1)$; no specific rate such as $o_p(n^{-1/4})$ is required. This contrasts with common causal-inference settings
\citep{van2011targeted,van2018targeted,chernozhukov2018double}, where product-error conditions typically imply an $n^{-1/4}$ rate (instead, double robustness can be guaranteed). For the semi-supervised mean, linearity and cross-fitting
reduce the requirement to $L_2$ consistency only.

\subsection{Single-fit prediction-powered inference with variance correction}
\label{sec:sf-ppi}
In this subsection, we propose a new variant of PPI and analyze its asymptotic properties. The central idea is to use an ML predictor \emph{without} cross-fitting and to introduce an explicit variance correction (SF-PPI-VC) to account for the reuse of labels in the rectifier term. 

We illustrate the procedure of SF-PPI-VC. The predictor $\widehat m$ is trained once on the labeled sample and then reused in both the unlabeled and labeled terms of the PPI score \eqref{eq:PPI-score-derivation}:
\begin{align*}
&\widehat U_{\mathrm{PPI}}^{\mathrm{sf}}(\theta) := \widehat U_{\mathrm{PPI}}(\theta; \widehat m(\cdot)) = \underbrace{\frac1N\sum_{i=1}^N U\!\big(\theta;X_i,\widehat m(X_i)\big)}_{ \substack{\text{measure of fit $m_\theta^{\mathrm{sf}}$ with }\\ \text{a fitted prediction rule $\widehat m(\cdot)$}}}  + \underbrace{\frac1n\sum_{j\in S}\!\Big\{U(\theta;X_j,Y_j)-U\!\big(\theta;X_j,\widehat m(X_j)\big)\Big\}.}_{\substack{\text{rectifier $\Delta_\theta^{\mathrm{sf}}$ with}\\ \text{a fitted prediction rule $\widehat m(\cdot)$}}}
\end{align*}
The SF-PPI-VC estimator is any solution to $\widehat U_{\mathrm{PPI}}^{\mathrm{sf}}(\theta)=0$, as in the original PPI implementation \citep{angelopous2023}, which does not use sample splitting.

As noted in our discussion of CF–PPI, reusing labels to both train $\widehat m$ and form residuals induces dependence between the residuals and the fitted rule, producing \emph{label-noise leakage} that typically appears in the fitted rectifier $\Delta_\theta^{\mathrm{sf}}$. Without a variance adjustment to absorb this leakage (when cross-fitting is not used), first-order validity is not guaranteed: root-$N$ consistency and asymptotic normality, as in Theorem~\ref{thm:ppi-known-m-compact}, may fail. See the Appendix for details.

A practical consideration of the SF-PPI-VC approach is that the variance correction must be tailored to the specific problem setting (e.g., mean estimation, generalized linear regression, etc.) and/or to the chosen prediction rule $\widehat m$ to guarantee the desired asymptotic properties. Developing a unified asymptotic theory may be quite challenging, as variance correction typically requires a specific adjustment method for each chosen class of predictors $\widehat m(\cdot)$.


We revisit the semi-supervised mean estimation problem, where the estimation is performed by SF-PPI-VC. 

\paragraph{Semi-supervised mean estimation (revisited).}
For the population mean $\theta_0=\mathbb{E}[Y]$ with $U(\theta;x,y)=y-\theta$ and
$Y=m_0(X)+\varepsilon$ ($\mathbb{E}[\varepsilon\mid X]=0$, $\sigma^2=\mathbb{E}[\varepsilon^2]<\infty$),
the SF-PPI estimator is
\begin{align}
\nonumber
\widehat\theta_{\mathrm{PPI}}^{\mathrm{sf}}
= \frac{1}{N}\sum_{i=1}^N \widehat m(X_i)
\;+\; \frac{1}{n}\sum_{j\in S}\!\big\{Y_j-\widehat m(X_j)\big\}.
\end{align}
Assume $\widehat m$ is a linear smoother trained on the labeled data $S$,
$\widehat m(x)=s(x)^\top Y_S + b(x)$, where $s(x)=(s_j(x))_{j\in S}\in\mathbb{R}^n$
are smoothing weights (depending on the labeled covariates $X_S$),
$Y_S:=(Y_j)_{j\in S}$ is the vector of labeled responses, and $b(x)$ is an offset.
We impose the mass-preserving property $\sum_{j\in S} s_j(x)=1$ for all $x$.


We fit $\widehat m$ using kernel ridge regression (KRR) with an unpenalized intercept. Write the Gram matrix $K\in\mathbb{R}^{n\times n}$ with $K_{ij}=k(X_i,X_j)$ for a
bounded positive semidefinite (PSD) kernel $k$ (so $\sup_x k(x,x)\le \kappa^2<\infty$).
We parameterize predictions as
\[
\widehat m(x)=\widehat\alpha+k_x^\top \widehat\beta,\,\,
k_x:=(k(X_1,x),\ldots,k(X_n,x))^\top,
\]
where $(\widehat\alpha,\widehat\beta)\in
\arg\min_{\alpha\in\mathbb{R},\,\beta\in\mathbb{R}^n}
\Big\{\frac{1}{n}\,\big\|Y_S-\alpha\,\mathbf 1_n-K\beta\big\|_2^2
+\lambda_n\,\beta^\top K\beta\Big\}$ with $\lambda_n>0$ the regularization parameter.
This construction yields a linear smoother $\widehat m(x)=s(x)^\top Y_S$ (here $b(x) = 0$) whose
weights obey the mass-preserving property.  Writing
$\widehat m_S:=(\widehat m(X_j))_{j\in S}$ for the fitted values on the labeled inputs, we have the following in-sample representation
\[
\widehat m_S \;=\; K\,\widehat\beta \;+\; \widehat\alpha\,\mathbf 1_n \;=\; H\,Y_S,
\]
where $H=[s(X_1) \cdots s(X_n)]^{\top}$ is symmetric $n$-by-$n$ KRR hat matrix with an unpenalized intercept, and it satisfies $H\mathbf 1_n=\mathbf 1_n$. Similarly, we define $N$-by-$n$ unlabeled weights matrix as $S_U=[s(X_1) \cdots s(X_N)]^{\top}$. The unlabeled leverage average is then given by $\overline{\|s(X)\|_2^2}=(1/N)\sum_{i=1}^N\|s(X_i)\|_2^2$. Based on these, we define the centering vector $c=(1/N)S_U^{\top} \mathbf 1_N $ $ + (1/n)\mathbf 1_n $ $- (1/n)H^{\top}\mathbf 1_n \in \mathbb{R}^n$. For additional details, see the Appendix. 

Then the following theorem holds.

\begin{theorem}\label{thm:sfppi_mean_krr}
Assume the same semi-supervised mean setting as Theorem~\ref{thm:cfppi_mean_unified}. Train a single predictor $\widehat m$ on the labeled sample using a KRR with a bounded PSD kernel $k$ and regularization $\lambda_n>0$, including an
unpenalized intercept so that $\sum_{j\in S}s_j(x)=1$.
Define the degrees-of-freedom-adjusted residual variance using the labeled hat matrix $H$:
\[
\widehat\sigma^2_{\mathrm{KRR}}:=\frac{\big\|(I_n-H)Y_S\big\|_2^2}{\,n-\tr(H)\,}.
\]
Assume the tuning satisfies $\lambda_n \downarrow 0$ and $n\lambda_n \to \infty$, and that the
stability and centering conditions for the induced smoother $\widehat m(x)=s(x)^\top Y_S+b(x)$ hold
(in particular, $\max_i\|s(X_i)\|_2^2=O_p(1/n)$, $\sqrt N\max_{j\in S}|c_j|\to 0$, and $N\|c\|_2^2\to 1/f$, see Appendix).

Then:

\noindent\textup{\textbf{(i) Consistency.}}
$\widehat\theta^{\mathrm{sf}}_{\mathrm{PPI}}\xrightarrow{p}\theta_0$.

\noindent\textup{\textbf{(ii) Asymptotic normality.}}
\begin{align*}
\sqrt{N}\,\big(\widehat\theta^{\mathrm{sf}}_{\mathrm{PPI}}-\theta_0\big)\;&\xrightarrow{d}\; \mathcal{N}\!\big(0,\sigma_f^2\big),\\
\sigma_f^2 \;=\; \Var\!\big(m_0(X)\big) \;&+\; \frac{1}{f}\,\Var\!\big(Y-m_0(X)\big).
\end{align*}

\noindent\textup{\textbf{(iii) Variance correction and studentized CLT.}}
With
\begin{align*}
\widehat{\Var}\!\big(\widehat\theta^{\mathrm{sf}}_{\mathrm{PPI}}\big)
&=\frac{1}{N}\Big\{\Var_N\!\big(\widehat m(X_1),\ldots,\widehat m(X_N)\big) -\widehat\sigma^2_{\mathrm{KRR}}\, \overline{\|s(X)\|_2^2}\Big\}+\widehat\sigma^2_{\mathrm{KRR}}\,\|c\|_2^2,    
\end{align*}
where $
\Var_N(Z_1,\ldots,Z_N)
= (1/N)\sum_{i=1}^N (Z_i-\bar Z)^2, \, \bar Z=(1/N)\sum_{i=1}^N Z_i$, 
we have $\widehat{\Var}\!\big(\widehat\theta^{\mathrm{sf}}_{\mathrm{PPI}}\big)\xrightarrow{p}\sigma_f^2/N$ and
\[
\frac{\widehat\theta^{\mathrm{sf}}_{\mathrm{PPI}}-\theta_0}
{\sqrt{\widehat{\Var}\!\big(\widehat\theta^{\mathrm{sf}}_{\mathrm{PPI}}\big)}}
\;\xrightarrow{d}\; \mathcal N(0,1).
\]
\end{theorem}
Parts (i)–(iii) of Theorem \ref{thm:sfppi_mean_krr} show that the SF-PPI estimator with variance correction---when the predictor is obtained by KRR with an unpenalized intercept and mild tuning/stability conditions---admits standard large-sample inference from a single model fit. Moreover, the plug-in variance $\widehat{\Var}\!\big(\widehat\theta^{\mathrm{sf}}_{\mathrm{PPI}}\big)$ is consistent, yielding a studentized CLT. Notably, $\sigma_f^2$ coincides with the variance that would obtain if $m_0$ were known (oracle form), suggesting first-order optimality.  

\section{Simulation experiments}\label{sec:Simulation experiments}
\subsection{Simulation setting}\label{subsec:Simulation Setting}
We study the super-population mean target \(\theta_0=\mathbb{E}[Y]\) under a semi-supervised design. 
For each replication, we draw \(X\sim\mathcal{N}(0,1)\) and generate
\[
Y = m_0(X) + \varepsilon,\qquad \varepsilon\sim\mathcal{N}(0,1),
\]
independent of \(X\). We adopt a finite-population SRSWOR design: a population of size \(N\) with all \(N\) covariates unlabeled, and a labeled subset \(S\subset\{1,\dots,N\}\) of size \(n\) drawn without replacement (labeling fraction \(f=n/N\)). We consider $N \in \{500,2000\}$ with $f\in \{0.1, 0.2, 0.3, 0.4, 0.5\}$.

We consider two ground-truth regression functions \(m_0\). These functions are displayed in Figure~\ref{fig:true_functions}. 

\begin{figure}[h]
  \centering
  \includegraphics[width=0.9\columnwidth]{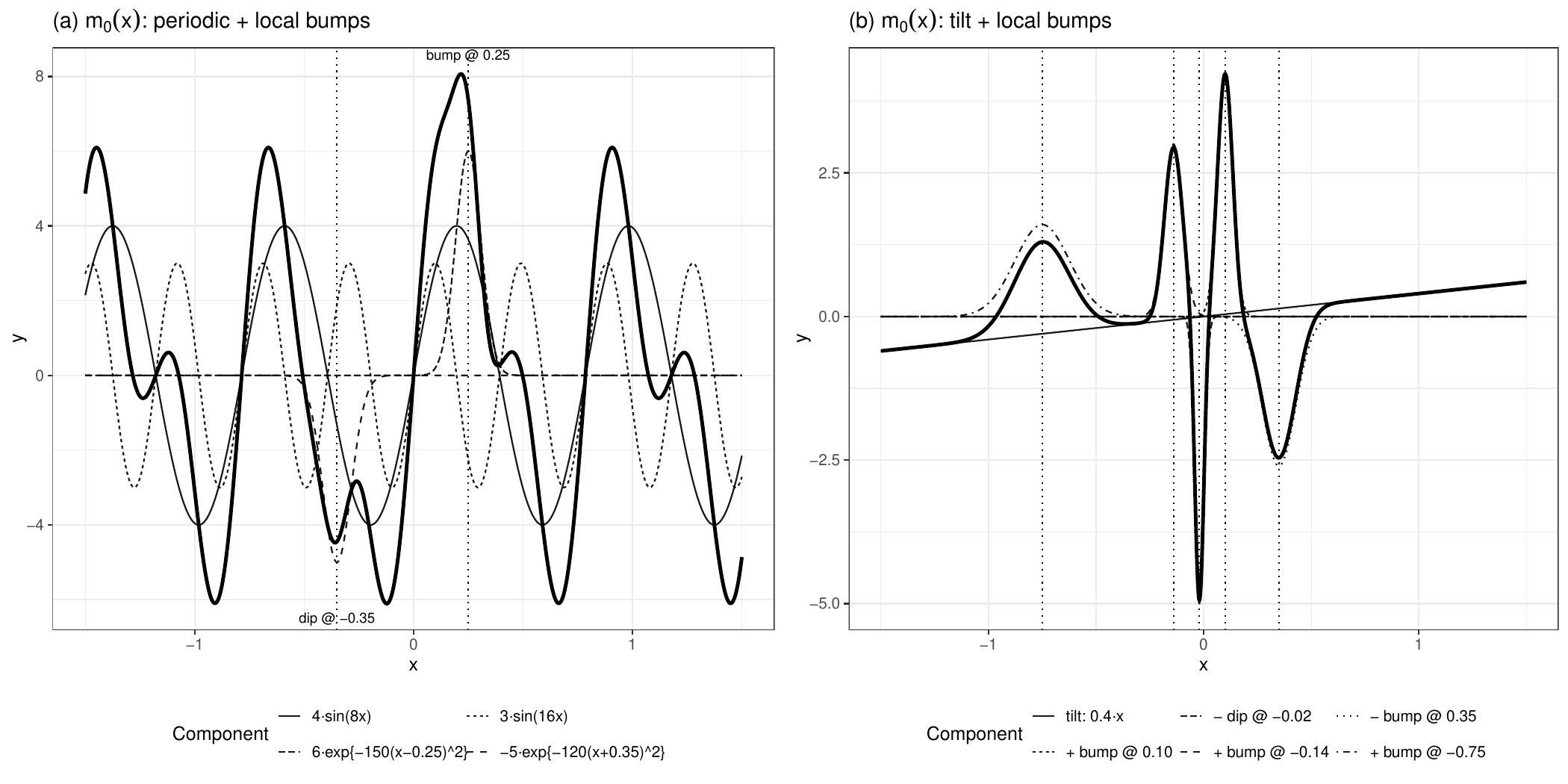}
  \caption{True regression functions $m_0(x)$ (shown as bold curves) used in simulation experiments: a periodic baseline with multiple local bumps (Panel (a)), and a tilted baseline with multiple local bumps and dips (Panel (b)). For each panel, the dominant trends are indicated by the curve ($y = 4\sin(8x)$) and the line ($y = 0.4x$).}
  \label{fig:true_functions}
\end{figure}

For the first simulation experiment, we set the \emph{periodic + local bumps} specification for \(m_0\),
\[
m_0(x)=4\sin(8x)+3\sin(16x)+6e^{-150(x-0.25)^2}-5e^{-120(x+0.35)^2},
\]
which combines a dominant low-frequency sinusoid \(4\sin(8x)\), a higher-frequency ripple \(3\sin(16x)\), and two sharp Gaussian perturbations—a positive bump near \(x=0.25\) and a negative dip near \(x=-0.35\)—yielding a periodic baseline with localized deviations. 

For the second simulation experiment, we set the \emph{tilt + local bumps} specification for \(m_0\),
\[
\begin{aligned}
m_0(x)=\;&0.4x
+4.2\,e^{-((x-0.10)/0.05)^2}
-5.0\,e^{-((x+0.02)/0.03)^2}\\
&+3.0\,e^{-((x+0.14)/0.06)^2}
-2.6\,e^{-((x-0.35)/0.11)^2}
+1.6\,e^{-((x+0.75)/0.18)^2},
\end{aligned}
\]
which superimposes a global linear tilt with several localized Gaussian features of varying width and sign: a sharp positive bump at \(0.10\), a nearby sharp negative dip at \(-0.02\), a medium positive bump at \(-0.14\), a broad negative dip at \(0.35\), and a very broad positive bump at \(0.75\). This yields a non-periodic regression surface with asymmetric local structure.

We evaluate three categories of estimators. 

First, the classical estimator 
\[
\hat\theta_{\mathrm{class}}=\frac{1}{n}\sum_{j\in S}Y_j,
\]
the sample average of the outcomes on the labeled dataset. 

Second, PPI with a \emph{prespecified} predictor \(m\),
\[
\hat\theta_{\mathrm{PPI}}=\frac{1}{N}\sum_{i=1}^Nm(X_i)
\;+\;\frac{1}{n}\sum_{j\in S}\{Y_j-m(X_j)\}.
\]

As for the fixed predictor \(m\), we consider (i) an oracle specification \(m=m_0\) and two misspecified fixed rules: (ii) a crude step function 
\(m(x)=2.2\,\mathbf{1}\{x>0.5\}-0.2\,\mathbf{1}\{x\le 0.5\}\), and (iii) a dominant-trend surrogate that keeps only the leading structure of the truth 
(i.e., \(m(x)=4\sin(8x)\) in the periodic case and \(m(x)=0.4x\) in the tilted case). 

Third, PPI with a \emph{learned} predictor \(\widehat m\). Specifically, we fit \(\widehat m\) via KRR with an unpenalized intercept, where the kernel is specified as a Gaussian kernel
\[
k(x,x')=\exp\!\left(-\frac{(x-x')^2}{2\ell^2}\right)
\quad\text{with}\quad \ell=0.1,
\]
and a ridge parameter \(\lambda=c\, n^{-\alpha}\) with \(\alpha=1/2\) and \(c=10^{-4}\). This regularization setting satisfies the standard KRR conditions \(\lambda\to0\) and \(n\lambda\to\infty\), providing a bias–variance tradeoff that is
compatible with asymptotic consistency and that performed well across our simulation scenarios. Recall that leaving the intercept unpenalized ensures that the mass-preserving property. 

We report three implementations that differ in how \(\widehat m\) is trained and how uncertainty is quantified.
(i) Vanilla PPI fits \(\widehat m\) once on all labeled data and uses the plug-in variance formula; this highlights the baseline behavior of PPI, originally proposed by \citep{angelopous2023}, without cross-fitting or variance adjustment. (ii) CF-PPI \citep{zrnic2024cross} employs \(K\)-fold cross-fitting (here \(K=5\)): for each fold, \(\widehat m(\cdot) = m^{(-)}(\cdot)\) is trained on
the remaining \(K-1\) folds and evaluated on the held-out fold, mitigating overfitting in the correction term and potentially
improving robustness of standard errors.
(iii) SF–PPI–VC fits the model once $\widehat m$ and corrects the variance by separating uncertainty from unlabeled predictions and noise in the labeled outcomes, using the KRR hat matrix and weights with a degrees-of-freedom adjustment. See the algorithms in Section~\ref{sec:ALGORITHM} for the implementation of these methods.

Across replications (\(R=10{,}000\)) we report (a) the bias, (b)
the mean absolute error (MAE), (c) the root mean squared error (RMSE), (d)
the 95\% coverage of the nominal 95\% confidence interval, 
and (e) the mean confidence-interval length. Here, the ground-truth \(\theta_0=\mathbb{E}[m_0(X)]\) is computed by one-dimensional numerical
integration under \(X\sim\mathcal{N}(0,1)\). That is, $\theta_0 = \mathbb{E}[m_0(X)]
= \int_{-\infty}^{\infty} m_0(x)\,p(x)\,dx
= \int_{-\infty}^{\infty} m_0(x)\,\phi(x)\,dx$, where \(\phi(x)=(2\pi)^{-1/2}\exp\!\bigl(-x^2/2\bigr)\) denotes the standard normal density. 

\subsection{Simulation experiment 1: periodic + local bumps specification}\label{subsec:Simulation Experiment 1}

\subsubsection{Monte Carlo numerical results at $N = 500$}
The simulation results are shown in Table \ref{tab:ppi_n500_simexp1}. As the labeling fraction $f$ increases from $0.1$ to $0.5$, all procedures exhibit monotonic decreases in MAE and RMSE and shorter intervals; biases are essentially zero throughout. PPI with the oracle $m=m_0$ uniformly improves upon the classical estimator for every $f$, yielding much smaller errors and substantially shorter intervals while maintaining approximately $95\%$ coverage. This suggests that, when implemented properly, PPI can substantially improve efficiency relative to using only the labeled data.

PPI with the prediction rule fixed to a dominant low-frequency sinusoid, $m(x)=4\sin(8x)$, outperforms the classical estimator in both accuracy and interval length, whereas PPI with the crude step function $m(x)=2.2\,\mathbf{1}\{x>0.5\}-0.2\,\mathbf{1}\{x\le 0.5\}$ is consistently the least accurate (largest MAE/RMSE and longer intervals). Nevertheless, the coverage of the nominal $95\%$ confidence intervals is maintained and the bias remains nearly zero across all scenarios, consistent with the theoretical result that misspecification affects only efficiency (interval width) and not the bias of PPI.

PPI with the fitted prediction rule $\widehat m$ without cross-fitting or variance correction (PPI (no CF, no VC)) yields deceptively short intervals but exhibits severe undercoverage when the labeling fraction is small—about $58\%$ at $f=0.1$, $84\%$ at $f=0.2$, and $91\%$ at $f=0.3$. Coverage improves and approaches the nominal $95\%$ as $f$ increases, but it does not reach the nominal level even at $f=0.5$. This pattern suggests that the vanilla PPI procedure \citep{angelopous2023} may suffer from undercoverage in low-label regimes, motivating the use of cross-fitting and/or variance correction.

PPI with the fitted prediction rule $\widehat m$ with cross-fitting (CF-PPI) and with variance correction (SF-PPI-VC) exhibits slight undercoverage at $f=0.1$, but coverage quickly recovers to near $95\%$ by $f=0.2$; by $f=0.3$, both methods essentially maintain the nominal $95\%$ coverage.

\begin{table*}[h!]
\caption{Results of Simulation Experiment 1 at $N=500$ for varying labeled fractions $f=n/N$.}
\label{tab:ppi_n500_simexp1}
\begin{tabular}{@{}llrrrrr@{}}
\hline
Sample size & Method &
\multicolumn{1}{c}{Bias} &
\multicolumn{1}{c}{MAE} &
\multicolumn{1}{c}{RMSE} &
\multicolumn{1}{c}{Int. len.} &
\multicolumn{1}{c@{}}{Cov. (\%)} \\
\hline

\multirow{7}{*}{\begin{tabular}{@{}l@{}}
$N=500$\\
$n=50$\\
$f=0.1$
\end{tabular}}
& Classical (labeled only) & 0.0010 & 0.3890 & 0.5810 & 2.2973 & 94.7 \\
& PPI (oracle $m=m_0$) & 0.0032 & 0.1569 & 0.2315 & 0.8978 & 94.8 \\
& PPI (fixed $m$: step) & 0.0023 & 0.4090 & 0.6115 & 2.4293 & 94.6 \\
& PPI (fixed $m$: sin-only) & $-$0.0010 & 0.2555 & 0.3821 & 1.4766 & 94.2 \\
& PPI (fitted $\widehat m$: no CF, no VC) & 0.0125 & 0.3387 & 0.5524 & 0.8277 & 57.9 \\
& CF-PPI (fitted $\widehat m$: CF) & 0.0072 & 0.3592 & 0.5624 & 1.9696 & 91.7 \\
& SF-PPI-VC (fitted $\widehat m$: VC) & 0.0125 & 0.3387 & 0.5524 & 1.8312 & 90.7 \\
\hline

\multirow{7}{*}{\begin{tabular}{@{}l@{}}
$N=500$\\
$n=100$\\
$f=0.2$
\end{tabular}}
& Classical (labeled only) & 0.0025 & 0.2750 & 0.4106 & 1.6258 & 95.1 \\
& PPI (oracle $m=m_0$) & 0.0036 & 0.1413 & 0.2081 & 0.8079 & 94.9 \\
& PPI (fixed $m$: step) & 0.0030 & 0.2899 & 0.4295 & 1.7247 & 95.3 \\
& PPI (fixed $m$: sin-only) & 0.0027 & 0.1960 & 0.2884 & 1.1019 & 94.4 \\
& PPI (fitted $\widehat m$: no CF, no VC) & 0.0019 & 0.1862 & 0.2812 & 0.7826 & 84.4 \\
& CF-PPI (fitted $\widehat m$: CF) & 0.0008 & 0.2020 & 0.3018 & 1.1202 & 93.7 \\
& SF-PPI-VC (fitted $\widehat m$: VC) & 0.0019 & 0.1862 & 0.2812 & 1.0249 & 93.6 \\
\hline

\multirow{7}{*}{\begin{tabular}{@{}l@{}}
$N=500$\\
$n=150$\\
$f=0.3$
\end{tabular}}
& Classical (labeled only) & 0.0046 & 0.2243 & 0.3367 & 1.3281 & 95.1 \\
& PPI (oracle $m=m_0$) & 0.0034 & 0.1354 & 0.1987 & 0.7756 & 95.0 \\
& PPI (fixed $m$: step) & 0.0048 & 0.2322 & 0.3500 & 1.4132 & 95.4 \\
& PPI (fixed $m$: sin-only) & 0.0032 & 0.1674 & 0.2499 & 0.9441 & 94.1 \\
& PPI (fitted $\widehat m$: no CF, no VC) & 0.0039 & 0.1513 & 0.2222 & 0.7627 & 91.3 \\
& CF-PPI (fitted $\widehat m$: CF) & 0.0041 & 0.1577 & 0.2350 & 0.9100 & 94.8 \\
& SF-PPI-VC (fitted $\widehat m$: VC) & 0.0039 & 0.1513 & 0.2222 & 0.8450 & 94.4 \\
\hline

\multirow{7}{*}{\begin{tabular}{@{}l@{}}
$N=500$\\
$n=200$\\
$f=0.4$
\end{tabular}}
& Classical (labeled only) & 0.0056 & 0.1934 & 0.2936 & 1.1502 & 94.9 \\
& PPI (oracle $m=m_0$) & 0.0036 & 0.1306 & 0.1943 & 0.7590 & 95.0 \\
& PPI (fixed $m$: step) & 0.0054 & 0.2018 & 0.3023 & 1.2280 & 95.7 \\
& PPI (fixed $m$: sin-only) & 0.0038 & 0.1546 & 0.2286 & 0.8545 & 93.7 \\
& PPI (fitted $\widehat m$: no CF, no VC) & 0.0037 & 0.1401 & 0.2060 & 0.7517 & 93.1 \\
& CF-PPI (fitted $\widehat m$: CF) & 0.0034 & 0.1452 & 0.2135 & 0.8332 & 94.9 \\
& SF-PPI-VC (fitted $\widehat m$: VC) & 0.0037 & 0.1401 & 0.2060 & 0.7876 & 94.6 \\
\hline

\multirow{7}{*}{\begin{tabular}{@{}l@{}}
$N=500$\\
$n=250$\\
$f=0.5$
\end{tabular}}
& Classical (labeled only) & 0.0039 & 0.1779 & 0.2620 & 1.0287 & 95.0 \\
& PPI (oracle $m=m_0$) & 0.0036 & 0.1281 & 0.1915 & 0.7488 & 95.0 \\
& PPI (fixed $m$: step) & 0.0036 & 0.1824 & 0.2684 & 1.1016 & 95.8 \\
& PPI (fixed $m$: sin-only) & 0.0030 & 0.1465 & 0.2157 & 0.7958 & 93.4 \\
& PPI (fitted $\widehat m$: no CF, no VC) & 0.0038 & 0.1331 & 0.1968 & 0.7441 & 93.9 \\
& CF-PPI (fitted $\widehat m$: CF) & 0.0041 & 0.1369 & 0.2013 & 0.7950 & 95.1 \\
& SF-PPI-VC (fitted $\widehat m$: VC) & 0.0038 & 0.1331 & 0.1968 & 0.7614 & 94.6 \\
\hline
\end{tabular}

\smallskip \begin{center} \begin{minipage}{0.88\textwidth} \footnotesize \raggedright \textit{Methods.} \textbf{Classical (labeled only)}: sample mean of $Y$ over the labeled set. \textbf{PPI (oracle $m=m_0$)}: prediction-powered inference using the true regression $m_0(x)=\mathbb{E}[Y\mid X=x]$. \textbf{PPI (fixed $m$: step)}: PPI with a misspecified, fixed step predictor (e.g., $m(x)=2.2\,\mathbf{1}\{x>0.5\}-0.2\,\mathbf{1}\{x\le 0.5\}$). \textbf{PPI (fixed $m$: sin-only)}: PPI with a misspecified dominant-trend surrogate for $m$ (e.g., $m(x)=4\sin(8x)$ in the periodic design). \textbf{PPI (fitted $\widehat m$: no CF, no VC)}: Vanilla PPI \citep{angelopous2023} with a learned $\widehat m$ fit once by KRR; plug-in variance---no cross-fitting or variance correction. \textbf{CF-PPI (fitted $\widehat m$: CF)}: CF-PPI \citep{zrnic2024cross} ($K=5$), training $\widehat m$ on $K-1$ folds and evaluating on the hold-out. \textbf{SF-PPI-VC (fitted $\widehat m$: VC)}: single-fit PPI with variance correction that adjusts for uncertainty in $\widehat m$ and outcome noise (degrees-of-freedom adjustment). \\
\smallskip \textit{Metrics.} \textbf{MAE} = median absolute error; \textbf{RMSE} = root mean squared error; \textbf{Cov.} = coverage; \textbf{Int. len.} = average $95\%$ central credible-interval length. Bias values with magnitude $<10^{-4}$ are reported as 0.0000. \end{minipage} \end{center}
\end{table*}

\subsubsection{Monte Carlo numerical results at $N = 2000$}
The simulation results are shown in Table~\ref{tab:ppi_n2000_simexp1} mirrors the $N{=}500$ findings (see Table~\ref{tab:ppi_n500_simexp1}): as $f$ increases from $0.1$ to $0.5$, MAE/RMSE and interval length decline monotonically while biases remain essentially zero. Classical coverage stays near $95\%$. PPI with the oracle predictor ($m{=}m_0$) uniformly dominates—smaller errors and substantially shorter intervals with approximately $95\%$ coverage. Among the misspecified fixed rules, PPI (fixed $m$: sin-only) improves on the classical estimator in both accuracy and interval length, whereas PPI (fixed $m$: step) is the least accurate across $f$ (largest MAE/RMSE and longer intervals), though its coverage remains near nominal.

As noted in the $N{=}500$ results (Table~\ref{tab:ppi_n500_simexp1}), the vanilla PPI (PPI (no CF, no VC)) exhibits substantial
undercoverage at small $f$ (e.g., about $87\%$ at $f{=}0.1$). Coverage improves as $f$ increases and approaches $95\%$ at $f=0.4$. On the other hand, introducing cross-fitting (CF-PPI) or variance correction (SF-PPI-VC) maintains nominal coverage across all scenarios (only about 1\% below nominal at $f=0.1$).

\begin{table*}[h!]
\caption{Results of Simulation Experiment 1 at $N=2000$ for varying labeled fractions $f=n/N$.}
\label{tab:ppi_n2000_simexp1}
\begin{tabular}{@{}llrrrrr@{}}
\hline
Sample size & Method &
\multicolumn{1}{c}{Bias} &
\multicolumn{1}{c}{MAE} &
\multicolumn{1}{c}{RMSE} &
\multicolumn{1}{c}{Int. len.} &
\multicolumn{1}{c@{}}{Cov. (\%)} \\
\hline

\multirow{7}{*}{\begin{tabular}{@{}l@{}}
$N=2000$\\
$n=200$\\
$f=0.1$
\end{tabular}}
& Classical (labeled only) & 0.0019 & 0.2009 & 0.2956 & 1.1508 & 94.9 \\
& PPI (oracle $m=m_0$) & 0.0020 & 0.0750 & 0.1130 & 0.4489 & 95.0 \\
& PPI (fixed $m$: step) & 0.0031 & 0.2066 & 0.3092 & 1.2167 & 94.9 \\
& PPI (fixed $m$: sin-only) & 0.0002 & 0.1296 & 0.1906 & 0.7390 & 94.7 \\
& PPI (fitted $\widehat m$: no CF, no VC) & 0.0014 & 0.0941 & 0.1408 & 0.4265 & 87.1 \\
& CF-PPI (fitted $\widehat m$: CF) & 0.0019 & 0.1004 & 0.1494 & 0.5615 & 94.0 \\
& SF-PPI-VC (fitted $\widehat m$: VC) & 0.0014 & 0.0941 & 0.1408 & 0.5326 & 94.2 \\
\hline

\multirow{7}{*}{\begin{tabular}{@{}l@{}}
$N=2000$\\
$n=400$\\
$f=0.2$
\end{tabular}}
& Classical (labeled only) & $-$0.0007 & 0.1386 & 0.2065 & 0.8141 & 95.2 \\
& PPI (oracle $m=m_0$) & 0.0011 & 0.0689 & 0.1015 & 0.4040 & 95.4 \\
& PPI (fixed $m$: step) & $-$0.0006 & 0.1439 & 0.2152 & 0.8633 & 95.5 \\
& PPI (fixed $m$: sin-only) & 0.0007 & 0.0987 & 0.1445 & 0.5513 & 94.4 \\
& PPI (fitted $\widehat m$: no CF, no VC) & 0.0008 & 0.0727 & 0.1075 & 0.3975 & 93.6 \\
& CF-PPI (fitted $\widehat m$: CF) & 0.0011 & 0.0745 & 0.1108 & 0.4329 & 95.0 \\
& SF-PPI-VC (fitted $\widehat m$: VC) & 0.0008 & 0.0727 & 0.1075 & 0.4201 & 94.8 \\
\hline

\multirow{7}{*}{\begin{tabular}{@{}l@{}}
$N=2000$\\
$n=600$\\
$f=0.3$
\end{tabular}}
& Classical (labeled only) & 0.0000 & 0.1147 & 0.1686 & 0.6647 & 95.1 \\
& PPI (oracle $m=m_0$) & 0.0013 & 0.0657 & 0.0971 & 0.3879 & 95.5 \\
& PPI (fixed $m$: step) & 0.0000 & 0.1197 & 0.1751 & 0.7072 & 95.7 \\
& PPI (fixed $m$: sin-only) & 0.0012 & 0.0857 & 0.1255 & 0.4724 & 94.3 \\
& PPI (fitted $\widehat m$: no CF, no VC) & 0.0013 & 0.0671 & 0.0993 & 0.3850 & 94.5 \\
& CF-PPI (fitted $\widehat m$: CF) & 0.0010 & 0.0682 & 0.1006 & 0.4006 & 95.2 \\
& SF-PPI-VC (fitted $\widehat m$: VC) & 0.0013 & 0.0671 & 0.0993 & 0.3934 & 95.1 \\
\hline

\multirow{7}{*}{\begin{tabular}{@{}l@{}}
$N=2000$\\
$n=800$\\
$f=0.4$
\end{tabular}}
& Classical (labeled only) & 0.0009 & 0.0974 & 0.1451 & 0.5756 & 95.2 \\
& PPI (oracle $m=m_0$) & 0.0016 & 0.0644 & 0.0952 & 0.3796 & 95.3 \\
& PPI (fixed $m$: step) & 0.0009 & 0.1008 & 0.1504 & 0.6143 & 95.9 \\
& PPI (fixed $m$: sin-only) & 0.0015 & 0.0774 & 0.1143 & 0.4275 & 93.9 \\
& PPI (fitted $\widehat m$: no CF, no VC) & 0.0018 & 0.0644 & 0.0964 & 0.3780 & 94.9 \\
& CF-PPI (fitted $\widehat m$: CF) & 0.0018 & 0.0655 & 0.0971 & 0.3868 & 95.4 \\
& SF-PPI-VC (fitted $\widehat m$: VC) & 0.0018 & 0.0644 & 0.0964 & 0.3819 & 95.1 \\
\hline

\multirow{7}{*}{\begin{tabular}{@{}l@{}}
$N=2000$\\
$n=1000$\\
$f=0.5$
\end{tabular}}
& Classical (labeled only) & 0.0013 & 0.0885 & 0.1308 & 0.5148 & 95.1 \\
& PPI (oracle $m=m_0$) & 0.0018 & 0.0636 & 0.0941 & 0.3745 & 95.4 \\
& PPI (fixed $m$: step) & 0.0013 & 0.0911 & 0.1347 & 0.5512 & 95.9 \\
& PPI (fixed $m$: sin-only) & 0.0015 & 0.0724 & 0.1076 & 0.3981 & 93.5 \\
& PPI (fitted $\widehat m$: no CF, no VC) & 0.0018 & 0.0635 & 0.0946 & 0.3735 & 95.2 \\
& CF-PPI (fitted $\widehat m$: CF) & 0.0018 & 0.0632 & 0.0951 & 0.3792 & 95.4 \\
& SF-PPI-VC (fitted $\widehat m$: VC) & 0.0018 & 0.0635 & 0.0946 & 0.3755 & 95.3 \\
\hline
\end{tabular}
\end{table*}

\subsection{Simulation experiment 2: tilt + local bumps specification}\label{subsec:Simulation Experiment 2}
\subsubsection{Monte Carlo numerical results at $N = 500$}
Table~\ref{tab:ppi_n500_simexp2} shows results for $N{=}500$. As the labeling fraction $f$ rises from $0.1$ to $0.5$, all procedures show monotonic decreases in MAE, RMSE, and interval length, while biases remain essentially zero. 

Among the misspecified fixed rules, PPI (fixed $m(x)=0.4x$: dom.trend) consistently outperforms PPI (fixed $m(x)=2.2\,\mathbf{1}\{x>0.5\}-0.2\,\mathbf{1}\{x\le 0.5\}$: step) in MAE/RMSE at every $f$, with intervals comparable to the classical estimator and coverage near $95\%$. By contrast, the step rule is the least efficient—largest MAE/RMSE and much longer intervals—though its coverage is typically near or above $95\%$. 

PPI (no CF, no VC) exhibits severe undercoverage at small $f$ (about $42\%$ at $f{=}0.1$, $70\%$ at $f{=}0.2$, and $83\%$ at $f{=}0.3$). Coverage improves with $f$ but remains far below the nominal level even at $f{=}0.5$. The loss of coverage for PPI in Simulation Experiment~2 (tilt $+$ local-bumps design) is generally greater than in Simulation Experiment~1 (periodic $+$ local-bumps design), indicating that the tilt design is more challenging for the naive PPI procedure.

Adding either cross-fitting (CF-PPI) or variance correction (SF-PPI-VC) successfully addresses this issue: both exhibit only mild undercoverage at $f{=}0.1$ (approximately $91\%$) and recover to approximately $95\%$ by $f{=}0.2$; by $f{=}0.3$ they essentially maintain nominal coverage.

\begin{table*}[h!]
\caption{Results of Simulation Experiment 2 at $N=500$ for varying labeled fractions $f=n/N$.}
\label{tab:ppi_n500_simexp2}
\begin{tabular}{@{}llrrrrr@{}}
\hline
Sample size & Method &
\multicolumn{1}{c}{Bias} &
\multicolumn{1}{c}{MAE} &
\multicolumn{1}{c}{RMSE} &
\multicolumn{1}{c}{Int. len.} &
\multicolumn{1}{c@{}}{Cov. (\%)} \\
\hline

\multirow{6}{*}{\begin{tabular}{@{}l@{}}
$N=500$\\
$n=50$\\
$f=0.1$
\end{tabular}}
& Classical (labeled only) & $-$0.0002 & 0.1531 & 0.2299 & 0.8799 & 94.2 \\
& PPI (oracle $m=m_0$) & $-$0.0006 & 0.1039 & 0.1525 & 0.5942 & 94.8  \\
& PPI (fixed $m$: step) & 0.0012 & 0.1723 & 0.2565 & 1.0250 & 95.1 \\
& PPI (fixed $m$: dom.trend) & 0.0002 & 0.1542 & 0.2285 & 0.8799 & 94.5 \\
& PPI (fitted $\widehat m$: no CF, no VC) & 0.0178 & 0.3053 & 0.5196 & 0.5342 & 41.8 \\
& CF-PPI (fitted $\widehat m$: CF) & 0.0090 & 0.3133 & 0.5108 & 1.7689 & 91.3 \\
& SF-PPI-VC (fitted $\widehat m$: VC) & 0.0178 & 0.3053 & 0.5196 & 1.7456 & 91.5 \\
\hline

\multirow{6}{*}{\begin{tabular}{@{}l@{}}
$N=500$\\
$n=100$\\
$f=0.2$
\end{tabular}}
& Classical (labeled only) & 0.0010 & 0.1107 & 0.1624 & 0.6248 & 94.5 \\
& PPI (oracle $m=m_0$) & $-$0.0002 & 0.0793 & 0.1157 & 0.4482 & 94.6 \\
& PPI (fixed $m$: step) & 0.0015 & 0.1215 & 0.1798 & 0.7395 & 96.0 \\
& PPI (fixed $m$: dom.trend) & 0.0010 & 0.1100 & 0.1614 & 0.6268 & 95.0 \\
& PPI (fitted $\widehat m$: no CF, no VC) & 0.0005 & 0.1330 & 0.2156 & 0.4215 & 70.0 \\
& CF-PPI (fitted $\widehat m$: CF) & $-$0.0028 & 0.1484 & 0.2338 & 0.8521 & 93.4 \\
& SF-PPI-VC (fitted $\widehat m$: VC) & 0.0005 & 0.1330 & 0.2156 & 0.7879 & 94.4 \\
\hline

\multirow{6}{*}{\begin{tabular}{@{}l@{}}
$N=500$\\
$n=150$\\
$f=0.3$
\end{tabular}}
& Classical (labeled only) & 0.0002 & 0.0879 & 0.1319 & 0.5107 & 94.6 \\
& PPI (oracle $m=m_0$) & $-$0.0004 & 0.0671 & 0.0997 & 0.3873 & 94.7 \\
& PPI (fixed $m$: step) & 0.0003 & 0.0963 & 0.1442 & 0.6144 & 96.6 \\
& PPI (fixed $m$: dom.trend) & 0.0000 & 0.0865 & 0.1310 & 0.5139 & 94.9 \\
& PPI (fitted $\widehat m$: no CF, no VC) & 0.0009 & 0.0904 & 0.1365 & 0.3735 & 83.4 \\
& CF-PPI (fitted $\widehat m$: CF) & $-$0.0005 & 0.1016 & 0.1496 & 0.5887 & 95.5 \\
& SF-PPI-VC (fitted $\widehat m$: VC) & 0.0009 & 0.0904 & 0.1365 & 0.5225 & 95.2 \\
\hline

\multirow{6}{*}{\begin{tabular}{@{}l@{}}
$N=500$\\
$n=200$\\
$f=0.4$
\end{tabular}}
& Classical (labeled only) & $-$0.0004 & 0.0766 & 0.1141 & 0.4427 & 94.7 \\
& PPI (oracle $m=m_0$) & $-$0.0002 & 0.0605 & 0.0906 & 0.3530 &  94.7 \\
& PPI (fixed $m$: step) & $-$0.0007 & 0.0830 & 0.1227 & 0.5412 & 97.1 \\
& PPI (fixed $m$: dom.trend) & $-$0.0007 & 0.0763 & 0.1134 & 0.4469 & 95.2 \\
& PPI (fitted $\widehat m$: no CF, no VC) & 0.0002 & 0.0731 & 0.1089 & 0.3456 & 89.5 \\
& CF-PPI (fitted $\widehat m$: CF) & $-$0.0016 & 0.0795 & 0.1187 & 0.4773 & 95.7 \\
& SF-PPI-VC (fitted $\widehat m$: VC) & 0.0002 & 0.0731 & 0.1089 & 0.4189 & 94.8 \\
\hline

\multirow{6}{*}{\begin{tabular}{@{}l@{}}
$N=500$\\
$n=250$\\
$f=0.5$
\end{tabular}}
& Classical (labeled only) & $-$0.0008 & 0.0694 & 0.1019 & 0.3961 & 94.9 \\
& PPI (oracle $m=m_0$) & $-$0.0003 & 0.0569 & 0.0850 & 0.3305 & 94.6 \\
& PPI (fixed $m$: step) & $-$0.0011 & 0.0728 & 0.1083 & 0.4918 & 97.7 \\
& PPI (fixed $m$: dom.trend) & $-$0.0009 & 0.0692 & 0.1016 & 0.4010 & 95.1 \\
& PPI (fitted $\widehat m$: no CF, no VC) & 0.0000 & 0.0624 & 0.0942 & 0.3261 & 91.7 \\
& CF-PPI (fitted $\widehat m$: CF) & $-$0.0010 & 0.0684 & 0.1022 & 0.4156 & 95.9 \\
& SF-PPI-VC (fitted $\widehat m$: VC) & 0.0000 & 0.0624 & 0.0942 & 0.3646 & 94.8 \\
\hline
\end{tabular}
\end{table*}

\subsubsection{Monte Carlo numerical results at $N = 2000$}
The results for Simulation Experiment 2 at $N{=}2000$ are reported in Table~\ref{tab:ppi_n2000_simexp2}. As $f$ increases from $0.1$ to $0.5$, all procedures show monotonic declines in MAE, RMSE, and interval length; biases are negligible throughout. The classical estimator maintains coverage near $95\%$ but with the longest intervals. PPI with the oracle predictor ($m{=}m_0$) dominates the classical estimator at every $f$, delivering smaller errors and substantially shorter intervals while keeping coverage at roughly $95\%$. Among the misspecified fixed rules, PPI (fixed $m$: dom.trend) continues to outperform the classical estimator in both accuracy and interval length, whereas PPI (fixed $m$: step) is the least accurate across $f$, though its coverage is near or slightly above nominal.

The learned–predictor variants mirror the $N{=}500$ case. PPI (no CF, no VC) exhibits severe undercoverage at small $f$—about $78.6\%$ at $f{=}0.1$, $89.5\%$ at $f{=}0.2$, and $92.9\%$ at $f{=}0.3$—with coverage improving as $f$ increases and approaching nominal only by $f{=}0.5$. Introducing cross-fitting (CF-PPI) or variance correction (SF-PPI-VC) successfully restores coverage: both show only mild undercoverage at $f{=}0.1$ and are near $95\%$ for $f\!\ge\!0.2$, while retaining markedly shorter intervals than the classical estimator. By $f{=}0.5$, their performance is very close to that of the oracle PPI.

\begin{table*}[h!]
\caption{Results of Simulation Experiment 2 at $N=2000$ for varying labeled fractions $f=n/N$.}
\label{tab:ppi_n2000_simexp2}
\begin{tabular}{@{}llrrrrr@{}}
\hline
Sample size & Method &
\multicolumn{1}{c}{Bias} &
\multicolumn{1}{c}{MAE} &
\multicolumn{1}{c}{RMSE} &
\multicolumn{1}{c}{Int. len.} &
\multicolumn{1}{c@{}}{Cov. (\%)} \\
\hline

\multirow{7}{*}{\begin{tabular}{@{}l@{}}
$N=2000$\\
$n=200$\\
$f=0.1$
\end{tabular}}
& Classical (labeled only) & 0.0002 & 0.0764 & 0.1130 & 0.4423 & 94.6 \\
& PPI (oracle $m=m_0$) & $-$0.0003 & 0.0513 & 0.0757 & 0.2976 & 95.1 \\
& PPI (fixed $m$: step) & 0.0013 & 0.0869 & 0.1287 & 0.5140 & 95.3 \\
& PPI (fixed $m$: dom.trend) & 0.0008 & 0.0754 & 0.1129 & 0.4423 & 94.8 \\
& PPI (fitted $\widehat m$: no CF, no VC) & 0.0005 & 0.0737 & 0.1103 & 0.2697 & 78.6 \\
& CF-PPI (fitted $\widehat m$: CF) & $-$0.0006 & 0.0787 & 0.1169 & 0.4286 & 93.4 \\
& SF-PPI-VC (fitted $\widehat m$: VC) & 0.0005 & 0.0737 & 0.1103 & 0.4206 & 94.6 \\
\hline

\multirow{7}{*}{\begin{tabular}{@{}l@{}}
$N=2000$\\
$n=400$\\
$f=0.2$
\end{tabular}}
& Classical (labeled only) & 0.0002 & 0.0538 & 0.0798 & 0.3131 & 95.3 \\
& PPI (oracle $m=m_0$) & $-$0.0012 & 0.0385 & 0.0575 & 0.2244 & 94.7 \\
& PPI (fixed $m$: step) & 0.0003 & 0.0603 & 0.0893 & 0.3703 & 96.1 \\
& PPI (fixed $m$: dom.trend) & 0.0003 & 0.0538 & 0.0799 & 0.3141 & 95.2 \\
& PPI (fitted $\widehat m$: no CF, no VC) & $-$0.0010 & 0.0436 & 0.0663 & 0.2163 & 89.5 \\
& CF-PPI (fitted $\widehat m$: CF) & $-$0.0013 & 0.0459 & 0.0690 & 0.2673 & 94.8 \\
& SF-PPI-VC (fitted $\widehat m$: VC) & $-$0.0010 & 0.0436 & 0.0663 & 0.2564 & 94.4 \\
\hline

\multirow{7}{*}{\begin{tabular}{@{}l@{}}
$N=2000$\\
$n=600$\\
$f=0.3$
\end{tabular}}
& Classical (labeled only) & $-$0.0003 & 0.0445 & 0.0653 & 0.2558 & 95.1 \\
& PPI (oracle $m=m_0$) & $-$0.0010 & 0.0324 & 0.0492 & 0.1939 & 95.0 \\
& PPI (fixed $m$: step) & $-$0.0002 & 0.0487 & 0.0719 & 0.3076 & 96.9 \\
& PPI (fixed $m$: dom.trend) & $-$0.0002 & 0.0443 & 0.0653 & 0.2574 & 95.4 \\
& PPI (fitted $\widehat m$: no CF, no VC) & $-$0.0008 & 0.0344 & 0.0525 & 0.1904 & 92.9 \\
& CF-PPI (fitted $\widehat m$: CF) & $-$0.0014 & 0.0357 & 0.0541 & 0.2164 & 95.3 \\
& SF-PPI-VC (fitted $\widehat m$: VC) & $-$0.0008 & 0.0344 & 0.0525 & 0.2074 & 94.9 \\
\hline

\multirow{7}{*}{\begin{tabular}{@{}l@{}}
$N=2000$\\
$n=800$\\
$f=0.4$
\end{tabular}}
& Classical (labeled only) & $-$0.0003 & 0.0374 & 0.0560 & 0.2216 & 95.3 \\
& PPI (oracle $m=m_0$) & $-$0.0007 & 0.0297 & 0.0444 & 0.1766 & 95.2 \\
& PPI (fixed $m$: step) & $-$0.0003 & 0.0417 & 0.0615 & 0.2709 & 97.2 \\
& PPI (fixed $m$: dom.trend) & $-$0.0003 & 0.0378 & 0.0561 & 0.2237 & 95.4 \\
& PPI (fitted $\widehat m$: no CF, no VC) & $-$0.0004 & 0.0307 & 0.0462 & 0.1749 & 94.2 \\
& CF-PPI (fitted $\widehat m$: CF) & $-$0.0007 & 0.0316 & 0.0473 & 0.1908 & 95.4 \\
& SF-PPI-VC (fitted $\widehat m$: VC) & $-$0.0004 & 0.0307 & 0.0462 & 0.1834 & 95.2 \\
\hline

\multirow{7}{*}{\begin{tabular}{@{}l@{}}
$N=2000$\\
$n=1000$\\
$f=0.5$
\end{tabular}}
& Classical (labeled only) & $-$0.0003 & 0.0340 & 0.0504 & 0.1983 & 95.2 \\
& PPI (oracle $m=m_0$) & $-$0.0005 & 0.0278 & 0.0417 & 0.1654 & 95.2 \\
& PPI (fixed $m$: step) & $-$0.0003 & 0.0374 & 0.0545 & 0.2461 & 97.9 \\
& PPI (fixed $m$: dom.trend) & $-$0.0004 & 0.0342 & 0.0504 & 0.2007 & 95.6 \\
& PPI (fitted $\widehat m$: no CF, no VC) & $-$0.0004 & 0.0286 & 0.0429 & 0.1644 & 94.1 \\
& CF-PPI (fitted $\widehat m$: CF) & $-$0.0006 & 0.0293 & 0.0436 & 0.1753 & 95.6 \\
& SF-PPI-VC (fitted $\widehat m$: VC) & $-$0.0004 & 0.0286 & 0.0429 & 0.1690 & 94.9 \\
\hline
\end{tabular}
\end{table*}

To summarize, the results of the simulation experiments (Table~\ref{tab:ppi_n500_simexp1}--\ref{tab:ppi_n2000_simexp2}) indicate that: (1) CF-PPI and SF-PPI-VC consistently outperform the classical estimator that relies solely on labeled data across all scenarios, ranging from a small to a modest labeling fraction and under both true function scenarios, highlighting the advantage of utilizing unlabeled data rather than discarding it; (2) a misspecified prediction rule, such as a crude step function or relying only on the dominant trend, may perform worse than the classical estimator, so PPI should be used with a reasonably accurate prediction rule (unbiasedness is guaranteed by design, but efficiency improvement depends on prediction accuracy); and (3) the efficiency of PPI, stemming from its systematic integration of the prediction rule, is best realized when either sample-splitting (CF-PPI) or variance adjustment is employed and the predictor is fitted only once (SF-PPI-VC). These strategies enable more reliable inference, particularly when the labeled sample is substantially smaller than the pool of unlabeled data. Vanilla PPI approaches without these safeguards may lead to undercoverage.

\section{Real data application: Energy Efficiency dataset}\label{sec:Real data applications}
We apply PPI methods to the \emph{Energy Efficiency} dataset, available from the \href{https://archive.ics.uci.edu/dataset/242/energy+efficiency}{UCI Machine Learning Repository}. The dataset comprises \(768\) building designs with eight continuous covariates: \(x_1\) = relative compactness, \(x_2\) = surface area, \(x_3\) = wall area, \(x_4\) = roof area, \(x_5\) = overall height, \(x_6\) = orientation, \(x_7\) = glazing area, and \(x_8\) = glazing area distribution. The data include two response variables, \emph{Heating Load} and \emph{Cooling Load}; in this application we set \(y\) to be \emph{Heating Load}. There are no missing values. This dataset was also analyzed in \cite{lee2026mec}.

We follow a semi-supervised setup, where a subset of observations is randomly designated as labeled and the remainder as unlabeled (covariates only). Specifically, we take \(n=115\) labeled pairs \(\{(X_j, Y_j): j \in S\}\) and \(N=653\) unlabeled covariates \(\{X_i: i=1,\ldots,N\}\), so the labeling fraction is \(f = n/N = 115/653 \approx 0.176\). Our target parameter is the population mean of the response, \(\theta = \mathbb{E}[Y]\) (i.e., the population mean of Heating Load), under the working model \(Y = m(X) + \varepsilon\) with \(\mathbb{E}[\varepsilon \mid X] = 0\), where $m$ is unknown. Because this is a real dataset, the true value of \(\theta\) is unknown; as a reference, we compute the full-sample mean using all 768 observations, obtaining \(\bar{Y}_{\mathrm{full}} = 22.307\), which we use as the ground truth for this application.

We compare the classical estimator $\bar{Y}_{n}$ based on the labeled subset ($n = 115$), the classical estimator $\bar{Y}_{\mathrm{full}}$ computed using the entire dataset (768 observations; used as a reference or pseudo–ground truth), vanilla PPI, CF--PPI, and SF--PPI--VC. Throughout, we use KRR with the same configuration as in the simulation experiments described in Subsection~\ref{subsec:Simulation Setting}.

\begin{figure}[h!]
    \centering
    \includegraphics[width=1\linewidth]{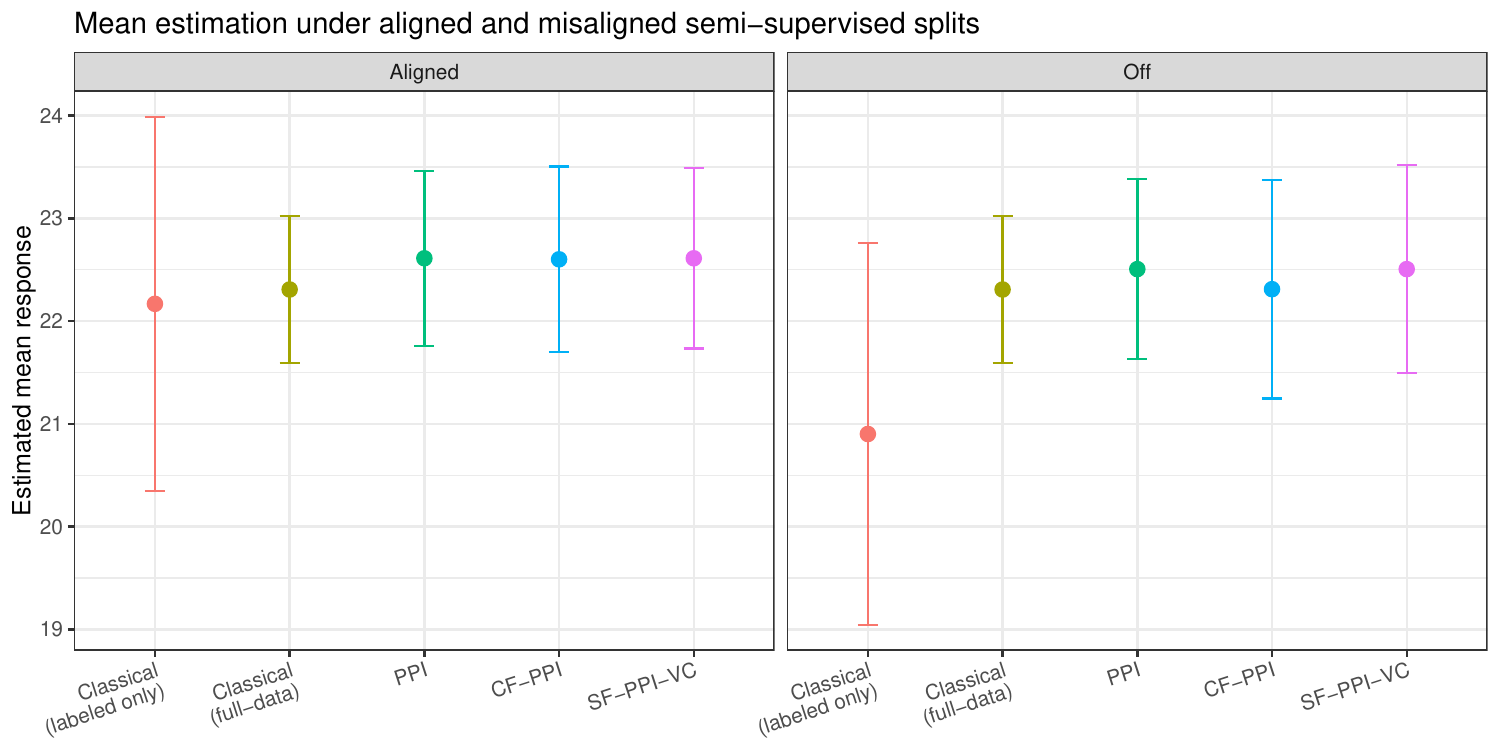}
    \caption{Mean estimation under aligned and misaligned semi-supervised splits in a real-data application for Energy Efficiency dataset. We compare the classical estimator based on the labeled subset, the classical estimator computed using the full dataset (used as a reference), and three PPI-based debiased estimators: vanilla PPI, CF--PPI, and SF--PPI--VC.
The labeled set is obtained by random subsampling, leading to two representative cases:
an \emph{aligned} case (left), in which the labeled-only estimate is close to the full-data reference, and a \emph{misaligned} case (right), in which the two differ due to sampling variability.
Error bars denote 95\% confidence intervals.
PPI-based methods produce estimates closer to the reference value with confidence intervals that are tighter than those of the labeled-only estimator.}
    \label{fig:real_data_application_plot}
\end{figure}




Because the labeled set is constructed by random subsampling from the full dataset, the labeled-only estimator $\bar{Y}_{n}$ may differ from the full-data estimator $\bar{Y}_{\mathrm{full}}$ due to sampling variability. To examine the robustness of the bias-correction mechanisms employed by PPI-based debiased estimators (vanilla PPI, CF--PPI, and SF--PPI--VC), we consider two representative scenarios: one in which $\bar{Y}_{n}$ is close to $\bar{Y}_{\mathrm{full}}$, and another in which the two quantities differ appreciably.

A desirable debiased estimator should satisfy two criteria: (i) its point estimate should be close to the reference value $\bar{Y}_{\mathrm{full}} = 22.307$, and (ii) its 95\% confidence interval should be narrower than that of the labeled-only estimator, $\bar{Y}_{n} \pm 1.96\,\widehat{\mathrm{SE}}_{n}$, while remaining no tighter than the benchmark interval based on the full dataset, $\bar{Y}_{\mathrm{full}} \pm 1.96\,\widehat{\mathrm{SE}}_{\mathrm{full}}$. The simultaneous appearance of these properties provides evidence that the debiasing mechanism is functioning as intended.

Figure~\ref{fig:real_data_application_plot} summarizes the results. In the aligned case (left panel), all estimators produce similar point estimates, with modest differences in interval width. The debiased methods yield confidence intervals that are consistently tighter than those of the classical labeled-only estimator. Among these, vanilla PPI attains the narrowest intervals, which may reflect overfitting issue due to the label reuse. In the misaligned case (right panel), the debiased estimators remain centered near the reference value, whereas the labeled-only estimator exhibits a noticeable deviation. Because this analysis is based on a single real-data instance, direct comparisons between CF--PPI and SF--PPI--VC are necessarily limited.

\section{Conclusion}\label{sec:Conclusion}
We studied the statistical optimality of PPI under both fixed and learned prediction rules, together with its associated semiparametric efficiency theory, from first principles. This work addressed a critical gap in the modern ML literature, where semi-supervised inference procedures have increasingly been built around PPI. While PPI and CF--PPI appeared in prior work \citep{angelopous2023,zrnic2024cross}, the theoretical analysis developed in this paper relied on novel empirical process and $M$-estimation techniques that were not employed in the original PPI or CF--PPI studies. Moreover, we derived the semiparametric efficiency bound for PPI primarily via the Riesz representation theorem, which enabled us to characterize the conditions under which efficiency was attained. Finally, we introduced SF--PPI--VC, a new variant of PPI that enabled valid inference without cross-fitting through an explicit variance correction and constituted the main methodological contribution of this work. Simulation studies and real-data applications supported our theoretical findings.

\begin{appendix}

\section{Proofs of theoretical results in the main document}
\thispagestyle{empty} 

\subsection{Preliminaries: setup, notation, and roadmap}\label{subsec:Preliminaries: Assumptions and Setup.}
Let $U(\theta;x,y)\in\mathbb{R}^p$ be a measurable full–data estimating function. 
Examples include the population mean ($p=1$, $\theta=\mathbb{E}[Y]$, $U(\theta;x,y)=y-\theta$), 
linear regression ($\theta\in\mathbb{R}^p$, $U(\beta;x,y)=x\{y-x^\top\theta\}$), 
logistic regression with the canonical link ($\theta\in\mathbb{R}^p$, $U(\beta;x,y)=x\{y-\mathrm{expit}(x^\top\theta)\}$, where $\mathrm{expit}(t)=1/(1+e^{-t})$), 
and the $\tau$–quantile of $Y$ ($p=1$, $\theta$, $U(\theta;x,y)=\tau-\mathbf{1}\{y\le \theta\}$; for example, the median corresponds to $\tau=1/2$).

Define the super–population score, the finite–population (oracle) score, and the prediction-powered inference (PPI) score \citep{angelopous2023} by
\begin{align*}
U(\theta)
&:= \mathbb{E}\{U(\theta;X,Y)\}
= \int U(\theta;x,y)\, dP_0(x,y),\\[2mm]
U_N(\theta)
&:= \frac{1}{N}\sum_{i=1}^N U(\theta;X_i,Y_i),\\[2mm]
\widehat U_{\mathrm{PPI}}(\theta)
&:= \frac{1}{N}\sum_{i=1}^N U\!\big(\theta;X_i,m(X_i)\big)
\;+\; \frac{1}{n}\sum_{j\in S}\!\Big\{U(\theta;X_j,Y_j)-U\!\big(\theta;X_j,m(X_j)\big)\Big\}.
\end{align*}
Note that $U(\theta)$ is the population moment condition under $P_0$, 
$U_N(\theta)$ is its finite–sample analogue based on the full data, 
and $\widehat U_{\mathrm{PPI}}(\theta)$ is the computable version that combines predictions $m(X_i)$ for all units with residual corrections from the labeled set $S$.

The corresponding roots are
\[
\theta_0 \in \{\theta:\ U(\theta)=0\}, \qquad
\theta_N \in \{\theta:\ U_N(\theta)=0\}, \qquad
\widehat\theta_{\mathrm{PPI}} \in \{\theta:\ \widehat U_{\mathrm{PPI}}(\theta)=0\}.
\]
Here, the inferential target (i.e., estimand) of PPI is $\theta_0$, the unique solution to the super–population moment condition $U(\theta)=0$. 
The quantity $\theta_N$ is the finite–population (oracle) analogue based on the full data, and $\widehat\theta_{\mathrm{PPI}}$ is the feasible estimator constructed from predictions and labeled residuals.


We begin by stating the regularity conditions for the PPI score, the finite–population (oracle) score, and the super–population score, together with their solutions (roots).

\paragraph{Assumption 1 (Design–level regularity).}\label{ass:design} 
\begin{enumerate}\itemsep4pt
\item[\textbf{(a)}] \textbf{Identification:} $U_N(\theta)$ has a unique zero $\theta_N\in\Theta$.
\item[\textbf{(b)}] \textbf{Uniform score consistency:} 
\[
\sup_{\theta\in\Theta}\,\big\|\widehat U_{\mathrm{PPI}}(\theta)-U_N(\theta)\big\|_2=o_p(1).
\]
Here, $o_p(1)$ and probability statements below are with respect to the
sampling design, conditional on the finite population $\mathcal F_N$.
\item[\textbf{(c)}] \textbf{Jacobian stability \& nonsingularity:}
\[
\sup_{\theta\in\Theta}\,\big\|\partial_\theta \widehat U_{\mathrm{PPI}}(\theta)-\partial_\theta U_N(\theta)\big\|_{\op}=o_p(1),
\qquad
I_N(\theta_N):=-\partial_\theta U_N(\theta_N)\ \text{ is nonsingular.}
\]
\end{enumerate}

\paragraph{Assumption 2 (Super–population regularity).}\label{ass:super}
Assume $(X_i,Y_i)\stackrel{\mathrm{i.i.d.}}{\sim}P_0$. Let $U(\theta):=\mathbb E\{U(\theta;X,Y)\}$.  
For every compact $K\subset\Theta$,
\[
\sup_{\theta\in K}\big\|U_N(\theta)-U(\theta)\big\|_2\xrightarrow{p}0.
\]
The map $U(\theta)$ has a unique zero $\theta_0\in\Theta$, and $U(\theta;x,y)$ is continuously
differentiable in a neighborhood of $\theta_0$ with
\[
I(\theta_0):=-\mathbb E\{\partial_\theta U(\theta_0;X,Y)\}\ \text{nonsingular.}
\]

\paragraph{Assumption 3 (Second moment regularity).}\label{ass:clt}
Finite second moments exist and a finite-population central limit theorem (CLT) applies under simple random sampling without replacement (SRSWOR):
\[
\mathbb E\big\|U(\theta_0;X,Y)\big\|_2^2<\infty,\quad
\mathbb E\big\|\Delta(\theta_0;X,Y)\big\|_2^2<\infty,\quad
\Delta(\theta;x,y):=U(\theta;x,y)-U\big(\theta;x,m(x)\big).
\]
With $n,N\to\infty$ and $n/N\to f\in(0,1)$, the centered averages of
$U(\theta_0;X_i,Y_i)$ and $\Delta(\theta_0;X_i,Y_i)$ satisfy a multivariate CLT under SRSWOR.

Before turning to the main results, we record several clarifying remarks on the assumptions and setup.

\paragraph{Asymptotic notation.}
Unless stated otherwise, limits are taken as the sample size \(N\to\infty\) (and \(n=n(N)\to\infty\) when present).
For deterministic sequences \(a_N,b_N>0\), we write \(a_N=O(b_N)\) if \(\sup_N a_N/b_N<\infty\), and \(a_N=o(b_N)\) if \(a_N/b_N\to 0\).
For random quantities \(X_N\) and positive scales \(a_N\), 
\(X_N=O_p(a_N)\) means \(X_N/a_N\) is tight (bounded in probability), and \(X_N=o_p(a_N)\) means \(X_N/a_N \xrightarrow{p} 0\).
We use \(\xrightarrow{p}\) and \(\xrightarrow{d}\) to denote convergence in probability and in distribution, respectively.

\paragraph{Labeling fraction.}
We define the labeling fraction by \(f_N := n/N\) and assume \(f_N \to f \in (0,1)\) as \(N \to \infty, \, n = n(N) \to \infty\) for a fixed constant \(f\). Under SRSWOR, \(f\) is the design inclusion probability (labeling propensity score) and is typically known. For notational convenience, we suppress the subscript and write simply \(f\) for \(f_N\) as this is clear from context. Throughout, we assume $f$ is bounded away from both $0$ and $1$. If $f \to 0$, only a vanishing proportion of units are labeled, and inference requires a different rare-label asymptotic regime, which is outside the scope of this work.
If $f \to 1$, the problem reduces to the fully supervised setting, in which case the rectifier is negligible; our theory extends to this case trivially, but it is not of practical interest for semi-supervised inference.
Hence, the intermediate regime $f \in (0,1)$ is the most relevant in practice for semi-supervised inference \citep{zhu2009introduction,sohn2020fixmatch,song2024general}.

\paragraph{Assumptions 1–3.}
Assumptions 1–3—design-level regularity, super-population regularity, and second-moment regularity—are mild and standard conditions used to establish the consistency and asymptotic normality of the PPI estimator and, more broadly, of estimators in semi-supervised learning; see \citep{song2024general,angelopous2023,fisch2024stratified}.
We invoke these assumptions only in Subsections~\ref{subsec:Consistency of the PPI Estimator} and~\ref{subsec:General $M$-Estimation Theory of Prediction-Powered Inference}, which analyze the baseline PPI under a fixed prediction rule \(m\) and develop consistency and a general \(M\)-estimation CLT.
By contrast, the semiparametric efficiency analysis and the asymptotic theory for learned predictors rely on self-contained assumption sets introduced in Subsection~\ref{subsec:Semiparametric Efficiency Theory of Prediction-Powered Inference} and in Subsections~\ref{subsec:Asymptotic Properties of CF-PPI} and~\ref{subsec:Asymptotic Properties of SF-PPI-VC}.

\paragraph{Empirical-process notation.}
In the proofs of Subsections \ref{subsec:Asymptotic Properties of CF-PPI} and \ref{subsec:Asymptotic Properties of SF-PPI-VC}, we adopt standard empirical‐process notation \citep{geer2000empirical,pollard1989asymptotics,kennedy2016semiparametric}.
Let $P$ denote the super–population distribution of $X$. Define the empirical measures
\[
P_N := \frac{1}{N}\sum_{i=1}^N \delta_{X_i},
\qquad
P_n := \frac{1}{n}\sum_{j\in S} \delta_{X_j},
\]
where $\delta_z$ is the Dirac measure at $z$ and $S\subset\{1,\ldots,N\}$ with $|S|=n$.
For any measurable $h$,
\begin{align*}
P_N h &:= \int h\,dP_N = \frac{1}{N}\sum_{i=1}^N h(X_i),\\
P_n h &:= \int h\,dP_n = \frac{1}{n}\sum_{j\in S} h(X_j),\\
P h &:= \int h\,dP = \mathbb{E}_P[h(X)].    
\end{align*}
When a function depends on both covariates and outcomes, we use the same notation with the obvious modification; e.g.,
\[
P_n h = \frac{1}{n}\sum_{j\in S} h(X_j,Y_j), \qquad P h = \mathbb{E}_P[h(X,Y)].
\]
For example, if $h(x)=\mathbf{1}\{x\in A\}$, then $P_n h$ is the empirical probability of the event $A$ within the labeled subset $S$.
By the law of large numbers, $P_N h \to P h$ and (under SRSWOR with $n/N\to f\in(0,1)$) $P_n h \to P h$ in probability. This notation streamlines decompositions of estimators and estimands and the statement of convergence results. 

Figure \ref{fig:roadmap} organizes the main theoretical results by how the prediction rule \(m\) is treated—fixed, treated as a nuisance (model-level), or learned—and shows where each result is proved.
\begin{figure}[h]
  \centering
  \includegraphics[width=0.8\columnwidth]{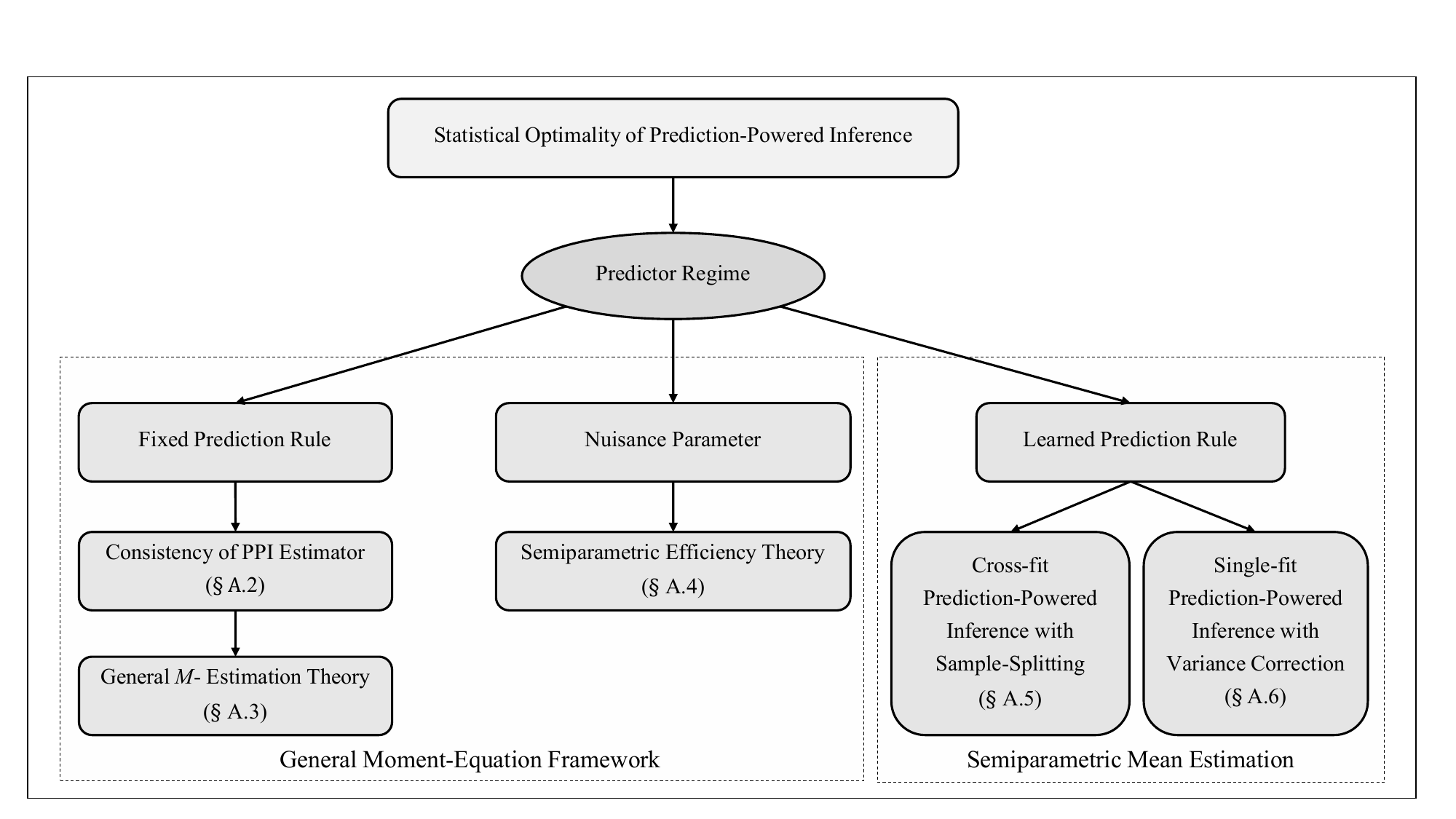}
  \caption{Roadmap of theoretical results by predictor regime (i.e., how the prediction rule \(m\) is treated).
  Left: with a fixed prediction rule, we establish consistency of the PPI estimator (\S~\ref{subsec:Consistency of the PPI Estimator}) and a general \(M\)-estimation CLT (\S~\ref{subsec:General $M$-Estimation Theory of Prediction-Powered Inference}).
  Center: treating \(m\) as a nuisance parameter yields the semiparametric efficiency theory (\S~\ref{subsec:Semiparametric Efficiency Theory of Prediction-Powered Inference}).
  Right: with a learned prediction rule, we analyze cross-fit PPI with sample-splitting (\S~\ref{subsec:Asymptotic Properties of CF-PPI}) and single-fit PPI with variance correction (\S~\ref{subsec:Asymptotic Properties of SF-PPI-VC}).}
  \label{fig:roadmap}
\end{figure}



\paragraph{Convex risk minimization.} In this paper, we mainly discuss the PPI estimator as a solution to the moment equation \eqref{eq:pop-moment}. It is important to note that any solution to the moment equation \eqref{eq:pop-moment} can also be characterized as a minimizer of a convex risk, and the authors of the original work on PPI \citep{angelopous2023} mainly discuss PPI as the minimizer of a convex risk.

More specifically, if there exists a measurable loss \(L:\Theta\times\mathcal X\times\mathcal Y\to\mathbb R\) that is convex in \(\theta\) and satisfies \(U(\theta;x,y)\in\partial_\theta L(\theta;x,y)\) for each \((x,y)\), then, under standard conditions allowing interchange of expectation and subdifferential, \(0\in\partial_\theta\,\E\!\big[L(\theta;X,Y)\big]
= \E\!\big[\partial_\theta L(\theta;X,Y)\big]
= \E\!\big[U(\theta;X,Y)\big]\). Hence any root \(\theta_0\) of the moment equation also satisfies
\(\theta_0\in\arg\min_{\theta\in\Theta}\E[L(\theta;X,Y)]\). Conversely, if \(\theta\mapsto\E[L(\theta;X,Y)]\) is strictly convex, its unique minimizer \(\theta_0\) satisfies \(\E[U(\theta_0;X,Y)]=0\).

\subsection{Consistency of the PPI estimator}\label{subsec:Consistency of the PPI Estimator}
The consistency of the PPI estimator was established in Subsection~\ref{subsec:ppi-consistency} of the main document. Here, we provide proofs of the propositions that were stated there without proof.

\begin{proposition}[Design--unbiasedness for an arbitrary fixed predictor]\label{thm:ppi-unbiased-detailed}
Let $\mathcal F_N=\{(X_i,Y_i)\}_{i=1}^N$ be a finite population and let
$U:\Theta\times\mathcal X\times\mathcal Y\to\mathbb R^p$ be measurable.
Fix any prediction rule $m:\mathcal X\to\mathcal Y$ (not necessarily $m_0(x)=\mathbb E[Y\mid X=x]$). Let $\widetilde U_i(\theta):=U\!\big(\theta; X_i, m(X_i)\big)$
denote the plug-in moment function, obtained by replacing the outcome $Y_i$ with its prediction $m(X_i)$, for each $i=1,\ldots,N$.

For $\theta\in\Theta$ define the measure of fit term
\[
m_\theta:=\frac{1}{N}\sum_{i=1}^N \widetilde U_i(\theta).
\]
Draw a labeled set $S\subset\{1,\dots,N\}$ of size $n$ by simple random sampling without replacement (SRSWOR),
and define the rectifier term
\[
\Delta_\theta:=\frac{1}{n}\sum_{j\in S}\Big\{U(\theta;X_j,Y_j)-\widetilde U_j(\theta)\Big\}.
\]
The PPI score is $\widehat U_{\mathrm{PPI}}(\theta):=m_\theta+\Delta_\theta$. Then, conditioning on $\mathcal F_N$,
\[
\mathbb E\!\left[\widehat U_{\mathrm{PPI}}(\theta)\mid\mathcal F_N\right]
=\frac{1}{N}\sum_{i=1}^N U(\theta;X_i,Y_i)
=U_N(\theta).
\]
That is, the PPI score is design-unbiased for the finite-population moment $U_N(\theta)$ under SRSWOR. In particular, if $\theta_N$ solves $U_N(\theta)=0$, then $\mathbb E[\widehat U_{\mathrm{PPI}}(\theta_N)\mid\mathcal F_N]=0$.
\end{proposition}

\begin{proof}
Introduce the sample–membership indicators $\delta_i:=\mathbf 1\{i\in S\}$ and the pointwise difference
\[
\Delta_i(\theta):=U(\theta;X_i,Y_i)-\widetilde U_i(\theta),\qquad i=1,\dots,N.
\]
Then the rectifier can be rewritten as a population sum weighted by indicators:
\begin{equation}\label{eq:rectifier-indicator}
\Delta_\theta=\frac{1}{n}\sum_{j\in S} \Delta_j(\theta)
=\frac{1}{n}\sum_{i=1}^N \delta_i\,\Delta_i(\theta).
\end{equation}
Under SRSWOR of size $n$, each unit has equal inclusion probability
\[
\mathbb E(\delta_i\mid\mathcal F_N)=\mathbb{P}(i\in S\mid\mathcal F_N)=\frac{n}{N},\qquad i=1,\dots,N.
\]
Taking conditional expectation of \eqref{eq:rectifier-indicator} and using linearity,
\begin{align*}
\mathbb E(\Delta_\theta\mid\mathcal F_N)
&=\frac{1}{n}\sum_{i=1}^N \mathbb E(\delta_i\mid\mathcal F_N)\,\Delta_i(\theta)
=\frac{1}{n}\sum_{i=1}^N \frac{n}{N}\,\Delta_i(\theta)
=\frac{1}{N}\sum_{i=1}^N \Delta_i(\theta)\\
&=\frac{1}{N}\sum_{i=1}^N \Big\{U(\theta;X_i,Y_i)-\widetilde U_i(\theta)\Big\}.
\end{align*}
Therefore,
\begin{align*}
\mathbb E\!\left[\widehat U_{\mathrm{PPI}}(\theta)\mid\mathcal F_N\right]
&= m_\theta + \mathbb E(\Delta_\theta\mid\mathcal F_N)\\
&=\frac{1}{N}\sum_{i=1}^N \widetilde U_i(\theta)
+ \frac{1}{N}\sum_{i=1}^N \Big\{U(\theta;X_i,Y_i)-\widetilde U_i(\theta)\Big\}\\
&=\frac{1}{N}\sum_{i=1}^N \Big[
\cancel{\widetilde U_i(\theta)} + U(\theta;X_i,Y_i) - \cancel{\widetilde U_i(\theta)}
\Big]
=\frac{1}{N}\sum_{i=1}^N U(\theta;X_i,Y_i)=U_N(\theta).
\end{align*}
No property of $m$ beyond being fixed (measurable) given $\mathcal F_N$ was used; in particular, $m$ need not equal the oracle regression \(m_0(x)=\mathbb{E}[Y\mid X=x]\).
\end{proof}

\begin{proposition}[Rectifier in Horvitz--Thompson form; design unbiasedness and variance]
Consider the rectifier
\[
\Delta_\theta \;:=\; \frac{1}{n}\sum_{j\in S}\Delta_j(\theta),
\qquad
\Delta_i(\theta):=U(\theta;X_i,Y_i)-U(\theta;X_i,m(X_i)),
\]
where $S\subset\{1,\dots,N\}$ is drawn by SRSWOR of size $n$. Let $\mathcal F_N=\{(X_i,Y_i)\}_{i=1}^N$ be a finite population. Let $\delta_i=\mathbf 1\{i\in S\}$ and $f:=n/N$. Let $m$ be any fixed predictor.

Then:
\begin{enumerate}
\item[\emph{\textbf{(a)}}] \textbf{\emph{Horvitz--Thompson identity.}} The rectifier can be expressed in the Horvitz--Thompson form \citep{horvitz1952}:
\[
\Delta_\theta \;=\; \frac{1}{N}\sum_{i=1}^N \frac{\delta_i}{f}\,\Delta_i(\theta).
\]
\item[\emph{\textbf{(b)}}] \textbf{\emph{Design unbiasedness.}} The rectifier is design--unbiased for the finite--population mean:
\[
\mathbb E\!\left[\Delta_\theta \,\middle|\, \mathcal F_N\right]
\;=\;
\overline{\Delta}_N(\theta)
\;:=\;
\frac{1}{N}\sum_{i=1}^N \Delta_i(\theta).
\]
\item[\emph{\textbf{(c)}}] \textbf{\emph{Design variance under SRSWOR.}} The rectifier has the following finite--population variance:
\[
\Var\!\left(\Delta_\theta \,\middle|\, \mathcal F_N\right)
\;=\;
\frac{1-f}{n}\cdot\frac{1}{N-1}
\sum_{i=1}^N \big\{\Delta_i(\theta)-\overline{\Delta}_N(\theta)\big\}
\big\{\Delta_i(\theta)-\overline{\Delta}_N(\theta)\big\}^{\!\top}.
\]
\end{enumerate}
\end{proposition}
\paragraph{Remark -- Implications for PPI.} 
The rectifier has Horvitz--Thompson form and is design–unbiased for the finite–population residual moment for any $m$; only its variance depends on $m$, through the dispersion of $\{\Delta_i(\theta)\}$. Thus $m$ can improve efficiency but cannot induce design bias.

\begin{proof}
$\quad$

\textbf{(a)} By definition,
\[
\Delta_\theta
= \frac{1}{n}\sum_{j\in S}\Delta_j(\theta)
= \frac{1}{n}\sum_{i=1}^N \delta_i\,\Delta_i(\theta)
= \sum_{i=1}^N \!\Big(\frac{1}{n}\delta_i\Big)\Delta_i(\theta)
= \sum_{i=1}^N \!\Big(\frac{1}{N}\frac{\delta_i}{f}\Big)\Delta_i(\theta),
\]
since $f=n/N$.

\textbf{(b)} Conditioning on $\mathcal F_N$, the quantities $\Delta_i(\theta)$ are fixed (all randomness comes from the SRSWOR indicators $\delta_i$). For SRSWOR of size $n$ from $\{1,\dots,N\}$,
\[
\mathbb E(\delta_i\mid\mathcal F_N)=f=\frac{n}{N},
\]
Taking conditional expectation and using linearity on the rectifier gives
\[
\mathbb E[\Delta_\theta\mid\mathcal F_N]
= \frac{1}{N}\sum_{i=1}^N \frac{\mathbb E[\delta_i\mid\mathcal F_N]}{f}\,\Delta_i(\theta)
= \frac{1}{N}\sum_{i=1}^N \Delta_i(\theta)
= \overline{\Delta}_N(\theta).
\]

\textbf{(c)} For clarity of exposition, we present the proof for the scalar case $p=1$; the vector case $p>1$ follows by replacing squares with outer products. We know that, by conditioning on $\mathcal F_N$, the quantities $\Delta_i(\theta)$ are fixed, hence, $\mathbb E(\delta_i\mid\mathcal F_N)=f=n/N$ and $\Var(\delta_i\mid\mathcal F_N)=f(1-f)$. Moreover, by exchangeability and the identity $\sum_{i=1}^N\delta_i=n$ (deterministic),
\begin{align*}
0
&= \Var\!\Big(\sum_{i=1}^N \delta_i \,\Big|\, \mathcal F_N\Big)
 = \sum_{i=1}^N \Var(\delta_i \mid \mathcal F_N)
   + 2\sum_{i<k}\Cov(\delta_i,\delta_k \mid \mathcal F_N) \\
&= N\,f(1-f) \;+\; 2\sum_{i<k}\Cov(\delta_i,\delta_k \mid \mathcal F_N) \\
&= N\,f(1-f) \;+\; 2 \binom{N}{2}\,\Cov(\delta_i,\delta_k \mid \mathcal F_N)
\qquad\text{(all off–diagonals equal by symmetry)} \\
&= N\,f(1-f) \;+\; N(N-1)\,\Cov(\delta_i,\delta_k \mid \mathcal F_N).
\end{align*}
which implies (since all off–diagonal covariances are equal)
\[
\Cov(\delta_i,\delta_k\mid\mathcal F_N)
\;=\; -\,\frac{f(1-f)}{N-1}\qquad(i\neq k).
\]

In general, the following variance identity holds
$$\Var\!\Big(\sum_{i=1}^N a_i Z_i\Big)
= \sum_{i=1}^N a_i^2\,\Var(Z_i)
+ 2\sum_{i<k} a_i a_k\,\Cov(Z_i,Z_k).$$ Thus, the variance of the rectifier given $\mathcal F_N$ is
\begin{align*}
\Var(\Delta_\theta\mid\mathcal F_N)
&= \Var\!\left(\frac{1}{n}\sum_{i=1}^N \delta_i\,\Delta_i(\theta)\,\Big|\,\mathcal F_N\right) = \frac{1}{n^2}\Var\!\left(\sum_{i=1}^N \Delta_i(\theta)\, \delta_i\,\Big|\,\mathcal F_N\right) \\
&= \frac{1}{n^2}\left\{
\sum_{i=1}^N \Delta_i(\theta)^2\,\Var(\delta_i\mid\mathcal F_N)
\;+\; 2\sum_{i<k} \Delta_i(\theta)\Delta_k(\theta)\,\Cov(\delta_i,\delta_k\mid\mathcal F_N)
\right\} \\
&= \frac{1}{n^2}\left\{
\sum_{i=1}^N f(1-f)\,\Delta_i(\theta)^2
\;+\; 2\sum_{i<k}\!\left(-\frac{f(1-f)}{N-1}\right)\Delta_i(\theta)\Delta_k(\theta)
\right\}\\
&= \frac{f(1-f)}{n^2}\left\{
\sum_{i=1}^N \,\Delta_i(\theta)^2
\;-\; \frac{2}{N-1}\sum_{i<k}\!\Delta_i(\theta)\Delta_k(\theta)
\right\}.
\end{align*}
Use
\(
\sum_{i<k} x_i x_k
=
\frac{1}{2}\left[
\left(\sum_{i=1}^N x_i\right)^{\!2}
- \sum_{i=1}^N x_i^2
\right]
\)
to obtain
\begin{align*}
\Var(\Delta_\theta\mid\mathcal F_N)
&= \frac{f(1-f)}{n^2}\left[
\sum_{i=1}^N \Delta_i(\theta)^2
- \frac{1}{N-1}\Big\{ \big(\sum_{i=1}^N\Delta_i(\theta)\big)^{\!2} - \sum_{i=1}^N \Delta_i(\theta)^2 \Big\}
\right] \\
&= \frac{f(1-f)}{n^2}\left[
\sum_{i=1}^N \Delta_i(\theta)^2
- \frac{1}{N-1}\big(\sum_{i=1}^N\Delta_i(\theta)\big)^{\!2}
+ \frac{1}{N-1}\sum_{i=1}^N \Delta_i(\theta)^2
\right] \\
&= \frac{f(1-f)}{n^2}\left[
\frac{N}{N-1}\sum_{i=1}^N \Delta_i(\theta)^2
- \frac{1}{N-1}\big(\sum_{i=1}^N\Delta_i(\theta)\big)^{\!2}
\right] \\
&= \frac{f(1-f)}{n^2}\cdot \frac{N}{N-1}
\left[
\sum_{i=1}^N \Delta_i(\theta)^2
- \frac{1}{N}\Big(\sum_{i=1}^N\Delta_i(\theta)\Big)^{\!2}
\right] \\
&= \frac{1-f}{n}\cdot \frac{1}{N-1}
\left[
\sum_{i=1}^N \Delta_i(\theta)^2
- \frac{1}{N}\Big(\sum_{i=1}^N\Delta_i(\theta)\Big)^{\!2}
\right]
\qquad(\text{since } f=\tfrac{n}{N}) \\
&= \frac{1-f}{n}\cdot \frac{1}{N-1}
\sum_{i=1}^N \big\{\Delta_i(\theta)-\overline{\Delta}_N(\theta)\big\}^{2},
\end{align*}
where $\overline{\Delta}_N(\theta)=\frac{1}{N}\sum_{i=1}^N\Delta_i(\theta).$
\end{proof}

\subsection{General $M$-estimation theory of prediction-powered inference}\label{subsec:General $M$-Estimation Theory of Prediction-Powered Inference}
This subsection provides the proofs of the lemmas and theorems for the general $M$-estimation theory of PPI, and it proves Theorem~3.1 in the main document.


\begin{lemma}[Oracle expansion of the full–data root]\label{thm:oracle-root-expansion}
Let $\theta_0$ be the unique solution to $\mathbb{E}\{U(\theta;X,Y)\}=0$, and let
$\theta_N\in\Theta$ solve the full–data moment equation $U_N(\theta)=N^{-1}\sum_{i=1}^N $ $U(\theta;X_i,Y_i)=0$.
Assume the super–population regularity and second–moment conditions (Assumptions~2–3).
Then, as $N\to\infty$,
\[
\theta_N-\theta_0
=
I(\theta_0)^{-1}\,\frac1N\sum_{i=1}^N U(\theta_0;X_i,Y_i)
\;+\; o_p(N^{-1/2}),
\]
equivalently,
\[
\sqrt{N}\,(\theta_N-\theta_0)
=
I(\theta_0)^{-1}\,\frac1{\sqrt{N}}\sum_{i=1}^N U(\theta_0;X_i,Y_i)
\;+\; o_p(1),
\]
where $I(\theta_0):=-\mathbb{E}\{\partial_\theta U(\theta_0;X,Y)\}$.
\end{lemma}

\begin{proof}
The proof is standard and is omitted here; see \citep{yuan1998} for a rigorous proof.
\end{proof}

\begin{lemma}[Representation of the PPI score at the oracle root]\label{lem:PPI-at-thetaN}
Fix the finite population $\mathcal F_N=\{(X_i,Y_i)\}_{i=1}^N$, a (fixed) prediction rule $m$, and define
\[
\Delta_i(\theta)\;:=\;U(\theta;X_i,Y_i)-U(\theta;X_i,m(X_i)) .
\]
Let $S\subset\{1,\ldots,N\}$ be the labeled set of size $n$, write $\delta_i=\mathbf 1\{i\in S\}$ and $f:=n/N$, and define
\[
\widehat U_{\mathrm{PPI}}(\theta)
=\frac1N\sum_{i=1}^N U(\theta;X_i,m(X_i))
+\frac1n\sum_{j\in S}\!\Big\{U(\theta;X_j,Y_j)-U(\theta;X_j,m(X_j))\Big\}.
\]
Let $U_N(\theta):=N^{-1}\sum_{i=1}^N U(\theta;X_i,Y_i)$ and let $\theta_N\in\Theta$ solve $U_N(\theta_N)=0$. Then
\begin{equation}\label{eq:PPI-repr}
\widehat U_{\mathrm{PPI}}(\theta_N)
=\frac1N\sum_{i=1}^N\!\Big(\frac{N}{n}\,\delta_i-1\Big)\,\Delta_i(\theta_N)
=\frac1N\sum_{i=1}^N\!\Big(\frac{\delta_i}{\,f\,}-1\Big)\,\Delta_i(\theta_N).
\end{equation}
\end{lemma}

\begin{proof}
Insert indicators to express the labeled average as a population average:
\[
\widehat U_{\mathrm{PPI}}(\theta)
=\frac1N\sum_{i=1}^N U(\theta;X_i,m(X_i))
+\frac1N\sum_{i=1}^N\Big(\frac{N}{n}\,\delta_i\Big)\,\Delta_i(\theta).
\]
Evaluate at $\theta=\theta_N$ and add–subtract $N^{-1}\sum_{i=1}^N U(\theta_N;X_i,Y_i)$:
\begin{align*}
\widehat U_{\mathrm{PPI}}(\theta_N)
&=\frac1N\sum_{i=1}^N\!\Big[U(\theta_N;X_i,m(X_i))-U(\theta_N;X_i,Y_i)\Big]\\
  &+\underbrace{\frac1N\sum_{i=1}^N U(\theta_N;X_i,Y_i)}_{=\,U_N(\theta_N)=0}
  +\frac1N\sum_{i=1}^N\!\Big(\frac{N}{n}\,\delta_i\Big)\Delta_i(\theta_N) \\
&=\frac1N\sum_{i=1}^N\!\Big(\frac{N}{n}\,\delta_i-1\Big)\,\Delta_i(\theta_N).
\end{align*}
Since $(N/n)\delta_i-1=\delta_i/f-1$, this gives \eqref{eq:PPI-repr}.
\end{proof}

\begin{lemma}[First–order expansion of the PPI root around the oracle root]\label{lem:ppi-expansion}
Fix the finite population $\mathcal F_N=\{(X_i,Y_i)\}_{i=1}^N$, a (fixed) prediction rule $m$, and define
\[
\Delta_i(\theta)\;:=\;U(\theta;X_i,Y_i)-U(\theta;X_i,m(X_i)).
\]
Let $S\subset\{1,\ldots,N\}$ be a simple random sample without replacement of size $n$,
write $\delta_i=\mathbf 1\{i\in S\}$ and $f:=n/N\in(0,1)$, and let
\[
U_N(\theta):=\frac1N\sum_{i=1}^NU(\theta;X_i,Y_i),\qquad U_N(\theta_N)=0.
\]
Let $\hat\theta_{\mathrm{PPI}}$ be any solution of $\widehat U_{\mathrm{PPI}}(\theta)=0$, where
\[
\widehat U_{\mathrm{PPI}}(\theta)
=\frac1N\sum_{i=1}^N U(\theta;X_i,m(X_i))
+\frac1n\sum_{j\in S}\!\Big\{U(\theta;X_j,Y_j)-U(\theta;X_j,m(X_j))\Big\}.
\]
Assume: (i) $U(\theta;x,y)$ is continuously differentiable in a neighborhood of $\theta_N$;
(ii) $I(\theta):=-\,\mathbb E\{\partial_\theta U(\theta;X,Y)\}$ exists and $I(\theta_N)$ is nonsingular; and
(iii) a uniform LLN and stochastic–equicontinuity hold so that
\(
\sup_{\theta\in\Theta}\big\|\partial_\theta \widehat U_{\mathrm{PPI}}(\theta)-\partial_\theta U_N(\theta)\big\|_{\op}
=o_p(1)
\)
as $n,N\to\infty$ with $n/N\to f$.
Then
\begin{equation}\label{eq:ppi-expansion}
\hat\theta_{\mathrm{PPI}}-\theta_N
=\{I(\theta_N)\}^{-1}\!\left[
\frac1N\sum_{i=1}^N\!\Big(\frac{\delta_i}{\,f\,}-1\Big)\,\Delta_i(\theta_N)
\right]+o_p(n^{-1/2}).
\end{equation}
\end{lemma}

\begin{proof}
$\quad$

\textbf{Step 1 -- Score at $\theta_N$.}
By Lemma~\ref{lem:PPI-at-thetaN},
\[
\widehat U_{\mathrm{PPI}}(\theta_N)
=\frac1N\sum_{i=1}^N\!\Big(\frac{\delta_i}{\,f\,}-1\Big)\,\Delta_i(\theta_N),
\]
a conditionally mean–zero sum with variance of order $n^{-1}$; hence
$\widehat U_{\mathrm{PPI}}(\theta_N)=O_p(n^{-1/2})$.

\smallskip
\textbf{Step 2 -- Linearization.}
A first–order Taylor expansion of $\widehat U_{\mathrm{PPI}}(\theta)$ around $\theta_N$ gives
\[
0=\widehat U_{\mathrm{PPI}}(\hat\theta_{\mathrm{PPI}})
=\widehat U_{\mathrm{PPI}}(\theta_N)
+\big[\widehat G_N(\tilde\theta)\big]\,
(\hat\theta_{\mathrm{PPI}}-\theta_N),
\]
where $\tilde\theta$ lies on the line segment between $\hat\theta_{\mathrm{PPI}}$ and $\theta_N$ and
$\widehat G_N(\theta):=\partial_\theta \widehat U_{\mathrm{PPI}}(\theta)$.
Thus
\[
\hat\theta_{\mathrm{PPI}}-\theta_N
=-\big[\widehat G_N(\tilde\theta)\big]^{-1}\widehat U_{\mathrm{PPI}}(\theta_N).
\]

\smallskip
\textbf{Step 3 -- Limit of the Jacobian.}
By assumption (iii) and $n/N\to f$,
\(
\widehat G_N(\tilde\theta)=\partial_\theta U_N(\theta_N)+o_p(1)
= -\,I(\theta_N)+o_p(1),
\)
so $\big[\widehat G_N(\tilde\theta)\big]^{-1}= -\,I(\theta_N)^{-1}+o_p(1)$.

\smallskip
\textbf{Step 4 -- Assemble.}
Combining the three steps yields
\begin{align*}
\hat\theta_{\mathrm{PPI}}-\theta_N
&=\{I(\theta_N)\}^{-1}\widehat U_{\mathrm{PPI}}(\theta_N)+o_p(n^{-1/2})\\
&=\{I(\theta_N)\}^{-1}\!\left[
\frac1N\sum_{i=1}^N\!\Big(\frac{\delta_i}{\,f\,}-1\Big)\,\Delta_i(\theta_N)
\right]+o_p(n^{-1/2}),    
\end{align*}
which is \eqref{eq:ppi-expansion}.
\end{proof}


\begin{theorem}[Asymptotic linearity of the PPI root]\label{thm:ppi-linear}
Let $\theta_0$ be the unique solution to $\mathbb{E}\{U(\theta;X,Y)\}=0$.  
Suppose the regularity conditions in Assumptions~1–3 hold and the labeled set
$S\subset\{1,\ldots,N\}$ is drawn by SRSWOR of size $n$.  Write
$\delta_i=\mathbf 1\{i\in S\}$ and $f:=n/N\in(0,1)$.  Let
$\hat\theta_{\mathrm{PPI}}$ be any solution to
$\widehat U_{\mathrm{PPI}}(\theta)=0$.  Then
\begin{equation}\label{eq:ppi-lin}
\hat\theta_{\mathrm{PPI}}-\theta_0
=\frac{1}{N}\sum_{i=1}^N \phi_i \;+\; o_p(N^{-1/2}),
\end{equation}
where the influence contribution of unit $i$ is
\begin{align*}
\phi_i
&=\phi(X_i,Y_i,\delta_i;\theta_0,m)\\
&:=\{I(\theta_0)\}^{-1}\!\left[
U(\theta_0;X_i,m(X_i))
+\frac{\delta_i}{\,f\,}\Big\{U(\theta_0;X_i,Y_i)-U(\theta_0;X_i,m(X_i))\Big\}
\right].
\end{align*}
\end{theorem}

\begin{proof}
$\quad$

\textbf{Step 1 -- Expansion around the oracle root.}
Let $\theta_N$ solve $U_N(\theta)=N^{-1}\sum_{i=1}^N U(\theta;X_i,Y_i) $ $=0$.
By Lemma~\ref{lem:ppi-expansion},
\begin{equation}\label{eq:ppi-lin-step1}
\hat\theta_{\mathrm{PPI}}-\theta_N
=\{I(\theta_N)\}^{-1}\!\left[
\frac{1}{N}\sum_{i=1}^N\!\Big(\frac{\delta_i}{\,f\,}-1\Big)
\{U(\theta_N;X_i,Y_i)-U(\theta_N;X_i,m(X_i))\}
\right]+o_p(n^{-1/2}).
\end{equation}

\textbf{Step 2 -- Replace $\theta_N$ by $\theta_0$.}
Since $\theta_N\!\xrightarrow{p}\!\theta_0$ and $U(\theta;\cdot),I(\theta)$ are continuous near
$\theta_0$, we have $\{I(\theta_N)\}^{-1}=\{I(\theta_0)\}^{-1}+o_p(1)$ and
$U(\theta_N;\cdot)=U(\theta_0;\cdot)+o_p(1)$ uniformly.  Hence
\begin{equation}\label{eq:ppi-lin-step2}
\hat\theta_{\mathrm{PPI}}-\theta_N
=\{I(\theta_0)\}^{-1}\!\left[
\frac{1}{N}\sum_{i=1}^N\!\Big(\frac{\delta_i}{\,f\,}-1\Big)
\{U(\theta_0;X_i,Y_i)-U(\theta_0;X_i,m(X_i))\}
\right]+o_p(n^{-1/2}).
\end{equation}

\textbf{Step 3 -- Oracle expansion.}
From the oracle expansion,
\[
\theta_N-\theta_0
=\{I(\theta_0)\}^{-1}\frac{1}{N}\sum_{i=1}^N U(\theta_0;X_i,Y_i)
+o_p(N^{-1/2}).
\]

\textbf{Step 4 -- Assemble.}
Adding the last display to \eqref{eq:ppi-lin-step2} gives
\begin{small}
\begin{align*}
\hat\theta_{\mathrm{PPI}}-\theta_0
&=\{I(\theta_0)\}^{-1}\frac{1}{N}\sum_{i=1}^N\!\left[
U(\theta_0;X_i,Y_i)
+\Big(\frac{\delta_i}{\,f\,}-1\Big)\{U(\theta_0;X_i,Y_i)-U(\theta_0;X_i,m(X_i))\}
\right]\\
& +
\underbrace{o_p(N^{-1/2})
+o_p(n^{-1/2})}_{\star}.    
\end{align*}
\end{small}
Algebra inside the brackets yields
\(
U(\theta_0;X_i,m(X_i))
+\frac{\delta_i}{\,f\,}\{U(\theta_0;X_i,Y_i)-U(\theta_0;X_i,m(X_i))\},
\)
which is precisely $\phi_i$ in the statement. 

\textbf{Step 5 -- Remainder term.} Recall that, in our decomposition, the $N^{-1/2}$ term comes from the oracle expansion, while the replacement and rectifier terms are governed by the labeled part and are $o_p(n^{-1/2})$.

For the asymptotic analysis for the remainder term ($\star$), let $N\to\infty$ and $n = n(N) \to \infty$ with $n/N\to f\in(0,1)$ by a slight abuse of notation, reuse the symbol $f$ for the limit. By definition of convergence, there exists a sequence $\varepsilon_N \to 0$ such that
\[
\frac{n}{N} = f + \varepsilon_N.
\]
Multiply both sides by $N$:
\[
n = fN + \varepsilon_N N = fN\Big(1 + \frac{\varepsilon_N}{f}\Big).
\]
Since $f>0$ and $\varepsilon_N \to 0$, we have $\varepsilon_N/f \to 0$, thus $\frac{\varepsilon_N}{f} = o(1)$. Then
\[
n = fN\{1+o(1)\}.
\]
Then, we have
\[
n^{-1/2} \;=\; f^{-1/2}\,N^{-1/2}\{1+o(1)\},
\]
so $o_p(n^{-1/2})=o_p(N^{-1/2})$. Hence the remainder term ($\star$) is
\[
o_p(N^{-1/2}) + o_p(n^{-1/2})
\;=\; o_p(N^{-1/2}) + o_p(N^{-1/2})
\;=\; o_p(N^{-1/2}).
\]
This concludes the proof.
\end{proof}

\begin{corollary}[Limit law, confidence intervals, and Wald tests for the PPI estimator]\label{cor:ppi-wald}
Assume the conditions of Theorem~\ref{thm:ppi-linear} hold. Write $f:=n/N\in(0,1)$, let $\delta_i=\mathbf 1\{i\in S\}$, and define the influence contribution for unit $i$
\begin{align}\label{eq:phi-known-m}
\phi_i &\equiv \phi(X_i,Y_i,\delta_i;\theta_0,m)\\
\nonumber
&:= \{I(\theta_0)\}^{-1}\!\left[
U(\theta_0;X_i,m(X_i))
+ \frac{\delta_i}{\,f\,}\Big\{U(\theta_0;X_i,Y_i)-U(\theta_0;X_i,m(X_i))\Big\}
\right].
\end{align}
Let $a^{\otimes 2}:=aa^\top$ denote the outer product and set
\begin{align*}
V_1 &:= \{I(\theta_0)\}^{-1}\,\mathbb{E}\!\big[\,U(\theta_0;X,Y)^{\otimes 2}\big]\,\{I(\theta_0)\}^{-1},\\
V_2 &:= \{I(\theta_0)\}^{-1}\,\mathbb{E}\!\big[\{U(\theta_0;X,Y)-U(\theta_0;X,m(X))\}^{\otimes 2}\big]\,\{I(\theta_0)\}^{-1}.
\end{align*}
Then, with $\Sigma_f := V_1 + (f^{-1}-1)\,V_2$:

\begin{enumerate}\itemsep4pt
\item[\textbf{\textup{(i)}}] \textbf{\emph{Asymptotic distribution.}}
\[
\sqrt{N}\,(\hat\theta_{\mathrm{PPI}}-\theta_0)\ \xrightarrow{d}\ \mathcal N(0,\Sigma_f),
\qquad
\text{equivalently}\quad
\mathrm{Var}(\hat\theta_{\mathrm{PPI}})\ \approx\ N^{-1}V_1 + (n^{-1}-N^{-1})V_2.
\]

\item[\textbf{\textup{(ii)}}] \textbf{\emph{Componentwise $(1-\alpha)$ Wald intervals.}}
Define
\[
\widehat I \;:=\; -\frac1n\sum_{j\in S}\partial_\theta U(\hat\theta_{\mathrm{PPI}};X_j,Y_j),\quad
\widehat V_1 \;:=\; \widehat I^{-1}\!\Big(\frac1n\sum_{j\in S} U(\hat\theta_{\mathrm{PPI}};X_j,Y_j)^{\otimes 2}\Big)\widehat I^{-1},
\]
\[
\Delta_j(\theta)\!:=\!U(\theta;X_j,Y_j)-U(\theta;X_j,m(X_j)),\qquad
\widehat V_2 \;:=\; \widehat I^{-1}\!\Big(\frac1n\sum_{j\in S}\Delta_j(\hat\theta_{\mathrm{PPI}})^{\otimes 2}\Big)\widehat I^{-1},
\]
and set
\[
\widehat\Sigma_{\hat\theta} \;:=\; N^{-1}\widehat V_1 + (n^{-1}-N^{-1})\widehat V_2.
\]
Then the $j$-th component interval is
\[
\Big[\ \hat\theta_{\mathrm{PPI},j}\ \pm\ z_{1-\alpha/2}\,\sqrt{(\widehat\Sigma_{\hat\theta})_{jj}}\ \Big].
\]

\item[\textbf{\textup{(iii)}}] \textbf{\emph{Wald test for \(H_0:\theta_0=\theta^\star\).}}
Let \(p=\dim(\theta)\). The Wald statistic
\[
W \;=\; (\hat\theta_{\mathrm{PPI}}-\theta^\star)^\top\ \widehat\Sigma_{\hat\theta}^{-1}\ (\hat\theta_{\mathrm{PPI}}-\theta^\star)
\]
converges in distribution to \(\chi^2_p\) under \(H_0\). 

For \(p=1\),
\(
z = \dfrac{\hat\theta_{\mathrm{PPI}}-\theta^\star}{\sqrt{(\widehat\Sigma_{\hat\theta})_{11}}}
\)
is asymptotically standard normal.
\end{enumerate}
\end{corollary}

\begin{proof}
By Theorem~\ref{thm:ppi-linear},
\[
\hat\theta_{\mathrm{PPI}}-\theta_0 \;=\; \frac{1}{N}\sum_{i=1}^N \phi_i \;+\; o_p(N^{-1/2}),
\qquad
\phi_i \equiv \phi(X_i,Y_i,\delta_i;\theta_0,m),
\]
so by a multivariate CLT \citep{hajek1960limiting,serfling2009approximation} (Assumption 3 - Second moment regularity),
\[
\sqrt{N}\,\Big(\hat\theta_{\mathrm{PPI}}-\theta_0\Big)
\;=\;
\frac{1}{\sqrt{N}}\sum_{i=1}^N \phi_i \;\xrightarrow{d}\; \mathcal N(0,\Sigma_f),
\quad\text{with}\quad
\Sigma_f \;=\; \Var\!\big\{\phi(X,Y,\delta;\theta_0,m)\big\}.
\]
We now compute $\Sigma_f$ explicitly. Write $f:=n/N\in(0,1)$ and 
\begin{align*}
\phi(X,Y,\delta;\theta_0,m)
&=\{I(\theta_0)\}^{-1}\!\Big[
U(\theta_0;X,m(X))+\frac{\delta}{f}\,\{U(\theta_0;X,Y)-U(\theta_0;X,m(X))\}
\Big]\\
&=\{I(\theta_0)\}^{-1}\!\Big[
U(\theta_0;X,m(X))+\frac{\delta}{f} \Delta(\theta_0;X,Y)
\Big]    
\end{align*}
by denoting $\Delta(\theta_0;X,Y):=U(\theta_0;X,Y)-U(\theta_0;X,m(X))$.

By the law of total variance,
\[
\Sigma_f
=
\underbrace{\Var\!\Big(\,\mathbb E\!\big[\phi(X,Y,\delta;\theta_0,m)\mid X,Y\big]\,\Big)}_{\star}
\;+\;
\underbrace{\mathbb E\!\Big(\,\Var\!\big[\phi(X,Y,\delta;\theta_0,m)\mid X,Y\big]\,\Big)}_{\ast}
.
\]

\paragraph{First term ($\star$).}
Under SRSWOR, $\delta\perp (X,Y)$ and $\mathbb E(\delta\mid X,Y)=\mathbb E(\delta)=f$. Hence
\begin{align*}
\mathbb E\!\big[\phi(X,Y,\delta;\theta_0,m)\mid X,Y\big]
&=\{I(\theta_0)\}^{-1}\!\Big[
U(\theta_0;X,m(X))+\tfrac{1}{f}\,f\,\Delta(\theta_0;X,Y)
\Big]\\
&=\{I(\theta_0)\}^{-1}U(\theta_0;X,Y),
\end{align*}
so
\[
\Var\!\Big(\,\mathbb E\!\big[\phi(X,Y,\delta;\theta_0,m)\mid X,Y\big]\,\Big)
=\{I(\theta_0)\}^{-1}\,\mathbb E\!\big[U(\theta_0;X,Y)^{\otimes 2}\big]\,
\{I(\theta_0)\}^{-1}
=: V_1.
\]

\paragraph{Second term ($\ast$).}
Conditionally on $(X,Y)$, the only randomness in $\phi$ is $\delta$, and under SRSWOR
\(
\Var(\delta\mid X,Y)=\Var(\delta)=f(1-f)
\). (Note that this is the finite-population correction.) Therefore,
\begin{align*}
\Var\!&\big[\phi(X,Y,\delta;\theta_0,m)\mid X,Y\big]\\
&=\Var
\!\Bigg[
\{
\underbrace{I(\theta_0)\}^{-1}}_{\text{Constant}}\!\Big\{
\underbrace{U(\theta_0;X,m(X))}_{\text{Constant}}
+\underbrace{\delta}_{\text{Random}}  \underbrace{\frac{1}{f} \Delta(\theta_0;X,Y)}_{\text{Constant}}
\Big\}
\mid X,Y\Bigg]
\\
&=
\Var
\!\Bigg[
\underbrace{\{
I(\theta_0)\}^{-1}\!\Big\{
\frac{1}{f} \Delta(\theta_0;X,Y)
\Big\}}_{\text{Constant}}
\delta
\mid X,Y\Bigg]\\
&= \Big(\frac{1}{f}\Big)^{\!2}\Var(\delta\mid X,Y)\;
\{I(\theta_0)\}^{-1}\,\Delta(\theta_0;X,Y)^{\otimes 2}\,\{I(\theta_0)\}^{-1}\\
&= \Big(\frac{1}{f}-1\Big)\;
\{I(\theta_0)\}^{-1}\,\Delta(\theta_0;X,Y)^{\otimes 2}\,\{I(\theta_0)\}^{-1}.
\end{align*}

Taking expectations gives
\begin{align*}
\mathbb E\!\Big(\,\Var\!\big[\phi(X,Y,\delta;\theta_0,m)\mid X,Y\big]\,\Big)
&=\Big(\frac{1}{f}-1\Big)\;
\{I(\theta_0)\}^{-1}\,\mathbb E\!\big[\Delta(\theta_0;X,Y)^{\otimes 2}\big]\,
\{I(\theta_0)\}^{-1}\\
&= (f^{-1}-1)\,V_2,
\end{align*}
where
\[
V_2:=\{I(\theta_0)\}^{-1}\,\mathbb E\!\big[\{U(\theta_0;X,Y)-U(\theta_0;X,m(X))\}^{\otimes 2}\big]\,
\{I(\theta_0)\}^{-1}.
\]

\paragraph{Combine.}
Thus
\[
\Sigma_f \;=\; V_1+(f^{-1}-1)\,V_2,
\]
establishing the variance in the limit law.  
Consistency of the plug-in estimators $\widehat I,\widehat V_1,\widehat V_2$ follows from LLN and the continuous mapping theorem under Assumptions~1–3, and parts (ii)–(iii) of the corollary then follow by Slutsky’s theorem. 
\end{proof}

\subsection{Semiparametric efficiency theory of prediction-powered inference}\label{subsec:Semiparametric Efficiency Theory of Prediction-Powered Inference}
This subsection provides the proofs of the lemmas and theorems for the semiparametric efficiency theory of PPI, and it proves Theorem~4.1 in the main document.

\begin{lemma}[Implicit functional and efficient influence function via a moment map]
\label{lem:implicit-moment-eif}
Let $O\sim P_0$ on a measurable space $(\mathcal O,\mathcal A)$, where $\mathcal O$ is the sample space, $\mathcal A$ is a $\sigma$-algebra of measurable subsets of $\mathcal O$, and $P_0$ is a probability measure on $(\mathcal O,\mathcal A)$. We write $\mathbb E_{P_0}[\cdot]$ for expectation with respect to $P_0$. Define
\[
L_0^2(P_0):=\big\{f\in L^2(P_0): \mathbb E_{P_0}[f]=0\big\}.
\]

Let $\mathcal M$ be a statistical model for the observed–data law. The \emph{tangent space} of $\mathcal M$ at $P_0$, denoted $\mathcal T \subset L_0^2(P_0)$, is the closed linear span in $L_0^2(P_0)$ of all score functions
\[
S(O)\;=\;\left.\frac{d}{d\varepsilon}\log\frac{dP_\varepsilon}{dP_0}(O)\right|_{\varepsilon=0} \in \mathcal{T },
\]
arising from regular parametric submodels $\{P_\varepsilon:\varepsilon\in(-\eta,\eta)\}\subset\mathcal M$ through $P_0$.

For $P\in\mathcal M$ and $\theta\in\Theta\subset\mathbb R^p$, define
\[
M(P,\theta):=\mathbb E_P[m(P;\theta)(O)]
=\int_{\mathcal O} m(P;\theta)(o)\,dP(o),\qquad
M:\mathcal M\times\Theta\to\mathbb R^p,
\]
where $m(P;\theta)\in L^{2}(P_0;\mathbb R^p)$ for $P$ near $P_0$ and $\theta$ near $\theta_0$. Here $L^2(P_0;\mathbb R^p)=(L^2(P_0))^p$. For each fixed $\theta$, the map $P\mapsto M(P,\theta)$ is pathwise differentiable at $P_0$
(i.e., Gâteaux differentiable along every regular parametric submodel with score $S\in\mathcal T$),
and for each fixed $P$, the map $\theta\mapsto m(P;\theta)$ is differentiable near $\theta_0$. Define the parameter $\Psi:\mathcal M\to\Theta$ implicitly by
\[
\Psi(P)=\theta(P)\quad\text{such that}\quad M(P,\theta(P))=0.
\]
\textbf{Assumptions.}
\begin{itemize}
\item[A1] (\emph{Identification}) $M(P_0,\theta)=0$ has a unique solution $\theta_0\in\Theta$.

\item[A2] (\emph{$\theta$–smoothness})
  \begin{enumerate}
  \item The Jacobian matrix 
  \[
  A:=\partial_\theta M(P_0,\theta_0) 
     = \mathbb E_{P_0}\!\big[\partial_\theta m(P_0;\theta_0)(O)\big]
  \]
  exists and is $p\times p$;
  \item $A$ is nonsingular (invertible).
  \end{enumerate}


\item[A3] (\emph{Pathwise differentiability  of $M$ in $P$})
There exists a moment representer $\varphi=(\varphi_1,\ldots,\varphi_p)$ $\in L_0^2(P_0;\mathbb R^p)$ such that for every
regular parametric submodel $\{P_\varepsilon:\varepsilon\in(-\eta,\eta)\}\subset\mathcal M$ through $P_0$
with score $S\in\mathcal T$,
\begin{align*}
\left.\frac{d}{d\varepsilon}M(P_\varepsilon,\theta_0)\right|_{\varepsilon=0}
&\;=\;
\mathbb E_{P_0}\!\big[\varphi(O)\,S(O)\big]\\
&\;=\;
\big(\,\mathbb E_{P_0}[\varphi_1(O)S(O)],\,\ldots,\,\mathbb E_{P_0}[\varphi_p(O)S(O)]\,\big)
\;\in\mathbb R^p,    
\end{align*}
i.e., the derivative acts componentwise via the $L^2(P_0)$ inner product with $s$. 
\end{itemize}

Then, the parameter $\Psi$ is pathwise differentiable at $P_0$ with the canonical gradient (i.e., efficient influence function)
\[
\phi^{\mathrm{eff}}(O)\;=\;-\,A^{-1}\,\Pi_{\mathcal T}\varphi(O),
\]
where $\Pi_{\mathcal T}:L_0^2(P_0;\mathbb R^p)\to\mathcal T^p$ denotes the orthogonal projection applied componentwise. If, in addition, the moment representer satisfies $\varphi\in\mathcal T^p$, then the projection is redundant and
\[
\phi^{\mathrm{eff}}(O)=-\,A^{-1}\,\varphi(O).
\]
\end{lemma}

\begin{proof}
We present the proof for the scalar case $p=1$. For general $p$, fix any $\nu\in \mathbb R^p$ and apply the same argument to the scalar composite $\nu^\top M(P, \Psi(P)) = 0$; since the resulting identity holds for all $\nu$ the vector formula follows.

Because $\Psi(P)$ is defined implicitly by $M(P,\Psi(P))=0$, the following identity holds
\begin{align}
\label{eq:moment_identity}
M(P,\Psi(P)) \;=\; 0 \qquad \text{for all $P$ in a neighborhood of $P_0$}.
\end{align}
Fix an arbitrary regular parametric submodel $\{P_\varepsilon:\varepsilon\in(-\eta,\eta)\}\subset\mathcal M$ through $P_0$ with score 
$S(O)=\left.\tfrac{d}{d\varepsilon}\log\frac{dP_\varepsilon}{dP_0}(O)\right|_{\varepsilon=0}\in\mathcal T$. For example, let $S\in L_0^2(P_0)$ with $\mathbb E_{P_0}[S]=0$. For $|\varepsilon|$ small enough so that $1+\varepsilon S(o)\ge 0$ $P_0$–a.s., define the (first–order) linear tilting submodel by
\[
\frac{dP_\varepsilon}{dP_0}(o)\;=\;1+\varepsilon\,S(o).
\]
Then $\log\!\left(\tfrac{dP_\varepsilon}{dP_0}(o)\right)=\log\!\big(1+\varepsilon S(o)\big)$ and hence
\[
S(o)\;=\;\left.\frac{d}{d\varepsilon}\log\frac{dP_\varepsilon}{dP_0}(o)\right|_{\varepsilon=0}
=\left.\frac{S(o)}{1+\varepsilon S(o)}\right|_{\varepsilon=0} \in L_0^2(P_0).
\]

Consider the composite map $g(\varepsilon):=M\!\big(P_\varepsilon,\Psi(P_\varepsilon)\big)$. 
By \eqref{eq:moment_identity}, we have $g(\varepsilon)=0$ for all $\varepsilon$ near $0$. 
Assumptions A2--A3 ensure that $g$ is differentiable at $\varepsilon=0$, and hence $g'(0)=0$. To compute $g'(0)$, we invoke the chain rule for $M$ in its two arguments: the distributional argument $P$ (handled via a Gâteaux derivative along the submodel $\{P_\varepsilon\}$) and the finite-dimensional argument $\theta$ (handled via the usual Jacobian). 
Specifically,
\begin{align}
\label{eq:semi_eq1}
\frac{d}{d\varepsilon} M\!\big(P_\varepsilon,\Psi(P_\varepsilon)\big)
=
\underbrace{
  \partial_{P} M\!\big(P_\varepsilon,\Psi(P_\varepsilon)\big)\,[\,\dot P_\varepsilon\,]
}_{\substack{\text{G\^ateaux derivative in $P$} \\ \text{}}}
\;+\;
\underbrace{\partial_{\theta} M\!\big(P_\varepsilon,\Psi(P_\varepsilon)\big)}_{\text{Jacobian in }\theta}
\cdot
\underbrace{\frac{d}{d\varepsilon}\Psi(P_\varepsilon)}_{\substack{\text{pathwise}\\ \text{derivative of }\Psi}}.    
\end{align}

Evaluating \eqref{eq:semi_eq1} at $\varepsilon=0$ (so $P_\varepsilon|_{\varepsilon=0}=P_0$ and $\Psi(P_0)=\theta_0$) yields
\begin{align}
\label{eq:semi_eq2}
0
=\left.\frac{d}{d\varepsilon}M\!\big(P_\varepsilon,\Psi(P_\varepsilon)\big)\right|_{\varepsilon=0}
=\left.\frac{d}{d\varepsilon}M(P_\varepsilon,\theta_0)\right|_{\varepsilon=0}
\;+\;
\partial_\theta M(P_0,\theta_0)\cdot \left.\frac{d}{d\varepsilon}\Psi(P_\varepsilon)\right|_{\varepsilon=0}.
\end{align}


Here we use the definition of the Gâteaux derivative for the first term on the
right–hand side of \eqref{eq:semi_eq2}. Fix \(\theta\) and write \(F(P,\theta):=M(P,\theta)\).
Let \(S\in L_0^2(P_0)\) and define a path \(P_\varepsilon\) by $dP_\varepsilon/dP_0 = 1+\varepsilon S$ for $|\varepsilon|$ small, so that \(P_{\varepsilon=0}=P_0\). The Gâteaux derivative of \(P\mapsto F(P,\theta)\) at \(P_0\)
in the direction \(S\) is
\[
\partial_P F(P_0,\theta)[S]
:=\lim_{\varepsilon\to 0}\frac{F(P_\varepsilon,\theta)-F(P_0,\theta)}{\varepsilon}
= \left.\frac{d}{d\varepsilon}F(P_\varepsilon,\theta)\right|_{\varepsilon=0},
\]
whenever the limit exists. In particular, with \(F=M\) and \(\theta=\theta_0\)
\[
\partial_P M(P_\varepsilon,\theta_0)[\dot P_\varepsilon] \bigg|_{\varepsilon = 0} = 
\partial_P M(P_0,\theta_0)[S]
= \left.\frac{d}{d\varepsilon}M(P_\varepsilon,\theta_0)\right|_{\varepsilon=0},
\]
which is the identity used in \eqref{eq:semi_eq2}.

Writing $A:=\partial_\theta M(P_0,\theta_0)$ (which exists and is invertible by A2), the equation \eqref{eq:semi_eq2} becomes
\begin{align}
\label{eq:semi_eq3}
0
=\left.\frac{d}{d\varepsilon}M(P_\varepsilon,\theta_0)\right|_{\varepsilon=0}
\;+\;
A\cdot \left.\frac{d}{d\varepsilon}\Psi(P_\varepsilon)\right|_{\varepsilon=0}.
\end{align}

By A3, for every regular submodel score $S\in\mathcal T$ the directional derivative of $P\mapsto M(P,\theta_0)$ along $\{P_\varepsilon\}$ admits the inner–product representation
\[
\left.\frac{d}{d\varepsilon}M(P_\varepsilon,\theta_0)\right|_{\varepsilon=0}
\,=\,\mathbb E_{P_0}\!\big[\varphi(O)\,S(O)\big]
\,=\,\langle \varphi,\,S\rangle_{L_0^2(P_0)},
\]
for some moment representer $\varphi\in L_0^2(P_0)$ (not necessarily in $\mathcal T$). \\
Here $S(O)=$ $\left.\tfrac{d}{d\varepsilon}\log \tfrac{dP_\varepsilon}{dP_0}(O)\right|_{\varepsilon=0} $ $\in\mathcal T$ is the submodel score and 
$\langle f,g\rangle_{L_0^2(P_0)}:=\mathbb E_{P_0}[f\,g]$ denotes the $L_0^2(P_0)$ inner product.

Equation \eqref{eq:semi_eq3} yields, for every score $S\in\mathcal T$,
\[
0 \;=\; \langle \varphi, S\rangle_{L_0^2(P_0)} \;+\; A \cdot \left.\frac{d}{d\varepsilon}\Psi(P_\varepsilon)\right|_{\varepsilon=0},
\]
hence
\begin{equation}
\nonumber
\left.\frac{d}{d\varepsilon}\Psi(P_\varepsilon)\right|_{\varepsilon=0} 
\;=\; -\,A^{-1}\, \langle \varphi, S\rangle_{L_0^2(P_0)} \qquad \text{for all } S\in\mathcal T .
\end{equation}
Define a linear map $L$ by
\begin{align}
\label{eq:functional_derivative_of_target_parameter}
L(S)\;:=\;-\,A^{-1}\,\langle \varphi, S\rangle_{L_0^2(P_0)}.    
\end{align}
Since $L$ is continuous and linear on the Hilbert space $(\mathcal T,\langle\cdot,\cdot\rangle_{L_0^2(P_0)})$, the Riesz representation theorem \citep{rudin1987real} guarantees a unique $\phi^{\mathrm{eff}}\in\mathcal T$ such that
\begin{equation}
\label{eq:Riesz_representer}
L(S)\;=\;\langle \phi^{\mathrm{eff}}, S\rangle_{L_0^2(P_0)} 
\qquad \text{for all } S\in\mathcal T.
\end{equation}
Comparing \eqref{eq:functional_derivative_of_target_parameter} and \eqref{eq:Riesz_representer} gives
\[
\langle A\phi^{\mathrm{eff}}+\varphi,\, S\rangle_{L_0^2(P_0)} \;=\; 0 \quad \text{for all } S\in\mathcal T.
\]
Thus $A\phi^{\mathrm{eff}}+\varphi\in \mathcal T^\perp$, so by the projection theorem
\[
\Pi_{\mathcal T}\!\big(A\phi^{\mathrm{eff}}+\varphi\big)=0
\quad\Longleftrightarrow\quad
A\,\Pi_{\mathcal T}\phi^{\mathrm{eff}}+\Pi_{\mathcal T}\varphi=0 .
\]
Because $\phi^{\mathrm{eff}}\in\mathcal T$, $\Pi_{\mathcal T}\phi^{\mathrm{eff}}=\phi^{\mathrm{eff}}$, and since $A$ is fixed we can pull it through the projection:
\[
A\,\phi^{\mathrm{eff}} \;=\; -\,\Pi_{\mathcal T}\varphi
\qquad\Longleftrightarrow\qquad
\phi^{\mathrm{eff}} \;=\; -\,A^{-1}\,\Pi_{\mathcal T}\varphi .
\]
In particular, if $\varphi\in\mathcal T$ then $\Pi_{\mathcal T}\varphi=\varphi$ and $\phi^{\mathrm{eff}}=-A^{-1}\varphi$. 

Note that this $\phi^{\text{eff}}$ satifies
$$\left.\frac{d}{d\varepsilon}\Psi(P_\varepsilon)\right|_{\varepsilon=0}
= \langle \phi^{\mathrm{eff}},\, S\rangle
= \left\langle -\,A^{-1}\Pi_{\mathcal T}\varphi,\, S \right\rangle, \quad \text{for all } S\in\mathcal T$$
by definition, the canonical gradient (i.e., efficient influence function) \citep{bickel1998,van2002semiparametric}.
\end{proof}


\begin{lemma}[Product rule for expectations along a regular submodel]\label{lemma:product_rule}
Let $\{P_\varepsilon:\varepsilon\in(-\eta,\eta)\}$ be a regular parametric submodel through $P_0$
dominated by a common measure $\mu$, with densities $p_\varepsilon=\frac{dP_\varepsilon}{d\mu}$.
Assume $p_\varepsilon$ is differentiable at $\varepsilon=0$ in $L_1(\mu) = \{f : \int |f| d\mu < \infty \}$ with norm $\|f\|_{L_1(\mu)} = \int |f| d\mu $, i.e., there exists
$\dot p_0\in L_1(\mu)$ such that
\[
\int \Big|\frac{p_\varepsilon-p_0}{\varepsilon}-\dot p_0\Big|\,d\mu \;\to\; 0
\quad\text{as }\varepsilon\to 0,
\]
and define the score
\[
S(O)\;:=\;\left.\frac{\partial}{\partial\varepsilon}\log p_\varepsilon(O)\right|_{\varepsilon=0}
\quad\text{so that}\quad
\left.\frac{\partial}{\partial\varepsilon}p_\varepsilon(O)\right|_{\varepsilon=0}
=p_0(O)\,S(O).
\]
Let $g_\varepsilon:\mathcal O\to\mathbb R^p$ be measurable functions such that
$\varepsilon\mapsto g_\varepsilon(o)$ is differentiable at $0$ for $P_0$-a.e.\ $o$, with
$\dot g_0(o):=\left.\frac{\partial}{\partial\varepsilon}g_\varepsilon(o)\right|_{\varepsilon=0}$,
and assume a dominated convergence condition allowing differentiation under the integral.
Then
\[
\left.\frac{d}{d\varepsilon}\,\mathbb E_{P_\varepsilon}\!\big[g_\varepsilon(O)\big]\right|_{\varepsilon=0}
\;=\;
\underbrace{\mathbb E_{P_0}\!\big[g_0(O)\,S(O)\big]}_{\text{\emph{change of law}}}
\;+\;
\underbrace{\mathbb E_{P_0}\!\big[\dot g_0(O)\big]}_{\text{\emph{change of integrand}}}.
\]
\end{lemma}

\begin{proof}
Write the expectation as an integral against the common dominating measure:
\[
\mathbb E_{P_\varepsilon}\!\big[g_\varepsilon(O)\big]
=\int g_\varepsilon(o)\,p_\varepsilon(o)\,d\mu(o).
\]
Differentiate at $\varepsilon=0$ and use the usual product rule:
\[
\left.\frac{d}{d\varepsilon}\int g_\varepsilon\,p_\varepsilon\,d\mu\right|_{\varepsilon = 0}
=\int \dot g_0(o)\,p_0(o)\,d\mu(o)\;+\;\int g_0(o)\,\dot p_0(o)\,d\mu(o).
\]
By definition of the score, $\dot p_0=p_0\,S$, so the second term equals
$\int g_0\,p_0\,S\,d\mu=\mathbb E_{P_0}[g_0\,S]$. The first term is
$\mathbb E_{P_0}[\dot g_0]$. This yields the identity.
\end{proof}


\begin{theorem}[Semiparametric efficiency and asymptotics of PPI under SRSWOR with known sampling fraction]
\label{thm:eif-ppi-srswor}
Let $O=(X,\delta,\delta Y)\sim P_0$, where $\delta\in\{0,1\}$ indicates whether $Y$ is labeled.
Let $\{O_i=(X_i,\delta_i,\delta_i Y_i) \in \mathcal X\times\{0,1\}\times\mathcal Y: i=1,\ldots,N\}$ denote the observed data from $N$ subjects.
Assume a simple random sampling without replacement (SRSWOR) design that is independent of $(X,Y)$, with a known labeling fraction $f\in(0,1)$ and $n/N\to f$.

Let $U(\theta;x,y)\in\mathbb R^p$ be a full–data estimating function and suppose the super–population target $\theta_0\in\Theta\subset\mathbb R^p$ is the unique solution of
\[
\mathbb E_{P_0}\!\big[U(\theta_0;X,Y)\big]=0.
\]
For any observed–data law $P$ on $(\mathcal X\times\{0,1\}\times\mathcal Y,\mathcal A)$, factored as
\[
P(\mathrm dx,\mathrm d\delta,\mathrm dy)
= P_X(\mathrm dx)\,P_\delta(\mathrm d\delta)\,P_{Y\mid X}(\mathrm dy\mid x),
\]
assume $P_\delta=\mathrm{Bernoulli}(f)$ is fixed (non–varying) and independent of $(X,Y)$.

Define the conditional moment
\[
\bar U(\theta;x;P)
:= \mathbb E_{P}\!\big[U(\theta;x,Y)\mid X=x\big]
= \int_{\mathcal Y} U(\theta;x,y)\, P_{Y\mid X}(\mathrm dy\mid x),
\]
where $P_{Y\mid X}(\cdot\mid x)$ denotes a conditional distribution of $Y$ given $X=x$ under $P$.

Define the information matrix
\[
I(\theta_0):=-\,\mathbb E_{P_0}\!\big[\partial_\theta U(\theta_0;X,Y)\big].
\]
Let \(\mathcal M\) be a statistical model for the observed–data law containing \(P_0\), and let \(\mathcal T\subset L_0^2(P_0)\) denote its tangent space at \(P_0\).
For each $P \in \mathcal M$, define the observed–data moment map
\begin{align}
\label{eq:M_p_theta}
M(P,\theta)&:=\mathbb E_{P}\!\big[m(P;\theta)(O)\big],\quad
m(P;\theta)(O)
:= \bar U(\theta;X;P)+\frac{\delta}{f}\Big\{U(\theta;X,Y)-\bar U(\theta;X;P)\Big\}.    
\end{align}

Given $\{X_i\}_{i=1}^N$ and a labeled SRSWOR subset $S\subset\{1,\ldots,N\}$ of size $n$, define the target PPI score
\[
\widehat U_{\mathrm{PPI}}^{\mathrm{target}}(\theta)
:= \frac{1}{N}\sum_{i=1}^N \bar U(\theta;X_i;P_0)
 \;+\; \frac{1}{n}\sum_{j\in S}\!\Big\{ U(\theta;X_j,Y_j) - \bar U(\theta;X_j;P_0) \Big\},
\]
and the computable PPI score used by the estimator
\[
\widehat U_{\mathrm{PPI}}(\theta)
:= \frac{1}{N}\sum_{i=1}^N \widetilde U(\theta;X_i)
 \;+\; \frac{1}{n}\sum_{j\in S}\!\Big\{ U(\theta;X_j,Y_j) - \widetilde U(\theta;X_j) \Big\},
\]
where $\widetilde U(\theta;x)=U(\theta;x,m(x))$ is built from a predictor $m$ (possibly data–dependent, e.g., via cross–fitting). Let $\widehat\theta_{\mathrm{PPI}}$ be any solution of $\widehat U_{\mathrm{PPI}}(\theta)=0$.

\textbf{Assumptions.}
\begin{itemize}
\item[A1] (\emph{Identification and smoothness})
  \begin{enumerate}
  \item[(i)] $\mathbb E_{P_0}[U(\theta;X,Y)]=0$ has a unique root $\theta_0$ in a neighborhood of the truth;
  \item[(ii)] $I(\theta_0)=-\mathbb E_{P_0}[\partial_\theta U(\theta_0;X,Y)]$ exists and is nonsingular.
  \end{enumerate}

\item[A2] (\emph{Moments}) \quad $\mathbb E_{P_0}\|U(\theta_0;X,Y)\|^2<\infty$. (By Jensen’s inequality, $\mathbb{E}_{P_0}\!\big\|\bar U(\theta_0;X;P_0)\big\|^2\le \mathbb{E}_{P_0}\!\big\|U(\theta_0;X,Y)\big\|^2<\infty$.)

\item[A3] (\emph{Regularity for the plug–in score})
There exists a deterministic neighborhood $\mathcal N\subset\Theta$ of $\theta_0$
(e.g., $\mathcal N=\{\theta:\|\theta-\theta_0\|\le r\}$ for some $r>0$) such that:
\begin{enumerate}
\item[(i)] $\displaystyle \sup_{\theta\in\mathcal N}\,
\mathbb E_{P_0}\!\big\|\widetilde U(\theta;X)-\bar U(\theta;X;P_0)\big\|^2=o_p(1)$;
\item[(ii)] $\displaystyle \sup_{\theta\in\mathcal N}\,
\big\|\partial_\theta \widehat U_{\mathrm{PPI}}(\theta)-\partial_\theta M(P_0,\theta)\big\|_{\mathrm{op}}
= o_p(1)$, and for any $\hat\theta$ with $\Pr(\hat\theta\in\mathcal N)\to 1$,
\[
\widehat U_{\mathrm{PPI}}(\hat\theta)-\widehat U_{\mathrm{PPI}}(\theta_0)
= \partial_\theta M(P_0,\theta_0)\,(\hat\theta-\theta_0)+o_p(N^{-1/2}).
\]
\end{enumerate}
\end{itemize}

Then the following hold:
\begin{itemize}
\item[\textbf{\textup{(i)}}] \textbf{\emph{Efficient influence function.}}
The parameter $\theta(P)$ defined implicitly by $M(P,\theta)=0$ is pathwise differentiable at $P_0$ with efficient influence function
\[
\phi^{\mathrm{eff}}(X,Y,\delta;\theta_0)
= I(\theta_0)^{-1}\!\left[
\bar U(\theta_0;X;P_0)
+\frac{\delta}{f}\Big\{U(\theta_0;X,Y)-\bar U(\theta_0;X;P_0)\Big\}
\right].
\]

\item[\textbf{\textup{(ii)}}] \textbf{\emph{Asymptotic linearity and normality.}}
\[
\widehat\theta_{\mathrm{PPI}}-\theta_0
= \frac{1}{N}\sum_{i=1}^N \phi^{\mathrm{eff}}(O_i) \;+\; o_p\!\left(N^{-1/2}\right),
\quad
\sqrt N\,(\widehat\theta_{\mathrm{PPI}}-\theta_0)
\;\xrightarrow{d}\;
\mathcal N\!\big(0,\; \operatorname{Var}_{P_0}(\phi^{\mathrm{eff}}(O))\big),
\]
with
\begin{align*}
&\operatorname{Var}_{P_0}(\phi^{\mathrm{eff}}(O))
= I(\theta_0)^{-1}\,\Sigma\,I(\theta_0)^{-1}, \\   
&\Sigma
= \operatorname{Var}_{P_0}\!\big(\bar U(\theta_0;X;P_0)\big)
\;+\;\frac{1}{f}\,\mathbb E_{P_0}\!\Big[\operatorname{Var}\!\big(U(\theta_0;X,Y)\mid X\big)\Big]
\end{align*}
\end{itemize}
\end{theorem}

\begin{proof}
We begin with a basic identity used repeatedly in the proof. Under SRSWOR with known sampling fraction \(f\in(0,1)\), the labeling indicator \(\delta\) is
independent of \((X,Y)\), so \(\mathbb E_{P_0}[\delta\mid X,Y]=\mathbb E_{P_0}[\delta]=f\) a.s., i.e.,
\(\mathbb E_{P_0}[\delta/f\mid X,Y]=1\). Hence, by the tower property, for any integrable
\(Z(X,Y)\),
\begin{align}
\label{eq:semi_trick1}
\mathbb E_{P_0}\!\Big[\tfrac{\delta}{f}\,Z(X,Y)\Big]
&= \mathbb E_{P_0}\!\Big[\mathbb E_{P_0}\!\big[\tfrac{\delta}{f}\,Z(X,Y)\mid X,Y\big]\Big]\\
\nonumber
&= \mathbb E_{P_0}\!\big[Z(X,Y)\,\mathbb E_{P_0}[\delta/f\mid X,Y]\big]
= \mathbb E_{P_0}[Z(X,Y)].
\end{align}

\paragraph{Identification.}
By definition \eqref{eq:M_p_theta},
\[
M(P_0,\theta_0)
= \mathbb E_{P_0}\!\Big[\bar U(\theta_0;X;P_0)
+ \frac{\delta}{f}\big\{U(\theta_0;X,Y)-\bar U(\theta_0;X;P_0)\big\}\Big].
\]
Thus
\begin{align*}
M(P_0,\theta_0)
&= \mathbb E_{P_0}\!\big[\bar U(\theta_0;X;P_0)\big]
  + \mathbb E_{P_0}\!\Big[\frac{\delta}{f}\big\{U(\theta_0;X,Y)-\bar U(\theta_0;X;P_0)\big\}\Big] \\
&= \mathbb E_{P_0}\!\big[\bar U(\theta_0;X;P_0)\big]
  + \mathbb E_{P_0}\!\big[U(\theta_0;X,Y)-\bar U(\theta_0;X;P_0)\big]
  \qquad\text{(by \eqref{eq:semi_trick1})} \\
  &= \cancel{\mathbb E_{P_0}\!\big[\bar U(\theta_0;X;P_0)\big]} 
  + \mathbb E_{P_0}\!\big[U(\theta_0;X,Y)\big]- \cancel{\mathbb E_{P_0}\!\big[\bar U(\theta_0;X;P_0)\big]} \\
&= \mathbb E_{P_0}\!\big[U(\theta_0;X,Y)\big]
= 0,
\end{align*}
where the last equality is Assumption~A1(i). More generally, for any \(\theta\in\Theta\),
\[
M(P_0,\theta)
= \mathbb E_{P_0}\!\Big[\bar U(\theta;X;P_0)
+ \tfrac{\delta}{f}\big\{U(\theta;X,Y)-\bar U(\theta;X;P_0)\big\}\Big]
= \mathbb E_{P_0}\!\big[U(\theta;X,Y)\big].
\]
Hence, by Assumption~A1(i), \(\theta_0\) is the unique solution to \(M(P_0,\theta)=0\).

\paragraph{Jacobian matrix.}
The Jacobian with respect to \(\theta\) (evaluated at \((P_0,\theta_0)\)) is
\begin{align}
\nonumber
A&:=\partial_\theta M(P_0,\theta_0)\\
&= \mathbb E_{P_0}\!\Big[\partial_\theta\bar U(\theta_0;X;P_0)
   +\tfrac{\delta}{f}\big\{\partial_\theta U(\theta_0;X,Y)-\partial_\theta\bar U(\theta_0;X;P_0)\big\}\Big]\nonumber\\
&= \mathbb E_{P_0}\!\Big[\partial_\theta\bar U(\theta_0;X;P_0)\Big]
   +
   \mathbb E_{P_0}\!\Big[ \tfrac{\delta}{f}\big\{\partial_\theta U(\theta_0;X,Y)-\partial_\theta\bar U(\theta_0;X;P_0)\big\}\Big]\nonumber\\
&= \mathbb E_{P_0}\!\Big[\partial_\theta\bar U(\theta_0;X;P_0)\Big]
   +
   \mathbb E_{P_0}\!\Big[ \partial_\theta U(\theta_0;X,Y)-\partial_\theta\bar U(\theta_0;X;P_0)\Big]\nonumber \qquad\text{(by \eqref{eq:semi_trick1})}\\
&= \cancel{\mathbb E_{P_0}\!\big[\partial_\theta\bar U(\theta_0;X;P_0)\big]} 
   \;+\; \mathbb E_{P_0}\!\big[\partial_\theta U(\theta_0;X,Y)\big]
   \;-\; \cancel{\mathbb E_{P_0}\!\big[\partial_\theta\bar U(\theta_0;X;P_0)\big]} \label{eq:use-trick}\\
&= \mathbb E_{P_0}\!\big[\partial_\theta U(\theta_0;X,Y)\big]
= -\,I(\theta_0).\label{eq:Jacobian_matrix}
\end{align}
Note that \(A\) is nonsingular by Assumption~A1(ii).

\paragraph{Observed–data tangent space.}
Under SRSWOR with known $f\in(0,1)$, the observed–data density factorizes as
\[
p(o)=p(x,\delta,\delta y)
=p_X(x)\,\big[(1-f)\big]^{1-\delta}\,\big[f\,p_{Y\mid X}(y\mid x)\big]^{\delta},
\]
so only $p_X$ and $p_{Y\mid X}$ vary along regular submodels; the sampling law
for $\delta$ is fixed by assumption. The observed–data tangent space at $P_0$ is the closed
linear span
\[
\mathcal T=\overline{\mathcal T_X\oplus\mathcal T_{Y\mid X}}\subset L_0^2(P_0),
\]
with
\[
\mathcal T_X=\Big\{\,S_X(X):\ \mathbb E_{P_0}[S_X(X)]
=\int_{\mathcal X} S_X(x)\,P_{0,X}(\mathrm dx)=0\,\Big\},
\]
and
\begin{align*}
\mathcal T_{Y\mid X}&=\Big\{\,\delta\,S_{Y\mid X}(Y\mid X):\
\mathbb E_{P_0}\!\big[S_{Y\mid X}(Y\mid X)\mid X=x\big]\\
&\quad\quad=\int_{\mathcal Y} S_{Y\mid X}(y\mid x)\,P_{0,Y\mid X}(\mathrm dy\mid x)=0
\ \text{for }P_{0,X}\text{-a.e. }x\,\Big\}.    
\end{align*}
Hence any regular observed–data submodel has score
\[
S(O)=S_X(X)+\delta\,S_{Y\mid X}(Y\mid X),
\quad\text{with}\quad
\mathbb E_{P_0}[S_X(X)]=0,\ \ \mathbb E_{P_0}[S_{Y\mid X}(Y\mid X)\mid X]=0.
\]

\paragraph{Moment representer $\varphi$.}
Set the moment representer, obtained by evaluating the $m(P;\theta)(O)$ \eqref{eq:M_p_theta} at $(P_0,\theta_0)$
\begin{align}
\label{eq:moment_representer}
\varphi(O) & :=\varphi(X,Y,\delta;P_0,\theta_0)=m(P_0;\theta_0)(O) \\
&=\bar U(\theta_0;X;P_0)+\frac{\delta}{f}\Big\{U(\theta_0;X,Y)-\bar U(\theta_0;X;P_0)\Big\}\nonumber\\
&= \mathbb E_{P_0}\!\big[U(\theta_0;X,Y)\mid X\big]
  +\frac{\delta}{f}\!\left\{U(\theta_0;X,Y)-\mathbb E_{P_0}\!\big[U(\theta_0;X,Y)\mid X\big]\right\}.\nonumber
\end{align}
Note that we have
\begin{align}
\nonumber
\mathbb E_{P_0}\big\|\varphi(O)\big\|^2
&=\mathbb E_{P_0}\Big\|\bar U(\theta_0;X;P_0)+\tfrac{\delta}{f}\big\{U(\theta_0;X,Y)-\bar U(\theta_0;X;P_0)\big\}\Big\|^2\\
\nonumber
&=\mathbb E_{P_0}\big\|\bar U(\theta_0;X;P_0)\big\|^2
+\frac{2}{f}\,\mathbb E_{P_0}\!\Big[\delta\,\big\langle \bar U(\theta_0;X;P_0),\,U(\theta_0;X,Y)-\bar U(\theta_0;X;P_0)\big\rangle\Big]\\
&\quad+\frac{1}{f^2}\,\mathbb E_{P_0}\!\Big[\delta^2\,\big\|U(\theta_0;X,Y)-\bar U(\theta_0;X;P_0)\big\|^2\Big].
\label{eq:L2_psi_proof}
\end{align}
Here, the cross-term in \eqref{eq:L2_psi_proof} becomes zero because
\begin{small}
\begin{align*}
&\mathbb E_{P_0}\!\Big[\delta\,\big\langle \bar U(\theta_0;X;P_0),\,U(\theta_0;X,Y)-\bar U(\theta_0;X;P_0)\big\rangle\Big]\\
&=\mathbb E_{P_0}\!\Big[\mathbb E_{P_0}\!\big[\delta\,\langle \bar U(\theta_0;X;P_0),\,U(\theta_0;X,Y)-\bar U(\theta_0;X;P_0)\rangle\mid X,Y\big]\Big]\\
&=\mathbb E_{P_0}\!\Big[\big\langle \bar U(\theta_0;X;P_0),\,U(\theta_0;X,Y)-\bar U(\theta_0;X;P_0)\big\rangle\,\mathbb E_{P_0}[\delta\mid X,Y]\Big]\\
&=f\mathbb E_{P_0}\!\Big[\big\langle \bar U(\theta_0;X;P_0),\,U(\theta_0;X,Y)-\bar U(\theta_0;X;P_0)\big\rangle\Big]\\
&=f\,\mathbb E_{P_0}\!\Big[\mathbb E_{P_0}\!\big[\big\langle \bar U(\theta_0;X;P_0),\,U(\theta_0;X,Y)-\bar U(\theta_0;X;P_0)\big\rangle\mid X\big]\Big]\\
&=f\,\mathbb E_{P_0}\!\Big[\big\langle \bar U(\theta_0;X;P_0),\,\mathbb E_{P_0}\!\big[U(\theta_0;X,Y)\mid X\big]-\bar U(\theta_0;X;P_0)\big\rangle\Big]\\
&=0,
\end{align*}
\end{small}
and since \(\delta^2=\delta\), the last term in \eqref{eq:L2_psi_proof} becomes
\[
\mathbb E_{P_0}\!\Big[\delta^2\,\big\|U(\theta_0;X,Y)-\bar U(\theta_0;X;P_0)\big\|^2\Big]
=f\,\mathbb E_{P_0}\big\|U(\theta_0;X,Y)-\bar U(\theta_0;X;P_0)\big\|^2.
\]

Hence,
\[
\mathbb E_{P_0}\big\|\varphi(O)\big\|^2
=\mathbb E_{P_0}\big\|\bar U(\theta_0;X;P_0)\big\|^2
+\frac{1}{f}\,\mathbb E_{P_0}\big\|U(\theta_0;X,Y)-\bar U(\theta_0;X;P_0)\big\|^2.
\]
By Jensen’s inequality,
\(
\mathbb E_{P_0}\big\|\bar U(\theta_0;X;P_0)\big\|^2
\le \mathbb E_{P_0}\big\|U(\theta_0;X,Y)\big\|^2,
\)
and by the \(L^2\) projection identity,
\begin{align*}
\mathbb E_{P_0}\big\|U(\theta_0;X,Y)-\bar U(\theta_0;X;P_0)\big\|^2
&=\mathbb E_{P_0}\big\|U(\theta_0;X,Y)\big\|^2-\mathbb E_{P_0}\big\|\bar U(\theta_0;X;P_0)\big\|^2\\
&\le \mathbb E_{P_0}\big\|U(\theta_0;X,Y)\big\|^2.    
\end{align*}
Therefore, by Assumption A2,
\begin{align}
\label{eq:finite_variance_condition}
\mathbb E_{P_0}\big\|\varphi(O)\big\|^2
\le \Big(1+\frac{1}{f}\Big)\,\mathbb E_{P_0}\big\|U(\theta_0;X,Y)\big\|^2
<\infty.    
\end{align}
Furthermore, by the identity \eqref{eq:semi_trick1} (valid under SRSWOR with known $f$),
\begin{align}
\label{eq:mean_zero_condition}
\mathbb E_{P_0}[\varphi(O)]
&=\mathbb E_{P_0}\!\big[\bar U(\theta_0;X;P_0)\big]
+\mathbb E_{P_0}\!\Big[\tfrac{\delta}{f}\{U(\theta_0;X,Y)-\bar U(\theta_0;X;P_0)\}\Big]\\
\nonumber
&=\mathbb E_{P_0}\!\big[U(\theta_0;X,Y)\big]=0.    
\end{align}
By \eqref{eq:finite_variance_condition} and \eqref{eq:mean_zero_condition}, it follows that
\(\varphi\in L_0^2(P_0) = \{\,g \;:\;
\mathbb E_{P_0}\!\big[\|g(O)\|^2\big]<\infty,\ \ \mathbb E_{P_0}[g(O)]=0\}\).

Define
\[
S_X^\star(X):=\bar U(\theta_0;X;P_0)
\quad\text{and}\quad
S_{Y\mid X}^\star(Y\mid X):=\frac{1}{f}\Big\{U(\theta_0;X,Y)-\bar U(\theta_0;X;P_0)\Big\}.
\]
Then by the tower property,
\begin{align*}
\mathbb E_{P_0}[S_X^\star(X)]
&=\mathbb E_{P_0}\!\big[\mathbb E_{P_0}\{U(\theta_0;X,Y)\mid X\}\big] \\
&=\int_{\mathcal X}\!\left\{\int_{\mathcal Y} U(\theta_0;x,y)\,
  P_{0,Y\mid X}(\mathrm dy\mid x)\right\} P_{0,X}(\mathrm dx) \\
&=\iint_{\mathcal X\times\mathcal Y} U(\theta_0;x,y)\,
   P_{0,Y\mid X}(\mathrm dy\mid x)\,P_{0,X}(\mathrm dx) \\
&=\iint_{\mathcal X\times\mathcal Y} U(\theta_0;x,y)\,
   P_{0,XY}(\mathrm dx,\mathrm dy)
=\mathbb E_{P_0}\!\big[U(\theta_0;X,Y)\big]=0 \quad \text{by A1 (i)},
\end{align*}
hence $S_X^\star\in\mathcal T_X=\{S_X(X):\mathbb E_{P_0}[S_X(X)]=0\}$. Moreover,
\begin{align*}
\mathbb E_{P_0}\!\big[S_{Y\mid X}^\star(Y\mid X)\mid &X\big]
= \mathbb E_{P_0}\!\left[\frac{1}{f}\Big\{U(\theta_0;X,Y)-\bar U(\theta_0;X;P_0)\Big\}\,\Big|\,X\right]\\
&= \frac{1}{f}\left\{\mathbb E_{P_0}\!\big[U(\theta_0;X,Y)\mid X\big]-\bar U(\theta_0;X;P_0)\right\}\\
&= \frac{1}{f}\left\{\int_{\mathcal Y} U(\theta_0;X,y)\,P_{0,Y\mid X}(\mathrm dy\mid X)
         - \int_{\mathcal Y} U(\theta_0;X,y)\,P_{0,Y\mid X}(\mathrm dy\mid X)\right\}\\
&= 0
\end{align*}
so $\delta\,S_{Y\mid X}^\star(Y\mid X)\in\mathcal T_{Y\mid X}
=\{\delta\,S_{Y\mid X}(Y\mid X):\mathbb E_{P_0}[S_{Y\mid X}(Y\mid X)\mid X]=0\}$.
Therefore
\[
\varphi(O)=S_X^\star(X)+\delta\,S_{Y\mid X}^\star(Y\mid X)\in
\mathcal T_X\oplus\mathcal T_{Y\mid X}\subset\mathcal T,
\quad\text{and hence } \Pi_{\mathcal T}\varphi=\varphi.
\]

\paragraph{Pathwise derivative of $M$ in $P$.}
Fix any regular observed–data submodel $\{P_\varepsilon:\varepsilon\in(-\eta,\eta)\}$ through $P_0$
with score $S\in\mathcal T$, which decomposes as
$S(O)=S_X(X)+\delta\,S_{Y\mid X}(Y\mid X)$, where
$\mathbb E_{P_0}[S_X(X)]=0$ and $\mathbb E_{P_0}[S_{Y\mid X}(Y\mid X)\mid X]=0$.
By definition \eqref{eq:M_p_theta},
\begin{align*}
&M(P_\varepsilon,\theta_0)=\mathbb E_{P_\varepsilon}\!\big[m(P_\varepsilon;\theta_0)(O)\big],\\    
&m(P_\varepsilon;\theta_0)(O)
=\bar U(\theta_0;X;P_\varepsilon)+\frac{\delta}{f}\Big\{U(\theta_0;X,Y)-\bar U(\theta_0;X;P_\varepsilon)\Big\}.
\end{align*}

Applying Lemma~\ref{lemma:product_rule} (with
$g_\varepsilon(O)=m(P_\varepsilon;\theta_0)(O)$) along smooth submodels, we obtain
\begin{align}
\label{eq:product_form_M}
\left.\frac{d}{d\varepsilon}M(P_\varepsilon,\theta_0)\right|_{\varepsilon=0}
&=\left.\frac{d}{d\varepsilon}\,\mathbb E_{P_\varepsilon}\!\big[m(P_\varepsilon;\theta_0)(O)\big]\right|_{\varepsilon=0}  \\
\nonumber
&=\underbrace{\mathbb E_{P_0}\!\big[m(P_0;\theta_0)(O)\,S(O)\big]}_{\text{change of law}}
\;+\;
\underbrace{\mathbb E_{P_0}\!\Big[\left.\frac{d}{d\varepsilon}m(P_\varepsilon;\theta_0)(O)\right|_{\varepsilon=0}\Big]}_{\text{change of integrand}}.
\end{align}
Here, rewrite $m(P;\theta_0)(O)$ by 
\begin{align*}
m(P;\theta_0)(O)&=\bar U(\theta_0;X;P)+\frac{\delta}{f}\{U(\theta_0;X,Y)-\bar U(\theta_0;X;P)\}\\
&=\left(1 - \frac{\delta}{f}\right)\bar U(\theta_0;X;P)
+\frac{\delta}{f}U(\theta_0;X,Y).  
\end{align*}
Note that $m(P;\theta_0)(O)$ depends on $P$ (more specifically, the conditional distribution $P_{Y\mid X}$) only through the conditional mean $\bar U(\theta_0;X;P) = 
\mathbb{E}_{P}\!\big[U(\theta_0;X,Y)\mid X\big]
= \int_{\mathcal Y} U(\theta_0;X,y)\, P_{Y\mid X}(dy\mid X).
$

Hence, the integrand in the second term on the right-hand side of \eqref{eq:product_form_M} can be simplified to
\[
\left.\frac{d}{d\varepsilon}m(P_\varepsilon;\theta_0)(O)\right|_{\varepsilon=0}
=\Big(1-\frac{\delta}{f}\Big)\left.\frac{d}{d\varepsilon}\bar U(\theta_0;X;P_\varepsilon)\right|_{\varepsilon=0}.
\]
For a regular observed–data submodel, recall that the score decomposes as
$S(O)=S_X(X)+\delta\,S_{Y\mid X}(Y\mid X)$ with $\mathbb E_{P_0}[S_X(X)]=0$ and
$\mathbb E_{P_0}[S_{Y\mid X}(Y\mid X)\mid X]=0$. The conditional score identity gives
\begin{align*}
\left.\frac{d}{d\varepsilon}\bar U(\theta_0;X;P_\varepsilon)\right|_{\varepsilon=0}
&= \left.\frac{d}{d\varepsilon}\int_{\mathcal Y} U(\theta_0;X,y)\,P_{\varepsilon,Y\mid X}(dy\mid X)\right|_{\varepsilon=0} \\
&= \left.\frac{d}{d\varepsilon}\int_{\mathcal Y} U(\theta_0;X,y)\,p_\varepsilon(y\mid X)\,d\nu(y)\right|_{\varepsilon=0} \\
&= \int_{\mathcal Y} U(\theta_0;X,y)\,\left.\frac{\partial}{\partial\varepsilon}p_\varepsilon(y\mid X)\right|_{\varepsilon=0}\,d\nu(y) \\
&= \int_{\mathcal Y} U(\theta_0;X,y)\,p_0(y\mid X)\,S_{Y\mid X}(y\mid X)\,d\nu(y) \\
&= \mathbb E_{P_0}\!\big[\,U(\theta_0;X,Y)\,S_{Y\mid X}(Y\mid X)\mid X\,\big]\;=:\;\xi(X),
\end{align*}
where $\nu$ is a $\sigma$-finite dominating measure on $\mathcal Y$, 
$p_\varepsilon(y\mid X):=\frac{dP_{\varepsilon,Y\mid X}}{d\nu}(y\mid X)$, 
$S_{Y\mid X}(y\mid X):=\left.\partial_\varepsilon\log p_\varepsilon(y\mid X)\right|_{\varepsilon=0}$, 
so $\left.\partial_\varepsilon p_\varepsilon(y\mid X)\right|_{\varepsilon = 0}=p_0(y\mid X)\,S_{Y\mid X}(y\mid X)$, 
$\mathbb E_{P_0}\!\big[S_{Y\mid X}(Y\mid X)\mid X\big]=0$, 
and dominated convergence justifies differentiation under the integral.

Under SRSWOR, using $\delta\!\perp\!(X,Y)$ and $\mathbb E_{P_0}[\delta|X,Y] = \mathbb E_{P_0}[\delta] =f$,
\[
\begin{aligned}
\mathbb E_{P_0}\!\bigg[\Big(1-\frac{\delta}{f}\Big)\,\xi(X)\bigg]
&= \mathbb E_{P_0}\!\big[\xi(X)\big]
   - \mathbb E_{P_0}\!\bigg[\frac{\delta}{f}\,\xi(X)\bigg] \\
&= \mathbb E_{P_0}\!\big[\xi(X)\big]
   - \mathbb E_{P_0}\!\Big[\,\mathbb E_{P_0}\!\big[ \tfrac{\delta}{f}\,\xi(X) \,\big|\, X,Y \big]\,\Big] \\
&= \mathbb E_{P_0}\!\big[\xi(X)\big]
   - \mathbb E_{P_0}\!\Big[\, \xi(X)\,\mathbb E_{P_0}\!\big[ \tfrac{\delta}{f} \,\big|\, X,Y \big] \,\Big] \\
&= \mathbb E_{P_0}\!\big[\xi(X)\big]
   - \mathbb E_{P_0}\!\big[\xi(X)\big] \;=\; 0 .
\end{aligned}
\]
Therefore the change–of–integrand term (i.e., the second term on the right–hand side of \eqref{eq:product_form_M}) vanishes, and we conclude
\begin{align*}
\left.\frac{d}{d\varepsilon}M(P_\varepsilon,\theta_0)\right|_{\varepsilon=0}
&=\mathbb E_{P_0}\!\big[m(P_0;\theta_0)(O)\,S(O)\big]\\
&=\mathbb E_{P_0}\!\big[\varphi(O)\,S(O)\big]
=\langle \varphi,\,S\rangle_{L_0^2(P_0)}\quad \text{for any score }S\in\mathcal T.    
\end{align*}
Moreover, since \(\varphi\in\mathcal T\) (as shown above), if we view the map
\[
\mathcal L':\mathcal T\to\mathbb R^p,\qquad 
\mathcal L'(S):=\left.\frac{d}{d\varepsilon}M(P_\varepsilon,\theta_0)\right|_{\varepsilon=0}(S),
\]
as a (coordinatewise) linear functional of the score direction \(S\), then
\[
\mathcal L'(S)=\langle \varphi,S\rangle_{L_0^2(P_0)}\quad\text{for all }S\in\mathcal T.
\]
By the Riesz representation theorem (applied componentwise), \(\varphi\) is the unique canonical moment representer (canonical gradient of the moment map) at \((P_0,\theta_0)\).

\paragraph{Efficient influence function $\phi^{\mathrm{eff}}$.}
Because $\varphi(O)$ in \eqref{eq:moment_representer} belongs to the tangent space $\mathcal T$ and the regularity conditions hold (we already proved), Lemma~\ref{lem:implicit-moment-eif} yields the efficient influence function
\begin{align}
\label{eq:EIF_PPI}
\phi^{\mathrm{eff}}(O) 
&=\phi^{\mathrm{eff}}(X,Y,\delta;P_0,\theta_0)
= -\,A^{-1}\,\varphi(O)\\
\nonumber
&= I(\theta_0)^{-1}\!\left[
  \bar U(\theta_0;X;P_0)
  +\frac{\delta}{f}\Big\{U(\theta_0;X,Y)-\bar U(\theta_0;X;P_0)\Big\}
\right]\\
\nonumber
&= I(\theta_0)^{-1}\!\left[
  \mathbb E_{P_0}\!\big[U(\theta_0;X,Y)\mid X\big]
  +\frac{\delta}{f}\Big\{U(\theta_0;X,Y)-\mathbb E_{P_0}\!\big[U(\theta_0;X,Y)\mid X\big]\Big\}
\right].
\end{align}
In other words, $\phi^{\mathrm{eff}}$ is the (unique) Riesz representer of the pathwise derivative of $\Psi(P)=\theta(P)$—i.e., the canonical gradient—satisfying, for any score $S\in\mathcal T$,
\begin{align}
\label{eq:EIF_Riez_Form}
\left.\frac{d}{d\varepsilon}\Psi(P_\varepsilon)\right|_{\varepsilon=0}
= \big\langle \phi^{\mathrm{eff}},\, S\big\rangle_{L_0^2(P_0)}
= \big\langle -\,A^{-1}\Pi_{\mathcal T}\varphi,\, S \big\rangle_{L_0^2(P_0)}
= \big\langle -\,A^{-1}\varphi,\, S \big\rangle_{L_0^2(P_0)},    
\end{align}
where $\langle g,S\rangle_{L_0^2(P_0)}:=\mathbb E_{P_0}[g(O)S(O)]$ (applied componentwise).

\paragraph{Asymptotic linearity and normality.}
A first–order expansion of the computable score at $\theta_0$ (Assumption A3(ii)) gives
\begin{align}
\label{eq:first_expansion_PPI_TARGET}
0=\widehat U_{\mathrm{PPI}}(\widehat\theta_{\mathrm{PPI}})
=\widehat U_{\mathrm{PPI}}(\theta_0)
+\partial_\theta M(P_0,\theta_0)\,(\widehat\theta_{\mathrm{PPI}}-\theta_0)
+ o_p(N^{-1/2}).    
\end{align}
Assumption A3(i), together with the SRSWOR identity, $(1/n)\sum_{j\in S} g(X_j,Y_j)= $ $(1/N)\sum_{i=1}^N $ $(\delta_i/f) g(X_i,Y_i)$ with $f=n/N$, implies
\begin{align}
\label{eq:second_expansion_PPI_TARGET}
\sqrt N\,\widehat U_{\mathrm{PPI}}(\theta_0)
=\sqrt N\,\widehat U_{\mathrm{PPI}}^{\mathrm{target}}(\theta_0)+o_p(1)
=\frac{1}{\sqrt N}\sum_{i=1}^N \varphi(O_i)+o_p(1),    
\end{align}
because
\begin{align*}
\widehat U_{\mathrm{PPI}}^{\mathrm{target}}(\theta_0)
&= \frac{1}{N}\sum_{i=1}^N \bar U(\theta_0;X_i;P_0)
  + \frac{1}{n}\sum_{j\in S}\!\Big\{U(\theta_0;X_j,Y_j)-\bar U(\theta_0;X_j;P_0)\Big\}\\
&= \frac{1}{N}\sum_{i=1}^N \bar U(\theta_0;X_i;P_0)
  + \frac{1}{N}\sum_{i=1}^N \frac{\delta_i}{f}\!\Big\{U(\theta_0;X_i,Y_i)-\bar U(\theta_0;X_i;P_0)\Big\}\\ 
&= \frac{1}{N}\sum_{i=1}^N \left[\bar U(\theta_0;X_i;P_0)
  +    \frac{\delta_i}{f} \!\Big\{U(\theta_0;X_i,Y_i)-\bar U(\theta_0;X_i;P_0)\Big\}\right]\\ 
&= \frac{1}{N}\sum_{i=1}^N \varphi(O_i). \quad \text{by definition of $\varphi(O)$ \eqref{eq:moment_representer}}
\end{align*}

Rearranging the expansion \eqref{eq:first_expansion_PPI_TARGET} to get
\[
\partial_\theta M(P_0,\theta_0)\,(\widehat\theta_{\mathrm{PPI}}-\theta_0)
= -\,\widehat U_{\mathrm{PPI}}(\theta_0) + o_p(N^{-1/2}).
\]
By \eqref{eq:Jacobian_matrix}, $\partial_\theta M(P_0,\theta_0)=-I(\theta_0)$, so
\[
-\,I(\theta_0)\,(\widehat\theta_{\mathrm{PPI}}-\theta_0)
= -\,\widehat U_{\mathrm{PPI}}(\theta_0) + o_p(N^{-1/2}),
\]
which yields (using nonsingularity of $I(\theta_0)$ from A1(ii))
\[
\sqrt N\,(\widehat\theta_{\mathrm{PPI}}-\theta_0)
= I(\theta_0)^{-1}\,\sqrt N\,\widehat U_{\mathrm{PPI}}(\theta_0) + o_p(1).
\]
By \eqref{eq:second_expansion_PPI_TARGET} and by Slutsky’s lemma,
\begin{align*}
\sqrt N (\widehat\theta_{\mathrm{PPI}}-\theta_0)
&= I(\theta_0)^{-1} \frac{1}{\sqrt N}\sum_{i=1}^N \varphi(O_i) + o_p(1)= \frac{1}{\sqrt N}\sum_{i=1}^N I(\theta_0)^{-1} \varphi(O_i) + o_p(1)\\
&= \frac{1}{\sqrt N}\sum_{i=1}^N \phi^{\mathrm{eff}}(O_i) + o_p(1),    
\end{align*}
since $\phi^{\mathrm{eff}}(O)=I(\theta_0)^{-1}\varphi(O)$ \eqref{eq:EIF_PPI}.

Since $\mathbb E_{P_0}[\varphi(O)]=0$ and $\mathbb E_{P_0}\|\varphi(O)\|^2<\infty$ (shown above using A2),
the multivariate CLT yields
\[
\sqrt N\,(\widehat\theta_{\mathrm{PPI}}-\theta_0)\ \xrightarrow{d}\ 
\mathcal N\!\big(0,\ \operatorname{Var}_{P_0}(\phi^{\mathrm{eff}}(O))\big),
\qquad
\widehat\theta_{\mathrm{PPI}}-\theta_0
= \frac{1}{N}\sum_{i=1}^N \phi^{\mathrm{eff}}(O_i) + o_p(N^{-1/2}).
\]

For the asymptotic variance, we first write the variance of the moment representer $\varphi(O)$ \eqref{eq:moment_representer}
\[
\Sigma\;:=\;\operatorname{Var}_{P_0}\!\big(\varphi(O)\big)
=\operatorname{Var}_{P_0}\!\Big(\bar U(\theta_0;X;P_0)
+\frac{\delta}{f}\big\{U(\theta_0;X,Y)-\bar U(\theta_0;X;P_0)\big\}\Big).
\]
Expanding,
\begin{align*}
\Sigma
&=\operatorname{Var}_{P_0}\!\big(\bar U(\theta_0;X;P_0)\big)
+2\operatorname{Cov}_{P_0}\!\Big(\bar U(\theta_0;X;P_0),\,\frac{\delta}{f}\{U(\theta_0;X,Y)-\bar U(\theta_0;X;P_0)\}\Big)\\
&\quad\quad\quad +\operatorname{Var}_{P_0}\!\Big(\frac{\delta}{f}\{U(\theta_0;X,Y)-\bar U(\theta_0;X;P_0)\}\Big).    
\end{align*}
The cross term is zero:
\begin{align*}
&\operatorname{Cov}_{P_0}\!\Big(\bar U(\theta_0;X;P_0),\,\tfrac{\delta}{f}\{U(\theta_0;X,Y)-\bar U(\theta_0;X;P_0)\}\Big)\\
&\quad\quad=\mathbb E_{P_0}\!\Big[\bar U(\theta_0;X;P_0)\,\tfrac{\delta}{f}\{U(\theta_0;X,Y)-\bar U(\theta_0;X;P_0)\}^\top\Big]\\
&\quad\quad =\mathbb E_{P_0}\!\Big[\bar U(\theta_0;X;P_0)\,\mathbb E_{P_0}\!\big[\tfrac{\delta}{f}\{U(\theta_0;X,Y)-\bar U(\theta_0;X;P_0)\}^\top\mid X,Y\big]\Big]\\
&\quad\quad=\mathbb E_{P_0}\!\Big[\bar U(\theta_0;X;P_0)\{U(\theta_0;X,Y)-\bar U(\theta_0;X;P_0)\}^\top\Big]\\
&\quad\quad=\mathbb E_{P_0}\!\Big[\bar U(\theta_0;X;P_0)\,\mathbb E_{P_0}\!\big[\{U(\theta_0;X,Y)-\bar U(\theta_0;X;P_0)\}^\top\mid X\big]\Big]=0,
\end{align*}
since $\mathbb E_{P_0}\!\big[U(\theta_0;X,Y)-\bar U(\theta_0;X;P_0)\mid X\big]=0$. 

For the last term,
\begin{align*}
&\operatorname{Var}_{P_0}\!\Big(\frac{\delta}{f}\{U(\theta_0;X,Y)-\bar U(\theta_0;X;P_0)\}\Big)\\
&\qquad =\mathbb E_{P_0}\!\Big[\Big(\frac{\delta}{f}\Big)^2\{U(\theta_0;X,Y)-\bar U(\theta_0;X;P_0)\}
\{U(\theta_0;X,Y)-\bar U(\theta_0;X;P_0)\}^\top\Big]\\    
&\qquad
=\frac{1}{f}\,\mathbb E_{P_0}\!\Big[\{U(\theta_0;X,Y)-\bar U(\theta_0;X;P_0)\}
\{U(\theta_0;X,Y)-\bar U(\theta_0;X;P_0)\}^\top\Big]\\
&\qquad=\frac{1}{f}\,\mathbb E_{P_0}\!\big[\operatorname{Var}\!\big(U(\theta_0;X,Y)\mid X\big)\big]
\end{align*}
where we used $\delta^2=\delta$ and $\mathbb E_{P_0}[\delta\mid X,Y]=f$. 
Hence, we have 
\[
\Sigma =\;\operatorname{Var}_{P_0}\!\big(\varphi(O)\big)
=\operatorname{Var}_{P_0}\!\big(\bar U(\theta_0;X;P_0)\big)
\;+\;\frac{1}{f}\,\mathbb E_{P_0}\!\Big[\operatorname{Var}\!\big(U(\theta_0;X,Y)\mid X\big)\Big].
\]
Finally, the variance of the efficient influence function $\phi^{\mathrm{eff}}(O)$ \eqref{eq:EIF_PPI} is 
$$
\operatorname{Var}_{P_0}\!\big(\phi^{\mathrm{eff}}(O)\big)
=
\operatorname{Var}_{P_0}\!\big(I(\theta_0)^{-1} \varphi(O)\big)=
I(\theta_0)^{-1} \operatorname{Var}_{P_0}\!\big(\varphi(O)\big)I(\theta_0)^{-1} =
I(\theta_0)^{-1}\,\Sigma\,I(\theta_0)^{-1}.$$
\end{proof}

\begin{corollary}[Variance equivalence with M–estimation]
\label{cor:var-equivalence-calibrated}
Assume the conditions of Theorems~\ref{thm:eif-ppi-srswor} and~\ref{thm:ppi-linear}.
Suppose the predictor $m$ is score-calibrated at the truth in the sense that:
\[
\widetilde U(\theta_0;X)=U(\theta_0;X,m(X))
=\bar U(\theta_0;X;P_0)=\mathbb E_{P_0}\!\big[U(\theta_0;X,Y)\mid X\big].
\]
Then the semiparametric efficiency variance equals the M–estimation variance:
\[
\operatorname{Var}_{P_0}\!\big(\phi^{\mathrm{eff}}(O)\big)
=\Sigma_f=I(\theta_0)^{-1}\,\Sigma\,I(\theta_0)^{-1}
=V_1+\big(f^{-1}-1\big)V_2,
\]
where
\[
\Sigma
=\operatorname{Var}_{P_0}\!\big(\bar U(\theta_0;X;P_0)\big)
+\frac{1}{f}\,\mathbb E_{P_0}\!\Big[\operatorname{Var}\!\big(U(\theta_0;X,Y)\mid X\big)\Big],
\]
and, in the notation of Corollary~\ref{cor:ppi-wald},
\[
V_1:=I(\theta_0)^{-1}\,\mathbb E_{P_0}\!\big[U(\theta_0;X,Y)^{\otimes2}\big]\,I(\theta_0)^{-1},
\qquad
V_2:=I(\theta_0)^{-1}\,\mathbb E_{P_0}\!\big[\Delta(\theta_0;X,Y)^{\otimes2}\big]\,I(\theta_0)^{-1},
\]
where $\Delta(\theta_0;X,Y):=U(\theta_0;X,Y)-U(\theta_0;X,m(X))$ and $a^{\otimes2}:=aa^\top$.

Equivalently,
\[
\Sigma
=\mathbb E_{P_0}\!\big[U(\theta_0;X,Y)^{\otimes2}\big]
+\Big(f^{-1}-1\Big)\,\mathbb E_{P_0}\!\big[\Delta(\theta_0;X,Y)^{\otimes2}\big],
\]
so $I(\theta_0)^{-1}\Sigma I(\theta_0)^{-1}=V_1+(f^{-1}-1)V_2$.
\end{corollary}

\begin{proof} 
Calibration condition
\[
\underbrace{\widetilde U(\theta_0;X)=U(\theta_0;X,m(X))}_{\text{used in M–estimation, Thm.~\ref{thm:ppi-linear}}}
\;=\;
\underbrace{\bar U(\theta_0;X;P_0)=\mathbb E_{P_0}\!\big[U(\theta_0;X,Y)\mid X\big]}_{\text{used in semiparametric efficiency, Thm.~\ref{thm:eif-ppi-srswor}}},
\]
implies that the predictor $m$ must be chosen to satisfy the \emph{\(U\)-}calibration equation $U(\theta_0; x, m(x)) $ $=\mathbb E_{P_0}\!\big[U(\theta_0; x, Y)\mid X=x\big]$ for $P_0$-a.e. $x$.

Introduce the shorthands
\[
U_0:=U(\theta_0;X,Y),\,
\bar U_0:=\bar U(\theta_0;X;P_0)=\mathbb E_{P_0}[U_0\mid X],\,
\Delta_0:=U(\theta_0;X,Y)-U(\theta_0;X,m(X)).
\]
Under the calibration condition, \(U(\theta_0;X,m(X))=\bar U_0\), so \(\Delta_0=U_0-\bar U_0\) and
\(\mathbb E_{P_0}[\Delta_0\mid X]=\mathbb E_{P_0}[U_0 - \bar U_0\mid X]=\mathbb E_{P_0}[U_0\mid X] - \mathbb E_{P_0}[\bar U_0\mid X]=\bar U_0 - \bar U_0 =0\). Also, by Assumption~A1(i) in Theorem~\ref{thm:eif-ppi-srswor},
\(\mathbb E_{P_0}[U_0]=0\) and hence \(\mathbb E_{P_0}[\bar U_0]=0\).

From Theorem~\ref{thm:eif-ppi-srswor}(ii), we have
\[
\Var_{P_0}\!\big(\phi^{\mathrm{eff}}(O)\big)
=I(\theta_0)^{-1}\,\Sigma\,I(\theta_0)^{-1},\qquad
\Sigma=\underbrace{\Var_{P_0}(\bar U_0)}_{\star} +\frac{1}{f}\,\mathbb E_{P_0}\!\big[\Var_{P_0}(U_0\mid X)\big].
\]
By the law of total variance, we have
\[
\Var_{P_0}(U_0)
= \Var_{P_0}\!\big(\,\mathbb E_{P_0}[U_0\mid X]\,\big)
  + \mathbb E_{P_0}\!\big[\Var_{P_0}(U_0\mid X)\big]
= \underbrace{\Var_{P_0}(\bar U_0)}_{\star}  + \mathbb E_{P_0}\!\big[\Var_{P_0}(U_0\mid X)\big],
\]
or equivalently, 
\[
\underbrace{\Var_{P_0}(\bar U_0)}_{\star}
= \Var_{P_0}(U_0) - \mathbb E_{P_0}\!\big[\Var_{P_0}(U_0\mid X)\big] = 
\mathbb E_{P_0}[U_0^{\otimes 2}] - \mathbb E_{P_0}\!\big[\Var_{P_0}(U_0\mid X)\big],
\]
since $\Var_{P_0}(U_0)=\mathbb E_{P_0}[U_0^{\otimes 2}]$ due to $\mathbb E_{P_0}[U_0]=0$. 

Therefore, $\Sigma$ can be expressed as
\begin{align}
\nonumber
\Sigma
&= 
\underbrace{\mathbb E_{P_0}[U_0^{\otimes 2}] - \mathbb E_{P_0}\!\big[\Var_{P_0}(U_0\mid X)\big]}_{\star}
+\frac{1}{f}\,\mathbb E_{P_0}\!\big[\Var_{P_0}(U_0\mid X)\big]\\
\label{eq:Sigma-rewrite-U0}
&=
\mathbb E_{P_0}[U_0^{\otimes2}]
+\Big(\tfrac{1}{f}-1\Big)\,\mathbb E_{P_0}\!\big[\Var_{P_0}(U_0\mid X)\big].
\end{align}

Under calibration, \(\Delta_0=U_0-\bar U_0\); hence
\[
\Var_{P_0}(U_0\mid X)
=\mathbb E_{P_0}\!\big[(U_0-\bar U_0)(U_0-\bar U_0)^\top\mid X\big]
=\mathbb E_{P_0}\!\big[\Delta_0^{\otimes2}\mid X\big].
\]
Taking expectations gives
\begin{equation}\label{eq:Delta-condvar}
\mathbb E_{P_0}\!\big[\Var_{P_0}(U_0\mid X)\big]=\mathbb E_{P_0}[\Delta_0^{\otimes2}].
\end{equation}

Substituting \eqref{eq:Delta-condvar} into \eqref{eq:Sigma-rewrite-U0},
\[
\Sigma
=\mathbb E_{P_0}[U_0^{\otimes2}]
+\Big(\tfrac{1}{f}-1\Big)\,\mathbb E_{P_0}[\Delta_0^{\otimes2}],
\]
and therefore
\begin{align*}
\Var_{P_0}\!\big(\phi^{\mathrm{eff}}(O)\big)
&=I(\theta_0)^{-1} \left\{ 
\mathbb E_{P_0}[U_0^{\otimes2}]
+\Big(\tfrac{1}{f}-1\Big)\,\mathbb E_{P_0}[\Delta_0^{\otimes2}]
\right\} I(\theta_0)^{-1} 
\\
&=
\underbrace{I(\theta_0)^{-1}\,\mathbb E_{P_0}[U_0^{\otimes2}]\,I(\theta_0)^{-1}}_{V_1}
+
\Big(\tfrac{1}{f}-1\Big)\,
\underbrace{I(\theta_0)^{-1}\,\mathbb E_{P_0}[\Delta_0^{\otimes2}]\,I(\theta_0)^{-1}}_{V_2}.    
\end{align*}
With the notation of Corollary~\ref{cor:ppi-wald},
\[
V_1:=I(\theta_0)^{-1}\,\mathbb E_{P_0}[U_0^{\otimes2}]\,I(\theta_0)^{-1},
\qquad
V_2:=I(\theta_0)^{-1}\,\mathbb E_{P_0}[\Delta_0^{\otimes2}]\,I(\theta_0)^{-1},
\]
so \(\Var_{P_0}(\phi^{\mathrm{eff}}(O))=V_1+(f^{-1}-1)V_2=\Sigma_f\).
\end{proof}

\subsection{Asymptotic properties of cross-fit prediction-powered inference with sample-splitting}\label{subsec:Asymptotic Properties of CF-PPI}
This subsection provides the proofs of the lemmas and theorems for cross-fit PPI with sample-splitting in the semisupervised mean estimation problem, and it proves Theorem~5.1 in the main document. We also provide proofs of the statements that were mentioned there without proof.

\begin{figure}[h!]
    \centering
    \includegraphics[width=0.75\linewidth]{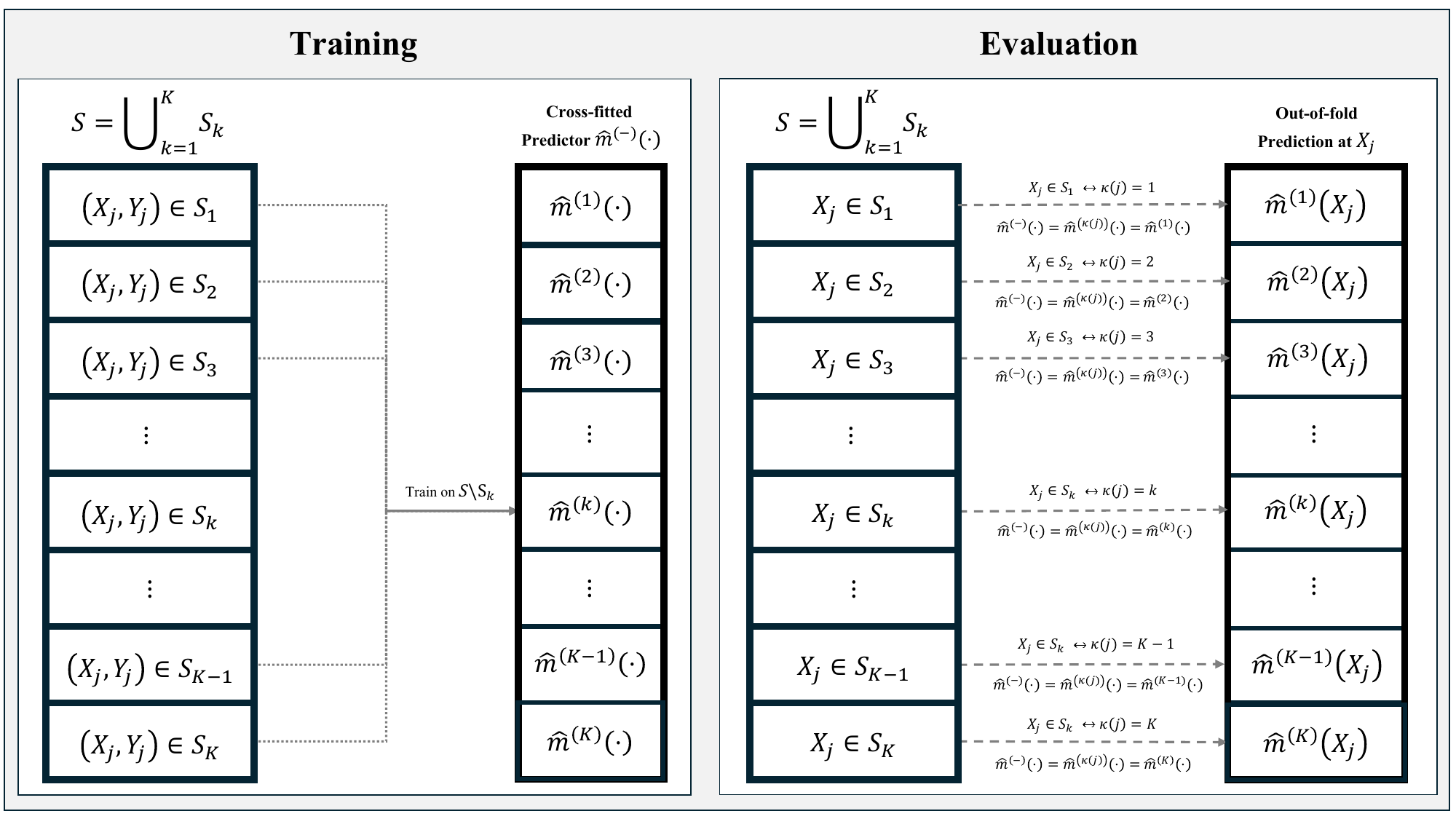}
    \caption{Illustration of the cross-prediction on the labeled set. The labeled set \(S\) is partitioned into \(K\) folds.
For each fold \(k\), a predictor \(\widehat m^{(k)}\) is trained on the dataset of covariates and responses from the complement \(S \setminus S_k\).
Each labeled covariate \(X_j \in S_k\) (or, symbolically, \(j \in S_k\)) is then evaluated using its corresponding out-of-fold predictor \(\widehat m^{(k)}(X_j)\),
ensuring that no labeled observation is evaluated by a model trained on its own label.
This induces conditional independence between the fitted predictor and the labeled residuals, thereby preventing label reuse.}
    \label{fig:Cross_Fitting_Idea}
\end{figure}

Before turning to the proof, we briefly recap the construction of the out-of-fold predictor function \(\widehat m^{(-)}(\cdot)\).
First, we partition the labeled set \(S\) into \(K\) folds \(S_1, \ldots, S_K\).
For the labeled covariates, for each fold \(k \in \{1, \ldots, K\}\), we fit a predictor \(\widehat m^{(k)}\) using only the labels in \(S \setminus S_k\).
We then use the resulting out-of-fold prediction \(\widehat m^{(k)}(X_i)\) for every labeled index \(i \in S_k\). Thus, the value of the out-of-fold predictor evaluated at a labeled covariate \(X_j\) is $\widehat m^{(-)}(X_j) = \widehat m^{(\kappa(j))}(X_j), \,\, j \in S,$ where \(\kappa(j)\) denotes the fold index such that \(j \in S_{\kappa(j)}\).
This construction is illustrated in Figure~\ref{fig:Cross_Fitting_Idea}.
For the unlabeled covariates, we use an aggregate predictor—either a single fit on all labeled data,
\(\widehat m^\star(x) = \widehat m_{\mathrm{all}}(x)\), or the average of the \(K\) fold-specific predictors, $\widehat m^\star(x) = \frac{1}{K} \sum_{k=1}^K \widehat m^{(k)}(x).$

\begin{theorem}[Consistency of the CF-PPI estimator for the mean]
\label{thm:cfppi_mean_consistency}
Let $(X,Y)$ have joint law $P$ with $X\sim P_X$ and $Y=m_0(X)+\varepsilon$, where $\mathbb E[\varepsilon\mid X]=0$ and $\mathbb E[Y^2]<\infty$. The target parameter is the population mean, $\theta_0:=\mathbb E[Y]$. Let $\{X_i\}_{i=1}^N \stackrel{\mathrm{i.i.d.}}{\sim} P_X$ be an unlabeled sample and $\{(X_j,Y_j)\}_{j\in S}$, $|S|=n$, be an independent labeled sample from $P$, with $N,n\to\infty$ and $n/N\to f\in(0,1)$.
Let $\widehat m^{(-)}$ be the cross–fitted predictor described above, and consider the CF–PPI estimator
\[
\widehat\theta^{\mathrm{cf}}_{\mathrm{PPI}}
=\frac{1}{N}\sum_{i=1}^N \widehat m^{(-)}(X_i)
+\frac{1}{n}\sum_{j\in S}\{Y_j-\widehat m^{(-)}(X_j)\}.
\]
If the out–of–fold error is stochastically bounded in $L_2(P_X)$,
\begin{align}
\|\widehat m^{(-)}-m_0\|_{L_2(P_X)}=\Big(\mathbb E\big[(\widehat m^{(-)}(X)-m_0(X))^2\big]\Big)^{1/2}=O_p(1),    
\label{eq:nuisance_condition_for_consistency_CF_PPI}
\end{align}
then $\widehat\theta^{\mathrm{cf}}_{\mathrm{PPI}}\xrightarrow{p}\theta_0$.
\end{theorem}

\begin{proof}

Using the notations of empirical process theory, the difference between the CF-PPI estimator and the target mean parameter can be written as follows:
\begin{align*}
&\widehat\theta^{\mathrm{cf}}_{\mathrm{PPI}}-\theta_0
= \Big\{P_N \widehat m^{(-)}\Big\} + \Big\{P_n\big(Y-\widehat m^{(-)}\big)\Big\} - P\,m_0
\\[2mm]
&\quad\text{(since $\theta_0=\mathbb E[Y]=P\,Y=P\,m_0$ as $Y=m_0+\varepsilon$ and $\mathbb E[\varepsilon\mid X]=0$)}
\\[2mm]
&= \underbrace{\big(P_N \widehat m^{(-)} - P_N m_0\big)}_{A_1}
  \;+\; \underbrace{\big(P_n(Y-\widehat m^{(-)}) - P_n(Y-m_0)\big)}_{A_2}
  \;+\; \underbrace{\Big(P_N m_0 + P_n(Y-m_0) - P m_0\Big)}_{A_3}
\\[2mm]
&\quad\text{(add and subtract $P_N m_0$ and $P_n(Y-m_0)$)}
\\[2mm]
&= \underbrace{\big(P_N-P_n\big)\big(\widehat m^{(-)}-m_0\big)}_{A_1+A_2}
  \;+\; \underbrace{\big(P_N-P\big)m_0}_{\text{from the $P_N m_0 - P m_0$ part}}
  \;+\; \underbrace{P_n(Y-m_0)}_{\text{remaining part of }A_3}
\\[2mm]
&= \big(P_N-P\big)m_0 \;+\; \underbrace{\big(P_n-P\big)(Y-m_0)}_{\text{since }P(Y-m_0)=0}
   \;+\; \big(P_N-P_n\big)\big(\widehat m^{(-)}-m_0\big)
\\[2mm]
&=\underbrace{(P_N-P)\,m_0}_{\substack{\text{unlabeled fluctuation}\\(A)}}
  \;+\;\underbrace{(P_n-P)\,(Y-m_0)}_{\substack{\text{labeled residual fluctuation}\\(B)}}
  \;+\;\underbrace{(P_N-P_n)\,(\widehat m^{(-)}-m_0)}_{\substack{\text{nuisance remainder}\\(C)}}. \tag{$\star$}
\end{align*}

We briefly sketch the proof, as it is central to analyzing the asymptotic behavior of CF–PPI (Subsection~5.2 of the main document) and, later, SF–PPI–VC (Subsection~5.3). The CF–PPI decomposition yields three pieces: the first two, $(P_N-P)\,m_0$ and $(P_n-P)\{Y-m_0\}$, are empirical–process terms with \emph{fixed} indices because $m_0(x)=\mathbb{E}[Y\mid X=x]$ is nonrandom (oracle). Hence they fluctuate at the $O_p(N^{-1/2})$ and $O_p(n^{-1/2})$ scales (with $n/N\to f\in(0,1)$, so $O_p(n^{-1/2})=O_p(N^{-1/2})$), constitute the \emph{first–order} part of the expansion, and require no Donsker/entropy control. The third piece, $(P_N-P_n)(\widehat m-m_0)$, would \emph{without} cross–fitting typically require a Donsker condition for the nuisance class (see Section~4.2 of \citep{kennedy2016semiparametric}); however, under cross–fitting the nuisance $\widehat m=\widehat m^{(-)}$ is trained on folds disjoint from the evaluation sample, rendering it conditionally fixed and thus $o_p(N^{-1/2})$, i.e., only a \emph{second–order} remainder. Intuitively, cross–fitting preserves design–unbiasedness of the score and pushes the learning error from $\widehat m$ out of the first–order limit.

Now, we are ready to analyze each of the three terms in \((\star)\) separately; this is central to the proof of the theorem and will also be used later in establishing asymptotic normality.

\paragraph{Unlabeled fluctuation.}
Note that the first term $(A)=(P_N-P)\,m_0$ is the empirical average of the fixed function $m_0$, and hence, by the CLT, it behaves asymptotically as a centered normal random variable with variance $\Var\{m_0(X)\}/N$, up to $o_p(N^{-1/2})$ error. More precisely, by the i.i.d. \ CLT, the unlabeled part follows:
\begin{align}
\label{eq:asymptotic_distribution_first_part}
&\sqrt{N}\,\big\{(P_N-P)\,m_0\big\}
\xrightarrow{d}
\mathcal{N}\!\big(0,\Var\{m_0(X)\}\big),\\
\nonumber
&(P_N-P)\,m_0=\frac{1}{\sqrt{N}}\,Z_1+o_p(N^{-1/2}),
\end{align}
with \(Z_1\sim\mathcal{N}(0,\Var\{m_0(X)\})\). Because $\sqrt{N}\,\big\{(P_N-P)\,m_0\big\} = O_p(1)$, it holds $(P_N-P)m_0 = o_p(1)$. (One can use the weak law of large numbers directly for the same conclusion of $(P_N-P)m_0 = o_p(1)$.)

\paragraph{Labeled residual fluctuation.} Next, the second term $(B) = (P_n-P)\,(Y-m_0)$ is the empirical fluctuation of the residuals $Y-m_0(X)$. Since the residuals have mean zero under the true distribution, this term also satisfies a CLT with variance $\Var\{Y-m_0(X)\}/n$. Like the first term, by CLT, the labeled residual part follows (note \(P(Y-m_0)=0\)): 
\begin{align}
\label{eq:asymptotic_distribution_second_part}
&\sqrt{n}\,\big\{(P_n-P)\,(Y-m_0)\big\}
\xrightarrow{d}
\mathcal{N}\!\big(0,\Var\{Y-m_0(X)\}\big),\\
\nonumber
&(P_n-P)\,(Y-m_0)=\frac{1}{\sqrt{n}}\,Z_2+o_p(n^{-1/2}),    
\end{align}
with \(Z_2\sim\mathcal{N}(0,\Var\{Y-m_0(X)\})\). Because $\sqrt{n}\,\big\{(P_n-P)\,(Y-m_0)\big\} = O_p(1)$, it holds $(P_n-P)\,(Y-m_0) = o_p(1)$.


\paragraph{Nuisance remainder.} Finally, we prove $(C)=(P_N-P_n)\big(\widehat m^{(-)}-m_0\big) = o_p(1)$. Let $\mathcal T :=\sigma\!\big(\{\widehat m^{(k)}\}_{k=1}^K,\widehat m^\star,\kappa\big)$ denote the $\sigma$–field generated by the trained objects used to score
observations (the $K$ fold–specific fits $\{\widehat m^{(k)}\}_{k=1}^K$, the aggregator
$\widehat m^\star$, and the fold map $\kappa$).  Define $\Delta(x):=\widehat m^{(-)}(x)-m_0(x)$. Conditional on $\mathcal T$, the function $\Delta$ is deterministic. Note that the remainder term in \((\star)\) can be further decomposed as
\begin{align}
\label{eq:remainder}
(C)=(P_N-P_n)\Delta &=
\underbrace{(P_N-P)\Delta}_{\substack{\text{Unlabeled average}\\R_N} }-\underbrace{(P_n-P)\Delta}_{\substack{\text{Labeled average}\\R_n} } 
\end{align}

This remainder term $(C)$ in \eqref{eq:remainder} captures the mismatch between the estimated regression function and the truth; for \emph{consistency}, it suffices that $\|\widehat m^{(-)}-m_0\|_{L_2(P_X)}=O_p(1)$. Later for \emph{asymptotic normality} we require $\|\widehat m^{(-)}-m_0\|_{L_2(P_X)}=o_p(1)$ so that $(C) = (P_N-P_n)\Delta=O_p(N^{-1/2}\|\Delta\|_{L_2(P_X)})=o_p(N^{-1/2})$ and the first two fluctuation terms in \((\star)\) dominate.
  
The unlabeled average $R_N$ in the display above is benign: the unlabeled covariates are independent of the fitted functions, so it behaves like a standard empirical average of a fixed function (recall definition of  $m^{(-)}(X)$). The potentially troublesome piece is the \emph{labeled} average $R_n$. Without cross-fitting, the same labeled observations used to train $\widehat m$ would also be used to evaluate it, creating dependence and bias in this term. Cross-fitting restores \emph{honesty}—each evaluation point is scored by a model trained on other folds—so that the empirical fluctuations admit the variance bounds used above without invoking restrictive entropy/Donsker conditions \citep{kennedy2024semiparametric}.

\paragraph{(i) Unlabeled average $R_N$.}
Since the unlabeled sample $\{X_i\}_{i=1}^N$ is independent of $\mathcal T$
(the fits are trained on labeled data only), $\{\,\Delta(X_i)\,\}_{i=1}^N$
are i.i.d.\ given $\mathcal T$ with mean $P\Delta:=\mathbb{E}[\Delta(X)\mid\mathcal T]$.
Writing $R_N$ in centered summation form,
\[
R_N = \frac{1}{N}\sum_{i=1}^N\Big(\Delta(X_i)-P\Delta\Big),
\]
we have, by independence,
\[
\mathbb{E}[R_N\mid\mathcal T]=0,
\qquad
\Var(R_N\mid\mathcal T)=\frac{1}{N}\Var\!\big(\Delta(X)\mid\mathcal T\big)
\le \frac{1}{N}\mathbb{E}[\Delta(X)^2\mid\mathcal T]
= \frac{1}{N}\|\Delta\|_{L_2(P_X)}^{\,2}.
\]
Now, we use the \emph{Chebyshev’s inequality}: For any random variable $Z$ with $\mathbb{E}[Z]=0$,
\[
\mathbb{P}(|Z|\ge t)\le \frac{\Var(Z)}{t^2}\qquad\text{for all }t>0 .
\]
Apply this to $Z=\sqrt{N}\,R_N$ (so that $\mathbb{E}[Z\mid\mathcal T]=0$ and
$\Var(Z\mid\mathcal T)=N\,\Var(R_N\mid\mathcal T)\le \|\Delta\|_{L_2(P_X)}^{\,2}$).
Then, for any $s>0$,
\begin{align}
\label{eq:concentration_ineq_for_R_N}
\mathbb{P}\!\left(\,|\sqrt{N}\,R_N|\ge s \,\middle|\,\mathcal T\right)
\le \frac{\|\Delta\|_{L_2(P_X)}^{\,2}}{s^2}.    
\end{align}
Taking expectations in both sides of (\ref{eq:concentration_ineq_for_R_N}) and using $\|\Delta\|_{L_2(P_X)}=O_p(1)$ (i.e., assumption (\ref{eq:nuisance_condition_for_consistency_CF_PPI})) yields
$\sqrt{N}\,R_N=O_p(1)$, thus \ $R_N=O_p(N^{-1/2})$.

Note that we write \(X_N = O_p(1)\) if, for every \(\varepsilon>0\), there exists a finite constant \(M>0\) such that 
\(\mathbb{P}(|X_N|>M)<\varepsilon\) for all sufficiently large \(N\). The upper bound 
\(\|\Delta\|_{L_2(P_X)}^{2}/s^{2}\) in \eqref{eq:concentration_ineq_for_R_N} can be made arbitrarily small by taking \(s\) sufficiently large—smaller than any given \(\varepsilon>0\)—and any such \(s\) can be regarded as \(M\); hence, \(\sqrt{N}\,R_N = O_p(1)\). Henceforth, we omit this reasoning, as it is straightforward from the context once a bounding inequality of the form \eqref{eq:concentration_ineq_for_R_N} is established.

\paragraph{(ii) Labeled average $R_n$.} We provide a detailed illustration, as this term is the core part where cross-fitting is most essential for proving consistency (and also for proving asymptotic normality in Theorem \ref{thm:cfppi_mean_clt}.).

Let the labeled indices be partitioned into $K$ folds $S=S_1\cup\cdots\cup S_K$,
and for each $k$ let $\widehat m^{(k)}$ be the predictor trained on the labeled data
$\{(X_j,Y_j):j\in S\setminus S_k\}$. Let $\kappa(j)=k$ denote the fold map if $j\in S_k$ and the
cross–fitted predictor $\widehat m^{(-)}(x):=\widehat m^{(\kappa(j))}(x)$ when scoring index $j$.

\emph{Decoupling by cross–fitting.} 
The term $R_n$ can be decomposed as
\begin{align}
  R_n &= (P_n-P)\Delta 
  = \frac{1}{n}\sum_{j\in S} \big\{\widehat m^{(-)}(X_j)-m_0(X_j)\big\} -P\Delta  \nonumber \\
  &= \frac{1}{n}\sum_{j\in S} \widehat m^{(-)}(X_j)
     -\frac{1}{n}\sum_{j\in S} m_0(X_j) -P\Delta \nonumber \\
  &= 
  \underbrace{\frac{1}{n}\sum_{j\in S} \widehat m^{(\kappa(j))}(X_j)}_{\text{Cross-fitted model fit}}
  -\frac{1}{n}\sum_{j\in S} m_0(X_j) -P\Delta ,\label{eq:cross_fit_model_fit} 
\end{align}
where the first term in \eqref{eq:cross_fit_model_fit} is the average prediction from the cross-fitted models
and the second term is the corresponding population regression truth evaluated on the labeled sample. Here, term \(P\Delta\) are understood conditional on the model that scores each \(j\) (i.e., the cross-fitted training set that excludes \(j\)); we keep the shorthand \(P\Delta\) for readability.

This decomposition makes clear why cross-fitting is essential: it ensures that, for each $j\in S_k$, the model used to evaluate $X_j$ -- namely $\widehat m^{(\kappa(j))}$ -- is trained without the pair $(X_j,Y_j)$. Hence
\[
\widehat m^{(\kappa(j))}(X_j)\ \perp\!\!\!\perp\ Y_j \mid X_j,
\]
i.e., there is no label leakage. This conditional independence yields an unbiased expansion and clean conditional variance control, which establish consistency and enable the CLT for asymptotic normality later. Thus, the cross-fitted term behaves like an average of i.i.d.\ random variables. 
In particular, conditional on $\mathcal{T}$, the variables
\[
\Delta(X_j) = \widehat m^{(-)}(X_j)-m_0(X_j), \qquad j \in S,
\]
are i.i.d.\ with the same distribution as $\Delta(X)$, and they depend only on $X_j$ (since the fitted model never uses $Y_j$ in training). Refer to \citep{newey2018cross,kennedy2024semiparametric,chernozhukov2018double} for more details on similar ideas used in the development of cross-fitting and double machine learning methods.

\emph{Conditional mean and variance.} Writing $W_j:=\Delta(X_j)-P\Delta$ (so that
$\mathbb{E}[W_j\mid\mathcal T]=0$), we have
\[
R_n=\frac{1}{n}\sum_{j\in S} W_j,\qquad
\mathbb{E}[R_n\mid\mathcal T]
=\frac{1}{n}\sum_{j\in S}\mathbb{E}[W_j\mid\mathcal T]=0,
\]
and, by conditional independence and identical distribution of the $W_j$'s,
\[
\Var(R_n\mid\mathcal T)
=\frac{1}{n^2}\sum_{j\in S}\Var(W_j\mid\mathcal T)
=\frac{1}{n}\Var\!\big(\Delta(X)\mid\mathcal T\big)
\le \frac{1}{n}\,\mathbb{E}\!\big[\Delta(X)^2\mid\mathcal T\big]
=\frac{1}{n}\,\|\Delta\|_{L_2(P_X)}^{\,2}.
\]

Recall that Chebyshev's inequality in its conditional form states that, for any random variable $Z$ with $\mathbb{E}[Z\mid\mathcal{T}]=0$, and for all $s>0$,  
$\mathbb{P}\!\left(|Z|\ge s \,\middle|\, \mathcal{T}\right) \le \Var(Z\mid \mathcal{T}) / s^2.$

Apply this to $Z=\sqrt{n}\,R_n$ to obtain, for any $s>0$,
\begin{align}
\label{eq:concentration_ineq_for_R_n}
\mathbb{P}\!\left(\,|\sqrt{n}\,R_n|\ge s \,\middle|\,\mathcal T\right)
\le \frac{\Var(\sqrt{n}\,R_n\mid\mathcal T)}{s^2}
= \frac{n\,\Var(R_n\mid\mathcal T)}{s^2}
\le \frac{\|\Delta\|_{L_2(P_X)}^{\,2}}{s^2}.    
\end{align}
Taking expectations in both sides of (\ref{eq:concentration_ineq_for_R_n}) and using the stochastic boundedness
$\|\Delta\|_{L_2(P_X)}=O_p(1)$ yields $\sqrt{n}\,R_n=O_p(1)$, i.e.\ $R_n=O_p(n^{-1/2})$.

One should note that, without cross–fitting, $\Delta$ would be
measurable with respect to the same labeled data used in $P_n$, so the $W_j$'s would no longer
be conditionally independent and centered given the training objects. Cross–fitting ensures that,
conditional on $\mathcal T$, $W_j$ are i.i.d.\ mean-zero, allowing the variance bound and the
Chebyshev control above to hold without further entropy/Donsker assumptions.

\paragraph{(iii) Conclusion on the nuisance remainder.}
Combining the two parts, we have
\[
(P_N-P_n)\Delta = R_N - R_n = O_p(N^{-1/2})+O_p(n^{-1/2}) \;\xrightarrow{p}\; 0
\quad\text{as } N,n\to\infty .
\]


\paragraph{Conclusion.}
Each underbraced term in \((\star)\) converges to $0$ in probability, so
$\widehat\theta^{\mathrm{cf}}_{\mathrm{PPI}}\to_p \theta_0$.
\end{proof}


\begin{lemma}[Cross-fitted empirical-process bound; cf.\ Lemma~1 in \cite{kennedy2024semiparametric}]
\label{lem:cf_external_sample_bound_n}
Let $\{W_i\}_{i=1}^n$ be i.i.d.\ from a distribution $P$, and let $\mathcal T$ be a $\sigma$–field
independent of $\sigma(W_1,\dots,W_n)$ (e.g., the $\sigma$–field generated by the training objects
used to construct a cross–fitted predictor from data disjoint from $\{W_i\}_{i=1}^n$).
For any $\mathcal T$–measurable function $h$ with $\|h\|_{L_2(P)}^2=\mathbb E[h(W)^2]<\infty$,
writing $P_n^{W} h := n^{-1}\sum_{i=1}^n h(W_i)$, we have
\[
\big(P_n^{W}-P\big)h
\;=\; O_p\!\Big(\|h\|_{L_2(P)}/\sqrt{n}\Big).
\]
Equivalently, for every $t>0$,
\[
\mathbb P\!\left(
\frac{|(P_n^{W}-P)h|}{\|h\|_{L_2(P)}/\sqrt{n}} \ge t
\,\middle|\, \mathcal T\right) \le \frac{1}{t^2},
\qquad\text{hence}\quad
\frac{|(P_n^{W}-P)h|}{\|h\|_{L_2(P)}/\sqrt{n}} = O_p(1).
\]
\end{lemma}

\begin{proof}
Independence implies $(W_1,\dots,W_n)\mid\mathcal T \sim P^{\otimes n}$ a.s., so given $\mathcal T$ the $W_i$'s
are i.i.d.\ with common law $P$. Since $h$ is $\mathcal T$–measurable, it is fixed when conditioning on $\mathcal T$.
Thus,
\[
\mathbb E[h(W_i)\mid\mathcal T]=\int h\,dP=:Ph,\qquad
\mathbb E[h(W_i)^2\mid\mathcal T]=\int h^2\,dP=\|h\|_{L_2(P)}^2,
\]
and $\Var(h(W_i)\mid\mathcal T)\le \|h\|_{L_2(P)}^2$. With $P_n^{W} h := n^{-1}\sum_{i=1}^n h(W_i)$,
\[
\mathbb E\!\big[(P_n^{W}-P)h\mid\mathcal T\big]=0,
\]
and conditional i.i.d.-ness yields
\[
\Var\!\big((P_n^{W}-P)h\mid\mathcal T\big)
=\Var\!\big(P_n^{W} h\mid\mathcal T\big)
=\frac{1}{n}\Var\!\big(h(W)\mid\mathcal T\big)
\le \frac{\|h\|_{L_2(P)}^2}{n}.
\]
Let $Z:=(P_n^{W}-P)h$. From the calculation above,
$\mathbb{E}[Z\mid\mathcal T]=0$ and 
$\Var(Z\mid\mathcal T)\le \|h\|_{L_2(P)}^2/n$.
Chebyshev’s inequality in conditional form (obtained by applying conditional
Markov inequality to $Z^2$):
\[
\mathbb P\!\left(|Z|\ge a \,\middle|\, \mathcal T\right)
= \mathbb P\!\left(Z^2\ge a^2 \,\middle|\, \mathcal T\right)
\le \frac{\mathbb E[Z^2\mid \mathcal T]}{a^2}
= \frac{\Var(Z\mid\mathcal T)+\big(\mathbb E[Z\mid\mathcal T]\big)^2}{a^2}
= \frac{\Var(Z\mid\mathcal T)}{a^2}.
\]
Finally, choose $a:= t\,\|h\|_{L_2(P)}/\sqrt{n}$ (with $t>0$). Then
\[
\mathbb P\!\left(\, |(P_n^{W}-P)h| \ge t\,\frac{\|h\|_{L_2(P)}}{\sqrt{n}}
\,\middle|\, \mathcal T\right)
\le 
\frac{\Var(Z\mid\mathcal T)}{t^2\,\|h\|_{L_2(P)}^2/n}
\le \frac{1}{t^2}.
\tag{$\dagger$}
\]
If $\|h\|_{L_2(P)}=0$, then $h=0$ $P$–a.s., so $Z\equiv0$ and the event on the left of $(\dagger)$ has probability $0$; the bound still holds. Taking expectations over $\mathcal T$ yields
\[
\mathbb P\!\left(\, |(P_n^{W}-P)h| \ge t\,\frac{\|h\|_{L_2(P)}}{\sqrt{n}} \right)
\le \frac{1}{t^2},
\]
which is equivalent to
\(
(P_n^{W}-P)h = O_p\!\big(\|h\|_{L_2(P)}/\sqrt{n}\big).
\)
\end{proof}

\begin{theorem}[Asymptotic normality of the CF-PPI estimator for the mean]
\label{thm:cfppi_mean_clt}
Assume the setup of Theorem~\ref{thm:cfppi_mean_consistency}. 
If, in addition, the cross–fitted predictor satisfies
\begin{align}
\label{eq:nuisance_condition_for_asymptotic_normality_CF_PPI}
\|\widehat m^{(-)}-m_0\|_{L_2(P_X)} = o_p(1),    
\end{align}
then
\[
\sqrt{N}\,\big(\widehat\theta^{\mathrm{cf}}_{\mathrm{PPI}}-\theta_0\big)
\;\xrightarrow{d}\; \mathcal{N}\!\big(0,\sigma_f^2\big),
\]
with asymptotic variance
\[
\sigma_f^2 \;=\; \Var\!\big(m_0(X)\big) \;+\; \frac{1}{f}\,\Var\!\big(Y-m_0(X)\big).
\]
\end{theorem}

\begin{proof}
We start from the decomposition \((\star)\) in \emph{Proof} of Theorem \ref{thm:cfppi_mean_consistency} and multiply by \(\sqrt{N}\):
\[
\sqrt{N}\big(\widehat\theta^{\mathrm{cf}}_{\mathrm{PPI}}-\theta_0\big)
= \underbrace{\sqrt{N}\,(P_N-P)\,m_0}_{\text{(I)}} 
+ \underbrace{\sqrt{N}\,(P_n-P)\,(Y-m_0)}_{\text{(II)}}
+ \underbrace{\sqrt{N}\,(P_N-P_n)\big(\widehat m^{(-)}-m_0\big)}_{\text{(III)}}.
\]

\paragraph{Leading terms (I) and (II).}
By the i.i.d.\ CLT, \eqref{eq:asymptotic_distribution_first_part} gives
\[
\sqrt{N}\,(P_N-P)\,m_0 \;\xrightarrow{d}\; \mathcal N\!\big(0,\Var\{m_0(X)\}\big).
\]
Similarly, \eqref{eq:asymptotic_distribution_second_part} yields
\[
\sqrt{n}\,(P_n-P)\,(Y-m_0) \;\xrightarrow{d}\; \mathcal N\!\big(0,\Var\{Y-m_0(X)\}\big),
\]
so
\[
\sqrt{N}\,(P_n-P)\,(Y-m_0)
= \sqrt{\frac{N}{n}}\;\sqrt{n}\,(P_n-P)\,(Y-m_0)
\;\xrightarrow{d}\; \mathcal N\!\Big(0,\frac{1}{f}\Var\{Y-m_0(X)\}\Big),
\]
because \(N/n\to 1/f\). Since the unlabeled and labeled samples are independent, the limits above are independent.

\paragraph{Remainder term (III).}
Let $\Delta:=\widehat m^{(-)}-m_0$. From the decomposition,
\[
\text{(III)}
=\sqrt{N}\,(P_N-P_n)\Delta
=\sqrt{N}\Big\{(P_N-P)\Delta-(P_n-P)\Delta\Big\}
=\sqrt{N}\,R_N-\sqrt{N}\,R_n,
\]
where $R_N:=(P_N-P)\Delta$ and $R_n:=(P_n-P)\Delta$.

Apply Lemma~\ref{lem:cf_external_sample_bound_n} with $h=\Delta$ to the two evaluation samples:
(i) the unlabeled sample $\{X_i\}_{i=1}^N$ (take $W=X$, sample size $n=N$), and
(ii) the labeled evaluation sample $\{X_j\}_{j\in S}$ (cross–fitted, so $X_j\perp\!\!\!\perp\mathcal T$, sample size $n$).  
The lemma gives
\[
R_N=O_p\!\Big(\|\Delta\|_{L_2(P_X)}/\sqrt{N}\Big),
\qquad
R_n=O_p\!\Big(\|\Delta\|_{L_2(P_X)}/\sqrt{n}\Big).
\]
Hence
\begin{align*}
\text{(III)}
&=\sqrt{N}\,R_N-\sqrt{N}\,R_n
=O_p\!\big(\|\Delta\|_{L_2(P_X)}\big)
\;+\;
O_p\!\Big(\sqrt{\tfrac{N}{n}}\;\|\Delta\|_{L_2(P_X)}\Big) =O_p\!\big(\|\Delta\|_{L_2(P_X)}\big),
\end{align*}
since $n/N\to f\in(0,1]$ implies $\sqrt{N/n}=O(1)$. Therefore, under the condition (\ref{eq:nuisance_condition_for_asymptotic_normality_CF_PPI}), we have
\[
\text{(III)}=\sqrt{N}\,(P_N-P_n)\big(\widehat m^{(-)}-m_0\big)=o_p(1).
\]
This controls the nuisance remainder at the $\sqrt{N}$ scale and completes the treatment of term (III).

\paragraph{Conclusion.}
By Slutsky’s theorem and independence of the two leading limits,
\[
\sqrt{N}\big(\widehat\theta^{\mathrm{cf}}_{\mathrm{PPI}}-\theta_0\big)
\;\xrightarrow{d}\; \mathcal N\!\Big(0,\ \Var\{m_0(X)\}+\tfrac{1}{f}\Var\{Y-m_0(X)\}\Big)
= \mathcal N\!\big(0,\sigma_f^2\big),
\]
which is the claimed result.
\end{proof}


\begin{theorem}[Consistency and asymptotic normality of the CF-PPI estimator for the mean]
\label{thm:cfppi_mean_unified_conclusion}
Let $(X,Y)$ have joint law $P$ with $X\sim P_X$ and $Y=m_0(X)+\varepsilon$, where
$\mathbb E[\varepsilon\mid X]=0$ and $\mathbb E[Y^2]<\infty$. The target is the population mean
$\theta_0:=\mathbb E[Y]$. Let $\{X_i\}_{i=1}^N \stackrel{\mathrm{i.i.d.}}{\sim} P_X$ be an unlabeled
sample and $\{(X_j,Y_j)\}_{j\in S}$, $|S|=n$, an independent labeled sample from $P$, with
$N,n\to\infty$ and $n/N\to f\in(0,1)$. Let $\widehat m^{(-)}$ be the cross–fitted predictor
(each labeled index is scored by a model trained without its own fold), and consider
\begin{equation}\label{eq:cfppi_mean_estimator}
\widehat\theta^{\mathrm{cf}}_{\mathrm{PPI}}
=\frac{1}{N}\sum_{i=1}^N \widehat m^{(-)}(X_i)
+\frac{1}{n}\sum_{j\in S}\{Y_j-\widehat m^{(-)}(X_j)\}.
\end{equation}
Then:

\medskip
\noindent \textbf{\textup{(i) Consistency.}} 
If the out–of–fold error is stochastically bounded in $L_2(P_X)$,
\[
\|\widehat m^{(-)}-m_0\|_{L_2(P_X)} = O_p(1),
\]
then $\widehat\theta^{\mathrm{cf}}_{\mathrm{PPI}}\xrightarrow{p}\theta_0$.

\medskip
\noindent \textbf{\textup{(ii) Asymptotic normality.}}  
If, in addition, the cross–fitted predictor is $L_2(P_X)$–consistent,
\[
\|\widehat m^{(-)}-m_0\|_{L_2(P_X)} = o_p(1),
\]
then
\[
\sqrt{N}\,\big(\widehat\theta^{\mathrm{cf}}_{\mathrm{PPI}}-\theta_0\big)
\;\xrightarrow{d}\; \mathcal{N}\!\big(0,\sigma_f^2\big),
\qquad
\sigma_f^2 \;=\; \Var\!\big(m_0(X)\big) \;+\; \frac{1}{f}\,\Var\!\big(Y-m_0(X)\big).
\]
\end{theorem}

\begin{proof}
The theorem follows from the decomposition used in the proof of
Theorem~\ref{thm:cfppi_mean_consistency} (consistency) together with the
remainder control and limit calculations in the proof of
Theorem~\ref{thm:cfppi_mean_clt} (asymptotic normality); no additional
ingredients are required.
\end{proof}

\subsection{Asymptotic properties of single-fit prediction-powered inference with variance correction}\label{subsec:Asymptotic Properties of SF-PPI-VC}
This subsection provides the proofs of the lemmas and theorems for single-fit PPI with variance correction in the semisupervised mean estimation problem, and it proves Theorem~5.2 in the main document.

\begin{proposition}[Linear–smoother notation and basic identities]
\label{prop:linear-smoother-notation}
Let $(X,Y)\sim P$ with $X\in\mathcal X\subset\mathbb R^d$, $Y\in\mathcal Y\subset\mathbb R$, and
$\mathbb E[Y^2]<\infty$. Suppose $Y=m_0(X)+\varepsilon$ with $\mathbb E[\varepsilon\mid X]=0$.
The target is the population mean $\theta_0:=\mathbb E[Y]$.


Let \(\{X_i\}_{i=1}^N\) be an unlabeled sample drawn i.i.d.\ from \(P_X\), and let
\(\{(X_j,Y_j)\}_{j=1}^n\) be an independent labeled sample drawn i.i.d.\ from \(P\),
with \(n/N\to f\in(0,1)\). For notational simplicity, we write
\(S=\{1,\dots,n\}\) for the labeled sample indices. Assume the prediction rule is a single–fit affine linear smoother trained on the labeled set $S$:
\[
\widehat m(x)=\sum_{j\in S} s_j(x)\,Y_j + b(x)=s(x)^\top Y_S+b(x):\mathcal X\to\mathbb R,
\]
where $s(x)=(s_j(x))_{j\in S}\in\mathbb R^n$ is the vector of smoothing weights at $x$, and
$b:\mathcal X\to\mathbb R$ is a covariate–only offset term. The weights $s_j(\cdot)$ and $b(\cdot)$
depend only on covariates (in particular, on $\{X_j:j\in S\}$ and the index set $S$), not on $Y_S$.

Define the labeled hat matrix $H\in\mathbb R^{n\times n}$ and the unlabeled weights matrix
$S_U\in\mathbb R^{N\times n}$ by
\begin{align}
\label{eq:matrix_H}
H
=\begin{bmatrix}
s(X_1)^\top\\
s(X_2)^\top\\
\vdots\\
s(X_n)^\top
\end{bmatrix}
=
\begin{bmatrix}
s_1(X_1) & s_2(X_1) & \cdots & s_n(X_1) \\
s_1(X_2) & s_2(X_2) & \cdots & s_n(X_2) \\
\vdots   & \vdots   & \ddots & \vdots   \\
s_1(X_n) & s_2(X_n) & \cdots & s_n(X_n)
\end{bmatrix}\in\mathbb R^{n\times n},
\end{align}
\begin{align}
\label{eq:matrix_S_U}
S_U
=\begin{bmatrix}
s(X_1)^\top\\
s(X_2)^\top\\
\vdots\\
s(X_N)^\top
\end{bmatrix}
=
\begin{bmatrix}
s_1(X_1) & s_2(X_1) & \cdots & s_n(X_1) \\
s_1(X_2) & s_2(X_2) & \cdots & s_n(X_2) \\
\vdots   & \vdots   & \ddots & \vdots   \\
s_1(X_N) & s_2(X_N) & \cdots & s_n(X_N)
\end{bmatrix}\in\mathbb R^{N\times n}.
\end{align}

Define the average squared norm of the smoothing weights by
\[
\overline{\|s(X)\|_2^2}
:=\frac{1}{N}\sum_{i=1}^N\|s(X_i)\|_2^2
=
\frac{1}{N} \|S_U \|_{F}^2
=
\frac{1}{N}\operatorname{tr}(S_U^\top S_U)\in [0,\infty),
\]
where $\|\cdot \|_F$ denotes the Frobenius norm.

Define three $n$–dimensional vectors summarizing the weights:
\begin{align}
\label{def:a_and_hbar}
a:=\frac{1}{N}S_U^\top\mathbf 1_N\in\mathbb R^n,\qquad
\bar h:=\frac{1}{n}H^\top\mathbf 1_n\in\mathbb R^n,    
\end{align}
and
\begin{align}
\label{def:c}
c:=\frac{1}{N}S_U^\top \mathbf 1_N
+\frac{1}{n}\mathbf 1_n
-\frac{1}{n}H^\top \mathbf 1_n
=\frac{1}{N}S_U^\top \mathbf 1_N+\frac{1}{n}(I - H^\top)\mathbf 1_n
= a+\frac{1}{n}\mathbf 1_n-\bar h\in \mathbb R^n.    
\end{align}

Let $P_N h:=N^{-1}\sum_{i=1}^N h(X_i)$ and $P_n h:=n^{-1}\sum_{j\in S}h(X_j)$. Then:
\begin{itemize}
  \item[\textbf{\emph{(a)}}] $P_N \widehat m= a^\top Y_S+P_N b.$
  \item[\textbf{\emph{(b)}}] $P_n \widehat m= \bar h^\top Y_S+P_n b.$
  \item[\textbf{\emph{(c)}}] $P_n(Y-\widehat m)= \frac{1}{n}\mathbf 1_n^\top Y_S-\bar h^\top Y_S-P_n b
  =\big(\frac{1}{n}\mathbf 1_n-\bar h\big)^\top Y_S-P_n b.$
  \item[\textbf{\emph{(d)}}] $P_N \widehat m+P_n(Y-\widehat m)=\big(a+\frac{1}{n}\mathbf 1_n-\bar h\big)^\top Y_S+(P_N-P_n)b
  = c^\top Y_S+(P_N-P_n)b.$
  \item[\textbf{\emph{(e)}}] With $Y_S=m_0(X_S)+\varepsilon_S$, one has
  \[
  c^\top Y_S=c^\top m_0(X_S)+c^\top \varepsilon_S.
  \]
  \item[\textbf{\emph{(f)}}] If $\sum_{j\in S}s_j(x)=1$ for all $x$ (i.e., mass preservation), then
  \[
  \mathbf 1_n^\top a=\mathbf 1_n^\top \bar h=1,\qquad
  \|c\|_2^2=\frac{1}{n}+\|a-\bar h\|_2^2.
  \]
\end{itemize}
\end{proposition}

\begin{proof}
Note that $s(x)=(s_j(x))_{j\in S}=(s_1(x),\dots,s_n(x))^\top \in \mathbb R^n$ (we assumed $S=\{1,\dots,n\} \subset\{1,\dots,N\}$ for simplicity), so the usual
vector–matrix algebra applies. In particular,
\[
\widehat m(x)=s(x)^\top Y_S+b(x)=\sum_{j\in S}s_j(x)\,Y_j+b(x),\qquad
\|s(x)\|_2^2=s(x)^\top s(x).
\]
Stacking rows $s(X_i)^\top$ gives the unlabeled weights matrix
$S_U=\big[s(X_1)^\top;\dots; s(X_N)^\top\big]\in\mathbb R^{N\times n}$ and, restricting to the
labeled indices, the hat matrix $H=\big[s(X_1)^\top;\dots; s(X_n)^\top\big]\in\mathbb R^{n\times n}$.
Consequently,
\[
a=\frac{1}{N}S_U^\top \mathbf 1_N,\qquad
\bar h=\frac{1}{n}H^\top \mathbf 1_n,\qquad
\overline{\|s(X)\|_2^2}=\frac{1}{N}\sum_{i=1}^N\|s(X_i)\|_2^2=\frac{1}{N}\operatorname{tr}(S_U^\top S_U),
\]
and all identities in Proposition~\ref{prop:linear-smoother-notation} follow by standard
matrix–vector manipulations (linearity of sums, associativity of matrix products).

Now, we are ready to prove (a) -- (f):

\paragraph{(a)}
\[
P_N \widehat m
=\frac{1}{N}\sum_{i=1}^N \big(s(X_i)^\top Y_S+b(X_i)\big)
=\Big(\frac{1}{N}\sum_{i=1}^N s(X_i)\Big)^\top Y_S+P_N b
=a^\top Y_S+P_N b.
\]

\paragraph{(b)}
\[
P_n \widehat m
=\frac{1}{n}\sum_{j\in S} \big(s(X_j)^\top Y_S+b(X_j)\big)
=\Big(\frac{1}{n}\sum_{j\in S} s(X_j)\Big)^\top Y_S+P_n b
=\bar h^\top Y_S+P_n b.
\]

\paragraph{(c)}
Since $P_n Y = \frac{1}{n}\sum_{j\in S} Y_j = \frac{1}{n}\mathbf 1_n^\top Y_S$ and by (b),
\[
P_n(Y-\widehat m)=\frac{1}{n}\mathbf 1_n^\top Y_S-\bar h^\top Y_S-P_n b
=\Big(\frac{1}{n}\mathbf 1_n-\bar h\Big)^\top Y_S-P_n b.
\]

\paragraph{(d)}
Add (a) and (c):
\begin{align*}
P_N\widehat m+P_n(Y-\widehat m)
&= a^\top Y_S+P_N b + \Big(\frac{1}{n}\mathbf 1_n-\bar h\Big)^\top Y_S-P_n b \\
&= \Big(a+\frac{1}{n}\mathbf 1_n-\bar h\Big)^\top Y_S+(P_N-P_n)b
=c^\top Y_S+(P_N-P_n)b,    
\end{align*}
using the definition of $c$.

\paragraph{(e)}
Insert $Y_S=m_0(X_S)+\varepsilon_S$ into $c^\top Y_S$:
\[
c^\top Y_S=c^\top m_0(X_S)+c^\top \varepsilon_S.
\]

\paragraph{(f)}
If $\sum_{j\in S}s_j(x)=1$ for all $x$, then
\begin{align*}
\mathbf 1_n^\top a
&= \mathbf{1}_n^\top\!\left(\tfrac{1}{N} S_U^\top \mathbf 1_N\right)
= \tfrac{1}{N}\,\mathbf 1_n^\top S_U^\top \mathbf 1_N
= \tfrac{1}{N}\,\mathbf 1_N^\top S_U \mathbf 1_n \\
&= \tfrac{1}{N}\sum_{i=1}^N \sum_{j=1}^n s_j(X_i)
= \tfrac{1}{N}\sum_{i=1}^N 1
= 1,
\end{align*}
since $\mathbf 1^\top A \mathbf 1 $ is sum of all elements of matrix $A$.

Similarly,
\begin{align*}
\mathbf 1_n^\top \bar h
=\mathbf 1_n^\top \left(\frac{1}{n}H^\top \mathbf 1_n\right) 
=\frac{1}{n}\mathbf 1_n^\top H^\top \mathbf 1_n 
=\frac{1}{n}\sum_{j=1}^n \sum_{k=1}^n s_k(X_j) 
=\frac{1}{n}\sum_{j=1}^n 1 =1.
\end{align*}

Since $c=\frac{1}{n}\mathbf 1_n+(a-\bar h)$, expand
\[
\|c\|_2^2=\Big\|\frac{1}{n}\mathbf 1_n\Big\|_2^2+\frac{2}{n}\mathbf 1_n^\top(a-\bar h)+\|a-\bar h\|_2^2
=
\left(\frac{1}{n}\right)^2
\Big\|\mathbf 1_n\Big\|_2^2+\frac{2}{n} (1 - 1)  +\|a-\bar h\|_2^2
=\frac{1}{n}+0+\|a-\bar h\|_2^2,
\]
because $\mathbf 1_n^\top(a-\bar h)=\mathbf 1_n^\top a-\mathbf 1_n^\top \bar h=0$.
\end{proof}


\begin{proposition}[Kernel ridge regression with an unpenalized intercept: closed form and mass preservation]\label{prop:KRR_MP}
Let $K\in\mathbb{R}^{n\times n}$ be a symmetric positive semidefinite kernel matrix built on the labeled inputs $(X_j)_{j\in S}$ and let $Y=Y_S\in\mathbb{R}^n$ be the response vector. 

Consider
\[
(\hat\beta,\hat\alpha)
=\arg\min_{\beta\in\mathbb{R}^n,\ \alpha\in\mathbb{R}}
\ \frac1n\big\|Y-(K\beta+\alpha\mathbf 1_n)\big\|_2^2+\lambda\,\beta^\top K\beta,
\qquad \lambda>0.
\]
Then the intercept and coefficients are
\[
\hat\alpha
=\frac{\mathbf 1_n^\top (K+n\lambda I_n)^{-1}Y}
        {\mathbf 1_n^\top (K+n\lambda I_n)^{-1}\mathbf 1_n} \in \mathbb R,
\qquad
\hat\beta
=(K+n\lambda I_n)^{-1}\big(Y-\hat\alpha\,\mathbf 1_n\big) \in \mathbb R^n.
\]
Let
\[
w^\top
:=\frac{\mathbf 1_n^\top (K+n\lambda I_n)^{-1}}
        {\mathbf 1_n^\top (K+n\lambda I_n)^{-1}\mathbf 1_n} \in \mathbb R^{1 \times n},
\qquad
H:=K(K+n\lambda I_n)^{-1}(I_n-\mathbf 1_n w^\top)+\mathbf 1_n w^\top \in \mathbb R^{n \times n}.
\]
Then the fitted values on the training points $(\widehat m_S:=(\widehat m(X_j))_{j\in S}\in\mathbb{R}^n)$ satisfy $$\widehat m_S:=K\hat\beta+\hat\alpha\mathbf 1_n=H\,Y,$$ and $H$ is symmetric. Moreover, for any new input $x$ with kernel vector $k_x\in\mathbb{R}^n$ (entries $(k_x)_j=k(X_j,x)$), the prediction can be written
\[
\widehat m(x)=s(x)^\top Y \in \mathbb R,\qquad
s(x)^\top=k_x^\top (K+n\lambda I_n)^{-1}(I_n-\mathbf 1_n w^\top)+w^\top \in \mathbb R^{1 \times n},
\]
and the weights sum to one, $s(x)^\top\mathbf 1_n=1$ for all $x$ (i.e., the mass–preserving property holds.)

\paragraph{\emph{Remark -- Offset term.}} 
Typically, we fit an affine linear smoother 
\(\widehat m(x) = s(x)^\top Y + b(x)\), 
with $s(x)$ denoting the smoothing weights and $b(x)$ an offset term, as stated in Proposition \ref{prop:linear-smoother-notation}. The proposition shows that, when using kernel ridge regression with an unpenalized intercept, 
the offset term automatically vanishes, i.e., \(b(x) = 0\).

\paragraph{\emph{Remark -- hat matrix $H$.}}
Stacking the row weights $s(X_i)^\top \in \mathbb R^{1 \times n}$ for $i\in S (=\{1,\ldots,n\}$ for notational convenience) gives
\[
H
=
\begin{bmatrix}
s(X_1)^\top\\
s(X_2)^\top\\
\vdots\\
s(X_n)^\top
\end{bmatrix}
=
\begin{bmatrix}
s_1(X_1) & s_2(X_1) & \cdots & s_n(X_1)\\
s_1(X_2) & s_2(X_2) & \cdots & s_n(X_2)\\
\vdots   & \vdots   & \ddots & \vdots\\
s_1(X_n) & s_2(X_n) & \cdots & s_n(X_n)
\end{bmatrix}
\in\mathbb R^{n\times n}.
\]
For kernel ridge regression with an unpenalized intercept,
\[
H
=
\underbrace{\begin{bmatrix}
k_{X_1}^\top\\
k_{X_2}^\top\\
\vdots\\
k_{X_n}^\top
\end{bmatrix}}_{=\,K}
\,(K+n\lambda I_n)^{-1}(I_n-\mathbf 1_n w^\top)
\;+\;
\mathbf 1_n w^\top,
\]
where $k_{X_i}^\top=(k(X_i,X_1),\ldots,k(X_i,X_n))$.

\paragraph{\emph{Remark -- $S_U$ matrix (Proposition \ref{prop:linear-smoother-notation}).}}
Stacking the row weights for all $N$ design points yields
\[
S_U
=
\begin{bmatrix}
s(X_1)^\top\\
s(X_2)^\top\\
\vdots\\
s(X_N)^\top
\end{bmatrix}
=
\begin{bmatrix}
s_1(X_1) & s_2(X_1) & \cdots & s_n(X_1)\\
s_1(X_2) & s_2(X_2) & \cdots & s_n(X_2)\\
\vdots   & \vdots   & \ddots & \vdots\\
s_1(X_N) & s_2(X_N) & \cdots & s_n(X_N)
\end{bmatrix}
\in\mathbb R^{N\times n}.
\]
For kernel ridge regression with an unpenalized intercept,
\[
S_U
=
\underbrace{\begin{bmatrix}
k_{X_1}^\top\\
k_{X_2}^\top\\
\vdots\\
k_{X_N}^\top
\end{bmatrix}}_{=\,K_U}
\,(K+n\lambda I_n)^{-1}(I_n-\mathbf 1_n w^\top)
\;+\;
\mathbf 1_N w^\top,
\]
where $(K_U)_{ij}=k(X_i,X_j)$ for $i=1,\ldots,N$ and $j\in S$, and $\mathbf 1_N\in\mathbb R^N$ is the all-ones vector.
\end{proposition}

\begin{proof}
Consider the loss function of kernel ridge regression \citep{cheng2024comprehensive} with an unpenalized intercept. Let
\[
\mathcal L(\beta,\alpha)
=\frac1n\big\|Y-(K\beta+\alpha\mathbf 1_n)\big\|_2^2+\lambda\,\beta^\top K\beta,
\qquad
r(\beta,\alpha):=Y-K\beta-\alpha\mathbf 1_n,
\]
so that
\(
\mathcal L(\beta,\alpha)=\frac1n\,r^\top r+\lambda\,\beta^\top K\beta.
\)

By the chain rule,
\[
\frac{\partial}{\partial\beta}\,(r^\top r)=2\Big(\frac{\partial r}{\partial\beta}\Big)^\top r
=2(-K)^\top r=-2K r,
\qquad
\frac{\partial}{\partial\alpha}\,(r^\top r)=2\Big(\frac{\partial r}{\partial\alpha}\Big)^\top r
=2(-\mathbf 1_n)^\top r,
\]
and
\[
\frac{\partial}{\partial\beta}\,(\beta^\top K\beta)=(K+K^\top)\beta=2K\beta.
\]
Therefore
\begin{align*}
\frac{\partial\mathcal L}{\partial\beta}
&=\frac1n(-2K)r+ \lambda(2K)\beta
= -\frac{2}{n}K\big(Y-K\beta-\alpha\mathbf 1_n\big)+2\lambda K\beta,\\
\frac{\partial\mathcal L}{\partial\alpha}
&=\frac1n(-2\mathbf 1_n^\top)r
= -\frac{2}{n}\,\mathbf 1_n^\top\big(Y-K\beta-\alpha\mathbf 1_n\big).
\end{align*}

Setting the gradients to zero and simplifying gives
\begin{align}
K\big(Y-K\beta-\alpha\mathbf 1_n\big)&=n\lambda\,K\beta, \label{eq:g1}\\
\mathbf 1_n^\top K\beta+n\alpha&=\mathbf 1_n^\top Y. \label{eq:g2}
\end{align}
For $\lambda>0$ and $K\succeq 0$, $\mathcal L$ is convex in $(\beta,\alpha)$, so
\eqref{eq:g1}--\eqref{eq:g2} characterize the set of minimizers.


From \eqref{eq:g1} a convenient exact solution is
\[
(K+n\lambda I_n)\beta=Y-\alpha\mathbf 1_n
\quad\Longrightarrow\quad
\beta=(K+n\lambda I_n)^{-1}(Y-\alpha\mathbf 1_n),
\]
because left–multiplying by $K$ recovers \eqref{eq:g1}. Substitute this into \eqref{eq:g2}:
\[
\mathbf 1_n^\top \underbrace{K (K+n\lambda I_n)^{-1}}_{\text{(A)}} (Y-\alpha\mathbf 1_n)+n\alpha=\mathbf 1_n^\top Y.
\]

At this point, we prove a useful identity to proceed on the computation for term (A). Start from
\[
K \;=\; (K+n\lambda I_n) - n\lambda I_n.
\]
Right–multiply by \((K+n\lambda I_n)^{-1}\) (which exists for \(\lambda>0\)):
\begin{align}
\nonumber
K(K+n\lambda I_n)^{-1}
&= \big[(K+n\lambda I_n) - n\lambda I_n\big](K+n\lambda I_n)^{-1}\\
\nonumber
&= (K+n\lambda I_n)(K+n\lambda I_n)^{-1} - n\lambda I_n (K+n\lambda I_n)^{-1}\\
\label{eq:kernel_equation}
&= I_n - n\lambda (K+n\lambda I_n)^{-1}.
\end{align}


Using the identity \eqref{eq:kernel_equation} in term (A) to get,
\begin{align*}
&\mathbf 1_n^\top \underbrace{[I_n-n\lambda (K+n\lambda I_n)^{-1}]}_{\text{(A)} \quad (\because\,\, \text{Identity}  \eqref{eq:kernel_equation})} (Y-\alpha\mathbf 1_n)+n\alpha
=
\mathbf 1_n^\top Y.
\end{align*}
After distribution, cancel \(\mathbf 1_n^\top Y\) on both sides and use \(-\alpha n+n\alpha=0\) to obtain
\begin{align*}
\cancel{\mathbf 1_n^\top Y}
- \cancel{\alpha n} 
-n\lambda\,\mathbf 1_n^\top (K+n\lambda I_n)^{-1}Y
+n\lambda\alpha\,\mathbf 1_n^\top (K+n\lambda I_n)^{-1}\mathbf 1_n
+
\cancel{n\alpha}
=
\cancel{\mathbf 1_n^\top Y},
\end{align*}
which is
\[
-n\lambda\,\mathbf 1_n^\top (K+n\lambda I_n)^{-1}Y
+n\lambda\alpha\,\mathbf 1_n^\top (K+n\lambda I_n)^{-1}\mathbf 1_n=0,
\]
or equivalently, 
\[
\alpha\big(\cancel{n\lambda} \,\mathbf 1_n^\top (K+n\lambda I_n)^{-1}\mathbf 1_n\big)
=\cancel{n\lambda} \,\mathbf 1_n^\top (K+n\lambda I_n)^{-1}Y.
\]
This derives the intercept and coefficients
\begin{align}
\label{eq:intercept_term_derived}
  \hat\alpha
&=\frac{\mathbf 1_n^\top (K+n\lambda I_n)^{-1}Y}
        {\mathbf 1_n^\top (K+n\lambda I_n)^{-1}\mathbf 1_n}=
        \underbrace{\left\{
        \frac{\mathbf 1_n^\top (K+n\lambda I_n)^{-1}}
        {\mathbf 1_n^\top (K+n\lambda I_n)^{-1}\mathbf 1_n}\right\}}_{w^\top}         Y,
\\
\hat\beta
&=(K+n\lambda I_n)^{-1}\big(Y-\hat\alpha\,\mathbf 1_n\big)
=
(K+n\lambda I_n)^{-1}\big(I_nY-\hat\alpha\,\mathbf 1_n\big).  
\nonumber
\end{align}

For the linear smoother form, considering the form of the equation, \eqref{eq:intercept_term_derived}, set
\[
w^\top=\frac{\mathbf 1_n^\top (K+n\lambda I_n)^{-1}}{\mathbf 1_n^\top (K+n\lambda I_n)^{-1}\mathbf 1_n} \quad \text{(it holds } w^\top\mathbf 1_n=1\,\text{)} \tag{$\star$},
\]
so the intercept and coefficients can be expressed as linear transformation of $Y$
\[
\begin{pmatrix}
\hat\alpha \\[6pt]
\hat\beta
\end{pmatrix}
=
\begin{pmatrix}
w^\top Y \\[6pt]
(K+n\lambda I_n)^{-1}\big(I_n-\mathbf 1_n w^\top\big)Y
\end{pmatrix}=
\underbrace{\begin{pmatrix}
w^\top \\[6pt]
(K+n\lambda I_n)^{-1}\!\big(I_n - \mathbf 1_n w^\top\big)
\end{pmatrix}}_{(n+1) \times n}
\underbrace{Y}_{n \times 1} \in \mathbb{R}^{n+1}.
\]
Then the vector of fitted values on the labeled (training) points,
\(
\widehat m_S := \big(\widehat m(X_j)\big)_{j\in S}\in\mathbb{R}^n,
\)
can be written as a linear smoother of \(Y\):
\begin{align*}
\widehat m_S
&=
\begin{pmatrix}
\mathbf 1_n \quad  K
\end{pmatrix}
\begin{pmatrix}
\hat\alpha \\[6pt]
\hat\beta
\end{pmatrix}
=
K\hat\beta+\hat\alpha\mathbf 1_n
=K(K+n\lambda I_n)^{-1}\big(I_n-\mathbf 1_n w^\top\big)Y+\mathbf 1_n w^\top Y\\
&=
\left\{K(K+n\lambda I_n)^{-1}\big(I_n-\mathbf 1_n w^\top\big)+\mathbf 1_n w^\top   \right\}Y\\
&=H\,Y,    
\end{align*}
with the (training-set) hat matrix $H \in \mathbb R^{n \times n}$ given by
\begin{align*}
    H=K(K+n\lambda I_n)^{-1}(I_n-\mathbf 1_n w^\top)+\mathbf 1_n w^\top.
\end{align*}

Now, we show the hat matrix $H$ is symmetric. Set $A:=K+n\lambda I_n$. Since $A$ is symmetric positive definite, $A^{-1}$ is symmetric. By \eqref{eq:kernel_equation}, we have the identity
\[
K A^{-1}=(A-n\lambda I_n)A^{-1}=I_n-n\lambda A^{-1}.
\]
Hence
\begin{align*}
H
&= KA^{-1}(I_n-\mathbf 1_n w^\top)+\mathbf 1_n w^\top\\
&=(I_n-n\lambda A^{-1})(I_n-\mathbf{1}_n w^\top)+\mathbf{1}_n w^\top \\
&= I_n-n\lambda A^{-1}-\mathbf{1}_n w^\top+n\lambda A^{-1}\mathbf{1}_n w^\top+\mathbf{1}_n w^\top \\
&= I_n-n\lambda A^{-1}+n\lambda A^{-1}\mathbf{1}_n w^\top .
\end{align*}
Taking transposes and using $A^{-1}=A^{-\top}$ gives
\[
H^\top= I_n-n\lambda A^{-1}+n\lambda\, w\,\mathbf{1}_n^\top A^{-1}.
\]
Let $b:=\mathbf{1}_n^\top A^{-1}\mathbf{1}_n>0$. By the definition of $w$ ($\star$), $A^{-1}\mathbf{1}_n=b\,w$ and $\mathbf{1}_n^\top A^{-1}=b\,w^\top$, so
\[
A^{-1}\mathbf{1}_n w^\top
= (b\,w) w^\top
= b\,w w^\top
= w (b\,w^\top)
= w\,\mathbf{1}_n^\top A^{-1}.
\]
Therefore the last terms in the expressions for $H$ and $H^\top$ coincide, and we conclude $H=H^\top$.

We verify mass preservation:
\begin{align*}
H\mathbf 1_n
&=
\left\{K(K+n\lambda I_n)^{-1}\big(I_n-\mathbf 1_n w^\top\big)+\mathbf 1_n w^\top   \right\}\mathbf 1_n\\
&=K(K+n\lambda I_n)^{-1}\big(\mathbf 1_n-\mathbf 1_n w^\top\mathbf 1_n\big)+\mathbf 1_n w^\top\mathbf 1_n\\
&=K(K+n\lambda I_n)^{-1}\cdot \mathbf 0+\mathbf 1_n\cdot 1
=\mathbf 1_n,    
\end{align*}
since $w^\top\mathbf 1_n=1$ ($\star$). Thus constants are reproduced exactly on the training set.

For a new input point $x$, we defined the kernel vector $k_x$ as kernel evaluations between $x$ and the labeled covariates
\[
k_x := 
\begin{pmatrix}
k(X_1,x) \\
k(X_2,x) \\
\vdots \\
k(X_n,x)
\end{pmatrix}
\in \mathbb{R}^n,
\]
so that its $j$th entry is $(k_x)_j = k(X_j,x)$. The same algebra used for the training fits, the fitted value at a new point $x$, $\widehat m(x) \in \mathbb R$, can be
written as a linear functional of $Y$:
\begin{align*}
\widehat m(x)
&=k_x^\top (K+n\lambda I_n)^{-1}(I_n-\mathbf 1_n w^\top)Y+w^\top Y\\
&=
\left\{k_x^\top (K+n\lambda I_n)^{-1}(I_n-\mathbf 1_n w^\top)+w^\top  \right\}Y \\
&=:s(x)^\top Y,
\end{align*}
where we define the (row) smoothing vector
\[
s(x)^\top
:= k_x^\top (K+n\lambda I_n)\!^{-1}\!\big(I_n-\mathbf 1_n w^\top\big) \;+\; w^\top
\;\in\; \mathbb R^{1\times n}.
\]
Equivalently,
\[
\widehat m(x) = 
\underbrace{s(x)^\top}_{1\times n}\,\underbrace{Y}_{n\times 1}.
\]

The sum of the weights equals one:
\[
s(x)^\top\mathbf 1_n
=k_x^\top (K+n\lambda I_n)^{-1}(I_n-\mathbf 1_n w^\top)\mathbf 1_n+w^\top \mathbf 1_n
=k_x^\top (K+n\lambda I_n)^{-1}\big(\mathbf 1_n-\mathbf 1_n\cdot 1\big)+1
=1.
\]
Hence the smoother vector $s(x)$ is mass–preserving pointwise as well.
\end{proof}


\begin{theorem}[Consistency and asymptotic normality of SF--PPI with variance correction (linear–smoother case)]
\label{thm:sfppi_mean_unified}
Let $(X,Y)\sim P$ with $X\in\mathcal X\subset\mathbb R^d$, $Y\in\mathcal Y\subset\mathbb R$, and $\mathbb E[Y^2]<\infty$.
Assume $Y=m_0(X)+\varepsilon$ with $\mathbb E[\varepsilon\mid X]=0$ and $\Var(\varepsilon|X)=\sigma^2$.
The target is $\theta_0:=\mathbb E[Y]=\mathbb E[m_0(X)]$.


Let \(\{X_i\}_{i=1}^N\) be an unlabeled sample drawn i.i.d.\ from \(P_X\), and let
\(\{(X_j,Y_j)\}_{j=1}^n\) be an independent labeled sample drawn i.i.d.\ from \(P\),
with \(n/N\to f\in(0,1)\). For notational simplicity, we write
\(S=\{1,\dots,n\}\) for the labeled sample indices. 

Let $\widehat m$ be a single–fit affine linear smoother trained on $S$:
\[
\widehat m(x)=s(x)^\top Y_S+b(x),\qquad s(x)=(s_j(x))_{j\in S}\in\mathbb R^n.
\]
Assume the smoother is \emph{mass–preserving (MP)}: $\sum_{j\in S}s_j(x)=1$ for all $x$, so
$\mathbf 1_n^\top a=\mathbf 1_n^\top\bar h=1$ and
$\|c\|_2^2=\frac{1}{n}+\|a-\bar h\|_2^2$ ($a$, $\bar h$, and $c$ defined in Proposition~\ref{prop:linear-smoother-notation}).
Define $S_U,H,a,\bar h,c$ and $\overline{\|s(X)\|_2^2}$ as in that proposition.
Consider the SF–PPI estimator
\[
\widehat\theta^{\mathrm{sf}}_{\mathrm{PPI}}
:= P_N\widehat m+P_n(Y-\widehat m)
= \frac{1}{N}\sum_{i=1}^N \widehat m(X_i)+\frac{1}{n}\sum_{j\in S}\{Y_j-\widehat m(X_j)\}.
\]

\noindent\textbf{Assumptions.}
\begin{enumerate}[label=\textup{A\arabic*}, leftmargin=3.2em]
\item \label{ass:A1} \emph{Predictor accuracy:} $\|\widehat m-m_0\|_{L_2(P_X)}=o_p(1)$.
\item \label{ass:A2} \emph{Smoother stability:} (i) $\max_{1\le i\le N}\|s(X_i)\|_2^2=O_p(1/n)$, (ii) $\max_{j\in S}|c_j|=o_p(N^{-1/2})$, and (iii) $\|a-\bar h\|_2^2=o_p(1/n)$ (equivalently, $N\|c\|_2^2 \to 1/f$ by MP).
\item \label{ass:A3} \emph{Offset regularity:}
      $(P_N-P_n)b=o_p(N^{-1/2})$ and $\|b\|_{L_2(P_X)}\to 0$.
\end{enumerate}

Then:

\smallskip
\noindent \textbf{\textup{(i) Consistency.}}
Under \((\ref{ass:A1})\)–\((\ref{ass:A3})\), \(\widehat\theta^{\mathrm{sf}}_{\mathrm{PPI}}\xrightarrow{p}\theta_0\).

\smallskip
\noindent \textbf{\textup{(ii) Asymptotic normality.}}
Under \((\ref{ass:A1})\)–\((\ref{ass:A3})\),
\[
\sqrt{N}\,(\,\widehat\theta^{\mathrm{sf}}_{\mathrm{PPI}}-\theta_0)
\ \xrightarrow{d}\ \mathcal N\!(0, \sigma_f^2 \,).
\]
where $\sigma_f^2=\Var(m_0(X))+\tfrac{1}{f} \sigma^2=\Var(m_0(X))+\tfrac{1}{f} \Var (Y - m_0(X))$.

\smallskip
\noindent \textbf{\textup{(iii) Variance correction and studentized CLT.} }
Let $\widehat\sigma^2\xrightarrow{p}\sigma^2$ be any consistent estimator and define
\[
\widehat{\Var}\!\big(\widehat\theta^{\mathrm{sf}}_{\mathrm{PPI}}\big)
:=\frac{1}{N}\Big\{\Var_N(\widehat m(X_1),\ldots,\widehat m(X_N))
-\widehat\sigma^2\,\overline{\|s(X)\|_2^2}\Big\}
+\widehat\sigma^2\,\|c\|_2^2.
\]
Under \((\ref{ass:A2})\)–\((\ref{ass:A3})\),
\(\widehat{\Var}(\widehat\theta^{\mathrm{sf}}_{\mathrm{PPI}})\xrightarrow{p}\sigma_f^2/N\) and
\[
\frac{\widehat\theta^{\mathrm{sf}}_{\mathrm{PPI}}-\theta_0}{\sqrt{\widehat{\Var}(\widehat\theta^{\mathrm{sf}}_{\mathrm{PPI}})}}
\ \xrightarrow{d}\ \mathcal N(0,1).
\]
\end{theorem}

\begin{proof}
Throughout write \(Y=m_0(X)+\varepsilon\) with \(\mathbb E[\varepsilon\mid X]=0\) and
\(\sigma^2:=\mathbb E[\varepsilon^2]=\Var(Y-m_0(X))<\infty\). We often use empirical process notations; for any integrable \(h:\mathcal X\to\mathbb R\), set $P_N h=(1/N)\sum_{i=1}^N h(X_i),\,\,
P_n h=(1/n)\sum_{j\in S} h(X_j),\,\,
P h=\mathbb E\{h(X)\}.$ Reserve \(\mathcal X\subset\mathbb R^d\) for the covariate space.
Let \(X_{1:N}:=(X_1,\ldots,X_N)\) denote the full list of population covariates and
\(X_S:=(X_j)_{j\in S}\) the labeled sublist.
Write \(\mathscr{X}_N:=\sigma(X_{1:N})\).
Expectations/variances “given \(\mathscr{X}_N\)” treat \(\varepsilon_S\) as the only source of randomness.

\paragraph{Smoother stability assumption (A2).} As for \ref{ass:A2}-(i), we have
\[
\overline{\|s(X)\|_2^2}
:= \frac1N\sum_{i=1}^N\|s(X_i)\|_2^2
\;\le\; \max_{1\le i\le N}\|s(X_i)\|_2^2
= O_p(n^{-1}).
\]

Furtherfmore, under mass preservation, we have $\|c\|_2^2 = \frac1n + \|a-\bar h\|_2^2$ (by Proposition~\ref{prop:linear-smoother-notation}-(f)), hence, using $N/n\to 1/f$, we have
\begin{align}
\label{eq:c_limit_behavior}
N\|c\|_2^2 \to \frac1f
\;\Longleftrightarrow\;
N\|a-\bar h\|_2^2 \to 0
\;\Longleftrightarrow\;
\|a-\bar h\|_2^2 = o_p(N^{-1})
\;\Longleftrightarrow\;
\|a-\bar h\|_2^2 = o_p(n^{-1}),
\end{align}
where $\Longleftrightarrow$ implies the equivalence. Therefore, we have $\|c\|_2^{2}
= \frac{1}{n} + \|a-\bar h\|_2^{2}
= \frac{1}{n} + o_p\!\left(\frac{1}{n}\right)$, which leads to $\frac{\|c\|_2^{2}}{1/n}
= 1 + n\|a-\bar h\|_2^{2}
\;\xrightarrow{p}\; 1.$ Thus, $\|c\|_2^{2}
= \frac{1}{n}\{1+o_p(1)\},$ or equivalently, $\frac{\|c\|_2^{2}}{1/n}\;\xrightarrow{p}\;1$, which means that \(\|c\|_2^{2}\) is asymptotically equivalent to \(1/n\).

\paragraph{Initial three-term decomposition of $\hat\theta^{\sf sf}_{\rm PPI}-\theta_0$.}
Note that the difference between the single–fit PPI estimator and the true parameter, 
$\widehat\theta^{\sf sf}_{\rm PPI}-\theta_0$, can be decomposed as
\begin{align}
\widehat\theta^{\sf sf}_{\rm PPI}-\theta_0
&= P_N\widehat m+P_n(Y-\widehat m) - P m_0 \nonumber\\
&=(P_N-P)\,m_0 + (P_n-P)\,(Y-m_0) + (P_N-P_n)\,(\widehat m-m_0)
\tag{$\star$}\\
&= \underbrace{(P_N-P)m_0}_{\text{unlabeled fluctuation}}
 \;+\; \underbrace{P_n\varepsilon}_{\text{labeled noise}}
 \;+\; \underbrace{(P_N-P_n)(\widehat m-m_0)}_{\text{remainder from learning }\widehat m}
 \qquad\text{(since $Y=m_0+\varepsilon$)}, 
\label{eq:basic-identity}
\end{align}
where decomposition ($\star$) has been derived in the proof of Theorem~\ref{thm:cfppi_mean_consistency}.

\paragraph{Remark -- remainder term of cross-fit PPI estimator.}
For the cross-fit PPI estimator, the predictor evaluated on labeled points is trained on a disjoint fold; denote the corresponding out-of-fold predictor by $\widehat m^{(-)}$. Then
\[
(P_N-P_n)\big(\widehat m^{(-)}-m_0\big)
= O_p\!\left(\|\widehat m^{(-)}-m_0\|_{L_2(P_X)}\left(\tfrac{1}{\sqrt{N}}+\tfrac{1}{\sqrt{n}}\right)\right),
\]
and consequently
\[
\sqrt{N}\,(P_N-P_n)\big(\widehat m^{(-)}-m_0\big)
= O_p\!\left(\|\widehat m^{(-)}-m_0\|_{L_2(P_X)}\right).
\]
Hence the remainder $\sqrt{N}(P_N-P_n)(\widehat m^{(-)}-m_0)$ is $O_p(1)$ if 
$\|\widehat m^{(-)}-m_0\|_{L_2(P_X)} = O_p(1)$, used for consistency (Theorem \ref{thm:cfppi_mean_consistency}), and it is $o_p(1)$ if $\|\widehat m^{(-)}-m_0\|_{L_2(P_X)}=o_p(1)$ used for asymptotic normality (Theorem \ref{thm:cfppi_mean_clt}).

\paragraph{Three-term decomposition of remainder term $(P_N-P_n)(\widehat m-m_0)$ in $\hat\theta^{\sf sf}_{\rm PPI}-\theta_0$.} Note that, we consider the fitted predictor  \(\widehat m\) expressed as an affine linear smoother trained on \(S\): $\widehat m(x)=s(x)^\top Y_S+b(x),$ and $Y_S=m_0(X_S)+\varepsilon_S$.

Then, for each $x\in\mathcal X$,
\begin{align*}
\widehat m(x)-m_0(x)
&= s(x)^\top\big(m_0(X_S)+\varepsilon_S\big) + b(x) - m_0(x)\\
&=\underbrace{\big(s(x)^\top m_0(X_S)-m_0(x)\big)}_{\text{bias transport}}
\;+\;\underbrace{s(x)^\top \varepsilon_S}_{\text{labeled-noise leakage}}
\;+\; \underbrace{b(x)}_{\text{offset}}.
\end{align*}
For brevity, we use the function notation
\[
s^\top m_0:\ x\mapsto s(x)^\top m_0(X_S),\qquad
s^\top \varepsilon_S:\ x\mapsto s(x)^\top \varepsilon_S,
\]
so that, as a function of $x$,
\[
\widehat m - m_0
=\underbrace{(s^\top m_0 - m_0)}_{\text{bias transport}}
\;+\;\underbrace{s^\top\varepsilon_S}_{\text{labeled-noise leakage}}
\;+\; \underbrace{b}_{\text{offset}}.
\]
Insert this into the remainder of \eqref{eq:basic-identity}:
\begin{equation}\label{eq:R-decomp}
(P_N-P_n)(\widehat m-m_0)
=(P_N-P_n)(s^\top m_0-m_0)+(P_N-P_n)(s^\top\varepsilon_S)+(P_N-P_n)b.
\end{equation}

We decomposed the remainder into three terms, each with a different role: (i). $(P_N-P_n)(s^\top m_0-m_0)$ is a \emph{bias–transport} piece (covariate–only); (ii). $(P_N-P_n)(s^\top\varepsilon_S)$ is the \emph{labeled–noise leakage}; and (iii). $(P_N-P_n)b$ is the \emph{offset} piece (covariate–only).

\paragraph{Final three–term decomposition of $\hat\theta^{\sf sf}_{\rm PPI}-\theta_0$ via regrouping.} By inserting remainder term (\ref{eq:R-decomp}) to the decomposition (\ref{eq:basic-identity}), we can now decompose the difference $\hat\theta^{\sf sf}_{\rm PPI}-\theta_0$ into three terms
\begin{align}
&\hat\theta^{\sf sf}_{\rm PPI}-\theta_0
= 
(P_N-P)m_0 + 
P_n\varepsilon +
(P_N-P_n)(\widehat m-m_0)
 \nonumber
 \\
 &= (P_N-P)m_0
 + P_n\varepsilon
 +(P_N-P_n)(s^\top m_0-m_0)+(P_N-P_n)(s^\top\varepsilon_S)+(P_N-P_n)b
 \label{eq:five-decomp_before}
 \\
 &=\underbrace{(P_N-P)m_0}_{\substack{\text{(U) unlabeled}\\\text{fluctuation}}}
\;+\;
\underbrace{\tfrac{1}{n}\mathbf 1_n^\top\varepsilon_S}_{\substack{\text{(L) raw}\\\text{labeled noise}}}
\;+\;
\underbrace{(P_N-P_n)\!\big(s^\top m_0-m_0\big)}_{\substack{\text{(T) bias}\\\text{transport}}}
\;+\;
\underbrace{(a-\bar h)^\top\varepsilon_S}_{\substack{\text{(R) noise}\\\text{leakage}}}
\;+\;
\underbrace{(P_N-P_n)b}_{\substack{\text{(B)}\\\text{offset}}}.\label{eq:five-decomp}\\
&=\underbrace{(P_N-P)m_0}_{\substack{\text{(U) unlabeled}\\\text{fluctuation}}}
\;+\;
\underbrace{\left(\tfrac{1}{n}\mathbf 1_n + (a-\bar h)\right)^\top \varepsilon_S}_{\substack{\text{(L) + (R) }\\\text{raw labeled noise + noise leakage}}}
\;+\;
\underbrace{(P_N-P_n)\!\big(s^\top m_0-m_0\big)}_{\substack{\text{(T) bias}\\\text{transport}}}
\;+\;
\underbrace{(P_N-P_n)b}_{\substack{\text{(B)}\\\text{offset}}}.\label{eq:four-decomp}\\
&=\underbrace{(P_N-P)m_0}_{\substack{\text{(U) unlabeled}\\\text{fluctuation} \\\text{$=:U_N$} }}
\;+\;
\underbrace{c^\top \varepsilon_S}_{\substack{\text{(N) }\\\text{label-noise leakage}\\ \text{$=:L_n$}}}
\;+\;
\underbrace{\underbrace{(P_N-P_n)\!\big(s^\top m_0-m_0\big)}_{\substack{\text{(T) bias}\\\text{transport}}}
\;+\;
\underbrace{(P_N-P_n)b}_{\substack{\text{(B)}\\\text{offset}}}.\label{eq:four-decomp_final}}_{\text{$=:R_{N,n}$}}
\end{align}
Here, from \eqref{eq:five-decomp_before} to \eqref{eq:five-decomp}, we used
$P_n\varepsilon
= \frac{1}{n}\sum_{j\in S}\varepsilon_j
= \tfrac{1}{n}\mathbf 1_n^\top \varepsilon_S$
to obtain term (L) and
\[
(P_N-P_n)\big(s^\top\varepsilon_S\big)
= P_N\big(s^\top\varepsilon_S\big)-P_n\big(s^\top\varepsilon_S\big) 
= \Big(\tfrac{1}{N}S_U^\top \mathbf 1_N\Big)^\top \varepsilon_S
   - \Big(\tfrac{1}{n}H^\top \mathbf 1_n\Big)^\top \varepsilon_S 
= (a-\bar h)^\top \varepsilon_S
\]
to obtain term (R), where $a:=N^{-1}S_U^\top\mathbf 1_N$ and $\bar h:=n^{-1}H^\top\mathbf 1_n$
with $S_U,H$ as in Proposition~\ref{prop:linear-smoother-notation}.
From \eqref{eq:five-decomp} to \eqref{eq:four-decomp}, we grouped the two noise terms (L) and (R),
and from \eqref{eq:four-decomp} to \eqref{eq:four-decomp_final}
we introduced $c:=a+\tfrac1n\mathbf 1_n-\bar h$.

Finally, we denote the three terms in \eqref{eq:four-decomp_final} by
\[
U_N := (P_N-P)m_0,\qquad
L_n := c^\top\varepsilon_S,\qquad
R_{N,n} := (P_N-P_n)(s^\top m_0 - m_0) + (P_N-P_n)b,
\]
which yields the final three-term decomposition
\begin{equation}\label{eq:UNLN}
\widehat\theta^{\sf sf}_{\rm PPI}-\theta_0 \;=\; U_N + L_n + R_{N,n}.
\end{equation}

\paragraph{Contrast with decomposition used for asymptotic behavior of CF-PPI.} Before proceeding with the main proof, we briefly outline the idea by contrasting it with CF-PPI decomposition used in the proofs of consistency/CLT for CF–PPI
(Theorems~\ref{thm:cfppi_mean_consistency}–\ref{thm:cfppi_mean_clt}). 

In the SF--PPI decomposition
\eqref{eq:five-decomp_before}--\eqref{eq:four-decomp_final}, the structure differs from CF-PPI decomposition (i.e., ($\star$) in the proof of Theorem \ref{thm:cfppi_mean_consistency}). The
final decomposition \eqref{eq:four-decomp_final} yields three principal components: (i) the \(U_N\) term---the
unlabeled fluctuation \((P_N-P)m_0\), a fixed--function empirical--process term of order \(O_p(N^{-1/2})\);
(ii) the \(L_n\) term---the label--noise leakage (terms (L) and (R) regrouped into \(L_n\)), arising because
the labeled noise \(\varepsilon_S\) appears both in fitting \(\widehat m\) and in evaluating residuals; and
(iii) the \(R_{N,n}\) term---the bias--transport piece \((P_N-P_n)(s^\top m_0 - m_0)\), which resembles a
standard residual fluctuation but is now entangled with the fitted rule, together with an offset
\((P_N-P_n)b\) that is often negligible (e.g., for KRR with an unpenalized intercept,
\((P_N-P_n)b=0\)).


One should note that the decomposition \eqref{eq:four-decomp_final} was obtained by
first expanding the remainder $(P_N-P_n)(\widehat m-m_0)$ and then
\emph{regrouping all terms that depend on the labeled noise} $\varepsilon_S$.
Indeed,
\[
P_n\varepsilon\;+\;(P_N-P_n)(s^\top\varepsilon_S)
=\Big(\tfrac{1}{n}\mathbf 1_n^\top+(a-\bar h)^\top\Big)\varepsilon_S
= c^\top\varepsilon_S,
\]
so the only randomness coming from the label noise (i.e., the residuals 
$\varepsilon_S = (Y_j - m_0(X_j))_{j \in S}$) is isolated in the single linear form $c^\top \varepsilon_S$. This step is essential in the single–fit setting: because \(\widehat m\) is trained on the
same labeled sample, the remainder \((P_N-P_n)(\widehat m - m_0)\) shares the label noise and,
after regrouping, contributes at the \(\sqrt N\)-scale via the “leakage’’ term
\(c^\top \varepsilon_S\), with
\[
\Var\!\big(\sqrt N\,c^\top \varepsilon_S \mid \mathcal X_N\big)
  \longrightarrow \frac{\sigma^2}{f}.
\]
The appearance of the noise variance \(\sigma^2\) in the first–order limit suggests that a variance correction is required for valid and efficient inference (rigorous details follow).

By contrast, in CF–PPI the cross-fitted predictor $\widehat m^{(-)}$ used on the labeled points is trained on disjoint folds, so $(P_N-P_n)(\widehat m^{(-)}-m_0)$ is (measurably) a function of the covariates only.  Consequently, no leakage term appears and the simpler three–term decomposition (unlabeled fluctuation $+$ raw labeled noise $+$ negligible remainder) suffices for the CF–PPI proofs. Refer to proof of Theoerem \ref{thm:cfppi_mean_consistency} for more detail.

We are now ready to proceed to the main part of the proofs:
\paragraph{Term (U): unlabeled fluctuation.}
Consider 
\[
U_N=(P_N-P)m_0=\frac1N\sum_{i=1}^N\{m_0(X_i)-\theta_0\},\qquad
\theta_0=\mathbb E[m_0(X)].
\]
Let \(Z_i:=m_0(X_i)-\theta_0\). Then \((Z_i)_{i=1}^N\) are i.i.d. with
\(\mathbb E[Z_i]=0\) and
\[
\Var(Z_i)=\Var\!\big(m_0(X)\big)<\infty,
\]
since $\mathbb E\!\big[m_0(X)^2\big]
=\mathbb E\!\big[(\mathbb E[Y\mid X])^2\big]
\le \mathbb E\!\big[\mathbb E(Y^2\mid X)\big]
=\mathbb E[Y^2]<\infty$. 

Hence, by the classical i.i.d.\ CLT,
\begin{align}
\label{eq:term_U_N}
\sqrt{N}\,U_N
=\frac{1}{\sqrt N}\sum_{i=1}^N Z_i
\ \xrightarrow{d}\ \mathcal N\!\big(0,\Var(m_0(X))\big),    
\end{align}
so that \(\mathbb E[U_N]=0\) and \(\Var(U_N)=\Var(m_0(X))/N\).
This step uses only i.i.d.\ sampling of \(X_i\) and the finite second moment of \(Y\),
and does not rely on Assumptions~\((\ref{ass:A1})\)–\((\ref{ass:A3})\) nor on mass preservation.

\paragraph{Term (L): label–noise leakage.}
Consider 
$$L_n=c^\top\varepsilon_S=\sum_{j\in S} c_j\,\varepsilon_j,$$
where the weights are
\[
c_j \;=\; a_j+\frac{1}{n}-\bar h_j
=\; \frac{1}{N}\sum_{i=1}^N s_j(X_i)\;+\;\frac{1}{n}
\;-\;\frac{1}{n}\sum_{i=1}^n s_j(X_i),
\]
since $a:=\frac{1}{N}S_U^\top\mathbf 1_N$ and $\bar h:=\frac{1}{n}H^\top\mathbf 1_n$.

Thus \(c=(c_j)_{j\in S}\) depends only on the covariates (and the realized labeled
index set \(S\)), and is measurable with respect to \(\mathscr X_N:=\sigma(X_{1:N})\).


Define $w_{N,j}:=\sqrt{N}\,c_j$. Then
\[
\sqrt{N}\,L_n=\sqrt{N}\,c^\top\varepsilon_S=\sum_{j\in S}\sqrt{N}\,c_j\,\varepsilon_j
=\sum_{j\in S} w_{N,j}\,\varepsilon_j .
\]
Hence, conditioning on $\mathscr X_N:=\sigma(X_{1:N})$ and using independence of
$\{\varepsilon_j\}_{j\in S}$ across $j$,
\begin{align}
\label{eq:v_N^2}
v_N^2
&:=\Var(\sqrt{N}\,L_n\mid \mathscr X_N)
= \sum_{j\in S} w_{N,j}^2 \Var(\varepsilon_j\mid \mathscr X_N) =\sigma^2 \displaystyle\sum_{j\in S} w_{N,j}^2 = \sigma^2 N\|c\|_2^2
\end{align}

By Proposition~\ref{prop:linear-smoother-notation}(f) (mass preservation) and Assumption~\ref{ass:A2}, we already show that it holds
\[
\|c\|_2^2=\frac{1}{n}+\|a-\bar h\|_2^2=\frac{1}{n}+o_p\!\left(\frac{1}{n}\right).
\]
Multiplying by \(N\) and using \(N/n \to 1/f\) together with
\(n\|a-\bar h\|_2^2 \xrightarrow{p}0\) (i.e., $\|a-\bar h\|_2^2 = o_p(1/n)$), we obtain $N\|c\|_2^2
= \frac{N}{n} + N\|a-\bar h\|_2^2
= \frac{N}{n} + \frac{N}{n}\,\big(n\|a-\bar h\|_2^2\big)
\xrightarrow{p} \frac{1}{f}.$ 

Therefore, we have
$$v_N^2 \xrightarrow{p} \frac{\sigma^2}{f}.$$


\paragraph{Term (L): label–noise leakage — CLT via Lindeberg.} The variance calculation $v_N^2 =\Var(\sqrt{N}\,c^\top\varepsilon_S\mid\mathcal X_N)=\Var(\sqrt{N}\,L_n\mid \mathscr X_N)=\sigma^2 N\|c\|_2^2\to\sigma^2/f$ identifies the scale of the label–noise contribution, but variance
convergence alone does not imply asymptotic normality. To upgrade to a
Gaussian limit we invoke a conditional Lindeberg–Feller CLT for triangular
arrays.

We first state a conditional Lindeberg–Feller CLT \citep{lindeberg1922neue} for triangular arrays.

\begin{theorem}[Conditional Lindeberg–Feller CLT]
\label{theorem:cond-lind}
For each $N$, let $\{Z_{N,j}\}_{j=1}^{n}$ be random variables that are conditionally independent on a $\sigma$–field
$\mathscr G_N$. Assume
\[
\mathbb E[Z_{N,j}\mid \mathscr G_N]=0,
\qquad
\Var(Z_{N,j}\mid \mathscr G_N)=\sigma_{N,j}^2 \in [0,\infty),
\]
and write $s_N^2:=\sum_{j=1}^{n}\sigma_{N,j}^2$.
If for some $s^2\in(0,\infty)$ we have
\[
s_N^2 \xrightarrow{p} s^2,
\]
and, for every $\tau>0$,
\[
L_N(\tau)
:=\frac{1}{s_N^2}\sum_{j=1}^{n}
\mathbb E\!\left[ Z_{N,j}^2\,\mathbf 1\{|Z_{N,j}|>\tau\}\,\middle|\,\mathscr G_N\right]
\xrightarrow{p} 0,
\]
then
\[
\frac{1}{s_N}\sum_{j=1}^{n} Z_{N,j}
\ \xrightarrow{d}\ \mathcal N(0,1)
\quad\text{conditionally on }\mathscr G_N\ \text{(in probability)}.
\]
Equivalently, unconditionally,
\(
\sum_{j=1}^{n} Z_{N,j} \xrightarrow{d} \mathcal N(0,s^2).
\)
\end{theorem}

\medskip
We now apply Theorem~\ref{theorem:cond-lind} to the leakage term
\(L_n=c^\top\varepsilon_S=\sum_{j\in S} c_j\varepsilon_j\).
Let the conditioning $\sigma$–field be $\mathscr G_N:=\mathscr X_N=\sigma(X_{1:N})$, and define
\[
Z_{N,j}:=w_{N,j}\,\varepsilon_j,
\qquad
w_{N,j}:=\sqrt N\,c_j.
\]
Given $\mathscr X_N$, the $\{\varepsilon_j:j\in S\}$ are i.i.d.\ with
$\mathbb E[\varepsilon_j\mid\mathscr X_N]=0$ and
$\Var(\varepsilon_j\mid\mathscr X_N)=\sigma^2$, hence
\[
\mathbb E[Z_{N,j}\mid \mathscr X_N]=0,
\qquad
\Var(Z_{N,j}\mid \mathscr X_N)=\sigma^2 w_{N,j}^2 .
\]
Therefore the conditional variance of the sum is
\[
s_N^2
=\sum_{j\in S}\sigma^2 w_{N,j}^2
= \sigma^2\sum_{j\in S} N c_j^2
= \sigma^2 N\|c\|_2^2
=: v_N^2 .
\]
By mass preservation and Assumption~\ref{ass:A2},
\(
\|c\|_2^2=\tfrac1n+\|a-\bar h\|_2^2
=\tfrac1n+o_p(1/n)
\),
so \(N\|c\|_2^2\to_p 1/f\) and hence
\[
s_N^2=v_N^2 \ \xrightarrow{p}\ \sigma^2/f.
\]

\paragraph{Term (L): label–noise leakage — Lindeberg condition verification.}
Recall $Z_{N,j}=w_{N,j}\varepsilon_j$ and
$v_N^2=\Var\!\big(\sum_{j\in S}Z_{N,j}\,\big|\,\mathscr X_N\big)
=\sigma^2\sum_{\ell\in S} w_{N,\ell}^2$.
For fixed $\tau>0$, set
\[
L_N(\tau)
:=\frac{1}{v_N^2}\sum_{j\in S}
\mathbb E\!\Big[Z_{N,j}^2\,\mathbf 1\{|Z_{N,j}|>\tau\}\,\Big|\,\mathscr X_N\Big].
\]
Then
\begin{align*}
L_N(\tau)
&=\frac{1}{\sigma^2\sum_{\ell}w_{N,\ell}^2}
\sum_{j\in S}\mathbb E\!\Big[(w_{N,j}\varepsilon_j)^2\mathbf 1\{|w_{N,j}\varepsilon_j|>\tau\}\Big]\\
&=\sum_{j\in S}\pi_{N,j}\,
\frac{\mathbb E\!\big[\varepsilon^2\mathbf 1\{|\varepsilon|>\tau/|w_{N,j}|\}\big]}{\sigma^2},
\qquad
\pi_{N,j}:=\frac{w_{N,j}^2}{\sum_\ell w_{N,\ell}^2}.
\end{align*}
Let $W_N:=\max_{j\in S}|w_{N,j}|$ and define $\phi(t):=\sigma^{-2}
\mathbb E\!\big[\varepsilon^2\,\mathbf 1\{|\varepsilon|>\tau/t\}\big]$ for $t>0$
(with $\phi(0):=0$).
Since $|w_{N,j}|\le W_N$,
\[
\mathbf 1\{|\varepsilon|>\tau/|w_{N,j}|\}\ \le\
\mathbf 1\{|\varepsilon|>\tau/W_N\},
\]
hence
\begin{align}
\nonumber
0\ \le\ L_N(\tau)
&=\sum_{j \in S}\pi_{N,j}\,\frac{\mathbb E\!\big[\varepsilon^2\mathbf 1\{|\varepsilon|>\tau/|w_{N,j}|\}\big]}{\sigma^2}
\le \sum_{j \in S}\pi_{N,j}\,\frac{\mathbb E\!\big[\varepsilon^2\mathbf 1\{|\varepsilon|>\tau/W_N\}\big]}{\sigma^2}\\
&=\,\frac{\mathbb E\!\big[\varepsilon^2\mathbf 1\{|\varepsilon|>\tau/W_N\}\big]}{\sigma^2} \cdot\sum_{j \in S}\pi_{N,j}
= \phi(W_N).
\tag{$\ast$}    
\end{align}
Because $\mathbb E[\varepsilon^2]<\infty$, we have $\phi(t)\downarrow 0$ as $t\downarrow 0$
(by monotone/dominated convergence). By Assumption~\ref{ass:A2},
$W_N=\max_j|w_{N,j}|=\sqrt N\max_j|c_j|\xrightarrow{p}0$,
so the right-hand side of \((\ast)\) satisfies $\phi(W_N)\xrightarrow{p}0$.
Therefore $L_N(\tau)\xrightarrow{p}0$ for every $\tau>0$, i.e.,
\[
\frac{1}{v_N^2}\sum_{j\in S}
\mathbb E\!\Big[Z_{N,j}^2\,\mathbf 1\{|Z_{N,j}|>\tau\}\,\Big|\,\mathscr X_N\Big]\ \xrightarrow{p}\ 0,
\]
which is the conditional Lindeberg condition.

Now, we can apply Theorem~\ref{theorem:cond-lind}, which yields
\[
\frac{1}{v_N}\sum_{j\in S} Z_{N,j}
=\frac{\sqrt N\,L_n}{\sqrt{\sigma^2 N\|c\|_2^2}}
\ \xrightarrow{d}\ \mathcal N(0,1)
\qquad\text{conditionally on } \mathscr X_N.
\]
Finally, since $v_N^2=\sigma^2 N\|c\|_2^2 \xrightarrow{p}\sigma^2/f$,
Slutsky’s theorem yields the unconditional limit
\begin{equation}\label{eq:N-CLT}
\sqrt N\,L_n =\sqrt N\,c^\top\varepsilon_S = \sum_{j\in S} w_{N,j}\varepsilon_j \ \xrightarrow{d}\ \mathcal N\!\Big(0,\frac{\sigma^2}{f}\Big).
\end{equation}


\paragraph{Term (R): remainder.} The remaining part is to show $\sqrt N\,R_{N,n}=o_p(1)$. Recall
$R_{N,n}=(P_N-P_n)(s^\top m_0-m_0)+(P_N-P_n)b$. Note that both
$s^\top m_0-m_0$ and $b$ are \emph{functions of the covariates only}
(measurable w.r.t.\ $\mathscr X_N$). 

We first state a proposition for a generic bound for $(P_N-P_n)h$ for any square–integrable function $h\in L_2(P_X)$,
\[
\sqrt N\,(P_N-P_n)h=\underbrace{\sqrt N\,(P_N-P)h}_{(I)}
-\underbrace{\sqrt N\,(P_n-P)h}_{(II)}.
\]

Fix any $h\in L_2(P_X)$ that is measurable w.r.t.\ the training sample.
Conditionally on $h$, the variables $\{h(X_i)\}_{i=1}^N$ are i.i.d.\ with
mean $Ph:=\E[h(X)]$ and variance $\Var(h(X))$.
Hence $\mathbb E\!\big[\sqrt N\,(P_N-P)h\mid h\big]=0$ and $\Var\!\big(\sqrt N\,(P_N-P)h\mid h\big)=\Var\!\big(h(X)\big).$

Since $\Var(h(X))=\mathbb E[h(X)^2]-(Ph)^2\le \mathbb E[h(X)^2]=\|h\|_{L_2(P_X)}^2$, we get
\[
\mathbb E\!\big[(\sqrt N\,(P_N-P)h)^2\mid h\big]
=\Var\!\big(\sqrt N\,(P_N-P)h\mid h\big)
\le \|h\|_{L_2(P_X)}^2.
\]

Now apply Chebyshev’s inequality (for any $t>0$),
\[
\mathbb P\!\Big(\big|\sqrt N\,(P_N-P)h\big|>t\,\|h\|_{L_2(P_X)}\,\Big|\,h\Big)
\;\le\; \frac{\mathbb E[(\sqrt N\,(P_N-P)h)^2\mid h]}{t^2\|h\|_{L_2(P_X)}^2}
\;\le\; \frac{1}{t^2}.
\]
Taking expectations in $h$ yields the unconditional bound
\[
\mathbb P \!\Big(\big|\sqrt N\,(P_N-P)h\big|>t\,\|h\|_{L_2}\Big)\le \frac{1}{t^2},
\]
so $\sqrt N\,(P_N-P)h=O_p(\|h\|_{L_2(P_X)})$.

Exactly the same argument with the labeled average gives
\[
\sqrt n\,(P_n-P)h=O_p(\|h\|_{L_2(P_X)}).
\]
Therefore,
\begin{align*}
\sqrt N\,(P_N-P_n)h
&= \sqrt N\,(P_N-P)h - \sqrt N\,(P_n-P)h\\
&= O_p(\|h\|_{L_2(P_X)}) - \sqrt{\tfrac{N}{n}}\;O_p(\|h\|_{L_2(P_X)})
= O_p(\|h\|_{L_2(P_X)}),    
\end{align*}
because $N/n\to 1/f\,\,\, (\text{a finite constant})$ is bounded (i.e., $1/f = O(1)$). Consequently,
\begin{equation}\label{eq:PNPn-generic-correct}
\sqrt N\,(P_N-P_n)h \;=\; O_p\!\big(\|h\|_{L_2(P_X)}\big).
\end{equation}

Now, we go back to our main proof. We first want to control $\|s^\top m_0-m_0\|_{L_2}$ (write $L_2 = L_2(P_X)$ for notational simplicity). Recall that, using the affine smoother representation $\widehat m=s^\top Y_S+b$ and $Y_S=m_0(X_S)+\varepsilon_S$, $\widehat m - m_0 =(s^\top m_0-m_0)\;+\;s^\top\varepsilon_S\;+\;b,$
hence, by the triangle inequality,
\begin{equation}\label{eq:bias-transport-L2}
\|s^\top m_0-m_0\|_{L_2}
\ \le\ \|\widehat m-m_0\|_{L_2}
      +\|s^\top\varepsilon_S\|_{L_2}
      +\|b\|_{L_2}.
\end{equation}
By Assumption~\ref{ass:A1}, $\|\widehat m-m_0\|_{L_2}=o_p(1)$.
By Assumption~\ref{ass:A3}, $\|b\|_{L_2}\to 0$.

Note that the following equality holds:
\begin{align}
\nonumber
\big\|s^\top \varepsilon_S\big\|_{L_2(P_N)}^{2} &= P_N\!\left\{ \big( s(X)^\top \varepsilon_S \big)^2 \right\}
= \frac{1}{N}\sum_{i=1}^{N} \big( s(X_i)^\top \varepsilon_S \big)^2 
\\
\nonumber
&= \frac{1}{N}\,\varepsilon_S^\top
\left(\sum_{i=1}^{N} s(X_i)\,s(X_i)^\top\right)\varepsilon_S
= \frac{1}{N}\,\varepsilon_S^\top S_U^\top S_U\,\varepsilon_S,
\end{align}
where $S_U$ is the $N$-by-$n$ matrix whose $i$-th row is $s(X_i)^\top$: See Proposition \ref{prop:linear-smoother-notation} for the notation. The first equality used the notational convention $\|g\|_{L_2(P_N)}^2 := \int g(x)^2\, dP_N(x)
= \frac{1}{N}\sum_{i=1}^N g(X_i)^2
= P_N\!\big(g(X)^2\big).$

For the middle term, condition on $\mathscr X_N$: we have
\begin{align*}
\mathbb E\!\big[\|s^\top\varepsilon_S\|_{L_2}^2\,\big|\,\mathscr X_N\big]
&= \frac{1}{N}
\mathbb{E}\!\left[\varepsilon_S^\top S_U^\top S_U\,\varepsilon_S \,\middle|\, \mathscr X_N\right] \\&= \frac{1}{N}\,
\mathrm{tr}\!\left(S_U^\top S_U\,\mathbb{E}\!\left[\varepsilon_S \varepsilon_S^\top \,\middle|\, \mathscr X_N\right]\right) = \frac{\sigma^2}{N}\,\mathrm{tr}\!\left(S_U^\top S_U\right)
= \sigma^2\,\overline{\|s(X)\|_2^2}=O_p\!\Big(\frac{1}{n}\Big)    
\end{align*}
by Assumption~\ref{ass:A2}. Thus, we have $\|s^\top\varepsilon_S\|_{L_2}=O_p\!\Big(\frac{1}{\sqrt n}\Big),$ after which we plug it into \eqref{eq:bias-transport-L2} yields
\[
\|s^\top m_0-m_0\|_{L_2}
=o_p(1)+O_p(n^{-1/2})+o(1)=o_p(1).
\]

Applying \eqref{eq:PNPn-generic-correct} with $h=s^\top m_0-m_0$ leads to 
\begin{align}
\label{eq:remainder_1}
\sqrt N\,(P_N-P_n)\!\big(s^\top m_0-m_0\big)
=O_p\!\big(\|s^\top m_0-m_0\|_{L_2}\big)
=o_p(1).    
\end{align}

As for the offset piece we may either apply \eqref{eq:PNPn-generic-correct} with $h=b$ and use
$\|b\|_{L_2}\to 0$, or invoke directly Assumption~\ref{ass:A3}, which states
\begin{align}
\label{eq:remainder_2}
(P_N-P_n)b=o_p(N^{-1/2})
\quad\Rightarrow\quad
\sqrt N\,(P_N-P_n)b=o_p(1).    
\end{align}

Combining the two displays \eqref{eq:remainder_1} and \eqref{eq:remainder_2},
\begin{align}
\label{eq:remainder_final}
\sqrt N\,R_{N,n}
=\sqrt N\,(P_N-P_n)\!\big(s^\top m_0-m_0\big)
+\sqrt N\,(P_N-P_n)b
=o_p(1)+o_p(1)=o_p(1).    
\end{align}


\paragraph{Conclusion for (i)–(ii).}
Combining \eqref{eq:term_U_N}, \eqref{eq:N-CLT}, and \eqref{eq:remainder_final} with the
decomposition \eqref{eq:UNLN}, and noting that
$U_N$ is $\mathscr X_N$–measurable while $L_n$ is a function only of
$\varepsilon_S$ (hence $U_N \perp\!\!\!\perp L_n$ because
$\varepsilon$ is independent of $X$), we have
\[
\sqrt{N}\big(\widehat\theta^{\sf sf}_{\rm PPI}-\theta_0\big)
=\underbrace{\sqrt{N}\,U_N}_{\xrightarrow{d}\ \mathcal N(0,\Var(m_0(X)))}\;
 +\;\underbrace{\sqrt{N}\,L_n}_{\xrightarrow{d}\ \mathcal N(0,\sigma^2/f)}\;
 +\;\underbrace{\sqrt{N}\,R_{N,n}}_{\xrightarrow{p}\ 0}.
\]
By independence and Slutsky’s theorem, the sum of the two limiting normal
terms is normal with variance equal to the sum of variances, and the
negligible remainder does not affect the limit. Therefore,
\begin{align}
\label{eq:asymptotic_distribution_sf_ppi}
\sqrt N\big(\widehat\theta^{\sf sf}_{\rm PPI}-\theta_0\big)
\ \xrightarrow{d}\ \mathcal N\!\Big(0,\ \Var(m_0(X))+\tfrac{1}{f}\sigma^2\Big).    
\end{align}
Without $\sqrt N$–scaling, $U_N=O_p(N^{-1/2})$,
$L_n=O_p(n^{-1/2})$ (which is also $O_p(N^{-1/2})$ since $n/N\to f \in (0,1)$),
and $R_{N,n}=o_p(N^{-1/2})$ by \eqref{eq:remainder_final}.
Hence $\widehat\theta^{\sf sf}_{\rm PPI}\xrightarrow{p}\theta_0$.

\paragraph{Studentization (iii) — Overview and motivation.} 
One should note that the variance in the asymptotic distribution
\eqref{eq:asymptotic_distribution_sf_ppi} involves the nuisance quantities
$\Var(m_0(X))$ and $\sigma^2$, so the CLT as stated is not yet feasible for
inference. A naive plug–in $\Var_N(\widehat m(X_1),\ldots,\widehat m(X_N))/N$
would overestimate $\Var(m_0(X))/N$, because $\widehat m(X_i)$ contains
in–sample label noise. Thus, in the following, we show how to \emph{studentize}
the statistic by constructing a consistent estimator of its variance:
\[
\widehat{\Var}\!\big(\widehat\theta^{\sf sf}_{\mathrm{PPI}}\big)
:=\frac{1}{N}\!\left\{\Var_N(\widehat m(X_1),\ldots,\widehat m(X_N))
-\widehat\sigma^2\,\overline{\|s(X)\|_2^2}\right\}
+\widehat\sigma^2\,\|c\|_2^2,
\]
where $\widehat\sigma^2\xrightarrow{p}\sigma^2$. The subtraction removes the
inflation $\sigma^2\,\overline{\|s(X)\|_2^2}/N$ caused by training noise in
$\widehat m(X_i)$, while the addition $\widehat\sigma^2\|c\|_2^2$ accounts for
the “leakage’’ of labeled noise coming from the $P_n(Y-\widehat m)$ term. Under
(A2)–(A3) this estimator satisfies
$\widehat{\Var}(\widehat\theta^{\sf sf}_{\mathrm{PPI}})
\xrightarrow{p}\{\Var(m_0(X))+\sigma^2/f\}/N$, so the studentized statistic
\[
\frac{\widehat\theta^{\sf sf}_{\mathrm{PPI}}-\theta_0}
{\sqrt{\widehat{\Var}(\widehat\theta^{\sf sf}_{\mathrm{PPI}})}}
\ \xrightarrow{d}\ \mathcal N(0,1),
\]
yielding feasible standard errors and confidence intervals.

\paragraph{Studentization (iii) — Recall notations} 
Recall that for each \(i\in\{1,\dots,N\}\), the linear smoother evaluated at $X_i$ can be expressed as (Proposition \ref{prop:linear-smoother-notation})
\[
\widehat m(X_i)
= s(X_i)^\top Y_S + b(X_i)
= s(X_i)^\top\{m_0(X_S)+\varepsilon_S\}+b(X_i),
\]
where \(Y_S=(Y_j)_{j\in S} \in \mathbb R^n\), \(m_0(X_S)=(m_0(X_j))_{j\in S} \in \mathbb R^n\), and
\(\varepsilon_S=(\varepsilon_j)_{j\in S} \in \mathbb R^n\).

Introduce the function
\begin{align}
\label{eq:sm0}
(s^\top m_0)(x)\;:=\; s(x)^\top m_0(X_S)
\;=\;\sum_{j\in S} s_j(x)\,m_0(X_j): \mathcal{X} \xrightarrow{} \mathbb R,    
\end{align}
so that evaluating at \(x=X_i\), \(i\in\{1,\dots,N\}\), yields \((s^\top m_0)(X_i)=s(X_i)^\top m_0(X_S)\). One should note that the function (\ref{eq:sm0}) is the weighted average of the true values $m_0(X_j)$ with weights $s_j(x)$, and it is a function of $x$ for a fixed training set $X_S$.

Then
\begin{align*}
\Delta_i &:= \widehat m(X_i)-m_0(X_i) = s(X_i)^\top\{m_0(X_S)+\varepsilon_S\}+b(X_i)-m_0(X_i) \\
&= \underbrace{\big[(s^\top m_0)(X_i)-m_0(X_i)\big]}_{=:B_i}
\;+\; \underbrace{s(X_i)^\top\varepsilon_S}_{=:E_i}
\;+\; \underbrace{b(X_i)}_{=:b_i}.
\end{align*}
Equivalently,
\[
\widehat m(X_i)=m_0(X_i)+\Delta_i,
\qquad
\Delta_i=B_i+E_i+b_i.
\]

In terms of measurability with respect to \(\mathscr X_N:=\sigma(X_{1:N})\), observe that $B_i=(s^\top m_0-m_0)(X_i)$ is \(\mathscr X_N\)-measurable because it depends only on the smoothing weights \(s(\cdot)\) (functions of the
covariates and the index set \(S\)) and on the values \(\{m_0(X_j):j\in S\}\); all of these are
determined by \(X_{1:N}\). Likewise, the offset \(b_i=b(X_i)\) is \(\mathscr X_N\)-measurable, as it is a
covariate-only function. In contrast, \(E_i=s(X_i)^\top\varepsilon_S\) is \emph{not}
\(\mathscr X_N\)-measurable because it depends on the label noises \(\varepsilon_S\); rather, it satisfies
\(\mathbb E[E_i\mid\mathscr X_N]=0\) and (under homoskedasticity) 
\(\Cov(E_i,E_k\mid\mathscr X_N)=\sigma^2\, s(X_i)^\top s(X_k)\).


If the smoother is mass–preserving, $\sum_{j\in S}s_j(x)=1$ for all $x$. Then, for
$B_i:=(s^\top m_0-m_0)(X_i)$,
\begin{align*}
B_i
&= (s^\top m_0)(X_i) - m_0(X_i)
= \sum_{j\in S} s_j(X_i)\,m_0(X_j) - m_0(X_i) \\
&= \sum_{j\in S} s_j(X_i)\,m_0(X_j)
   - \sum_{j\in S} s_j(X_i)\,m_0(X_i)
   \qquad\big(\text{since }\sum_{j\in S}s_j(X_i)=1\big) \\
&= \sum_{j\in S} s_j(X_i)\,\big\{m_0(X_j)-m_0(X_i)\big\}.
\end{align*}

Abbreviate \(Z_i:=\widehat m(X_i)\) and recall the empirical variance
\begin{align*}
\Var_N(Z_1,\ldots,Z_N)&:=
P_N\!\big((Z-\bar Z_N)^2\big)=P_N(Z^2)-\{P_N Z\}^2
=\frac{1}{N}\sum_{i=1}^N Z_i^2
-\Big(\frac{1}{N}\sum_{i=1}^N Z_i\Big)^2,\\
P_N h&:=\frac1N\sum_{i=1}^N h(X_i),\quad \bar Z_N:=P_N Z.
\end{align*}
The empirical covariance $\Cov_N(A,B) = P_N [ (A - P_NA) (B - P_N B)]$ can also be similarly obtained.


\paragraph{Studentization (iii) — Step 1: Conditional expansion.}


Conditioning on $\mathscr X_N:=\sigma(X_{1:N})$ and using $Z_i=\widehat m(X_i)=m_0(X_i)+\Delta_i$,
\[
\Var_N(Z)\;=\;\Var_N(m_0)+\Var_N(\Delta)+2\,\Cov_N(m_0,\Delta),
\]
where $\Var_N,\Cov_N$ are taken with respect to \ $P_N$. With $\Delta=B+E+b$,
\[
\Var_N(\Delta)=\Var_N(E)+\Var_N(B+b)+2\,\Cov_N(E,B+b).
\]
Since $\{\varepsilon_j:j\in S\}$ are independent of $\mathscr X_N$ with
$\mathbb E[\varepsilon_j\mid\mathscr X_N]=0$ and $\Var(\varepsilon_j\mid\mathscr X_N)=\sigma^2$,
\[
\mathbb E[\,\Cov_N(m_0,E)\mid\mathscr X_N]=0,
\qquad
\mathbb E[\,\Cov_N(E,B+b)\mid\mathscr X_N]=0.
\]

For the $E$–part,
\begin{align*}
&\mathbb E\!\big[\Var_N(E)\mid\mathscr X_N\big]
=\mathbb E\!\big[P_N(E_i^2)-\{P_NE_i\}^2\mid\mathscr X_N\big]\\
&\quad = P_N\Big(\mathbb E\!\big[E_i^2\mid\mathscr X_N\big]\Big)
   \;-\;\mathbb E\!\Big[\big(P_NE\big)^2\,\Big|\,\mathscr X_N\Big]\\
&\quad= \frac{1}{N}\sum_{i=1}^N 
   \mathbb E\!\big[(s(X_i)^\top\varepsilon_S)^2\mid\mathscr X_N\big]
   \;-\;\mathbb E\!\Big[\Big(\frac{1}{N}\sum_{i=1}^N s(X_i)\Big)^\top\varepsilon_S
                       \Big(\frac{1}{N}\sum_{k=1}^N s(X_k)\Big)^\top\varepsilon_S
                       \,\Big|\,\mathscr X_N\Big]\\
&\quad= \frac{1}{N}\sum_{i=1}^N 
   s(X_i)^\top\,\mathbb E\!\big[\varepsilon_S\varepsilon_S^\top\mid\mathscr X_N\big]\,s(X_i)
   \;-\;\mathbb E\!\big[(a^\top\varepsilon_S)^2\mid\mathscr X_N\big]\\
&\quad= \frac{1}{N}\sum_{i=1}^N s(X_i)^\top(\sigma^2 I_n)s(X_i)
   \;-\;\sigma^2\,\|a\|_2^2\\
&\quad= \sigma^2\cdot \frac{1}{N}\sum_{i=1}^N \|s(X_i)\|_2^2
   \;-\;\sigma^2\,\|a\|_2^2\\
&\quad= \sigma^2\,\overline{\|s(X)\|_2^2}\;-\;\sigma^2\,\|a\|_2^2,
\end{align*}
where $a:=N^{-1}\sum_{i=1}^N s(X_i)=N^{-1}S_U^\top\mathbf 1_N$ and
$\overline{\|s(X)\|_2^2}:=N^{-1}\sum_{i=1}^N\|s(X_i)\|_2^2$. (Refer Proposition \ref{prop:linear-smoother-notation} for the notations for $S_U$.)

Putting the pieces together we obtain
\begin{align*}
& \mathbb{E}\!\big[\Var_N(\widehat m(X_1),\ldots,\widehat m(X_N)) \,\big|\, \mathscr X_N\big]
= \mathbb{E}\!\big[\Var_N(Z) \,\big|\, \mathscr X_N\big]  \\
&= \Var_N(m_0) + \mathbb{E}\!\big[\Var_N(\Delta) \,\big|\, \mathscr X_N\big]
   + 2\,\mathbb{E}\!\big[\Cov_N(m_0,\Delta) \,\big|\, \mathscr X_N\big] \\
&= \Var_N(m_0)
   + \underbrace{\mathbb{E}\!\big[\Var_N(E) \,\big|\, \mathscr X_N\big]}_{ =\sigma^2\,\overline{\|s(X)\|_2^2}\;-\;\sigma^2\,\|a\|_2^2} 
   + \underbrace{\mathbb{E}\!\big[\Var_N(B+b) \,\big|\, \mathscr X_N\big] + 2\,\mathbb{E}\!\big[\Cov_N(E,B+b) \,\big|\, \mathscr X_N\big]}_{R_N}  \\
&\qquad   
   + \underbrace{2\,\mathbb{E}\!\big[\Cov_N(m_0,E) \,\big|\, \mathscr X_N\big]}_{=0}
   + \underbrace{2\,\mathbb{E}\!\big[\Cov_N(m_0,B+b) \,\big|\, \mathscr X_N\big]}_{=0},
\end{align*}
and we get
\begin{align}\label{eq:Step1}
&\mathbb E\!\big[\Var_N(\widehat m(X_1),\ldots,\widehat m(X_N))\mid\mathscr X_N\big]\\
\nonumber
&\quad\quad =\Var_N(m_0(X_1),\ldots,m_0(X_N))
+\sigma^2\overline{\|s(X)\|_2^2}-\sigma^2\|a\|_2^2+R_N,
\end{align}
with the remainder
\[
R_N:=\mathbb E\!\big[\Var_N(B+b)\mid\mathscr X_N\big]
   \;+\;2\,\mathbb E\!\big[\Cov_N(m_0,B+b)\mid\mathscr X_N\big].
\]

\paragraph{Studentization (iii) — Step 2: Size of the remainder $R_N$.}
By the triangle inequality together with Jensen’s inequality for conditional
expectations,
\begin{align*}
|R_N|
&= \Big|\mathbb E\!\big[\Var_N(B+b)\mid\mathscr X_N\big]
      + 2\,\mathbb E\!\big[\Cov_N(m_0,B+b)\mid\mathscr X_N\big]\Big|\\
&\le \mathbb E\!\big[\,\Var_N(B+b)\,\bigm|\mathscr X_N\big]
   + 2\,\mathbb E\!\big[\,|\Cov_N(m_0,B+b)|\,\bigm|\mathscr X_N\big].   
\end{align*}
Moreover, for any function $g$, $\Var_N(g)=P_N(g^2)-\{P_N g\}^2 \;\le\; P_N(g^2),$ and, using $(u+v)^2\le 2u^2+2v^2$,
\[
\Var_N(B+b)\;\le\;P_N\!\big((B+b)^2\big)
\;\le\;2\,P_N(B^2)+2\,P_N(b^2).
\]

For the covariance term, by Cauchy–Schwarz under the empirical measure $P_N$, we know that it holds $|\Cov_N(U,V)|
= \big|P_N\{(U-\bar U)(V-\bar V)\}\big|
   \le \sqrt{P_N\{(U-\bar U)^2\}}\sqrt{P_N\{(V-\bar V)^2\}}      
= \sqrt{\Var_N(U)}\,\sqrt{\Var_N(V)}
   \le \sqrt{P_N(U^2)}\,\sqrt{P_N(V^2)}.$
   
Applying this with $U=m_0$ and $V=B+b$ and then using $(B+b)^2\le 2(B^2+b^2)$ gives
\[
|\Cov_N(m_0,B+b)|
\;\le\;\sqrt{P_N(m_0^2)}\,\sqrt{P_N\!\big((B+b)^2\big)}
\;\le\;\sqrt{2\,P_N(m_0^2)}\,\sqrt{P_N(B^2)+P_N(b^2)}.
\]
By summing up all the inequalities and using $2uv \leq u^2 + v^2$, we have
\begin{align}
\nonumber
  |R_N| &\leq 2\,P_N(B^2)+2\,P_N(b^2) + 2\sqrt{2\,P_N(m_0^2)}\,\sqrt{P_N(B^2)+P_N(b^2)}  \\
  \nonumber
  &\leq 2\,P_N(B^2)+2\,P_N(b^2) + 2P_N(m_0^2) + P_N(B^2)+P_N(b^2)\\
  \label{eq:R_N}
  &\leq 3\{\underbrace{P_N(B^2)}_{\textbf{(i)}} + \underbrace{P_N(b^2)}_{\textbf{(ii)}} \} + 2 \underbrace{P_N(m_0^2)}_{\textbf{(iii)}} 
\end{align}

Now, we show that the three terms on the right hand side of the inequality \eqref{eq:R_N} is $O_p(1)$.
\begin{itemize}
    \item[]\textbf{Term (i).}  Recall
\[
B_i \;=\; (s^\top m_0-m_0)(X_i)
      \;=\; \sum_{j\in S}s_j(X_i)\,\{m_0(X_j)-m_0(X_i)\}.
\]
Write \(w_{ij}:=s_j(X_i)\) and \(d_{ij}:=m_0(X_j)-m_0(X_i)\).  
By Cauchy–Schwarz for the inner product \(\sum_{j} w_{ij}d_{ij}\),
\begin{equation}\label{eq:Bi-CS}
B_i^2 \;\le\; \Big(\sum_{j\in S} w_{ij}^2\Big)\Big(\sum_{j\in S} d_{ij}^2\Big)
          \;=\; \|s(X_i)\|_2^2 \sum_{j\in S}\{m_0(X_j)-m_0(X_i)\}^2 .
\end{equation}
Using \((a-b)^2\le 2a^2+2b^2\),
\begin{equation}\label{eq:dij-bound}
\sum_{j\in S}\{m_0(X_j)-m_0(X_i)\}^2
\;\le\; 2\sum_{j\in S} m_0(X_j)^2 + 2n\,m_0(X_i)^2 .
\end{equation}
Combining \eqref{eq:Bi-CS}–\eqref{eq:dij-bound} and averaging over \(i\) gives
\begin{align}
P_N(B^2)
&:= \frac1N\sum_{i=1}^N B_i^2
 \;\le\; \frac1N\sum_{i=1}^N \|s(X_i)\|_2^2
        \Big(2\sum_{j\in S} m_0(X_j)^2 + 2n\,m_0(X_i)^2\Big) \notag\\
&= 2\,\overline{\|s(X)\|_2^2}\;\sum_{j\in S} m_0(X_j)^2
  \;+\; \frac{2n}{N}\sum_{i=1}^N \|s(X_i)\|_2^2\,m_0(X_i)^2,
\label{eq:PNB2-two-terms}
\end{align}
where \(\overline{\|s(X)\|_2^2}:=N^{-1}\sum_{i=1}^N\|s(X_i)\|_2^2\).

By Assumption \ref{ass:A2}, note that it holds \(\overline{\|s(X)\|_2^2}=O_p(1/n)\) since $\overline{\|s(X)\|_2^2}$ $ \leq \max_{1\le i\le N} $ $\|s(X_i)\|_2^2 $ $ =O_p(1/n)$. By \(\mathbb E[Y^2]<\infty\) we have (by the LLN)
\[
P_n\!\big(m_0(X)^2\big)=\frac1n\sum_{j\in S}m_0(X_j)^2=O_p(1),
\qquad
P_N\!\big(m_0(X)^2\big)=\frac1N\sum_{i=1}^N m_0(X_i)^2=O_p(1).
\]
Hence the first term on the right-hand side of \eqref{eq:PNB2-two-terms} is
\[
2\,\overline{\|s(X)\|_2^2}\;\sum_{j\in S} m_0(X_j)^2
\;=\; 2\,O_p(1/n)\cdot n\,P_n\!\big(m_0(X)^2\big)
\;=\; O_p(1).
\]

For the second term in the equation \eqref{eq:PNB2-two-terms}, it holds
\begin{align*}
\frac{2n}{N}\sum_{i=1}^N \|s(X_i)\|_2^2\,m_0(X_i)^2
&\;\le\; \frac{2n}{N}\,\Big(\max_{i}\|s(X_i)\|_2^2\Big)\sum_{i=1}^N m_0(X_i)^2\\
&\;=\; 2\,O_p(1)\cdot P_N\!\big(m_0(X)^2\big)
\;=\; O_p(1),    
\end{align*}
where \(\max_{1\le i\le N}\|s(X_i)\|_2^2=O_p(1/n)\) is from Assumption \ref{ass:A2}.

Therefore,
\[
P_N(B^2)=O_p(1).
\]
\item[]\textbf{Term (ii).}
Set
\(T_N:=P_N(b^2)=\frac1N\sum_{i=1}^N b(X_i)^2\ge 0\).
For any $\varepsilon>0$, by Markov’s inequality and the tower property,
\begin{align*}
\mathbb P(T_N>\varepsilon)
&\;\le\; \frac{1}{\varepsilon}\,\mathbb E[T_N]
\;=\; \frac{1}{\varepsilon}\,\mathbb E\!\left[\frac1N\sum_{i=1}^N b(X_i)^2\right]\\
&\;=\; \frac{1}{\varepsilon}\,\mathbb E\!\left[\int b(x)^2\,dP_X(x)\right]
\;=\; \frac{1}{\varepsilon}\,\mathbb E\!\big[\|b\|_{L_2(P_X)}^2\big].   
\end{align*}

Thus, if in addition to \ref{ass:A3} we assume the integrability
\(\mathbb E[\|b\|_{L_2(P_X)}^2]\to 0\) (which holds, e.g., whenever
$\|b\|_{L_2(P_X)}\xrightarrow{p}0$ and the sequence
$\{\|b\|_{L_2(P_X)}^2\}_N$ is uniformly integrable), then
\(\mathbb P(T_N>\varepsilon)\to 0\) for every $\varepsilon>0$ and hence
\[
P_N(b^2)\;=\;o_p(1).
\]
\item[]\textbf{Term (iii).} By $\mathbb E[Y^2]<\infty$ we have
$\mathbb E[m_0(X)^2]<\infty$, hence, by the LLN,
\[
P_N(m_0^2)=\frac1N\sum_{i=1}^N m_0(X_i)^2 \ \xrightarrow{p}\ \mathbb E[m_0(X)^2],
\]
so $P_N(m_0^2)=O_p(1)$.
\end{itemize}

Plugging these into \eqref{eq:R_N} gives \(R_N=O_p(1)\).
Consequently,
\begin{align}
\label{eq:R_N/N}
\frac{R_N}{N}=O_p\!\Big(\frac{1}{N}\Big)=o_p(1)    
\end{align}
which is the form needed in the subsequent steps.

\paragraph{Studentization (iii) — Step 3: Consistency of the plug–in variance for the $U$–part.}
Divide \eqref{eq:Step1} by $N$, and add and subtract $\hat{\sigma}^2$ in the term $\sigma^2\overline{\|s(X)\|_2^2}/N$ by any $\widehat\sigma^2\xrightarrow{p}\sigma^2$:
\begin{align*}
 &\frac{1}{N} \left\{ \mathbb E\!\big[\Var_N(\widehat m(X_1),\ldots,\widehat m(X_N))\mid\mathscr X_N\big] \right \}\\
&\quad=
\frac{1}{N}
\left\{
\Var_N(m_0(X_1),\ldots,m_0(X_N))
+\sigma^2\overline{\|s(X)\|_2^2}-\sigma^2\|a\|_2^2+R_N
\right\}\\
&\quad =
\frac{\Var_N(m_0(X_1),\ldots,m_0(X_N))}{N}
+
\frac{(\sigma^2 - \widehat{\sigma}^2 + \widehat{\sigma}^2)\overline{\|s(X)\|_2^2}}{N}-
\frac{\sigma^2\|a\|_2^2}{N}
+ \frac{R_N}{N} 
\end{align*}
Now, moving the term $ \widehat \sigma^2\overline{\|s(X)\|_2^2}/N$ from the right hand side to the left hand side provides
\begin{align}
\nonumber
&\frac1N\!\left\{\Var_N(\widehat m(X_1),\ldots,\widehat m(X_N))
-\widehat\sigma^2\,\overline{\|s(X)\|_2^2}\right\}\\
&\quad\quad\quad=\frac{\Var_N(m_0(X_1),\ldots,m_0(X_N))}{N} 
+\frac{(\sigma^2-\widehat\sigma^2)\,\overline{\|s(X)\|_2^2}}{N}
     -\frac{\sigma^2\|a\|_2^2}{N}
     +\frac{R_N}{N}. \label{eq:Step3}
\end{align}
By the LLN, $\Var_N(m_0(X_1),\ldots,m_0(X_N))/N\to \Var(m_0(X))/N$.
Since $\overline{\|s(X)\|_2^2}=O_p(1/n)$ by \ref{ass:A2} and $\widehat\sigma^2\to_p\sigma^2$,
the second term in \eqref{eq:Step3} is $o_p(1/N)$.
Moreover, Jensen gives $\|a\|_2^2=\big\|\tfrac1N\sum_i s(X_i)\big\|_2^2
\le \tfrac1N\sum_i\|s(X_i)\|_2^2=\overline{\|s(X)\|_2^2}=O_p(1/n)$, hence
$\sigma^2\|a\|_2^2/N=O_p(1/(nN))=o_p(1/N)$.
Together with \eqref{eq:R_N/N} (i.e., $R_N/N=o_p(1)$) we conclude that
\begin{equation}\label{eq:Step3-limit}
\frac1N\!\left\{\Var_N(\widehat m(X_1),\ldots,\widehat m(X_N))
-\widehat\sigma^2\,\overline{\|s(X)\|_2^2}\right\}
\ \xrightarrow{p}\ \frac{\Var(m_0(X))}{N}.
\end{equation}

\paragraph{Studentization (iii) — Step 4: Add the labeled–noise piece.}
From the analysis of Term (L) we already have
$\Var(\sqrt N\,L_n\mid\mathscr X_N)=\sigma^2 N\|c\|_2^2$ (see Equation \eqref{eq:v_N^2}) and, under \ref{ass:A2} with MP,
$N\|c\|_2^2\to 1/f$ (equivalently, $\|c\|_2^2=\tfrac1n+o_p(1/n)$).
Thus, with any $\widehat\sigma^2\xrightarrow{p}\sigma^2$,
\begin{equation}\label{eq:Step4}
\widehat\sigma^2\,\|c\|_2^2
=\frac{\widehat\sigma^2}{N}\,\big(N\|c\|_2^2\big)
\ \xrightarrow{p}\ \frac{\sigma^2}{f\,N}.
\end{equation}

\paragraph{Studentization (iii) — Step 5: Consistency and studentized CLT.}
Combining \eqref{eq:Step3-limit} and \eqref{eq:Step4} yields
\[
\widehat{\Var}\!\big(\widehat\theta^{\sf sf}_{\mathrm{PPI}}\big)
:=\frac{1}{N}\Big\{\Var_N(\widehat m(X_1),\ldots,\widehat m(X_N))
-\widehat\sigma^2\,\overline{\|s(X)\|_2^2}\Big\}
+\widehat\sigma^2\,\|c\|_2^2
\ \xrightarrow{p}\ \frac{\Var(m_0(X))}{N}+\frac{\sigma^2}{f\,N}.
\]
By part (ii) we already know
$\sqrt N(\widehat\theta^{\sf sf}_{\mathrm{PPI}}-\theta_0)\xrightarrow{d}
\mathcal N\!\big(0,\Var(m_0(X))+\sigma^2/f\big)$.
Since the denominator $\sqrt{\widehat{\Var}(\widehat\theta^{\sf sf}_{\mathrm{PPI}})}$
converges in probability to the same positive constant,
Slutsky’s lemma gives the studentized limit
\[
\frac{\widehat\theta^{\sf sf}_{\mathrm{PPI}}-\theta_0}
     {\sqrt{\widehat{\Var}(\widehat\theta^{\sf sf}_{\mathrm{PPI}})}}
\ \xrightarrow{d}\ \mathcal N(0,1).
\]
\end{proof}


\begin{proposition}[KRR with a bounded positive semi-definite (PSD) kernel satisfies MP, (A1)–(A3), and yields a consistent $\widehat\sigma^2$]
\label{prop:krr-gauss-A123}
Assume $Y=m_0(X)+\varepsilon$ with $\mathbb E[\varepsilon\mid X]=0$ and
$\Var(\varepsilon\mid X)=\sigma^2\in(0,\infty)$, and $\mathbb E[Y^2]<\infty$. Let $k:\mathcal X\times\mathcal X\to\mathbb R$ be a symmetric positive–semidefinite kernel with $\sup_{x\in\mathcal X} k(x,x)\ \le\ \kappa^2\ <\ \infty.$ For example, for Gaussian kernel. $k(x,x')=\exp\!\big(-\|x-x'\|_2^2/(2\ell^2)\big)$, so $k(x,x)=1$ for all $x$ and we may take $\kappa=1$. 

Let the labeled sample be $S=\{(X_j,Y_j)\}_{j=1}^n$ and write $K\in\mathbb R^{n\times n}$ for the Gram matrix on $X_S$, $K_{ij}=k(X_i,X_j)$, and $A:=K+n\lambda I_n$ with $\lambda=\lambda_n>0$.

Define
\[
w^\top:=\frac{\mathbf 1_n^\top A^{-1}}{\mathbf 1_n^\top A^{-1}\mathbf 1_n},
\qquad
H:=K A^{-1}(I_n-\mathbf 1_n w^\top)+\mathbf 1_n w^\top,
\]
and for $x\in\mathcal X$ set $k_x:=(k(X_1,x),\ldots,k(X_n,x))^\top$ and
\[
s(x)^\top:=k_x^\top A^{-1}(I_n-\mathbf 1_n w^\top)+w^\top,\qquad
\widehat m(x):=s(x)^\top Y_S .
\]
For unlabeled covariates $X_1,\dots,X_N$ (i.i.d.\ copies of $X$, independent of $S$) define the matrix $S_U\in\mathbb R^{N\times n}$
with $i$th row $s(X_i)^\top$:
\[
S_U:=K_U (K+n\lambda I_n)^{-1}(I_n-\mathbf 1_n w^\top)+\mathbf 1_N w^\top,
\qquad (K_U)_{ij}=k(X_i,X_j),
\]
and set
\[
a:=\frac{1}{N}S_U^\top\mathbf 1_N,\qquad
\bar h:=\frac{1}{n}H^\top\mathbf 1_n,\qquad
c:=a+\frac{1}{n}\mathbf 1_n-\bar h.
\]
Assume $N,n\to\infty$ with $n/N\to f\in(0,1)$.

Then the following hold:
\begin{itemize}[leftmargin=*, itemsep=2pt]

\item[] \textup{\textbf{Mass preservation (MP).}}
$H\mathbf 1_n=\mathbf 1_n$ and $\sum_{j=1}^n s_j(x)=1$ for all $x$.


\item[] \textup{\textbf{A1 Predictor accuracy.}}
If $m_0$ can be approximated arbitrarily well in $L_2(P_X)$ by functions generated by the kernel $k$
(for example, when $k$ is universal and $m_0$ is bounded and uniformly continuous on
$\mathcal X:=\mathrm{supp}(P_X)$), then $\|\widehat m - m_0\|_{L_2(P_X)} = o_p(1)$.

\item[] \textup{\textbf{A2 Smoother stability.}} (i) $\max_{1\le i\le N}\|s(X_i)\|_2^2=O_p (1/n)$, (ii) $\sqrt N\,\max_{1\le j\le n}|c_j|\ \to\ 0$ in probability, and (iii) $N\|c\|_2^2\ \to\ 1/f$.

\item[] \textup{\textbf{A3 Offset regularity.}} $b\equiv 0$, so $(P_N-P_n)b=0$ and $\|b\|_{L_2(P_X)}=0$.

\item[] \textup{\textbf{Plug–in variance estimator for SF–PPI with variance correction.}}

The degrees-of-freedom–adjusted residual variance
\[
\widehat\sigma^2_{\mathrm{KRR}}
:=\frac{\|(I_n-H)Y_S\|_2^2}{\,n-\mathrm{tr}(H)\,}
\]
satisfies $\widehat\sigma^2_{\mathrm{KRR}}\xrightarrow{p}\sigma^2$. With the above definitions,
\begin{align*}
\widehat{\Var}\!\big(\widehat\theta^{\sf sf}_{\mathrm{PPI}}\big)
&= \frac{1}{N}\Big\{\Var_N\big(\widehat m(X_1),\ldots,\widehat m(X_N)\big)
    - \widehat{\sigma}^2_{\mathrm{KRR}}\,\overline{\|s(X)\|_2^2}\Big\}
   + \widehat{\sigma}^2_{\mathrm{KRR}}\,\|c\|_2^2.
\end{align*}
\end{itemize}
\end{proposition}

\paragraph{Remark -- tuning of $\lambda$.}
The assumptions $\lambda=\lambda_n\downarrow0$ and $n\lambda_n\to\infty$ are needed for A1 and for the consistency of $\widehat\sigma^2_{\mathrm{KRR}}$ to control the regularization bias. Items in A2 remain valid uniformly for any $\lambda\ge\lambda_\star>0$. 
\paragraph{Remark -- notation (matrix norms).}
For a matrix $A\in\mathbb R^{p\times q}$,
\[
\|A\|_{\op}\ :=\ \sup_{\|x\|_2=1}\|Ax\|_2\ =\ \sqrt{\lambda_{\max}(A^\top A)},
\qquad
\|A\|_{F}\ :=\ \sqrt{\mathrm{tr}(A^\top A)}\ =\ \Big(\sum_{i=1}^p\sum_{j=1}^q A_{ij}^2\Big)^{1/2}.
\]
Here $\lambda_{\max}(B)$ denotes the largest eigenvalue of the (symmetric) matrix $B$.
\emph{Examples.} If $A=\mathrm{diag}(d_1,\dots,d_r)$, then $\|A\|_{\op}=\max_i |d_i|$ and
$\|A\|_{F}=(\sum_i d_i^2)^{1/2}$. If $A=uv^\top$ with $u\in\mathbb R^p$, $v\in\mathbb R^q$, then
$\|A\|_{\op}=\|u\|_2\|v\|_2$ and $\|A\|_{F}=\|u\|_2\|v\|_2$ (rank-one case).

\begin{proof}$\quad$
\paragraph{MP.} From the definitions,
$(I_n-\mathbf 1_n w^\top)\mathbf 1_n=\mathbf 1_n-(\mathbf 1_n w^\top)\mathbf 1_n=0$,
so $H\mathbf 1_n=\{ K A^{-1}(I_n-\mathbf 1_n w^\top)+\mathbf 1_n w^\top \}\mathbf 1_n = K A^{-1}0+\mathbf 1_n w^\top\mathbf 1_n=\mathbf 1_n$; similarly
$s(x)^\top\mathbf 1_n= \{k_x^\top A^{-1}(I_n-\mathbf 1_n w^\top)+w^\top \}\mathbf 1_n =k_x^\top A^{-1}(I_n\mathbf 1_n-\mathbf 1_n w^\top\mathbf 1_n) +w^\top\mathbf 1_n=k_x^\top A^{-1}0+w^\top\mathbf 1_n=1$.

\paragraph{A1.}
Let $k:\mathcal X\times\mathcal X\to\mathbb R$ be a symmetric positive–semidefinite kernel on $\mathbb{R}^d$ with $\sup_{x\in\mathcal X} k(x,x)\ $ $ \le\ \kappa^2\ <\ \infty$. Throughout the proof, without loss of generality, we assume that $\kappa = 1$ (which holds for many bounded PSD kernels including Gaussian and Mat\'ern kernels \citep{hastie2009elements,hofmann2008kernel}). 


Let $\mathcal H_{k}$ be its reproducing kernel Hilbert space (RKHS) \citep{wahba1990spline}. Define the kernel–ridge estimator
\[
\widehat m_{n,\lambda}
:=\arg\min_{m\in\mathcal H_{k}}
\left\{\frac1n\sum_{j\in S}(Y_j-m(X_j))^2+\lambda\|m\|_{\mathcal H_{k}}^2\right\},
\qquad \lambda=\lambda_n>0.
\]
Let $m_\lambda:=\arg\min_{m\in\mathcal H_{k}}
\{\mathbb E[(Y-m(X))^2]+\lambda\|m\|_{\mathcal H_{k}}^2\}$ be the population Tikhonov minimizer.
Since $k(x,x)\le \kappa^2(=1)$ for all $x$ and $\mathbb E[Y^2]<\infty$, standard RKHS learning bounds
\citep{caponnetto2007optimal,smale2007learning} yield, for some constant $C$,
\begin{equation}\label{eq:eq1_A1}
\mathbb E\!\left[\|\widehat m_{n,\lambda}-m_\lambda\|_{L_2(P_X)}^2\right]
\;\le\; \frac{C}{n\lambda}.
\end{equation}
Moreover, if
$m_0\in\overline{\mathcal H_{k}}^{\,L_2(P_X)}$ (e.g., this holds when $k$ is universal and $m_0$ is bounded and uniformly
continuous on $\mathcal{X} = \mathrm{supp}(P_X)$ \citep{micchelli2006universal,sriperumbudur2011universality}), then
\begin{equation}\label{eq:eq2_A1}
\|m_\lambda-m_0\|_{L_2(P_X)}\xrightarrow[\lambda\downarrow0]{}0 .
\end{equation}
By the triangle inequality,
$\|\widehat m_{n,\lambda}-m_0\|_{L_2(P_X)}
\le \|\widehat m_{n,\lambda}-m_\lambda\|_{L_2(P_X)}
   +\|m_\lambda-m_0\|_{L_2(P_X)}$.
Choose any $\lambda_n\downarrow0$ with $n\lambda_n\to\infty$. Then \eqref{eq:eq1_A1} and Markov’s
inequality give $\|\widehat m_{n,\lambda_n}-m_{\lambda_n}\|_{L_2(P_X)}=o_p(1)$, while
\eqref{eq:eq2_A1} gives $\|m_{\lambda_n}-m_0\|_{L_2(P_X)}\to0$. Therefore
\[
\|\widehat m_{n,\lambda_n}-m_0\|_{L_2(P_X)}\ \xrightarrow{p}\ 0,
\]
which is (A1).

\paragraph{A2--(i).} We first derive the deterministic operator-norm bounds. From $| k(x,x')|\le \kappa^2 (= 1)$ (by RKHS Cauchy–Schwarz) we have
\[
\|A^{-1}\|_{\op}\le \frac{1}{n\lambda},
\qquad
\|k_x\|_2 \le \sqrt n.
\]
(Indeed, $K\succeq0$ implies $\lambda_{\min}(A)\ge n\lambda$, while
$\|k_x\|_2^2 = \sum_{j=1}^n k(X_j,x)^2\le n$.)

For $w = \frac{A^{-1}\mathbf 1_n}{\mathbf 1_n^\top A^{-1}\mathbf 1_n},$ Rayleigh bounds give
\[
\|w\|_2^2 
= \frac{\mathbf 1_n^\top A^{-2}\mathbf 1_n}{(\mathbf 1_n^\top A^{-1}\mathbf 1_n)^2}
\;\le\; \frac{\lambda_{\max}(A)^2}{n\,\lambda_{\min}(A)^2}
\;\le\; \frac{(1+\lambda)^2}{n\lambda^2},
\]
hence $\|w\|_2 \le \tfrac{1+\lambda}{\sqrt n\,\lambda}$. Moreover,
\[
\|I_n - \mathbf 1_n w^\top\|_{\op}
\le \|I_n\|_{\op}+\|\mathbf 1_n w^\top\|_{\op}
= 1 + \|\mathbf 1_n\|_2\,\|w\|_2
\le 1+\frac{1+\lambda}{\lambda}.
\]

By the triangle inequality and $\|Mv\|_2\le \|M\|_{\op}\|v\|_2$,
\begin{align*}
\|s(x)\|_2
&=
\|(I_n - \mathbf 1_n w^\top) A^{-1} k_x+ w\|_2
\le\|I_n - \mathbf 1_n w^\top\|_{\op}\,\|A^{-1}\|_{\op}\,\|k_x\|_2 + \|w\|_2\\
&\leq \left(1+\frac{1+\lambda}{\lambda} \right) \frac{1}{n\lambda} \sqrt{n} + \frac{1+\lambda}{\sqrt{n} \lambda} = \frac{2+1/\lambda}{\sqrt{n}\,\lambda}
   + \frac{1+\lambda}{\sqrt{n}\,\lambda}\\
&= \frac{\,3+\lambda+1/\lambda\,}{\sqrt{n}\,\lambda}
 = \frac{\, 1 + 3/\lambda + 1/\lambda^{2}\,}{\sqrt{n}}
 = \frac{C_1(\lambda)}{\sqrt{n}},
\end{align*}
where $C_1(\lambda) =  1 + 3/\lambda  + 1/\lambda^{2}.$

Therefore, it holds $\|s(x)\|_2 \le C_1(\lambda)/\sqrt n$, thus, $\|s(x)\|_2^2\le C(\lambda)/n,$ uniformly in $x$ and for every realization of $X_S$ for constants $C_1(\lambda)$ and $C(\lambda) = C_1(\lambda)^2$. This immediately gives (i) $\max_{1\le i\le N}\|s(X_i)\|_2^2=O_p\!(C(\lambda)/n)$; if $\inf_n \lambda_n>0$ then $C(\lambda_n)\le C_\star<\infty$ and 
$\max_i\|s(X_i)\|_2^2 = O_p(1/n)$. (One may alternatively use the standard stability to yield the similar conclusion \citep{bousquet2002stability,cortes2010generalization}.)

\paragraph{A2--(iii).} By mass preservation, $S_U\mathbf 1_n=\mathbf 1_N$ and $H\mathbf 1_n=\mathbf 1_n$, hence
$\mathbf 1_n^\top a=\mathbf 1_n^\top \bar h=1$ (Prop.~\ref{prop:linear-smoother-notation}(f)).
Using $A^{-1}K=I_n-n\lambda A^{-1}$ (since $A=K+n\lambda I_n$), we have $H^\top=(I_n-w\mathbf 1_n^\top)A^{-1}K+w\mathbf 1_n^\top
      =I_n-n\lambda A^{-1}+n\lambda\,w\,\mathbf 1_n^\top A^{-1},$ and because $A^{-1}\mathbf 1_n=(\mathbf 1_n^\top A^{-1}\mathbf 1_n)\,w$, it follows that
$H^\top\mathbf 1_n=\mathbf 1_n$, i.e.
\begin{equation}\label{eq:hbar-uniform}
\bar h=\tfrac{1}{n}H^\top\mathbf 1_n=\tfrac{1}{n}\mathbf 1_n.
\end{equation}
Therefore
\begin{equation}\label{eq:c-a}
c=a+\tfrac{1}{n}\mathbf 1_n-\bar h=a,
\end{equation}
and by Prop.~\ref{prop:linear-smoother-notation}(f),
\begin{equation}\label{eq:c-split}
\|c\|_2^2=\frac{1}{n}+\|a-\bar h\|_2^2
=\frac{1}{n}+\Big\|a-\tfrac{1}{n}\mathbf 1_n\Big\|_2^2.
\end{equation}
Multiplying \eqref{eq:c-split} by $N$ gives
\begin{equation}\label{eq:Nc2-decomp}
N\|c\|_2^2=\frac{N}{n}+N\Big\|a-\tfrac{1}{n}\mathbf 1_n\Big\|_2^2.
\end{equation}

Recall $a=\frac{1}{N}\sum_{i=1}^N s(X_i)$ and $\sum_{j=1}^n s_j(x)=1$ by mass preservation. Fix the labeled sample $S$. Then the map $x\mapsto s(x)$ is deterministic, and
$X_1,\dots,X_N$ are i.i.d.\ draws from $P_X$, independent of $S$. By Jensen's inequality, since $v\mapsto \|v\|_2^2$ is convex,
\begin{align}
\Big\|a-\tfrac{1}{n}\mathbf 1_n\Big\|_2^2
=\left\|\frac{1}{N}\sum_{i=1}^N \big(s(X_i)-\tfrac{1}{n}\mathbf 1_n\big)\right\|_2^2
\ \le\ \frac{1}{N}\sum_{i=1}^N \big\|s(X_i)-\tfrac{1}{n}\mathbf 1_n\big\|_2^2 .
\label{eq:jensen}
\end{align}
Taking conditional expectation given $S$ and using i.i.d.\ of $X_i$,
\begin{align}
\mathbb E\!\left[\Big\|a-\tfrac{1}{n}\mathbf 1_n\Big\|_2^2 \,\middle|\, S\right]
&\le \frac{1}{N}\sum_{i=1}^N \mathbb E\!\left[\big\|s(X_i)-\tfrac{1}{n}\mathbf 1_n\big\|_2^2 \,\middle|\, S\right]= \frac{1}{N}\, \mathbb E\!\left[\big\|s(X)-\tfrac{1}{n}\mathbf 1_n\big\|_2^2 \,\middle|\, S\right].
\label{eq:cond-exp}
\end{align}

Note that for any vector $u$,
\(
\|u-\tfrac{1}{n}\mathbf 1_n\|_2^2
= \|u\|_2^2 - \tfrac{2}{n}\mathbf 1_n^\top u + \tfrac{1}{n}.
\)
Applying this with $u=s(X)$ and using MP ($\mathbf 1_n^\top s(X)=1$ deterministically),
\begin{align}
\big\|s(X)-\tfrac{1}{n}\mathbf 1_n\big\|_2^2
= \|s(X)\|_2^2 - \frac{2}{n}\cdot 1 + \frac{1}{n}
= \|s(X)\|_2^2 - \frac{1}{n}.
\label{eq:expand}
\end{align}
Plugging \eqref{eq:expand} into \eqref{eq:cond-exp} gives
\begin{align}
\mathbb E\!\left[\Big\|a-\tfrac{1}{n}\mathbf 1_n\Big\|_2^2 \,\middle|\, S\right]
\le \frac{1}{N}\left\{\mathbb E\!\left[\|s(X)\|_2^2 \,\middle|\, S\right]-\frac{1}{n}\right\}.
\label{eq:main-line}
\end{align}

From the derivation of A2-(i), we have $\mathbb E\!\left[\|s(X)\|_2^2 \,\middle|\, S\right] = O\!\left(1/n\right)$. Also, by Cauchy--Schwarz,
\(
(\mathbf 1_n^\top s(X))^2 \le \|\mathbf 1_n\|_2^2\,\|s(X)\|_2^2 = n\,\|s(X)\|_2^2,
\)
and since $\mathbf 1_n^\top s(X)=1$, we have the \emph{deterministic} lower bound
\[
\|s(X)\|_2^2 \ \ge\ \frac{1}{n}\qquad\text{(hence the bracket in \eqref{eq:main-line} is nonnegative).}
\]
Multiplying \eqref{eq:main-line} by $N$,
\begin{align}
\mathbb E\!\left[ N\Big\|a-\tfrac{1}{n}\mathbf 1_n\Big\|_2^2 \,\middle|\, S\right]
\le \mathbb E\!\left[\|s(X)\|_2^2 \,\middle|\, S\right]-\frac{1}{n}
= O\!\left(\frac{1}{n}\right).
\label{eq:ENfluct}
\end{align}

For any $\varepsilon>0$, Markov’s inequality applied to \eqref{eq:ENfluct} yields
\[
\mathbb P\!\left( N\Big\|a-\tfrac{1}{n}\mathbf 1_n\Big\|_2^2 > \varepsilon \,\middle|\, S\right)
\ \le\ \frac{1}{\varepsilon}\,
\mathbb E\!\left[ N\Big\|a-\tfrac{1}{n}\mathbf 1_n\Big\|_2^2 \,\middle|\, S\right]
\ =\ O\!\left(\frac{1}{n}\right)\xrightarrow[n\to\infty]{}0.
\]
Hence
\[
N\Big\|a-\tfrac{1}{n}\mathbf 1_n\Big\|_2^2 \ \xrightarrow{p}\ 0.
\]

From \eqref{eq:Nc2-decomp} and $n/N\to f\in(0,1)$,
\[
N\|c\|_2^2 \;=\; \frac{N}{n} + o_p(1) \;\longrightarrow\; \frac{1}{f}.
\qedhere
\]

\paragraph{A2--(ii).} Recall that we have $H^\top\mathbf 1_n=\mathbf 1_n$, hence
$\bar h=\tfrac{1}{n}\mathbf 1_n$ and therefore $c = a+\tfrac{1}{n}\mathbf 1_n-\bar h = a$. 
Write
\[
\delta \;:=\; a-\tfrac{1}{n}\mathbf 1_n \in\mathbb R^n,
\qquad\text{so that}\qquad
c=\delta+\tfrac{1}{n}\mathbf 1_n.
\]
Then, for any $\varepsilon>0$,
\begin{equation}\label{eq:triangle}
\sqrt N\,\max_{1\le j\le n}|c_j|
\;\le\; \sqrt N\,\max_{j}|\delta_j| \;+\; \frac{\sqrt N}{n}
\;\le\; \sqrt N\,\|\delta\|_2 \;+\; \frac{\sqrt N}{n},
\end{equation}
where we used $\|\delta\|_\infty\le \|\delta\|_2$.

Thus it suffices to prove $\sqrt N\,\|\delta\|_2 \xrightarrow{p}0$ and
$\sqrt N/n\to 0$. The latter is trivial since $n/N\to f\in(0,1)$ implies $\frac{\sqrt N}{n} \;=\; \frac{1}{\sqrt{f\,n}}\;\longrightarrow\; 0$. 

Recall $a=\tfrac{1}{N}\sum_{i=1}^N s(X_i)$ with $X_1,\dots,X_N$ i.i.d.\ (independent of $S$),
and $\sum_{j=1}^n s_j(x)=1$ (MP). Conditioning on $S$, Jensen’s inequality for the convex map $v\mapsto\|v\|_2^2$ gives
\begin{align}
\mathbb E\!\left[\|\delta\|_2^2 \,\middle|\, S\right]
&=\mathbb E\!\left[\Big\|a-\tfrac{1}{n}\mathbf 1_n\Big\|_2^2 \,\middle|\, S\right]
\le \frac{1}{N}\,\mathbb E\!\left[\big\|s(X)-\tfrac{1}{n}\mathbf 1_n\big\|_2^2 \,\middle|\, S\right]
= \frac{1}{N}\left\{\mathbb E\!\left[\|s(X)\|_2^2 \,\middle|\, S\right]-\frac{1}{n}\right\}.
\label{eq:delta-second-moment}
\end{align}
(The last equality uses
$\|s-\tfrac{1}{n}\mathbf 1_n\|_2^2=\|s\|_2^2-\tfrac{2}{n}\mathbf 1_n^\top s+\tfrac{1}{n}
=\|s\|_2^2-\tfrac{1}{n}$ and MP: $\mathbf 1_n^\top s(X)=1$.)

By A2-(i) we already established $\mathbb E\!\left[\|s(X)\|_2^2 \,\middle|\, S\right]=O\!(1/n)$, 
and by Cauchy--Schwarz with mass preservation one has the (deterministic) lower bound
$\|s(X)\|_2^2\ge \tfrac{1}{n}$, so the bracket in
\eqref{eq:delta-second-moment} is nonnegative. Hence
\begin{equation}\label{eq:N-delta-mean}
\mathbb E\!\left[\,N\|\delta\|_2^2 \,\middle|\, S\right]
\;\le\; \mathbb E\!\left[\|s(X)\|_2^2 \,\middle|\, S\right]-\frac{1}{n}
\;=\; O\!\left(\frac{1}{n}\right).
\end{equation}

For any $\varepsilon>0$, by Markov’s inequality applied conditionally on $S$,
\[
\mathbb P\!\left(\sqrt N\,\|\delta\|_2>\varepsilon \,\middle|\, S\right)
\;=\; \mathbb P \!\left(N\|\delta\|_2^2>\varepsilon^2 \,\middle|\, S\right)
\;\le\; \frac{1}{\varepsilon^2}\,
\mathbb E\!\left[N\|\delta\|_2^2 \,\middle|\, S\right]
\;=\; O\!\left(\frac{1}{n}\right).
\]
Therefore $\sqrt N\,\|\delta\|_2 \xrightarrow{p} 0$.

Combine the two limits with \eqref{eq:triangle}:
\[
\sqrt N\,\max_{1\le j\le n}|c_j|
\;\le\; \sqrt N\,\|\delta\|_2 + \frac{\sqrt N}{n}
\;\xrightarrow{p}\; 0 + 0 \;=\; 0.
\qedhere
\]
\paragraph{A3.}
By construction (unpenalized intercept), $\widehat m$ is linear in $Y_S$ with weights summing to one,
so $b\equiv 0$.

\paragraph{Variance correction -- plug–in variance estimator for SF–PPI.} Recall the notations $Y_S=m_S+\varepsilon$, where $m_S=(m_0(X_1),\ldots,m_0(X_n))^\top$
and $\varepsilon=(\varepsilon_1,\ldots,\varepsilon_n)^\top$ with
$\mathbb E[\varepsilon\mid X_S]=0$ and $\Var(\varepsilon\mid X_S)=\sigma^2 I_n$.
Let $R:=(I_n-H)Y_S$. We can decompose $\|R\|_2^2$ into three terms
\begin{align}
\|R\|_2^2
&= \big\|(I_n-H)Y_S\big\|_2^2
 = \big\|(I_n-H)(m_S+\varepsilon)\big\|_2^2 \nonumber = \big\|(I_n-H)m_S + (I_n-H)\varepsilon\big\|_2^2  \nonumber\\
&= \|(I_n-H)m_S\|_2^2
   + 2\,\big\langle (I_n-H)m_S,(I_n-H)\varepsilon\big\rangle
   + \|(I_n-H)\varepsilon\|_2^2  \nonumber\\
&= \|(I_n-H)m_S\|_2^2
   + 2\,m_S^\top (I_n-H)^\top (I_n-H)\varepsilon
   + \varepsilon^\top (I_n-H)^\top (I_n-H)\varepsilon  \nonumber\\
&= \underbrace{\|(I_n-H)m_S\|_2^2}_{\text{Unscaled bias}} 
   + \underbrace{2\,\varepsilon^\top (I_n-H)^2 m_S}_{\text{Unscaled cross}} 
   + \underbrace{\varepsilon^\top (I_n-H)^2 \varepsilon}_{\text{Unscaled noise}}.
\label{eq:res-expansion}
\end{align}

Recall $A=K+n\lambda I_n$ and 
\[
H \;=\; K A^{-1}(I_n-\mathbf 1_n w^\top)+\mathbf 1_n w^\top,
\qquad
w=\frac{A^{-1}\mathbf 1_n}{\mathbf 1_n^\top A^{-1}\mathbf 1_n}.
\]
Using $\mathrm{tr}(UV)=\mathrm{tr}(VU)$,
\begin{align}
\label{eq:tag1_vc}
\mathrm{tr}(H)
= \mathrm{tr}\!\big(KA^{-1}(I_n-\mathbf 1_n w^\top)\big) + \mathrm{tr}(\mathbf 1_n w^\top)
= \mathrm{tr}(KA^{-1}) - w^\top K A^{-1}\mathbf 1_n + 1.    
\end{align}

Since $K\succeq 0$ and $A=K+n\lambda I_n$, the eigenvalues of $KA^{-1}$ are
$\tfrac{\lambda_i(K)}{\lambda_i(K)+n\lambda}\in[0,1]$. Hence
\begin{align}
\label{eq:tag_2_vc}
\|KA^{-1}\|_{\op}\le 1,
\qquad
\mathrm{tr}(KA^{-1})=\sum_{i=1}^n \frac{\lambda_i(K)}{\lambda_i(K)+n\lambda}
\le \frac{1}{n\lambda}\sum_{i=1}^n \lambda_i(K)
= \frac{\mathrm{tr}(K)}{n\lambda}\le \frac{1}{\lambda},
\end{align}
because $0\le k(x,x)\le 1$ implies $\mathrm{tr}(K)\le n$.

At this point, we derive a tight upper bound on $\|w\|_2^2$. Recall that we have
\[
\|w\|_2^2
= \frac{\mathbf 1_n^\top A^{-2}\mathbf 1_n}{\big(\mathbf 1_n^\top A^{-1}\mathbf 1_n\big)^2}.
\]
Let $z:=A^{-1/2}\mathbf 1_n$. Then
\[
\mathbf 1_n^\top A^{-2}\mathbf 1_n = z^\top A^{-1} z,
\qquad
\mathbf 1_n^\top A^{-1}\mathbf 1_n = z^\top z,
\]
so
\begin{equation}
\label{eq:rayleigh}
\|w\|_2^2
= \frac{z^\top A^{-1} z}{(z^\top z)^2}
\le \frac{\lambda_{\max}(A^{-1})\,z^\top z}{(z^\top z)^2}
= \frac{\lambda_{\max}(A^{-1})}{\mathbf 1_n^\top A^{-1}\mathbf 1_n}.
\end{equation}

By Rayleigh’s inequality,
\begin{equation}
\label{eq:denom-lb}
\mathbf 1_n^\top A^{-1}\mathbf 1_n
\;\ge\; \frac{\|\mathbf 1_n\|_2^2}{\lambda_{\max}(A)}
\;=\; \frac{n}{\lambda_{\max}(A)}.
\end{equation}

Using $\lambda_{\max}(A^{-1})=1/\lambda_{\min}(A)$ and \eqref{eq:rayleigh}–\eqref{eq:denom-lb},
\begin{equation}
\label{eq:master}
\|w\|_2^2
\le \frac{\lambda_{\max}(A)}{n\,\lambda_{\min}(A)}.
\end{equation}

Since $K\succeq 0$ and $k(x_i,x_i)\le 1$, $\lambda_{\max}(K)\le \operatorname{tr}(K)=\sum_{i=1}^n k(x_i,x_i)\le n.$ Hence
\[
\lambda_{\max}(A)\le \lambda_{\max}(K)+n\lambda \le n(1+\lambda),
\qquad
\lambda_{\min}(A)\ge n\lambda.
\]
Plugging these into \eqref{eq:master} yields a bound of $\|w\|_2^2$
\begin{align}
\label{eq:tight_upper_bound_w^2}
\|w\|_2^2
\le \frac{n(1+\lambda)}{n\cdot n\lambda}
= \frac{1+\lambda}{n\lambda}.    
\end{align}

For the mixed term, by Cauchy–Schwarz and $\|\mathbf 1_n\|_2=\sqrt n$,
\begin{align}
\label{eq:tag3_vc_mixterm}
\big|\,w^\top K A^{-1}\mathbf 1_n\,\big|
\le \|KA^{-1}\|_{\op}\,\|w\|_2\,\|\mathbf 1_n\|_2
\le \sqrt{\frac{1+\lambda}{\lambda}},    
\end{align}
where we used the bound $\|w\|_2^2$ \eqref{eq:tight_upper_bound_w^2}. Combining \eqref{eq:tag1_vc}, 
\eqref{eq:tag_2_vc}, and \eqref{eq:tag3_vc_mixterm}
\begin{align}
\label{eq:trace_H}
\mathrm{tr}(H)\ \le\ \frac{1}{\lambda} + 1 + \sqrt{\frac{1+\lambda}{\lambda}}.
\ 
\end{align}
Thus, we have $\frac{\mathrm{tr}(H)}{n}
\leq \frac{1}{ n  \lambda } + \frac{1}{ n  }+ \frac{1}{n} \sqrt{\frac{1+\lambda}{\lambda}}$, which implies $\mathrm{tr}(H) = o(n)$ since $n \lambda \xrightarrow{} \infty$.

Therefore,
\[
n-\mathrm{tr}(H)\ \ge\ n - O\!\Big(\frac{1}{\lambda}\Big)
\ \xrightarrow[n\to\infty]{}\ \infty
\qquad\text{whenever } n\lambda\to\infty.
\]

Set $B:=(I_n-H)^2$. Since $H$ is symmetric, $I_n-H$ is symmetric, hence $B\succeq0$ and
\[
\varepsilon^\top B\varepsilon=\|(I_n-H)\varepsilon\|_2^2\ge0.
\]
Conditional on $X_S$ (so $H$ and $B$ are deterministic), with $\varepsilon=(\varepsilon_1,\ldots,\varepsilon_n)^\top$
i.i.d.\ mean $0$, variance $\sigma^2$, and finite fourth moment, the standard quadratic–form
identities give
\begin{align}
\label{eq:QF-moments_E}
\mathbb E\!\left[\varepsilon^\top B\varepsilon \,\middle|\, X_S\right]
&= \mathbb E\!\left[\operatorname{tr}\!\big(B\,\varepsilon\varepsilon^\top\big)\,\middle|\, X_S\right]
 = \operatorname{tr}\!\Big(B\,\mathbb E[\varepsilon\varepsilon^\top\mid X_S]\Big)
 = \operatorname{tr}(B\,\sigma^2 I_n)
 = \sigma^2\,\operatorname{tr}(B),\\
 \label{eq:QF-moments_V}
\Var\!\left(\varepsilon^\top B\varepsilon \,\middle|\, X_S\right)
&= 2\sigma^4\,\operatorname{tr}(B^2)\;+\;\kappa_4\,\operatorname{tr}(B\circ B)
 \;\le\; \big(2\sigma^4+|\kappa_4|\big)\,\operatorname{tr}(B^2)
 \;=\; C\,\operatorname{tr}(B^2),
\end{align}
for a constant $C$ depending only on $\mathbb E[\varepsilon_1^4]$. Here, $B\circ B$ denotes the Hadamard (entrywise) square of B, and we used $\operatorname{tr}(B\circ B)
= \sum_{i=1}^n B_{ii}^2
\le \sum_{i=1}^n\sum_{j=1}^n B_{ij}^2
= \|B\|_F^2
= \operatorname{tr}(B^\top B)
= \operatorname{tr}(B^2)$ since $B$ is symmetric.

Next we bound $\mathrm{tr}(B)$ and $\mathrm{tr}(B^2)$ in terms of $n$.
Using $A=K+n\lambda I_n$ and $KA^{-1}=I_n-n\lambda A^{-1}$, one can write
\begin{align*}
I_n-H
&= I_n - \big(KA^{-1}(I_n-\mathbf 1_n w^\top)+\mathbf 1_n w^\top\big)
= \big(I_n - KA^{-1}\big) + \big(KA^{-1}\mathbf 1_n w^\top - \mathbf 1_n w^\top\big)\\
&= \big(I_n - KA^{-1}\big) + \big(KA^{-1}-I_n\big)\mathbf 1_n w^\top = \big(I_n - KA^{-1}\big)\big(I_n-\mathbf 1_n w^\top\big)\\
&= n\lambda\,A^{-1}(I_n-\mathbf 1_n w^\top).    
\end{align*}
Hence
\[
B=(I_n-H)^2=(n\lambda)^2\,(I_n-\mathbf 1_n w^\top)^\top A^{-2}(I_n-\mathbf 1_n w^\top)\succeq0.
\]
From $\|A^{-1}\|_{\op}\le 1/(n\lambda)$ and 
$\|MN\|_{\op} \;\le\; \|M\|_{\op}\,\|N\|_{\op}$ (i.e., submultiplicativity of spectral norm), we obtain
\begin{align*}
\|I_n-H\|_{\op}\;&\le\; n\lambda\,\|A^{-1}\|_{\op}\,\|I_n-\mathbf 1_n w^\top\|_{\op}
\;\le\; 1+\|\mathbf 1_n w^\top\|_{\op}\\
&\;=\; 1+\sqrt n\,\|w\|_2
\leq
1 + \sqrt{\frac{1 + \lambda}{\lambda}} =:C(\lambda),    
\end{align*}
using $\sqrt n\|w\|_2\le \sqrt{(1+\lambda)/\lambda} $ proved earlier. 

Thus, with $B=(I_n-H)^2$,
\begin{align}
\mathrm{tr}(B^2)
&= \mathrm{tr}\!\big((I_n-H)^4\big)
 = \mathrm{tr}\!\big(((I_n-H)^2)^2\big) \notag\\
&\le \big\|(I_n-H)^2\big\|_{\op}\;\mathrm{tr}\!\big((I_n-H)^2\big)
\qquad\text{(since $M\succeq0\Rightarrow \mathrm{tr}(M^2)\le \|M\|_{\op}\mathrm{tr}(M)$)} \notag\\
&\leq \|I_n-H\|_{\op}^2\;\mathrm{tr}(B)
 \;\le\; C(\lambda)^2\,\mathrm{tr}(B),
 \label{eq:traceB2-bound}
\end{align}

Moreover,
\begin{equation}\label{eq:traceB-identity}
\mathrm{tr}(B)=\mathrm{tr}\!\big((I_n-H)^2\big)=n-2\,\mathrm{tr}(H)+\mathrm{tr}(H^2).
\end{equation}

We already have $\mathrm{tr}(H)=O(1/\lambda)=o(n)$ since $n\lambda\to\infty$ \eqref{eq:trace_H}. A convenient bound for $\mathrm{tr}(H^2)$ follows from the decomposition
\[
H \;=\; KA^{-1} + \big(I_n-KA^{-1}\big)\,\mathbf 1_n w^\top.
\]
Write
\[
Q:=KA^{-1},\qquad T:=(I_n-Q)\,\mathbf 1_n w^\top\quad(\text{rank one matrix}),
\]
so that $H=Q+T$. Then
\[
\mathrm{tr}(H^2)=\mathrm{tr}(Q^2)+2\,\mathrm{tr}(QT)+\mathrm{tr}(T^2).
\]

Since $A=K+n\lambda I_n$ is a polynomial in $K$, $K$ and $A$ commute and share eigenvectors. If $K v_i=\lambda_i(K)v_i$ with $\lambda_i(K)\ge0$, then
\[
Q v_i
= K A^{-1} v_i
= \frac{\lambda_i(K)}{\lambda_i(K)+n\lambda}\,v_i
=: \mu_i v_i,\qquad \mu_i\in[0,1].
\]
Hence
\[
\mathrm{tr}(Q^2)=\sum_i \mu_i^2 \;\le\; \sum_i \mu_i=\mathrm{tr}(Q).
\]
Moreover,
\[
\mathrm{tr}(Q)=\sum_i \mu_i
\le \sum_i \frac{\lambda_i(K)}{n\lambda}
= \frac{\mathrm{tr}(K)}{n\lambda}
\le \frac{n}{n\lambda}=\frac{1}{\lambda},
\]
since $k(X_i,X_i)\le 1$ implies $\mathrm{tr}(K)=\sum_i K_{ii}\le n$. Note that, if $n\lambda\to\infty$ then $\frac{1}{\lambda}=o(n)$ because
$\frac{(1/\lambda)}{n}=\frac{1}{n\lambda}\to0$. Hence, it holds 
\begin{align}
\label{eq:trace_S2}
  \mathrm{tr}(Q^2)\le \mathrm{tr}(Q)=\mathrm{tr}(KA^{-1})\le 1/\lambda=o(n).  
\end{align}

Moreover,
\[
\mathrm{tr}(QT)=\mathrm{tr}\!\big(Q(I_n-Q)\,\mathbf 1_n w^\top\big)=w^\top Q(I_n-Q)\,\mathbf 1_n,
\]
so
\begin{align}
\label{eq:tract_ST}
|\mathrm{tr}(QT)|
  &\le \|Q(I_n-Q)\|_{\op}\,\|w\|_2\,\|\mathbf 1_n\|_2
   \le \Big(\max_{x\in[0,1]}x(1-x)\Big)\cdot \sqrt{\frac{1+\lambda}{n\lambda}}\cdot\sqrt n \\
   \nonumber
  &= \frac{1}{4}\,\sqrt{\frac{1+\lambda}{\lambda}}
   = O\!\Big(\frac{1}{\sqrt{\lambda}}\Big)
   = o(n).
\end{align}
where the second inequality uses that $Q=KA^{-1}$ is symmetric with spectrum in $[0,1]$ (hence
$\|Q(I_n-Q)\|_{\op}\le \max_{x\in[0,1]} x(1-x)=\tfrac14$), together with
$\|w\|_2 \le \sqrt{\frac{1+\lambda}{n\lambda}}$ and $\|\mathbf 1_n\|_2=\sqrt n$.

Finally, for $T=(I_n-Q)\mathbf 1_n w^\top$ we have
\[
\mathrm{tr}(T^2)=\big(w^\top (I_n-Q)\mathbf 1_n\big)^2.
\]

Using $I_n-Q=I_n-KA^{-1}=n\lambda\,A^{-1}$ and 
$w=\dfrac{A^{-1}\mathbf 1_n}{\mathbf 1_n^\top A^{-1}\mathbf 1_n}$,
\begin{align*}
w^\top (I_n-Q)\mathbf 1_n
&= \left(\frac{A^{-1}\mathbf 1_n}{\mathbf 1_n^\top A^{-1}\mathbf 1_n}\right)^\top
   \big(n\lambda\,A^{-1}\mathbf 1_n\big) = n\lambda\,\frac{(A^{-1}\mathbf 1_n)^\top A^{-1}\mathbf 1_n}{\mathbf 1_n^\top A^{-1}\mathbf 1_n} \\
&= n\lambda\,\frac{\mathbf 1_n^\top A^{-2}\mathbf 1_n}{\mathbf 1_n^\top A^{-1}\mathbf 1_n}
\;\le\; n\lambda\,\lambda_{\max}(A^{-1})
\;\le\; n\lambda\,\frac{1}{\lambda_{\min}(A)}
\;\le\; n\lambda\,\frac{1}{n\lambda}
\;=\;1,
\end{align*}
where the first inequality is a Rayleigh–quotient bound,
\[
\frac{\mathbf 1_n^\top A^{-2}\mathbf 1_n}{\mathbf 1_n^\top A^{-1}\mathbf 1_n}
=\frac{z^\top A^{-1} z}{z^\top z}\le \lambda_{\max}(A^{-1})\quad
(z:=A^{-1/2}\mathbf 1_n),
\]
the second uses eigenvalue reciprocity for SPD matrices,
\(\lambda_{\max}(A^{-1})=1/\lambda_{\min}(A)\), and the third uses
\(\lambda_{\min}(A)\ge n\lambda\) since \(A=K+n\lambda I_n\) with \(K\succeq 0\).

Therefore, 
\begin{align}
\label{eq:tract_T2}
 \mathrm{tr}(T^2)\le 1=O(1)=o(n).
\end{align}

Therefore, by summing \eqref{eq:trace_S2}, \eqref{eq:tract_ST}, and \eqref{eq:tract_T2},
\[
0\le \mathrm{tr}(H^2)
= \mathrm{tr}(Q^2)+2\,\mathrm{tr}(QT)+\mathrm{tr}(T^2)
\le \mathrm{tr}(Q^2)+2\,|\mathrm{tr}(QT)|+\mathrm{tr}(T^2)
\le \frac{1}{\lambda}+
\frac{1}{2}\sqrt{\frac{1 + \lambda}{\lambda }}
+1.
\]
Dividing by $n$ and using $n\lambda\to\infty$,
\[
\frac{\mathrm{tr}(H^2)}{n}
\le \frac{1}{\lambda n}
+
\frac{1}{2n}\sqrt{\frac{1 + \lambda}{\lambda }}
+\frac{1}{n}
\;\xrightarrow[n\to\infty]{}\;0,
\]
hence $\mathrm{tr}(H^2)=o(n)$.

Recall that $B=(I_n-H)^2=I_n-2H+H^2,$ so taking traces gives $\mathrm{tr}(B)=\mathrm{tr}(I_n)-2\,\mathrm{tr}(H)+\mathrm{tr}(H^2)
= n-2\,\mathrm{tr}(H)+\mathrm{tr}(H^2).$ From the bounds already established, $\mathrm{tr}(H)=o(n)$ and $\mathrm{tr}(H^2)=o(n)$, hence
\[
\mathrm{tr}(B)=n+o(n).
\]
Consequently,
\begin{align}
\label{eq:trB_over_nminustrH}
\frac{\mathrm{tr}(B)}{\,n-\mathrm{tr}(H)\,}
=\frac{n-2\,\mathrm{tr}(H)+\mathrm{tr}(H^2)}{\,n-\mathrm{tr}(H)\,}
=1-\frac{\mathrm{tr}(H)-\mathrm{tr}(H^2)}{\,n-\mathrm{tr}(H)\,}
\;\longrightarrow\;1,
\end{align}
because the second term vanishes due to $\mathrm{tr}(H)-\mathrm{tr}(H^2)=o(n)$ and $n-\mathrm{tr}(H)\sim n$ (i.e., $(n-\mathrm{tr}(H))/n\to 1$).

Our eventual goal is to prove the df–adjusted residual variance is consistent:
\[
\widehat\sigma^2_{\mathrm{KRR}}
:=\frac{\|(I_n-H)Y_S\|_2^2}{\,n-\mathrm{tr}(H)\,}
\ \xrightarrow{p}\ \sigma^2 .
\]
Using \eqref{eq:res-expansion} and writing $B:=(I_n-H)^2$, we divide by $n-\mathrm{tr}(H)$ to obtain
\begin{align}
\widehat\sigma^2_{\mathrm{KRR}}
&=\underbrace{\frac{\|(I_n-H)m_S\|_2^2}{\,n-\mathrm{tr}(H)\,}}_{\text{Bias}} 
  \;+\; \underbrace{\frac{2\,\varepsilon^\top B\,m_S}{\,n-\mathrm{tr}(H)\,}}_{\text{Cross}} 
  \;+\; \underbrace{\frac{\varepsilon^\top B\,\varepsilon}{\,n-\mathrm{tr}(H)\,}}_{\text{Noise}}.
\label{eq:three-terms}
\end{align}
Thus it suffices to show the following three statements:
\begin{align}
&\frac{\|(I_n-H)m_S\|_2^2}{\,n-\mathrm{tr}(H)\,} \xrightarrow{p} 0,
\qquad\text{(Bias vanishes)}
\label{eq:bias-vanish}\\[2mm]
&\frac{\varepsilon^\top B\,m_S}{\,n-\mathrm{tr}(H)\,} \xrightarrow{p} 0,
\qquad\text{(Cross vanishes)}
\label{eq:cross-vanish}\\[2mm]
&\frac{\varepsilon^\top B\,\varepsilon}{\,n-\mathrm{tr}(H)\,} \xrightarrow{p} \sigma^2.
\qquad\text{(Noise converges)}
\label{eq:noise-converge}
\end{align}

We established the following facts, which yield the conclusion:
\begin{itemize}[leftmargin=*,itemsep=2pt]
\item[--]$\mathrm{tr}(B)=\mathrm{tr}\!\big((I_n-H)^2\big)=n-2\,\mathrm{tr}(H)+\mathrm{tr}(H^2)=n+o(n)$ and
$\dfrac{\mathrm{tr}(B)}{\,n-\mathrm{tr}(H)\,}\to 1$;
\item[--]$\mathrm{tr}(B^2)\le C(\lambda)^2\,\mathrm{tr}(B)$ for a finite constant $C(\lambda)$, and
$\|I_n-H\|_{\op}\le C(\lambda)$;
\item[--]conditional quadratic–form moments:
$\mathbb E[\varepsilon^\top B\varepsilon\mid X_S]=\sigma^2\,\mathrm{tr}(B)$ and
$\Var(\varepsilon^\top B\varepsilon\mid X_S)\le C\,\mathrm{tr}(B^2)$ for some constant $C$;
\item[--]$n-\mathrm{tr}(H)\sim n$ as $n\to\infty$ (here $a\sim b$ denotes asymptotic equivalence, i.e., $a/b\to 1$; equivalently $\operatorname{tr}(H)=o(n)$).
\end{itemize}

\paragraph{Variance correction -- bias term.} Let $\tilde m(x):=s(x)^\top m_S$ be the KRR predictor trained on noiseless labels; then
$\|(I_n-H)m_S\|_2^2 = \sum_{j=1}^n(\tilde m(X_j)-m_0(X_j))^2$.
By A1 and the uniform stability of KRR with square loss, $\frac{1}{n}\,\|(I_n-H)m_S\|_2^2 \xrightarrow{p} 0.$ Since $n-\mathrm{tr}(H)\sim n$, it follows that
\[
\frac{\|(I_n-H)m_S\|_2^2}{\,n-\mathrm{tr}(H)\,}
\xrightarrow{p}  0,
\]
so \eqref{eq:bias-vanish} holds.

\paragraph{Variance correction -- noise term.} Using the moment bounds in \eqref{eq:QF-moments_E}–\eqref{eq:QF-moments_V}, define
\[
Z_n\;:=\;\frac{\varepsilon^\top B\varepsilon-\sigma^2 \,\mathrm{tr}(B)}{\,n-\mathrm{tr}(H)\,}.
\]
Conditioning on $X_S$, we have $\mathbb E[Z_n\mid X_S]=0$ and
\[
\Var(Z_n\mid X_S)
=\frac{\Var(\varepsilon^\top B\varepsilon\mid X_S)}{(n-\mathrm{tr}(H))^2}
\;\le\;\frac{C\,\mathrm{tr}(B^2)}{(n-\mathrm{tr}(H))^2}
\;\le\;\frac{C\,C(\lambda)^2\,\mathrm{tr}(B)}{(n-\mathrm{tr}(H))^2},
\]
where we used $\mathrm{tr}(B^2)\le C(\lambda)^2\,\mathrm{tr}(B)$ \eqref{eq:traceB2-bound}. Since $\mathrm{tr}(B)=n+o(n)$ and
$n-\mathrm{tr}(H)\sim n$, the right-hand side is $O(1/n)\to 0$. By Chebyshev’s inequality, for any $\epsilon>0$,
\[
\mathbb P\!\left(\left|Z_n\right|>\epsilon\,\middle|\,X_S\right)
\;\le\;\frac{\Var(Z_n\mid X_S)}{\epsilon^2}\;\longrightarrow\;0,
\]
hence
\[
\frac{\varepsilon^\top B\varepsilon-\sigma^2\,\mathrm{tr}(B)}{\,n-\mathrm{tr}(H)\,}\xrightarrow{p}0.
\]
Moreover,
\[
\frac{\sigma^2\,\mathrm{tr}(B)}{\,n-\mathrm{tr}(H)\,}
=\sigma^2\,\frac{\mathrm{tr}(B)}{\,n-\mathrm{tr}(H)\,}\ \longrightarrow\ \sigma^2,
\]
because $\mathrm{tr}(B)/(n-\mathrm{tr}(H))\to 1$. Combining the two displays yields
\[
\frac{\varepsilon^\top (I_n-H)^2\varepsilon}{\,n-\mathrm{tr}(H)\,}\xrightarrow{p}\sigma^2.
\]

\paragraph{Variance correction -- cross term.} 
By Cauchy–Schwarz,
\[
\frac{\big|\varepsilon^\top B m_S\big|}{\,n-\mathrm{tr}(H)\,}
\;\le\;
\frac{\|(I_n-H)\varepsilon\|_2}{\sqrt{\,n-\mathrm{tr}(H)\,}}
\cdot
\frac{\|(I_n-H)m_S\|_2}{\sqrt{\,n-\mathrm{tr}(H)\,}}.
\]
The first factor equals
\[
\frac{\|(I_n-H)\varepsilon\|_2}{\sqrt{\,n-\mathrm{tr}(H)\,}}
=\sqrt{\frac{\varepsilon^\top (I_n-H)^2\varepsilon}{\,n-\mathrm{tr}(H)\,}},
\]
which converges in probability to $\sigma$ by the noise term analysis and the continuous mapping theorem. The second factor equals
\[
\frac{\|(I_n-H)m_S\|_2}{\sqrt{\,n-\mathrm{tr}(H)\,}}
=\sqrt{\frac{\|(I_n-H)m_S\|_2^2}{\,n-\mathrm{tr}(H)\,}}
\ \xrightarrow{p}\ 0
\]
by \eqref{eq:bias-vanish}. Therefore, by Slutsky’s theorem (product of a sequence
converging to a constant and a sequence converging to $0$),
\[
\frac{\big|\varepsilon^\top B m_S\big|}{\,n-\mathrm{tr}(H)\,}
\ \xrightarrow{p}\ 0,
\]
which proves \eqref{eq:cross-vanish}.

\paragraph{Variance correction -- conclusion.} Combining \eqref{eq:bias-vanish}–\eqref{eq:noise-converge} in \eqref{eq:three-terms} yields
$\widehat\sigma^2_{\mathrm{KRR}}\xrightarrow{p}\sigma^2$, completing the proof.
\end{proof}



\newpage
\clearpage

\section{Algorithms}\label{sec:ALGORITHM}
\subsection*{Algorithm 1: PPI point estimator for general moment equation with a fixed predictor}


\begin{algorithm}[H]
\scriptsize
\setlength{\baselineskip}{9.5pt}
\caption{PPI point estimator for general moment equation with a fixed predictor}
\label{alg:ppi_point_estimation}
\KwIn{Labeled set $S=\{(x_j,y_j)\}_{j=1}^n$, unlabeled covariates $\{x_i\}_{i=1}^N$, predictor $m(\cdot)$, estimating function $U(\theta;x,y)$.}
\KwOut{$\hat\theta_{\mathrm{PPI}}$ such that $\widehat U_{\mathrm{PPI}}(\hat\theta_{\mathrm{PPI}})=0$.}

\For{$j \in S$}{compute $\tilde U_j(\theta)\leftarrow U(\theta;x_j,m(x_j))$\;}
\For{$j \in S$}{compute $U_j(\theta)\leftarrow U(\theta;x_j,y_j)$\;}

\textbf{Rectifier:}\quad
$\displaystyle \Delta_\theta \leftarrow \frac{1}{n}\sum_{j\in S}\!\big[\,U_j(\theta)-\tilde U_j(\theta)\,\big]$\;

\textbf{Model-based fit:}\quad
\For{$i=1$ \KwTo $N$}{compute $\tilde U_i(\theta)\leftarrow U(\theta;x_i,m(x_i))$\;}
$\displaystyle m_\theta \leftarrow \frac{1}{N}\sum_{i=1}^N \tilde U_i(\theta)$\;

\textbf{Bias-corrected equation:}\quad
$\widehat U_{\mathrm{PPI}}(\theta)\leftarrow m_\theta+\Delta_\theta$\;

\textbf{Solve:}\quad
$\hat\theta_{\mathrm{PPI}} \leftarrow \arg\{\theta:\widehat U_{\mathrm{PPI}}(\theta)=0\}$\;  

\Return{$\hat\theta_{\mathrm{PPI}}$}\;
\end{algorithm}

\subsection*{Algorithm 2: 95\% Wald confidence intervals for the PPI estimator for general moment equation with a fixed predictor}

\begin{algorithm}[H]
\scriptsize
\setlength{\baselineskip}{9.5pt}
\caption{95\% Wald confidence intervals for the PPI estimator for general moment equation with a fixed predictor}
\label{alg:ppi_ci}
\KwIn{Labeled set $S=\{(x_j,y_j)\}_{j=1}^n$, unlabeled covariates $\{x_i\}_{i=1}^N$, predictor $m(\cdot)$, estimating function $U(\theta;x,y)$, and a root $\hat\theta_{\mathrm{PPI}}$ solving $\widehat U_{\mathrm{PPI}}(\hat\theta_{\mathrm{PPI}})=0$.}
\KwOut{Componentwise $95\%$ Wald confidence intervals for $\theta_0$.}

\textbf{Jacobian (information) estimate:}\quad
$\displaystyle \widehat I \leftarrow -\,\frac{1}{n}\sum_{j\in S}\partial_\theta U\!\big(\hat\theta_{\mathrm{PPI}};x_j,y_j\big)$\;

\textbf{Score-variance part:}\quad
$\displaystyle \widehat V_1 \leftarrow \widehat I^{-1}\!\Big(\frac{1}{n}\sum_{j\in S} U\!\big(\hat\theta_{\mathrm{PPI}};x_j,y_j\big)^{\otimes 2}\Big)\widehat I^{-1}$\;

\textbf{Rectifier-variance part:}\quad
\For{$j \in S$}{
$\displaystyle \Delta_j(\hat\theta_{\mathrm{PPI}})\leftarrow U\!\big(\hat\theta_{\mathrm{PPI}};x_j,y_j\big)-U\!\big(\hat\theta_{\mathrm{PPI}};x_j,m(x_j)\big)$\;
}
$\displaystyle \widehat V_2 \leftarrow \widehat I^{-1}\!\Big(\frac{1}{n}\sum_{j\in S}\Delta_j(\hat\theta_{\mathrm{PPI}})^{\otimes 2}\Big)\widehat I^{-1}$\;

\textbf{Asymptotic covariance of $\hat\theta_{\mathrm{PPI}}$:}\quad
$\displaystyle \widehat\Sigma_{\hat\theta} \leftarrow N^{-1}\widehat V_1 \;+\; \big(n^{-1}-N^{-1}\big)\widehat V_2$\;

\textbf{95\% CIs (componentwise):}\quad
\For{$j=1$ \KwTo $p$}{
Output $\displaystyle \Big[\ \hat\theta_{\mathrm{PPI},j}\ \pm\ 1.96\,\sqrt{(\widehat\Sigma_{\hat\theta})_{jj}}\ \Big]$\;
}
\Return{$\big\{\,[\hat\theta_{\mathrm{PPI},j}\pm 1.96\sqrt{(\widehat\Sigma_{\hat\theta})_{jj}}]\,\big\}_{j=1}^p$}\;
\end{algorithm}

\subsection*{Algorithm 3: CF-PPI point estimator for general moment equation}

\begin{algorithm}[H]
\scriptsize
\setlength{\baselineskip}{9.5pt}
\caption{CF-PPI point estimator for general moment equation}
\label{alg:cf_ppi_point_estimation}
\KwIn{Labeled set $S=\{(x_j,y_j)\}_{j=1}^n$, unlabeled covariates $\{x_i\}_{i=1}^N$, learning algorithm $\mathcal A$ (maps a labeled subset to a predictor), estimating function $U(\theta;x,y)$, folds $K\ge2$.}
\KwOut{$\hat\theta^{\mathrm{cf}}_{\mathrm{PPI}}$ such that $\widehat U^{\mathrm{cf}}_{\mathrm{PPI}}(\hat\theta^{\mathrm{cf}}_{\mathrm{PPI}})=0$.}

\textbf{Split (cross–fitting):}\quad Partition $S$ into $K$ disjoint folds $S_1,\ldots,S_K$.

\textbf{Train $K$ out–of–fold predictors:}\quad
For each $k\in\{1,\ldots,K\}$, fit $\widehat m^{(k)}\leftarrow \mathcal A(S\setminus S_k)$.

\textbf{Unlabeled model–fit (fold average):}\quad
\[
\widehat m^{\mathrm{cf}}_\theta \;\leftarrow\; \frac{1}{N}\sum_{i=1}^N
U\!\big(\theta; x_i, \widehat m^\star(x_i)\big)
\quad\text{such that }\widehat m^\star(x)=\frac1K\sum_{k=1}^K\widehat m^{(k)}(x)
\]

\textbf{Labeled rectifier (out–of–fold residuals):}\quad
\[
\widehat \Delta^{\mathrm{cf}}_\theta \;\leftarrow\; \frac{1}{n}\sum_{k=1}^K\sum_{i\in S_k}
\Big[\,U(\theta; x_i, y_i)-U\!\big(\theta; x_i, \widehat m^{(k)}(x_i)\big)\,\Big].
\]

\textbf{Cross–fit score:}\quad
\[
\widehat U^{\mathrm{cf}}_{\mathrm{PPI}}(\theta)\;\leftarrow\; \widehat m^{\mathrm{cf}}_\theta+ \widehat \Delta^{\mathrm{cf}}_\theta.
\]

\textbf{Solve:}\quad
$\widehat\theta^{\mathrm{cf}}_{\mathrm{PPI}} \leftarrow
\arg\{\theta:\widehat U^{\mathrm{cf}}_{\mathrm{PPI}}(\theta)=0\}$\; 

\Return{$\widehat\theta^{\mathrm{cf}}_{\mathrm{PPI}}$}\;

\textit{\textbf{$\ast$Note.}} If the estimating equation is linear in its second argument (e.g., mean, GLM scores, etc): \[
\widehat m^{\mathrm{cf}}_\theta \;\leftarrow\; 
\frac{1}{N}\sum_{i=1}^N
U\!\big(\theta; x_i, \widehat m^\star(x_i)\big)=
\frac{1}{KN}\sum_{k=1}^K\sum_{i=1}^N
U\!\big(\theta; x_i, \widehat m^{(k)}(x_i)\big)
\]
\end{algorithm}

\subsection*{Algorithm 4: CF–PPI point estimator for semiparametric mean estimation}

\FloatBarrier
\begin{algorithm}[H]
\scriptsize
\setlength{\baselineskip}{9.5pt}
\caption{CF–PPI point estimator for semiparametric mean estimation}
\label{alg:cf_ppi_mean_point}
\KwIn{Labeled set $S=\{(x_j,y_j)\}_{j=1}^n$, unlabeled covariates $\{x_i\}_{i=1}^N$, learning algorithm $\mathcal A$ (maps a labeled subset to a predictor), folds $K\ge2$.}
\KwOut{$\hat\theta^{\mathrm{cf}}_{\mathrm{PPI}}$.}

\textbf{Split (cross–fitting):}\quad Partition $S$ into $K$ disjoint folds $S_1,\ldots,S_K$.

\textbf{Train $K$ out–of–fold predictors:}\quad
\For{$k\in\{1,\ldots,K\}$}{
\Indp Fit $\widehat m^{(k)} \leftarrow \mathcal A(S\setminus S_k)$; for $j\in S_k$, set $\widehat m^{(-)}(x_j)\leftarrow \widehat m^{(k)}(x_j)$ (out–of–fold).\;
\Indm
}

\textbf{Fold–average predictor (for unlabeled $X$):}\quad
Define $\displaystyle \widehat m^{\star}(x)\leftarrow \frac{1}{K}\sum_{k=1}^K \widehat m^{(k)}(x)$.

\textbf{Unlabeled model–based mean:}\quad
$\displaystyle \widehat m^{\mathrm{cf}} \leftarrow \frac{1}{N}\sum_{i=1}^N \widehat m^{\star}(x_i)$.

\textbf{Rectifier (cross–fit residuals):}\quad
\[
\widehat\Delta^{\mathrm{cf}} \leftarrow \frac{1}{n}\sum_{j\in S}\big\{y_j-\widehat m^{(-)}(x_j)\big\}
\;=\; \frac{1}{n}\sum_{k=1}^K\sum_{j\in S_k}\big\{y_j-\widehat m^{(k)}(x_j)\big\}.
\]

\textbf{Estimator:}\quad
$\displaystyle \widehat\theta^{\mathrm{cf}}_{\mathrm{PPI}} \leftarrow \widehat m^{\mathrm{cf}}+\widehat\Delta^{\mathrm{cf}}$.

\Return{$\widehat\theta^{\mathrm{cf}}_{\mathrm{PPI}}$}\;
\smallskip
\textit{\textbf{$\ast$Note.}} The CF–PPI score is
$\widehat U^{\mathrm{cf}}_{\mathrm{PPI}}(\theta)
= \widehat m^{\mathrm{cf}}_{\theta}+\widehat{\Delta}^{\mathrm{cf}}_{\theta}
= (\widehat m^{\mathrm{cf}}-\theta)+\widehat{\Delta}^{\mathrm{cf}}$,
where $\widehat m^{\mathrm{cf}}:=N^{-1}\sum_{i=1}^N \widehat m^{\star}(x_i)$,
$\widehat m^{\mathrm{cf}}_{\theta}=\widehat m^{\mathrm{cf}}-\theta$, and
$\widehat{\Delta}^{\mathrm{cf}}_{\theta}:=n^{-1}\sum_{j\in S}\{y_j-\widehat m^{(-)}(x_j)\}$. In the mean case the rectifier does not depend on $\theta$ because
$U(\theta;x,y)-U(\theta;x,\widehat m(x))=y-\widehat m(x)$.
\end{algorithm}
\FloatBarrier

\subsection*{Algorithm 5: 95\% Wald confidence interval for the CF-PPI mean estimator for semiparametric mean estimation}

\begin{algorithm}[H]
\scriptsize
\setlength{\baselineskip}{9.5pt}
\caption{95\% Wald confidence interval for $\hat\theta^{\mathrm{cf}}_{\mathrm{PPI}}$ for semiparametric mean estimation}
\label{alg:cf_ppi_mean_ci}
\KwIn{Labeled set $S=\{(x_j,y_j)\}_{j=1}^n$, unlabeled covariates $\{x_i\}_{i=1}^N$, fold--average predictor $\widehat{m}^{\star}(\cdot)$ and out-of-fold (O.O.F.) labeled predictions $\widehat{m}^{(-)}(x_j)$ from Algorithm~\ref{alg:cf_ppi_mean_point}; estimator $\hat\theta^{\mathrm{cf}}_{\mathrm{PPI}}=\widehat{m}^{\mathrm{cf}}+\widehat{\Delta}^{\mathrm{cf}}$ with
$\widehat{m}^{\mathrm{cf}}=N^{-1}\sum_{i=1}^N \widehat{m}^{\star}(x_i)=\big(KN\big)^{-1}\sum_{k=1}^K\sum_{i=1}^N \widehat{m}^{(k)}(x_i)$
and $\widehat{\Delta}^{\mathrm{cf}}=n^{-1}\sum_{j\in S}\{y_j-\widehat{m}^{(-)}(x_j)\}$.}
\KwOut{A 95\% Wald confidence interval for $\theta_0=\mathbb E[Y]$.}

\textbf{Unlabeled variance (model--fit term):}\quad
$$\displaystyle \widehat V_U \;\leftarrow\; \frac{1}{N}\,\widehat{\Var}\!\big(\widehat{m}^{\star}(X)\big)
\;=\; \frac{1}{N}\cdot\frac{1}{N-1}\sum_{i=1}^N \big(\widehat{m}^{\star}(x_i)-\widehat{m}^{\mathrm{cf}}\big)^2.$$

\textbf{Rectifier variance (out-of-fold (O.O.F.) residuals):}\quad
$$\displaystyle \widehat V_S \;\leftarrow\; \frac{1}{n}\,\widehat{\Var}\!\big(Y-\widehat{m}^{(-)}(X)\big)
\;=\; \frac{1}{n}\cdot\frac{1}{n-1}\sum_{j\in S}\Big(\,y_j-\widehat{m}^{(-)}(x_j)-\widehat{\Delta}^{\mathrm{cf}}\,\Big)^2.$$

\textbf{Asymptotic variance:}\quad
$\displaystyle \widehat{\Var}\!\big(\hat\theta^{\mathrm{cf}}_{\mathrm{PPI}}\big)
\;\leftarrow\; \widehat V_U + \widehat V_S.$

\textbf{Standard error:}\quad
$\displaystyle \widehat{\mathrm{se}} \;\leftarrow\; \sqrt{\widehat{\Var}\!\big(\hat\theta^{\mathrm{cf}}_{\mathrm{PPI}}\big)}.$

\textbf{return (95\% Wald CI):}\quad
$\displaystyle \Big[\,\hat\theta^{\mathrm{cf}}_{\mathrm{PPI}} \ \pm\ 1.96\,\widehat{\mathrm{se}}\,\Big]$.

\smallskip
\textit{\textbf{$\ast$Note.}} The variance estimator equals
$\widehat{\Var}(\widehat{m}^{\star}(X))/N + \widehat{\Var}(Y-\widehat{m}^{(-)}(X))/n$.
\end{algorithm}

\subsection*{Algorithm 6: SF--PPI point estimator for semiparametric mean estimation}
\begin{algorithm}[H]
\scriptsize
\setlength{\baselineskip}{9.5pt}
\caption{SF–PPI point estimator for semiparametric mean estimation}
\label{alg:sf_ppi_mean_point}
\KwIn{Labeled set $S=\{(x_j,y_j)\}_{j=1}^n$, unlabeled covariates $\{x_i\}_{i=1}^N$, learning algorithm $\mathcal A$ (maps a labeled set to a predictor).}
\KwOut{$\hat\theta^{\mathrm{sf}}_{\mathrm{PPI}}$.}

\textbf{Fit single predictor:}\quad
$\widehat m \leftarrow \mathcal A(S)$.

\textbf{Unlabeled model--based mean:}\quad
$\displaystyle \widehat m^{\mathrm{sf}} \leftarrow \frac{1}{N}\sum_{i=1}^N \widehat m(x_i)$.

\textbf{Rectifier (in--sample residuals):}\quad
$\displaystyle \widehat\Delta^{\mathrm{sf}} \leftarrow \frac{1}{n}\sum_{j\in S}\!\big\{\,y_j-\widehat m(x_j)\,\big\}$.

\textbf{Estimator (closed form):}\quad
$\displaystyle \hat\theta^{\mathrm{sf}}_{\mathrm{PPI}} \leftarrow \widehat m^{\mathrm{sf}}+\widehat\Delta^{\mathrm{sf}}$.

\Return{$\hat\theta^{\mathrm{sf}}_{\mathrm{PPI}}$}.

\smallskip
\textit{\textbf{$\ast$Note.}} The SF--PPI point estimator is the vanilla PPI plug--in with a single fitted predictor: the model is trained once on the labeled set $S$ and reused both to impute $\widehat m(x_i)$ for the unlabeled covariates and to compute the in–sample labeled residuals $y_j-\widehat m(x_j)$. No cross–fitting is performed; the “VC’’ in SF-PPI-VC only affects variance/CI, not the point estimate.

\smallskip
\textit{\textbf{$\ast$Note.}} For the learning algorithm, we adopt kernel ridge regression (Gaussian kernel) with an unpenalized intercept; any supervised learner may be substituted.
\end{algorithm}

\subsection*{Algorithm 7: 95\% Wald confidence interval for the SF--PPI estimator with variance correction}

\begin{algorithm}[H]
\tiny
\setlength{\baselineskip}{6.0pt}
\caption{95\% Wald confidence interval for $\widehat\theta^{\mathrm{sf}}_{\mathrm{PPI}}$ with variance correction (KRR + unpenalized intercept)}
\label{alg:sf_ppi_vc_ci}
\KwIn{Labeled set $S=\{(x_j,y_j)\}_{j=1}^n$, unlabeled covariates $\{x_i\}_{i=1}^N$, Gaussian kernel $k_\ell(x,x')=\exp\!\big(-(x-x')^2/(2\ell^2)\big)$, ridge penalty $\lambda>0$.}
\KwOut{A 95\% Wald confidence interval for $\theta_0=\mathbb E[Y]$.}

\textbf{KRR with unpenalized intercept.}\\
Compute the Gram matrix $K\in\mathbb R^{n\times n}$, $K_{jj'}=k_\ell(x_j,x_{j'})$; set
$A\leftarrow K+n\lambda I_n$, $\mathbf{1}_n=(1,\ldots,1)^\top$, and
\[
w^\top \;\leftarrow\; \frac{\mathbf{1}_n^\top A^{-1}}{\mathbf{1}_n^\top A^{-1}\mathbf{1}_n}\,.
\]

\textbf{Hat matrix on labeled points ($n\times n$):}\\
\[
H \;\leftarrow\; K A^{-1}\big(I_n-\mathbf{1}_n w^\top\big)+\mathbf{1}_n w^\top \in\mathbb R^{n\times n}.
\]

\textbf{Unlabeled weights ($s(x_i)\in\mathbb R^{n}$):}\\
\[
s(x_i)^\top \;\leftarrow\; k(x_i,X_S)^\top A^{-1}\big(I_n-\mathbf{1}_n w^\top\big)+\mathbf{1}_n w^\top .
\]

\textbf{Weight matrix ($N\times n$):}\\
\[
S_U \;=\; \begin{bmatrix}
s(x_1)^\top\\[-2pt]\vdots\\[-2pt]s(x_N)^\top
\end{bmatrix} \in \mathbb R^{N\times n}.
\]

\textbf{Predictions:}\\
\[
\widehat m_S \leftarrow H Y_S\in\mathbb R^n,\qquad
\widehat m_U \leftarrow S_U Y_S\in\mathbb R^N.
\]

\textbf{SF--PPI point estimate:}\\
\[
\widehat m^{\mathrm{sf}} \leftarrow \frac1N\sum_{i=1}^N \widehat m_U[i],\qquad
\widehat\Delta^{\mathrm{sf}} \leftarrow \frac1n\sum_{j\in S}\!\big(y_j-\widehat m_S[j]\big),\qquad
\widehat\theta^{\mathrm{sf}}_{\mathrm{PPI}} \leftarrow \widehat m^{\mathrm{sf}}+\widehat\Delta^{\mathrm{sf}}.
\]

\textbf{Residual variance (df–corrected):}\\
\[
\widehat\sigma^2_{\mathrm{KRR}} \;\leftarrow\; \frac{\|Y_S-\widehat m_S\|_2^2}{\,n-\mathrm{tr}(H)\,}.
\]

\textbf{Unlabeled signal variance:}\\
\[
s_U^2 \leftarrow \frac1N\sum_{i=1}^N \|s(x_i)\|_2^2
= \frac{1}{N}\|S_U\|_F^2,\qquad
\widehat{\Var}(\widehat m_U) \leftarrow \frac{1}{N-1}\sum_{i=1}^N\big(\widehat m_U[i]-\widehat m^{\mathrm{sf}}\big)^2,
\]
\[
\widehat V_{\mathrm{signal}} \;\leftarrow\; \frac{\widehat{\Var}(\widehat m_U)-\widehat\sigma^2_{\mathrm{KRR}}\, s_U^2}{N}.
\]

\textbf{Labeled noise (variance--correction) term:}\\
\[
a \leftarrow \frac1N S_U^\top \mathbf{1}_N\in\mathbb R^n,\qquad
b \leftarrow \frac1n\big(I_n-H^\top\big)\mathbf{1}_n\in\mathbb R^n,\qquad
c \leftarrow a+b,\qquad
\widehat V_{\mathrm{VC}} \leftarrow \widehat\sigma^2_{\mathrm{KRR}}\,\|c\|_2^2.
\]

\textbf{Asymptotic variance and SE:}\;
\[
\widehat{\Var}\!\big(\widehat\theta^{\mathrm{sf}}_{\mathrm{PPI}}\big) \leftarrow \widehat V_{\mathrm{signal}} + \widehat V_{\mathrm{VC}},\qquad
\widehat{\mathrm{se}} \leftarrow \sqrt{\widehat{\Var}\!\big(\widehat\theta^{\mathrm{sf}}_{\mathrm{PPI}}\big)}.
\]

\textbf{Return (95\% Wald CI):} $\quad \Big[\,\widehat\theta^{\mathrm{sf}}_{\mathrm{PPI}} \ \pm\ 1.96\,\widehat{\mathrm{se}}\,\Big].$

\textit{\textbf{$\ast$Note.}} With an unpenalized intercept, $H\mathbf{1}_n=\mathbf{1}_n$, hence $b=\mathbf 0$ and $c=a$.
\end{algorithm}
\end{appendix}


\begin{funding}
No funding received for this research.
\end{funding}

\begin{supplement}
\stitle{Program Codes}
\sdescription{The R code implementing the simulation experiments and real-data application, as well as the Energy Efficiency dataset, is available in the Supplemental Content.}
\end{supplement}

\bibliographystyle{imsart-number}
\bibliography{ref1}
\end{document}